\documentclass[11pt]{amsart}
\usepackage{amsthm,amsmath,amssymb}

\usepackage{geometry}
\usepackage{graphicx, cite}
\usepackage{color}

\numberwithin{equation}{section}
\numberwithin{figure}{section}
\theoremstyle{plain}
\newtheorem{thm}{Theorem}[section]
\newtheorem{lem}[thm]{Lemma}
\newtheorem{cor}[thm]{Corollary}

\newtheorem{con}[thm]{Conjecture}

\theoremstyle{remark}
\newtheorem{rmk}[thm]{Remark}

\newcommand{\M}{\operatorname{M}}
\newcommand{\Hf}{\operatorname{H}}

\newcommand{\PP}{\operatorname{PP}}
\begin{document}

\title{Tiling Enumeration of Hexagons with Off-central Holes}

\author{Tri Lai}
\address{Department of Mathematics, University of Nebraska -- Lincoln, Lincoln, NE 68588}
\email{tlai3@unl.edu}
\thanks{This research was supported in part  by Simons Foundation Collaboration Grant (\# 585923).}



\date{\today}

\dedicatory{}

\begin{abstract}
In the prequel of the paper (\texttt{arXiv:1803.02792}), we considered exact enumerations of the cored versions of a doubly-intruded hexagon. The result generalized Ciucu's work about $F$-cored hexagons (\emph{Adv. Math. 2017}). In this paper, we provide an extensive list of 30 tiling enumerations of hexagons with three collinear chains of triangular holes with alternating orientations. Besides two chains of holes attaching to the boundary of the hexagon, we remove one more chain of triangles that is slightly off the center of the hexagon. Two of our enumerations imply two conjectures posed by Ciucu, Eisenk\"{o}lbl, Krattenthaler, and Zare (\emph{J. Combin. Theory Ser. A 2001}) as two very special cases.
\end{abstract}

\maketitle

\section{Introduction}
Motivated by MacMahon's elegant product formula for lozenge tilings a hexagon \cite{Mac}, James Propp posed a problem asking for a tiling formula for an almost regular hexagon with the central unit triangle removed (see Problem 2 in \cite{Propp}). The problem has been solved and generalized by a number of authors (see e.g. \cite{Ciu2, HG, OK,CEKZ,CK,Ciu1,HoleDent}). A milestone on this line of work is when Ciucu, Eisenk\"{o}lbl, Krattenthaler and Zare \cite{CEKZ} enumerated the \emph{cored hexagon} (or \emph{punctured hexagon}) $C_{x,y,z}(m)$ that are obtained by removing the central equilateral triangle of side-length $m$ from the hexagon $H$ of side-lengths $x, y+m, z,x+m,y,z+m$. We define this region in details in the next paragraph.

\begin{figure}\centering
\includegraphics[width=12cm]{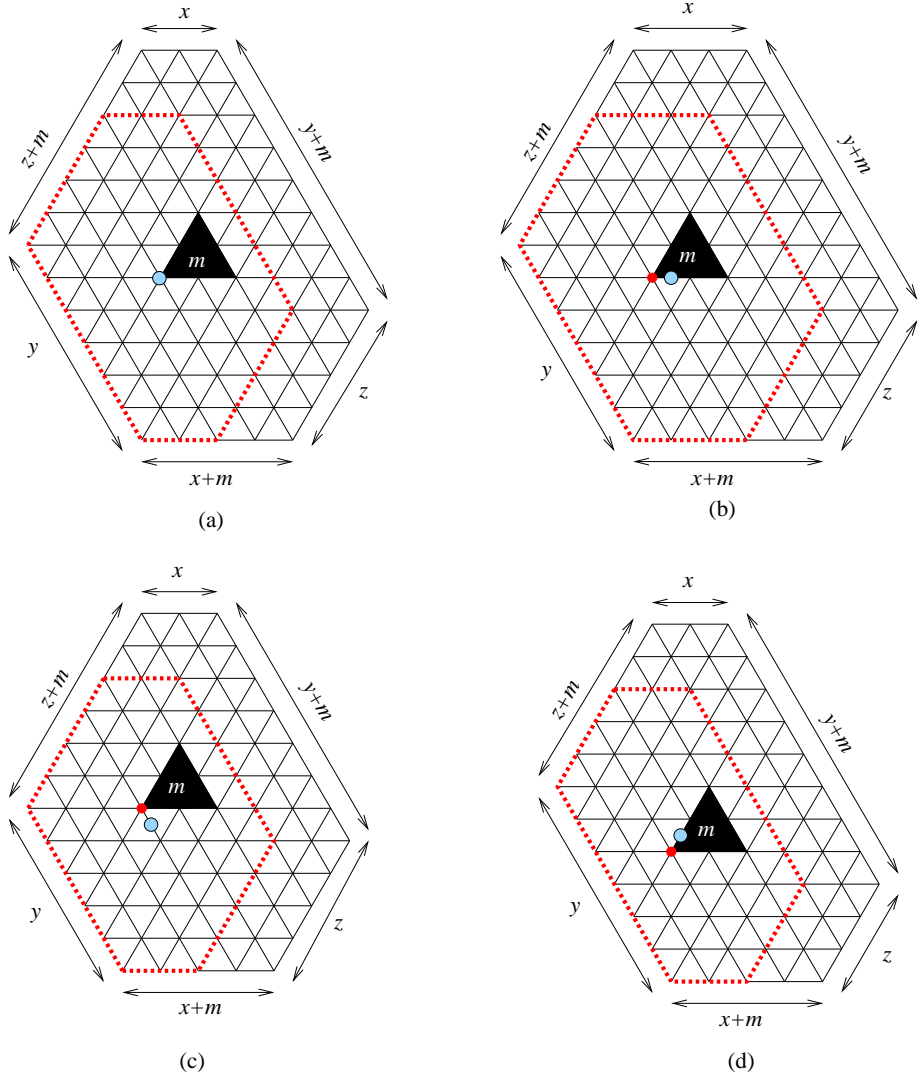}
\caption{(a) The cored hexagon $C_{2,6,4}(2)$. (b) The cored hexagon $C_{3,6,4}(2)$. (c) The cored hexagon $C_{2,5,4}(2)$. (d) The cored hexagon $C_{2,6,3}(2)$. }\label{fig:core}
\end{figure}

We start with an \emph{auxiliary hexagon} $H_0$ with side-lengths\footnote{From now on, we always list the side-lengths of a hexagon in the clockwise order from the north side.} $x,y,z,x,y,z$ (see the hexagons with the dashed boundary in Figure \ref{fig:core}). Next, we push the north, the northeast, and the southeast sides of the hexagon $m$ units outward, and keep other sides staying put. This way, we get a larger hexagon $H$, called the \emph{base hexagon}, of side-lengths $x, y+m, z,x+m,y,z+m$.  Finally, we remove an up-pointing $m$-triangle such that its left vertex is located at the closest lattice point to the center of the auxiliary hexagon $H_0$. Precisely, there are four cases to distinguish based on the parities of $x,y,z$.  When $x$, $y$ and $z$ have the same parity, the center of the hexagon is a lattice vertex and our removed triangle has the left vertex at the center. One readily sees that, in this case, the triangular hole stays evenly between each pair of parallel sides of the hexagon $H$. In particular, the distance between the north side of the hexagon and the top of the triangular hole and the distance between the base of the triangular hole and the south side of the hexagon are both $\frac{y+z}{2}$; the distances corresponding to the northeast and southwest sides of the hexagon are both $\frac{x+z}{2}$; the distances corresponding to the northwest and southeast sides of the hexagon are both $\frac{x+y}{2}$ (see Figure \ref{fig:core}(a); the hexagon with dashed boundary indicates the auxiliary hexagon).  Next, we consider the case when $x$ has parity different from that of $y$ and $z$. In this case, the center of the auxiliary hexagon $H_0$ is \emph{not} a lattice vertex anymore. It is the middle point of a horizontal unit lattice interval. We now place the triangular hole such that its leftmost is $1/2$ unit to the left of the center of the auxiliary hexagon (illustrated in Figure \ref{fig:core}(b); the larger shaded dot indicates the center of the auxiliary hexagon). Similarly, if $y$ has the opposite parity  to $x$ and $z$, we place the triangular hole $1/2$ unit to the northwest of the center of the auxiliary hexagon $H_0$ (shown in Figure \ref{fig:core}(c)). Finally, if $z$ has parity different from that of $x$ and $y$, the hole is located $1/2$ unit to the southwest of the center of $H_0$ (see Figure \ref{fig:core}(d)).

The work of Ciucu, Eisenk\"{o}lbl, Krattenthaler and Zare has been generalized further by Ciucu and Krattenthaler when the central triangular hole was extended to a cluster of  four triangles, called a `\emph{shamrock}' in \cite{CK},  and later to a chain of alternating triangles, called `\emph{fern}' in \cite{Ciu1} (see e.g. \cite{CL, Halfhex1,Halfhex2,Halfhex3} for more recent work about the fern structure). Recently, in the prequel of the paper \cite{HoleDent}, the author extended even more the work of Ciucu to a triple of removed ferns, besides the central fern removed, we cut off two more ferns from the boundary of the hexagon.

We note that in the above results, the hole must be located in the center of the region. In the general case, if the hole is moved away from the center, the tiling formula seems not to be simple anymore. However, Ciucu, Krattenthaler, Eisenk\"{o}lbl and Zare \cite{CEKZ} observed that when the triangular hole in the cored hexagon $C_{x,y,z}(m)$ is moved $1$ unit away, the number of tilings  seems to accept only several minor changes. In particular, they conjectured two explicit tiling formulas for the case when $x,y,z$ have the same parity and the triangular hole is 1-unit to the left of the center (see Figure \ref{fig:originoff}(a)), and the case  when $x$ has parity different from that of $y$ and $z$ and the triangular hole is $3/2$-unit to the left of the center (see Figure \ref{fig:originoff}(b)).

Next, we definition of the hyperfactorial function as follows:
\begin{equation}\label{hyper2}
\Hf(n)=\begin{cases}
\prod_{k=0}^{n-1}\Gamma(k+1) & \text{for $n$ a positive integer;}\\
\prod_{k=0}^{n-\frac{1}{2}}\Gamma(k+\frac{1}{2}) & \text{for $n$ a positive half-integer.}
\end{cases}
\end{equation}
where $\Gamma$ denotes the classical gamma function. Recall that $\Gamma(n+1)=n!$ and $\Gamma(n+\frac{1}{2})=\frac{(2n)!}{4^nn!}\sqrt{\pi}$, for a nonnegative integer $n$.

\begin{con}[Ciucu--Eisenk\"{o}lbl--Krattenthaler--Zare; Conjecture 1 in \cite{CEKZ}]\label{con1}
Let $x,y,z,m$ be nonnegative integers, $x,y,z$ having the same parity. The number of the lozenge tilings of a hexagon with side-lengths $x,y+m,z,x+m,y,z+m$, with an equilateral triangle of side-length $m$ removed at $1$ unit to the left of the center, equals
\begin{align}\label{phieq}
\Phi_{x,y,z}(m):=&\frac{1}{4} P_1(x,y,z,m)\frac{\Hf(m+x)\Hf(m+y)\Hf(m+z)\Hf(m+x+y+z)}{\Hf(m+x+y)\Hf(m+y+z)\Hf(m+z+x)}\notag\\
&\times \frac{\Hf(m+\left\lfloor\frac{x+y+z}{2}\right\rfloor )\Hf(m+\left\lceil\frac{x+y+z}{2}\right\rceil)}{\Hf(m+\frac{x+y}{2}+1)\Hf(m+\frac{y+z}{2})\Hf(m+\frac{z+x}{2}-1)}\notag\\
&\times \frac{\Hf(\frac{m}{2})^2\Hf(\left\lfloor\frac{x}{2}\right\rfloor)\Hf(\left\lceil\frac{x}{2}\right\rceil)\Hf(\left\lfloor\frac{y}{2}\right\rfloor)\Hf(\left\lceil\frac{y}{2}\right\rceil)\Hf(\left\lfloor\frac{z}{2}\right\rfloor)\Hf(\left\lceil\frac{z}{2}\right\rceil)}{\Hf(\frac{m}{2}+\left\lfloor\frac{x}{2}\right\rfloor)\Hf(\frac{m}{2}+\left\lceil\frac{x}{2}\right\rceil)\Hf(\frac{m}{2}+\left\lfloor\frac{y}{2}\right\rfloor)\Hf(\frac{m}{2}+\left\lceil\frac{y}{2}\right\rceil)\Hf(\frac{m}{2}+\left\lfloor\frac{z}{2}\right\rfloor)\Hf(\frac{m}{2}+\left\lceil\frac{z}{2}\right\rceil)}\notag\\
&\times \frac{\Hf(\frac{m}{2}+\left\lfloor\frac{x+y}{2}\right\rfloor)\Hf(\frac{m}{2}+\left\lceil\frac{x+y}{2}\right\rceil)\Hf(\frac{m}{2}+\frac{y+z}{2})^2\Hf(\frac{m}{2}+\left\lfloor\frac{z+x}{2}\right\rfloor)
\Hf(\frac{m}{2}+\left\lceil\frac{z+x}{2}\right\rceil)}{\Hf(\frac{m}{2}+\left\lfloor\frac{x+y+z}{2}\right\rfloor)\Hf(\frac{m}{2}+\left\lceil\frac{x+y+z}{2}\right\rceil)
\Hf(\frac{x+y}{2}-1)\Hf(\frac{y+z}{2})\Hf(\frac{z+x}{2}+1)}, \tag{T-1}
\end{align}
where $P_1(x,y,z,m)$ is the polynomial given by
\begin{equation}
P_1(x,y,z,m)=
\begin{cases}
(x+y)(x+z)+2xm & \text{ if $x$ is even,}\\
(x+y)(x+z)+2(x+y+z+m)m &\text{ if $x$ is odd.}
\end{cases}
\end{equation}
\end{con}

\begin{figure}\centering
\includegraphics[width=12cm]{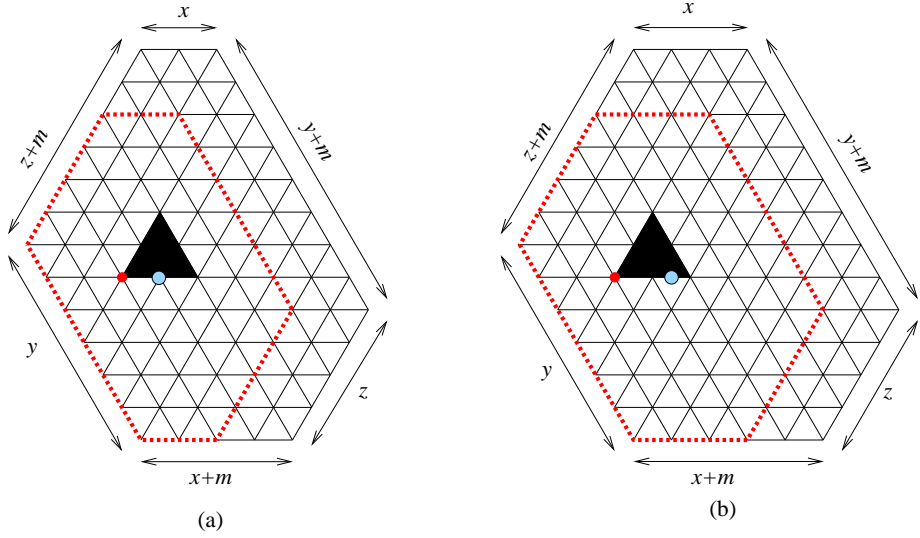}
\caption{(a) Removal of the triangle which is off-center by one unit in the case $x,y,z$ have the same parity. (b) Removal of the triangle which is off-center by $3/2$ units in the case when $x$ has parity opposite to $y$ and $z$.}\label{fig:originoff}
\end{figure}

\begin{con}[Ciucu--Eisenk\"{o}lbl--Krattenthaler--Zare; Conjecture 2 in \cite{CEKZ}]\label{con2}
Let $x,y,z,m$ be nonnegative integers, $x$ is of parity different from the parity of $y$ and $z$. The number of the lozenge tilings of a hexagon of side-lengths $x,y+m,z,x+m,y,z+m$, with an equilateral triangle of side-length $m$ removed at $3/2$ unit to the left of the center, equals
\begin{align}\label{psieq1}
\Psi_{x,z,y}(m):=&\frac{1}{16}P_2(x,y,z,m)\frac{\Hf(m+x)\Hf(m+y)\Hf(m+z)\Hf(m+x+y+z)}{\Hf(m+x+y)\Hf(m+y+z)\Hf(m+z+x)}\notag\\
&\times \frac{\Hf(\frac{m}{2})^2\Hf(\left\lfloor\frac{x}{2}\right\rfloor)\Hf(\left\lceil\frac{x}{2}\right\rceil)\Hf(\left\lfloor\frac{y}{2}\right\rfloor)\Hf(\left\lceil\frac{y}{2}\right\rceil)\Hf(\left\lfloor\frac{z}{2}\right\rfloor)\Hf(\left\lceil\frac{z}{2}\right\rceil)}{\Hf(\frac{m}{2}+\left\lfloor\frac{x}{2}\right\rfloor)\Hf(\frac{m}{2}+\left\lceil\frac{x}{2}\right\rceil)\Hf(\frac{m}{2}+\left\lfloor\frac{y}{2}\right\rfloor)\Hf(\frac{m}{2}+\left\lceil\frac{y}{2}\right\rceil)\Hf(\frac{m}{2}+\left\lfloor\frac{z}{2}\right\rfloor)\Hf(\frac{m}{2}+\left\lceil\frac{z}{2}\right\rceil)}\notag\\
&\times \frac{\Hf(\frac{m}{2}+\left\lfloor\frac{x+y}{2}\right\rfloor)\Hf(\frac{m}{2}+\left\lceil\frac{x+y}{2}\right\rceil)\Hf(\frac{m}{2}+\frac{y+z}{2})^2\Hf(\frac{m}{2}+\left\lfloor\frac{z+x}{2}\right\rfloor)
\Hf(\frac{m}{2}+\left\lceil\frac{z+x}{2}\right\rceil)}{\Hf(\frac{m}{2}+\left\lfloor\frac{x+y+z}{2}\right\rfloor)\Hf(\frac{m}{2}+\left\lceil\frac{x+y+z}{2}\right\rceil)
\Hf(\left\lfloor\frac{x+y}{2}\right\rfloor-1)\Hf(\frac{y+z}{2})\Hf(\left\lceil\frac{z+x}{2}\right\rceil+1)}\notag\\
&\times \frac{\Hf(m+\left\lfloor\frac{x+y+z}{2}\right\rfloor )\Hf(m+\left\lceil\frac{x+y+z}{2}\right\rceil)}{\Hf(m+\left\lceil\frac{x+y}{2}\right\rceil+1)\Hf(m+\frac{y+z}{2})\Hf(m+\left\lfloor\frac{z+x}{2}\right\rfloor-1)}, \tag{T-2}
\end{align}
where $P_2(x,y,z,m)$ is the polynomial given by
\begin{equation}
P_2(x,y,z,m)=
\begin{cases}
((x+y)^2-1)((x+z)^2-1)+4xm(x^2+2xy+y^2+2xz+3yz+z^2\\
\quad +2xm+3ym+3zm+2m^2-1) & \text{ if $x$ is even,}\\
((x+y)^2-1)((x+z)^2-1)+4(x+y+z+m)m(x^2+xy-1) &\text{ if $x$ is odd.}
\end{cases}
\end{equation}
\end{con}

The formulas (\ref{phieq}) and (\ref{psieq1}) almost agree with the corresponding tiling formulas for cored hexagons. The only remarkable differences are the appearances of the quadratic  factors $P_1(x,y,z,m)$ and $P_2(x,y,z,m)$.

It is worth noticing that these conjectures were recently proved by Rosengren \cite{Rosen}, using a different method from that in this paper. More precisely, he used  lattice path combinatorics and Selberg integral to obtain a complicated determinantal formula for a weighted sum of lozenge tilings of a hexagon with a triangular hole at an arbitrary position. Then, by evaluating the determinant for the special cases when the triangular hole is close to the  center, he was able to verify Ciucu--Eisenk\"{o}lbl--Krattenthaler--Zare's Conjectures \ref{con1} and \ref{con2}.

Inspired by the above two conjectures and by the previous work in \cite{HoleDent}, we would like to investigate similar situations for the case of hexagons with three ferns removed. Intuitively,  we remove a fern from the interior of the hexagon so that its leftmost vertex (called the `\emph{root}' of the fern in this paper) is located at a lattice point next to the center of the hexagon. Then we remove two more ferns at the same level such that the new ferns are touching the boundary of the hexagons.  Base on the parity of the side-lengths of our hexagons and the position of the inner fern, there are total \emph{thirty} different regions to enumerate. One may be surprised that the number of situations needed to considered is significantly larger than that in the case of cored hexagons. The reason is that, unlike the cored hexagons, $60^{\circ}$-rotations do \emph{not} preserve the structures of our regions.

The rest of the paper is organized as follows. In Section 2, we define carefully all 30 off-central regions and state their exact tiling enumerations. In Section 3, we quote several fundamental results in enumeration of tilings, especially two versions of Kuo's graphical condensation \cite{Kuo}, that will be the key of our proof. Section 4 is devoted to the detailed proof of our main theorems.

\section{Precise statement of the main result}\label{Statement1}

Recall that in the prequel paper \cite{HoleDent}, we enumerated $8$ different families of hexagons with three ferns removed (one fern is removed from the center and two more ferns are removed from the boundary). These families are divided into two groups based on the relative position of the lattice line $l$, that contains the three ferns and the west and east vertices of the hexagon. In particular, we have four `\emph{$R$-families}', $R^{\odot}$, $R^{\leftarrow}$, $R^{\nwarrow}$, and $R^{\swarrow}$, in the case when $l$ separates the east and west vertices of the hexagon, and  four `\emph{$Q$-families}', $Q^{\odot}$, $Q^{\leftarrow}$, $Q^{\nwarrow}$, and $Q^{\nearrow}$, in the case when $l$ leaves both east and west vertices of the hexagon on a same side (see detailed definition in Section 2 of \cite{HoleDent}).

Each of the above eight families of regions has several off-central counterparts, that will be defined carefully in the next eight subsections.

In order to state our next tilling enumerations, we define five more functions, besides the functions $\Phi$ and $\Psi$ in Conjectures \ref{con1} and \ref{con2}, as follows.

We define the functions $\Theta_{x,y,z}(m)$ and $\Theta'_{x,y,z}(m)$ by:
\begin{align}\label{thetaeq1}
\Theta_{x,y,z}(m):=&Q_{1}(y,z,m) \cdot \frac{\Hf(m+x)\Hf(m+y)\Hf(m+z)\Hf(m+x+y+z)}{\Hf(m+x+y)\Hf(m+y+z)\Hf(m+z+x)}\notag\\
&\times \frac{\Hf(m+\left\lfloor\frac{x+y+z}{2}\right\rfloor )\Hf(m+\left\lceil\frac{x+y+z}{2}\right\rceil)}{\Hf(m+\left\lfloor\frac{x+y}{2}\right\rfloor)\Hf(m+\frac{y+z}{2}+1)\Hf(m+\left\lfloor\frac{z+x}{2}\right\rfloor)}\notag\\
&\times \frac{\Hf(\frac{m}{2})^2\Hf(\left\lfloor\frac{x}{2}\right\rfloor)\Hf(\left\lceil\frac{x}{2}\right\rceil)\Hf(\left\lfloor\frac{y}{2}\right\rfloor)\Hf(\left\lceil\frac{y}{2}\right\rceil)\Hf(\left\lfloor\frac{z}{2}\right\rfloor)\Hf(\left\lceil\frac{z}{2}\right\rceil)}{\Hf(\frac{m}{2}+\left\lfloor\frac{x}{2}\right\rfloor)\Hf(\frac{m}{2}+\left\lceil\frac{x}{2}\right\rceil)\Hf(\frac{m}{2}+\left\lfloor\frac{y}{2}\right\rfloor)\Hf(\frac{m}{2}+\left\lceil\frac{y}{2}\right\rceil)\Hf(\frac{m}{2}+\left\lfloor\frac{z}{2}\right\rfloor)\Hf(\frac{m}{2}+\left\lceil\frac{z}{2}\right\rceil)}\notag\\
&\times \frac{\Hf(\frac{m}{2}+\left\lfloor\frac{x+y}{2}\right\rfloor)\Hf(\frac{m}{2}+\left\lceil\frac{x+y}{2}\right\rceil)
\Hf(\frac{m}{2}+\frac{y+z}{2})^2\Hf(\frac{m}{2}+\left\lfloor\frac{z+x}{2}\right\rfloor)\Hf(\frac{m}{2}+\left\lceil\frac{z+x}{2}\right\rceil)}
{\Hf(\frac{m}{2}+\left\lfloor\frac{x+y+z}{2}\right\rfloor)\Hf(\frac{m}{2}+\left\lceil\frac{x+y+z}{2}\right\rceil)\Hf(\left\lceil\frac{x+y}{2}\right\rceil)\Hf(\frac{y+z}{2})\Hf(\left\lceil\frac{z+x}{2}\right\rceil)} \tag{T-3}
\end{align}
and
\begin{align}\label{thetaeq2}
\Theta'_{x,y,z}(m):=&Q_{1}(y,z,m) \cdot \frac{\Hf(m+x)\Hf(m+y)\Hf(m+z)\Hf(m+x+y+z)}{\Hf(m+x+y)\Hf(m+y+z)\Hf(m+z+x)}\notag\\
&\times \frac{\Hf(m+\left\lfloor\frac{x+y+z}{2}\right\rfloor )\Hf(m+\left\lceil\frac{x+y+z}{2}\right\rceil)}{\Hf(m+\left\lceil\frac{x+y}{2}\right\rceil)\Hf(m+\frac{y+z}{2}-1)\Hf(m+\left\lceil\frac{z+x}{2}\right\rceil)}\notag\\
&\times \frac{\Hf(\frac{m}{2})^2\Hf(\left\lfloor\frac{x}{2}\right\rfloor)\Hf(\left\lceil\frac{x}{2}\right\rceil)\Hf(\left\lfloor\frac{y}{2}\right\rfloor)\Hf(\left\lceil\frac{y}{2}\right\rceil)\Hf(\left\lfloor\frac{z}{2}\right\rfloor)\Hf(\left\lceil\frac{z}{2}\right\rceil)}{\Hf(\frac{m}{2}+\left\lfloor\frac{x}{2}\right\rfloor)\Hf(\frac{m}{2}+\left\lceil\frac{x}{2}\right\rceil)\Hf(\frac{m}{2}+\left\lfloor\frac{y}{2}\right\rfloor)\Hf(\frac{m}{2}+\left\lceil\frac{y}{2}\right\rceil)\Hf(\frac{m}{2}+\left\lfloor\frac{z}{2}\right\rfloor)\Hf(\frac{m}{2}+\left\lceil\frac{z}{2}\right\rceil)}\notag\\
&\times \frac{\Hf(\frac{m}{2}+\left\lfloor\frac{x+y}{2}\right\rfloor)\Hf(\frac{m}{2}+\left\lceil\frac{x+y}{2}\right\rceil)
\Hf(\frac{m}{2}+\frac{y+z}{2})^2\Hf(\frac{m}{2}+\left\lfloor\frac{z+x}{2}\right\rfloor)\Hf(\frac{m}{2}+\left\lceil\frac{z+x}{2}\right\rceil)}
{\Hf(\frac{m}{2}+\left\lfloor\frac{x+y+z}{2}\right\rfloor)\Hf(\frac{m}{2}+\left\lceil\frac{x+y+z}{2}\right\rceil)\Hf(\left\lfloor\frac{x+y}{2}\right\rfloor)\Hf(\frac{y+z}{2}+1)\Hf(\left\lfloor\frac{z+x}{2}\right\rfloor)}, \tag{T'-3}
\end{align}
where $Q_{1}(y,z,m)$ is $m+\frac{y+z}{2}$ if $x$ is even, and is $\frac{y+z}{2}$ if $x$ is odd. One readily sees that the $\Theta$- and $\Theta'$-functions are almost the same, except for certain differences at several hyperfactorial factors.

We next define the function $\Lambda_{x,y,z}(m)$ and its variation $\Lambda'_{x,y,z}(m)$ by:
\begin{align}\label{lambdaeq1}
\Lambda_{x,y,z}(m):=&\frac{1}{8}Q_{2}(x,y,z,m) \cdot \frac{\Hf(m+x)\Hf(m+y)\Hf(m+z)\Hf(m+x+y+z)}{\Hf(m+x+y)\Hf(m+y+z)\Hf(m+z+x)}\notag\\
&\times \frac{\Hf(m+\left\lfloor\frac{x+y+z}{2}\right\rfloor)\Hf(m+\left\lceil\frac{x+y+z}{2}\right\rceil)}
{\Hf(m+\left\lceil\frac{x+y}{2}\right\rceil)\Hf(m+\frac{y+z}{2}+1)\Hf(m+\left\lfloor\frac{z+x}{2}\right\rfloor-1)}\notag\\
&\times \frac{\Hf(\frac{m}{2})^2\Hf(\left\lfloor\frac{x}{2}\right\rfloor)\Hf(\left\lceil\frac{x}{2}\right\rceil)\Hf(\left\lfloor\frac{y}{2}\right\rfloor)\Hf(\left\lceil\frac{y}{2}\right\rceil)\Hf(\left\lfloor\frac{z}{2}\right\rfloor)\Hf(\left\lceil\frac{z}{2}\right\rceil)}{\Hf(\frac{m}{2}+\left\lfloor\frac{x}{2}\right\rfloor)\Hf(\frac{m}{2}+\left\lceil\frac{x}{2}\right\rceil)\Hf(\frac{m}{2}+\left\lfloor\frac{y}{2}\right\rfloor)\Hf(\frac{m}{2}+\left\lceil\frac{y}{2}\right\rceil)\Hf(\frac{m}{2}+\left\lfloor\frac{z}{2}\right\rfloor)\Hf(\frac{m}{2}+\left\lceil\frac{z}{2}\right\rceil)}\notag\\
&\times \frac{\Hf(\frac{m}{2}+\left\lfloor\frac{x+y}{2}\right\rfloor)\Hf(\frac{m}{2}+\left\lceil\frac{x+y}{2}\right\rceil)
\Hf(\frac{m}{2}+\frac{y+z}{2})^2\Hf(\frac{m}{2}+\left\lfloor\frac{z+x}{2}\right\rfloor)\Hf(\frac{m}{2}+\left\lceil\frac{z+x}{2}\right\rceil)}
{\Hf(\frac{m}{2}+\left\lfloor\frac{x+y+z}{2}\right\rfloor)\Hf(\frac{m}{2}+\left\lceil\frac{x+y+z}{2}\right\rceil)\Hf(\left\lfloor\frac{x+y}{2}\right\rfloor)\Hf(\frac{y+z}{2}-1)\Hf(\left\lceil\frac{z+x}{2}\right\rceil+1)} \tag{T-4}
\end{align}
and
\begin{align}\label{lambdaeq2}
\Lambda' _{x,y,z}(m):=&\frac{1}{8}Q_{2}(x,y,z,m) \cdot \frac{\Hf(m+x)\Hf(m+y)\Hf(m+z)\Hf(m+x+y+z)}{\Hf(m+x+y)\Hf(m+y+z)\Hf(m+z+x)}\notag\\
&\times \frac{\Hf(m+\left\lfloor\frac{x+y+z}{2}\right\rfloor)\Hf(m+\left\lceil\frac{x+y+z}{2}\right\rceil)}
{\Hf(m+\left\lfloor\frac{x+y}{2}\right\rfloor)\Hf(m+\frac{y+z}{2}-1)\Hf(m+\left\lceil\frac{z+x}{2}\right\rceil+1)}\notag\\
&\times \frac{\Hf(\frac{m}{2})^2\Hf(\left\lfloor\frac{x}{2}\right\rfloor)\Hf(\left\lceil\frac{x}{2}\right\rceil)\Hf(\left\lfloor\frac{y}{2}\right\rfloor)\Hf(\left\lceil\frac{y}{2}\right\rceil)\Hf(\left\lfloor\frac{z}{2}\right\rfloor)\Hf(\left\lceil\frac{z}{2}\right\rceil)}{\Hf(\frac{m}{2}+\left\lfloor\frac{x}{2}\right\rfloor)\Hf(\frac{m}{2}+\left\lceil\frac{x}{2}\right\rceil)\Hf(\frac{m}{2}+\left\lfloor\frac{y}{2}\right\rfloor)\Hf(\frac{m}{2}+\left\lceil\frac{y}{2}\right\rceil)\Hf(\frac{m}{2}+\left\lfloor\frac{z}{2}\right\rfloor)\Hf(\frac{m}{2}+\left\lceil\frac{z}{2}\right\rceil)}\notag\\
&\times \frac{\Hf(\frac{m}{2}+\left\lfloor\frac{x+y}{2}\right\rfloor)\Hf(\frac{m}{2}+\left\lceil\frac{x+y}{2}\right\rceil)
\Hf(\frac{m}{2}+\frac{y+z}{2})^2\Hf(\frac{m}{2}+\left\lfloor\frac{z+x}{2}\right\rfloor)\Hf(\frac{m}{2}+\left\lceil\frac{z+x}{2}\right\rceil)}
{\Hf(\frac{m}{2}+\left\lfloor\frac{x+y+z}{2}\right\rfloor)\Hf(\frac{m}{2}+\left\lceil\frac{x+y+z}{2}\right\rceil)\Hf(\left\lceil\frac{x+y}{2}\right\rceil)\Hf(\frac{y+z}{2}+1)\Hf(\left\lfloor\frac{z+x}{2}\right\rfloor-1)}, \tag{T'-4}
\end{align}
in which the polynomial $Q_{2}(x,y,z,m)$ is defined as
 \begin{equation}
Q_{2}(x,y,z,m)=
\begin{cases}
(x+z+2m-1)((z+1)(x+z-1)+2x(m-1))\\\ \ \ \ \ \ \ \ \ \ \ \ \ \ \ \ \ +(y+1)((x+z)^2+4xm-1) &\text{ if $x$ is even,}\\
 z((x+z+2m)^2-1)+y((x+z+2m)^2-4m(x+m)-1) &\text{if $x$ is odd.}
  \end{cases}
\end{equation}
Finally, we define a counterpart of the function $\Psi$ in Conjecture \ref{con2}:
\begin{align}\label{psieq2}
\Psi'_{x,y,z}(m):=&\frac{1}{16}\frac{\Hf(m+x)\Hf(m+y)\Hf(m+z)\Hf(m+x+y+z)}{\Hf(m+x+y)\Hf(m+y+z)\Hf(m+z+x)}\notag\\
&\times \frac{\Hf(\frac{m}{2})^2\Hf(\left\lfloor\frac{x}{2}\right\rfloor)\Hf(\left\lceil\frac{x}{2}\right\rceil)\Hf(\left\lfloor\frac{y}{2}\right\rfloor)\Hf(\left\lceil\frac{y}{2}\right\rceil)\Hf(\left\lfloor\frac{z}{2}\right\rfloor)\Hf(\left\lceil\frac{z}{2}\right\rceil)}{\Hf(\frac{m}{2}+\left\lfloor\frac{x}{2}\right\rfloor)\Hf(\frac{m}{2}+\left\lceil\frac{x}{2}\right\rceil)\Hf(\frac{m}{2}+\left\lfloor\frac{y}{2}\right\rfloor)\Hf(\frac{m}{2}+\left\lceil\frac{y}{2}\right\rceil)\Hf(\frac{m}{2}+\left\lfloor\frac{z}{2}\right\rfloor)\Hf(\frac{m}{2}+\left\lceil\frac{z}{2}\right\rceil)}\notag\\
&\times \frac{\Hf(\frac{m}{2}+\left\lfloor\frac{x+y}{2}\right\rfloor)\Hf(\frac{m}{2}+\left\lceil\frac{x+y}{2}\right\rceil)\Hf(\frac{m}{2}+\frac{y+z}{2})^2\Hf(\frac{m}{2}+\left\lfloor\frac{z+x}{2}\right\rfloor)
\Hf(\frac{m}{2}+\left\lceil\frac{z+x}{2}\right\rceil)}{\Hf(\frac{m}{2}+\left\lfloor\frac{x+y+z}{2}\right\rfloor)\Hf(\frac{m}{2}+\left\lceil\frac{x+y+z}{2}\right\rceil)
\Hf(\left\lceil\frac{x+y}{2}\right\rceil+1)\Hf(\frac{y+z}{2})\Hf(\left\lfloor\frac{z+x}{2}\right\rfloor -1)}\notag\\
&\times \frac{\Hf(m+\left\lfloor\frac{x+y+z}{2}\right\rfloor )\Hf(m+\left\lceil\frac{x+y+z}{2}\right\rceil)}{\Hf(m+\left\lfloor\frac{x+y}{2}\right\rfloor-1)\Hf(m+\frac{y+z}{2})\Hf(m+\left\lceil\frac{z+x}{2}\right\rceil +1)} P_2(x,y,z,m), \tag{T'-2}
\end{align}
where $P_2(x,y,z,m)$ is the polynomial defined in the Conjecture \ref{con2}.

We also need the following well-known formula of Cohn, Larsen, and Propp for tiling number of dented semihexagon in the statements of our main theorems.

For a sequence $\textbf{a}:=(a_i)_{i=1}^{m}$, we denote $o_a:=\sum_{\text{$i$ odd}} a_i$ and $e_a:=\sum_{\text{$i$ even}}a_i$.  Let $S(a_1,a_2,\dotsc,a_m)$ denote the upper half of a hexagon of side-lengths $e_a,o_a,o_a,e_a,o_a,o_a$ in which $k:=\lfloor\frac{m}{2}\rfloor$ triangles of side-lengths $a_1,a_3,a_5,\dots,a_{2k+1}$ removed from the base, such that the distance between the $i$-th and the $(i+1)$-th removed triangles is $a_{2i}$ (see Figure \ref{fig:semihex} for an example).  We call  the region $S(a_1,a_2,\dotsc,a_m)$ a \emph{dented semihexagon}. Cohn, Larsen and Propp \cite{CLP} interpreted semi-strict Gelfand--Tsetlin patterns as lozenge tilings of  the dented semihexagon $S(a_1,a_2,\dotsc,a_m)$, and obtained the following tiling formula
\begin{align}\label{semieq}
s(a_1,a_2,\dots,a_{2l-1})&=s(a_1,a_2,\dots,a_{2l})\\
&=\dfrac{1}{\Hf(a_1+a_{3}+a_{5}+\dotsc+a_{2l-1})}\dfrac{\prod_{\substack{1\leq i<j\leq 2l-1\\
                  \text{$j-i$ odd}}}\Hf(a_i+a_{i+1}+\dotsc+a_{j})}{\prod_{\substack{1\leq i<j\leq 2l-1\\
                  \text{$j-i$ even}}}\Hf(a_i+a_{i+1}+\dotsc+a_{j})},
\end{align}
where $s(a_1,a_2,\dotsc,a_m)$ denotes the number of tilings
 of $S(a_1,a_2,\dotsc,a_m)$.

\begin{figure}\centering
\setlength{\unitlength}{3947sp}%
\begingroup\makeatletter\ifx\SetFigFont\undefined%
\gdef\SetFigFont#1#2#3#4#5{%
  \reset@font\fontsize{#1}{#2pt}%
  \fontfamily{#3}\fontseries{#4}\fontshape{#5}%
  \selectfont}%
\fi\endgroup%
\resizebox{8cm}{!}{
\begin{picture}(0,0)%
\includegraphics{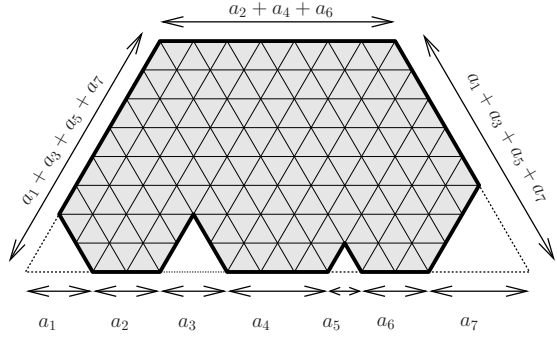}%
\end{picture}%
%
%

\begin{picture}(6994,4016)(1923,-5559)
\put(4502,-1877){\makebox(0,0)[lb]{\smash{{\SetFigFont{14}{16.8}{\rmdefault}{\mddefault}{\itdefault}{$a_2+a_4+a_6$}%
}}}}
\put(2178,-4122){\rotatebox{60.0}{\makebox(0,0)[lb]{\smash{{\SetFigFont{14}{16.8}{\rmdefault}{\mddefault}{\itdefault}{$a_1+a_3+a_5+a_7$}%
}}}}}
\put(2296,-5544){\makebox(0,0)[lb]{\smash{{\SetFigFont{14}{16.8}{\rmdefault}{\mddefault}{\itdefault}{$a_1$}%
}}}}
\put(3136,-5544){\makebox(0,0)[lb]{\smash{{\SetFigFont{14}{16.8}{\rmdefault}{\mddefault}{\itdefault}{$a_2$}%
}}}}
\put(3916,-5544){\makebox(0,0)[lb]{\smash{{\SetFigFont{14}{16.8}{\rmdefault}{\mddefault}{\itdefault}{$a_3$}%
}}}}
\put(4771,-5544){\makebox(0,0)[lb]{\smash{{\SetFigFont{14}{16.8}{\rmdefault}{\mddefault}{\itdefault}{$a_4$}%
}}}}
\put(5589,-5544){\makebox(0,0)[lb]{\smash{{\SetFigFont{14}{16.8}{\rmdefault}{\mddefault}{\itdefault}{$a_5$}%
}}}}
\put(6219,-5529){\makebox(0,0)[lb]{\smash{{\SetFigFont{14}{16.8}{\rmdefault}{\mddefault}{\itdefault}{$a_6$}%
}}}}
\put(7186,-5536){\makebox(0,0)[lb]{\smash{{\SetFigFont{14}{16.8}{\rmdefault}{\mddefault}{\itdefault}{$a_7$}%
}}}}
\put(7263,-2603){\rotatebox{300.0}{\makebox(0,0)[lb]{\smash{{\SetFigFont{14}{16.8}{\rmdefault}{\mddefault}{\itdefault}{$a_1+a_3+a_5+a_7$}%
}}}}}
\end{picture}}
\caption{The dented semihexagon $S(2,2,2,3,1,2,4)$.}
\label{fig:semihex}
\end{figure}

\subsection{The $E^{(i)}$-type regions}
We will define our region similar to the cored hexagons. In particular, we start from an auxiliary hexagon $H_0$, and push our its sides in a certain way to obtain a larger hexagon, called the base hexagon. Then we will remove three collinear ferns from the base hexagon in a particular way to obtain our regions.

Let us investigate first the case when  the auxiliary hexagon $H_0$ of side-lengths $x,z,z,x,z,z$, for $x$ and $z$ of the same parity, and the middle fern is removed  at $1$ unit from the center of the auxiliary hexagon. In this case, the center of $H_0$ is a lattice point, and there are six different lattice points, which have distance $1$ from the center (shown in Figure \ref{fig:offposition}(a); the darker shaded nodes indicates the center of the auxiliary hexagon $H_0$, and the lighter shaded nodes represent the possible off-central positions for the leftmost of the middle fern). We note that, in the case of cored hexagons, we only need to consider the regions corresponding to one of these six positions, since regions corresponding to other positions can be obtained from the latter region by $60^{\circ}$ rotations. However, this is \emph{not} true anymore for our hexagons with three ferns removed.

\begin{figure}\centering
\includegraphics[width=10cm]{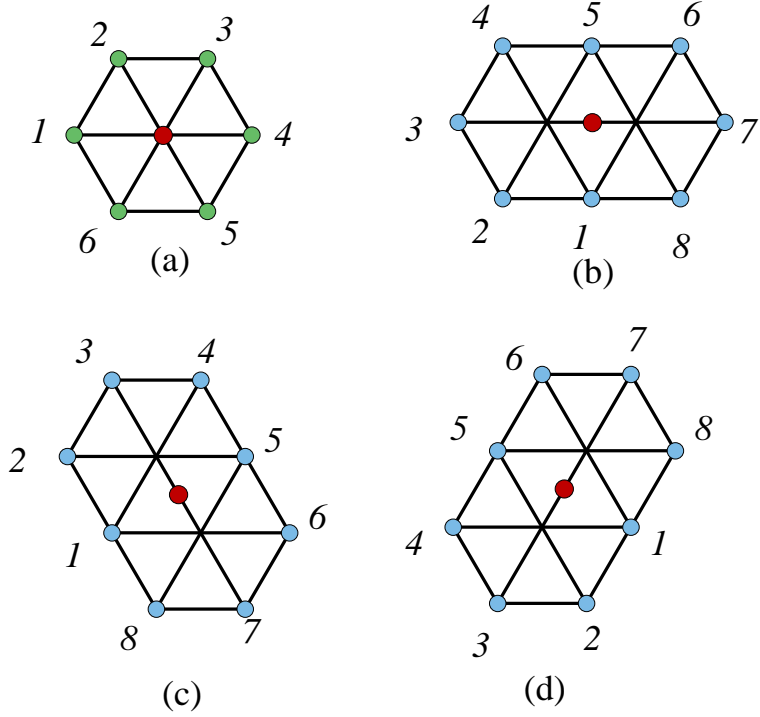}
\caption{(a) Six off-central positions for $E^{(i)}$- and $\overline{E}^{(i)}$-type regions. (b) Eight off-central positions  for $F^{(i)}$- and $\overline{F}^{(i)}$-type regions. (c) Eight off-central positions for $G^{(i)}$- and $\overline{G}^{(i)}$-type regions. (d) Eight off-central positions for $K^{(i)}$- an $\overline{K}^{(i)}$-type regions.}\label{fig:offposition}
\end{figure}

Assume that $x,z$ are nonnegative integers with the same parity, that $\textbf{a}=(a_1,a_2,\dotsc,a_m)$, $\textbf{b}:=(b_1,b_2,\dotsc,b_n)$, $\textbf{c}=(c_1,c_2,\dotsc,c_k)$ are three (may be empty) sequences  of nonnegative integers, and that $y$ is an integer that may take negative values in certain condition. The three sequences $\textbf{a}, \textbf{b}, \textbf{c}$ determine the side-lengths of triangles in the left, the right, and the central ferns, respectively.  Set
\begin{align}
e_a:=\sum_{i\ even}a_i, &\quad o_a:=\sum_{i \ odd} a_i,\notag\\
e_b:=\sum_{j\ even}b_j, &\quad o_b:=\sum_{j \ odd} b_j,\notag\\
e_c:=\sum_{t\ even}c_t, &\quad o_c:=\sum_{t \ odd} c_t,
\end{align}
and $a:=a_1+a_2+\dotsc$, $b:=b_1+b_2+\dotsc$, $c:=c_1+c_2+\dotsc$.

We first define the regions corresponding to the off-central positions $1$ and $4$ in Figure \ref{fig:offposition}(a). In this case, we assume further that $y\geq 0$ (the domain of the $y$-parameter changes in the other off-central cases). Let us start with the auxiliary hexagon $H_0$ of side-lengths $x,z,z,x,z,z$. We perform the following side-pushing process. We push outward all six sides  of the hexagon (in the clockwise order, starting from the north side) a distance of $e_a+o_b+o_c, y+\max(b-a,0),\  b+c,\ b+c+y+\max(a-b,0),\ o_a+e_b+e_c+y+\max(a-b,0),\ a,\ a+y+\max(b-a,0)$ units, respectively. We get a larger hexagon $H$, called the \emph{base hexagon}, has side-lengths $x+o_a+e_b+e_c,$  $2y+z+e_a+o_b+o_c+ |a-b|$,  $z+o_a+e_b+e_c,$ $x+e_a+o_b+o_c$, $2y+z+o_a+e_b+e_c+ |a-b|,$ $z+e_a+o_b+o_c$. This is illustrated by Figure \ref{pushing1} (for $j=0$); the auxiliary hexagon is the shaded one with the dashed boundary, and the base hexagon is the one restricted by the bold contour.

\begin{figure}\centering
 \centering
  \setlength{\unitlength}{3947sp}%
\begingroup\makeatletter\ifx\SetFigFont\undefined%
\gdef\SetFigFont#1#2#3#4#5{%
  \reset@font\fontsize{#1}{#2pt}%
  \fontfamily{#3}\fontseries{#4}\fontshape{#5}%
  \selectfont}%
\fi\endgroup%
\resizebox{15cm}{!}{
\begin{picture}(0,0)%
\includegraphics{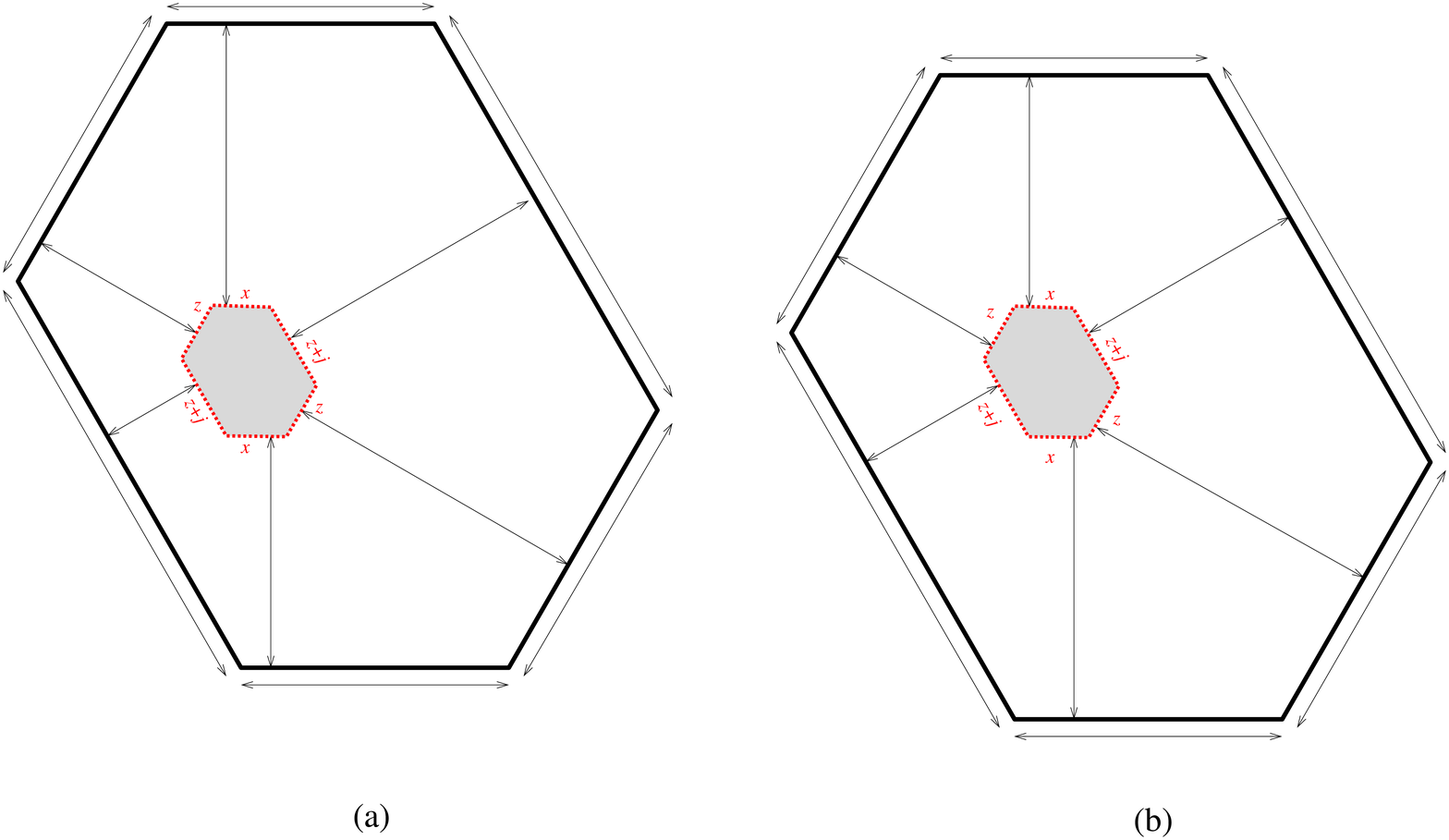}%
\end{picture}%
\begin{picture}(19869,11857)(3262,-12422)
\put(6445,-841){\makebox(0,0)[lb]{\smash{{\SetFigFont{14}{16.8}{\rmdefault}{\mddefault}{\itdefault}{\color[rgb]{0,0,0}$x+o_a+e_b+e_c$}%
}}}}
\put(3786,-3503){\rotatebox{60.0}{\makebox(0,0)[lb]{\smash{{\SetFigFont{14}{16.8}{\rmdefault}{\mddefault}{\itdefault}{\color[rgb]{0,0,0}$z+e_a+o_b+o_c$}%
}}}}}
\put(3683,-6143){\rotatebox{300.0}{\makebox(0,0)[lb]{\smash{{\SetFigFont{14}{16.8}{\rmdefault}{\mddefault}{\itdefault}{\color[rgb]{0,0,0}$2y+z+o_a+e_b+e_c+(b-a)+j$}%
}}}}}
\put(7366,-10617){\makebox(0,0)[lb]{\smash{{\SetFigFont{14}{16.8}{\rmdefault}{\mddefault}{\itdefault}{\color[rgb]{0,0,0}$x+e_a+o_b+o_c$}%
}}}}
\put(11282,-9200){\rotatebox{60.0}{\makebox(0,0)[lb]{\smash{{\SetFigFont{14}{16.8}{\rmdefault}{\mddefault}{\itdefault}{\color[rgb]{0,0,0}$z+o_a+e_b+e_c$}%
}}}}}
\put(10127,-1891){\rotatebox{300.0}{\makebox(0,0)[lb]{\smash{{\SetFigFont{14}{16.8}{\rmdefault}{\mddefault}{\itdefault}{\color[rgb]{0,0,0}$2y+z+e_a+o_b+o_c+(b-a)+j$}%
}}}}}
\put(8604,-4599){\rotatebox{30.0}{\makebox(0,0)[lb]{\smash{{\SetFigFont{14}{16.8}{\rmdefault}{\mddefault}{\itdefault}{\color[rgb]{0,0,0}$b+c$}%
}}}}}
\put(6708,-4476){\rotatebox{90.0}{\makebox(0,0)[lb]{\smash{{\SetFigFont{14}{16.8}{\rmdefault}{\mddefault}{\itdefault}{\color[rgb]{0,0,0}$y+e_a+o_b+o_c+(b-a)$}%
}}}}}
\put(8695,-7066){\rotatebox{330.0}{\makebox(0,0)[lb]{\smash{{\SetFigFont{14}{16.8}{\rmdefault}{\mddefault}{\itdefault}{\color[rgb]{0,0,0}$y+b+c$}%
}}}}}
\put(7211,-9436){\rotatebox{90.0}{\makebox(0,0)[lb]{\smash{{\SetFigFont{14}{16.8}{\rmdefault}{\mddefault}{\itdefault}{\color[rgb]{0,0,0}$y+o_a+e_b+e_c$}%
}}}}}
\put(5279,-6720){\rotatebox{30.0}{\makebox(0,0)[lb]{\smash{{\SetFigFont{14}{16.8}{\rmdefault}{\mddefault}{\itdefault}{\color[rgb]{0,0,0}$a$}%
}}}}}
\put(4502,-4232){\rotatebox{330.0}{\makebox(0,0)[lb]{\smash{{\SetFigFont{14}{16.8}{\rmdefault}{\mddefault}{\itdefault}{\color[rgb]{0,0,0}$y+a+(b-a)$}%
}}}}}
\put(17083,-1405){\makebox(0,0)[lb]{\smash{{\SetFigFont{14}{16.8}{\rmdefault}{\mddefault}{\itdefault}{\color[rgb]{0,0,0}$x+o_a+e_b+e_c$}%
}}}}
\put(14338,-4240){\rotatebox{60.0}{\makebox(0,0)[lb]{\smash{{\SetFigFont{14}{16.8}{\rmdefault}{\mddefault}{\itdefault}{\color[rgb]{0,0,0}$z+e_a+o_b+o_c$}%
}}}}}
\put(18208,-11326){\makebox(0,0)[lb]{\smash{{\SetFigFont{14}{16.8}{\rmdefault}{\mddefault}{\itdefault}{\color[rgb]{0,0,0}$x+e_a+o_b+o_c$}%
}}}}
\put(21907,-10263){\rotatebox{60.0}{\makebox(0,0)[lb]{\smash{{\SetFigFont{14}{16.8}{\rmdefault}{\mddefault}{\itdefault}{\color[rgb]{0,0,0}$z+o_a+e_b+e_c$}%
}}}}}
\put(20739,-2672){\rotatebox{300.0}{\makebox(0,0)[lb]{\smash{{\SetFigFont{14}{16.8}{\rmdefault}{\mddefault}{\itdefault}{\color[rgb]{0,0,0}$2y+z+e_a+o_b+o_c+(a-b)+j$}%
}}}}}
\put(18836,-4620){\rotatebox{30.0}{\makebox(0,0)[lb]{\smash{{\SetFigFont{14}{16.8}{\rmdefault}{\mddefault}{\itdefault}{\color[rgb]{0,0,0}$b+c$}%
}}}}}
\put(17747,-4476){\rotatebox{90.0}{\makebox(0,0)[lb]{\smash{{\SetFigFont{14}{16.8}{\rmdefault}{\mddefault}{\itdefault}{\color[rgb]{0,0,0}$y+e_a+o_b+o_c$}%
}}}}}
\put(19845,-7295){\rotatebox{330.0}{\makebox(0,0)[lb]{\smash{{\SetFigFont{14}{16.8}{\rmdefault}{\mddefault}{\itdefault}{\color[rgb]{0,0,0}$y+b+c+(a-b)$}%
}}}}}
\put(18361,-10263){\rotatebox{90.0}{\makebox(0,0)[lb]{\smash{{\SetFigFont{14}{16.8}{\rmdefault}{\mddefault}{\itdefault}{\color[rgb]{0,0,0}$y+o_a+e_b+e_c+(a-b)$}%
}}}}}
\put(15677,-6777){\rotatebox{30.0}{\makebox(0,0)[lb]{\smash{{\SetFigFont{14}{16.8}{\rmdefault}{\mddefault}{\itdefault}{\color[rgb]{0,0,0}$a$}%
}}}}}
\put(15446,-5113){\rotatebox{330.0}{\makebox(0,0)[lb]{\smash{{\SetFigFont{14}{16.8}{\rmdefault}{\mddefault}{\itdefault}{\color[rgb]{0,0,0}$y+a$}%
}}}}}
\put(14246,-6734){\rotatebox{300.0}{\makebox(0,0)[lb]{\smash{{\SetFigFont{14}{16.8}{\rmdefault}{\mddefault}{\itdefault}{\color[rgb]{0,0,0}$2y+z+e_a+o_b+o_c+(a-b)+j$}%
}}}}}
\end{picture}}
\caption{The edge-pushing procedure used in the definitions of the $E^{(i)}$-, $F^{(i)}$-, $G^{(i)}$-, $K^{(i)}$-type regions: (a) the case $a\leq b$, (b) the case $a\geq b$.}\label{pushing1}
\end{figure}

For $i=1,4$, we denote $E^{(i)}_{x,y,z}(\textbf{a};\textbf{c};\textbf{b})$ the region obtained from the base hexagon $H$  by removing three ferns as follows. The middle fern consists of equilateral triangles of side-lengths $c_1,c_2,\dots,c_k$ as they appear from left to right. The first triangle (of side-length $c_1$) is up-pointing and the next triangles in the fern are oriented in alternating orientations. We remove the middle fern, such that its root (i.e. its leftmost point) is at the off-central position $i$ as shown in Figure \ref{fig:offposition}(a),  for $i=1,4$. The left fern consists of triangles of side-lengths $a_1,a_2,\dots,a_m$ running from left to right and starts by a down-pointing triangle; while the right fern consists of triangles of side-lengths $b_1,b_2,\dots, b_n$ running from \emph{right to left} and starts by an \emph{up-pointing} triangles. These latter two ferns are removed such that they are at the same level as the middle fern, that the left fern is touching the southwest side of the hexagon, and that the right fern is touching the northeast side of the hexagon. See Figure \ref{fig:off1}(a) for an example.

For $i=2,3,5,6$, unlike the case $i=1,4$ above,  we start with an auxiliary hexagon of side-lengths $x,z+2,z,x,z+2,z$, we still perform the same side-pushing procedure to get the base hexagon of side-lengths $x+o_a+e_b+e_c,$  $2y+z+e_a+o_b+o_c+ |a-b|+2$,  $z+o_a+e_b+e_c,$ $x+e_a+o_b+o_c$, $2y+z+o_a+e_b+e_c+ |a-b|+2,$ $z+e_a+o_b+o_c$.
If $i=2,3$, we allow $y$ to take any integer values greater or equal to $\max(a-b,-2)$. In particular, $y\geq 0$ if $a\leq b$, however, $y$ may be negative when $a>b$. Our side-pushing process still works well in the same way as the definition of the $E^{(1)}$-type region, as even though $y$ may be negative, all the pushing distances are still nonnegative. Next, the three ferns are removed similarly, such that the root of the middle fern is at the off-central position $i$ in Figure \ref{fig:offposition}(a),  for $i=2,3$.
If $i=5,6$, we assume in addition that $y\geq \max (b-a,-2)$, and the  $E^{(5)}$- and $E^{(6)}$-type regions are defined similarly.
By, the symmetry it is enough to enumerate only three of the above regions, namely the $E^{(1)}$-, $E^{(2)}$- and $E^{(6)}$-type regions (see Figure \ref{fig:off1} for examples\footnote{In this section, we ignore the four shaded triangles of labels $u,v,w,s$ in the illustrative figures; these triangles will be used later in Section \ref{sec:proofoff}.}), as the remaining regions can be obtained from these three regions by $180^{\circ}$-rotations.

\begin{rmk}\label{rmkE}
In the definition of the $E^{(1)}$-type regions, if we remove the three ferns such that the root of the middle one is at the center of the auxiliary hexagon $H_0$, then we obtain the region $R^{\odot}_{x,y,z}(\textbf{a};\textbf{c};\textbf{b})$ in Theorem 2.2 of \cite{HoleDent}. In some sense, the $E^{(i)}$-type regions can be viewed as counterparts of the $R^{\odot}$-type regions investigated in \cite{HoleDent}.
\end{rmk}

We note that the positions of the two side ferns are determined uniquely by the level of the the middle fern. Thus, from now on, when define the three removed ferns, we only need to care about the position of the middle fern.

\begin{figure}\centering
\setlength{\unitlength}{3947sp}%
\begingroup\makeatletter\ifx\SetFigFont\undefined%
\gdef\SetFigFont#1#2#3#4#5{%
  \reset@font\fontsize{#1}{#2pt}%
  \fontfamily{#3}\fontseries{#4}\fontshape{#5}%
  \selectfont}%
\fi\endgroup%
\resizebox{15cm}{!}{
\begin{picture}(0,0)%
\includegraphics{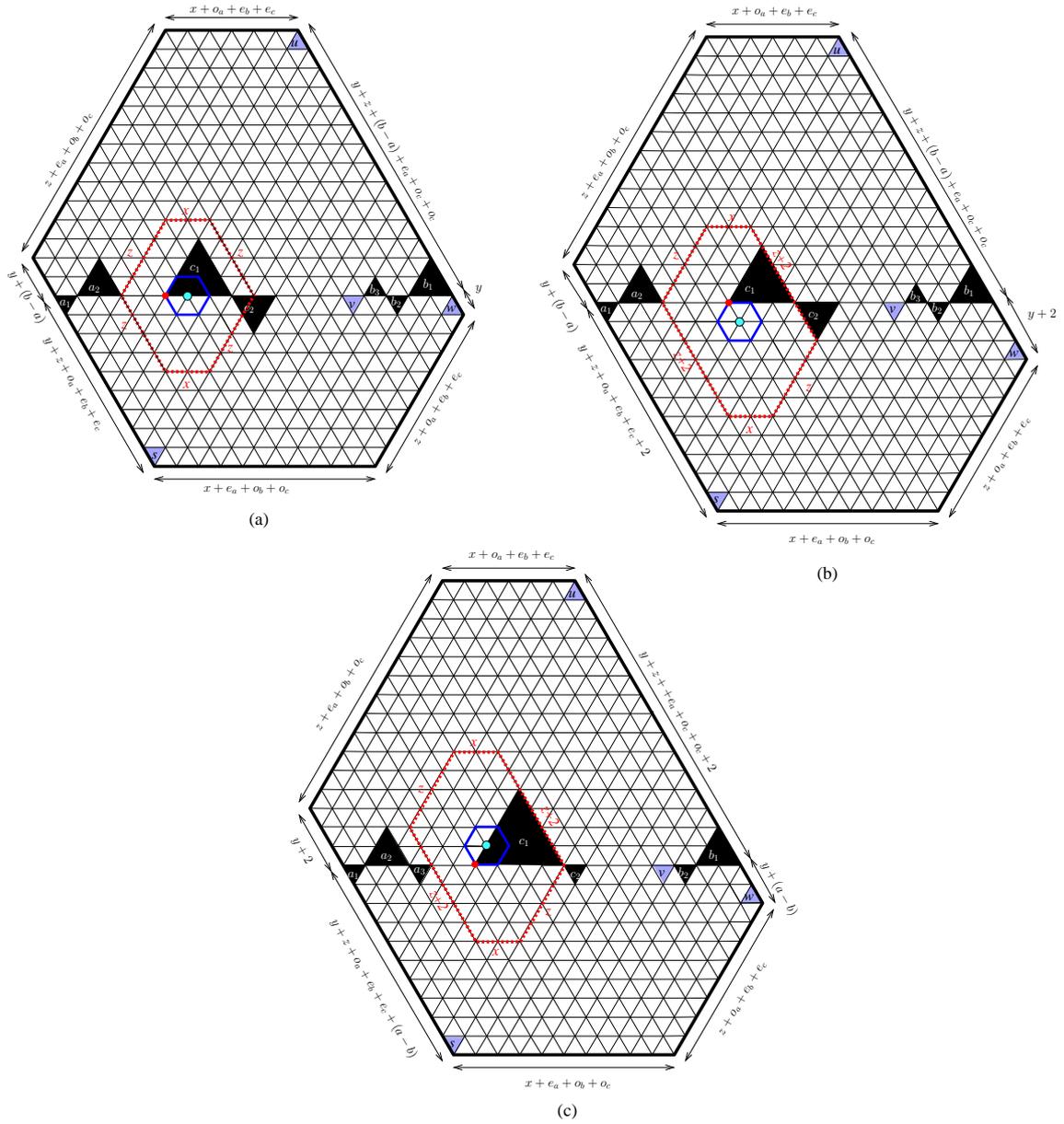}%
\end{picture}%
%
%

\begin{picture}(19395,20979)(2759,-21207)
\put(16776,-16402){\rotatebox{300.0}{\makebox(0,0)[lb]{\smash{{\SetFigFont{14}{16.8}{\familydefault}{\mddefault}{\updefault}{\color[rgb]{0,0,0}$y+(a-b)$}%
}}}}}
\put(14537,-12451){\rotatebox{300.0}{\makebox(0,0)[lb]{\smash{{\SetFigFont{14}{16.8}{\familydefault}{\mddefault}{\updefault}{\color[rgb]{0,0,0}$y+z++e_a+o_c+o_c+2$}%
}}}}}
\put(16068,-19597){\rotatebox{60.0}{\makebox(0,0)[lb]{\smash{{\SetFigFont{14}{16.8}{\familydefault}{\mddefault}{\updefault}{\color[rgb]{0,0,0}$z+o_a+e_b+e_c$}%
}}}}}
\put(12395,-20542){\makebox(0,0)[lb]{\smash{{\SetFigFont{14}{16.8}{\familydefault}{\mddefault}{\updefault}{\color[rgb]{0,0,0}$x+e_a+o_b+o_c$}%
}}}}
\put(8798,-17463){\rotatebox{300.0}{\makebox(0,0)[lb]{\smash{{\SetFigFont{14}{16.8}{\familydefault}{\mddefault}{\updefault}{\color[rgb]{0,0,0}$y+z+o_a+e_b+e_c+(a-b)$}%
}}}}}
\put(8082,-15929){\rotatebox{300.0}{\makebox(0,0)[lb]{\smash{{\SetFigFont{14}{16.8}{\familydefault}{\mddefault}{\updefault}{\color[rgb]{0,0,0}$y+2$}%
}}}}}
\put(8598,-13906){\rotatebox{60.0}{\makebox(0,0)[lb]{\smash{{\SetFigFont{14}{16.8}{\familydefault}{\mddefault}{\updefault}{\color[rgb]{0,0,0}$z+e_a+o_b+o_c$}%
}}}}}
\put(11349,-10666){\makebox(0,0)[lb]{\smash{{\SetFigFont{14}{16.8}{\familydefault}{\mddefault}{\updefault}{\color[rgb]{0,0,0}$x+o_a+e_b+e_c$}%
}}}}
\put(13518,-3976){\rotatebox{60.0}{\makebox(0,0)[lb]{\smash{{\SetFigFont{14}{16.8}{\familydefault}{\mddefault}{\updefault}{\color[rgb]{0,0,0}$z+e_a+o_b+o_c$}%
}}}}}
\put(12704,-5521){\rotatebox{300.0}{\makebox(0,0)[lb]{\smash{{\SetFigFont{14}{16.8}{\familydefault}{\mddefault}{\updefault}{\color[rgb]{0,0,0}$y+(b-a)$}%
}}}}}
\put(13401,-6716){\rotatebox{300.0}{\makebox(0,0)[lb]{\smash{{\SetFigFont{14}{16.8}{\familydefault}{\mddefault}{\updefault}{\color[rgb]{0,0,0}$y+z+o_a+e_b+e_c+2$}%
}}}}}
\put(17287,-10381){\makebox(0,0)[lb]{\smash{{\SetFigFont{14}{16.8}{\familydefault}{\mddefault}{\updefault}{\color[rgb]{0,0,0}$x+e_a+o_b+o_c$}%
}}}}
\put(20960,-9436){\rotatebox{60.0}{\makebox(0,0)[lb]{\smash{{\SetFigFont{14}{16.8}{\familydefault}{\mddefault}{\updefault}{\color[rgb]{0,0,0}$z+o_a+e_b+e_c$}%
}}}}}
\put(21709,-6167){\makebox(0,0)[lb]{\smash{{\SetFigFont{14}{16.8}{\familydefault}{\mddefault}{\updefault}{\color[rgb]{0,0,0}$y+2$}%
}}}}
\put(19429,-2290){\rotatebox{300.0}{\makebox(0,0)[lb]{\smash{{\SetFigFont{14}{16.8}{\familydefault}{\mddefault}{\updefault}{\color[rgb]{0,0,0}$y+z+(b-a)+e_a+o_c+o_c$}%
}}}}}
\put(16224,-511){\makebox(0,0)[lb]{\smash{{\SetFigFont{14}{16.8}{\familydefault}{\mddefault}{\updefault}{\color[rgb]{0,0,0}$x+o_a+e_b+e_c$}%
}}}}
\put(3571,-3748){\rotatebox{60.0}{\makebox(0,0)[lb]{\smash{{\SetFigFont{14}{16.8}{\familydefault}{\mddefault}{\updefault}{\color[rgb]{0,0,0}$z+e_a+o_b+o_c$}%
}}}}}
\put(2836,-5346){\rotatebox{300.0}{\makebox(0,0)[lb]{\smash{{\SetFigFont{14}{16.8}{\familydefault}{\mddefault}{\updefault}{\color[rgb]{0,0,0}$y+(b-a)$}%
}}}}}
\put(3489,-6583){\rotatebox{300.0}{\makebox(0,0)[lb]{\smash{{\SetFigFont{14}{16.8}{\familydefault}{\mddefault}{\updefault}{\color[rgb]{0,0,0}$y+z+o_a+e_b+e_c$}%
}}}}}
\put(6181,-478){\makebox(0,0)[lb]{\smash{{\SetFigFont{14}{16.8}{\familydefault}{\mddefault}{\updefault}{\color[rgb]{0,0,0}$x+o_a+e_b+e_c$}%
}}}}
\put(9196,-1806){\rotatebox{300.0}{\makebox(0,0)[lb]{\smash{{\SetFigFont{14}{16.8}{\familydefault}{\mddefault}{\updefault}{\color[rgb]{0,0,0}$y+z+(b-a)+e_a+o_c+o_c$}%
}}}}}
\put(11476,-5683){\makebox(0,0)[lb]{\smash{{\SetFigFont{14}{16.8}{\familydefault}{\mddefault}{\updefault}{\color[rgb]{0,0,0}$y$}%
}}}}
\put(10419,-8571){\rotatebox{60.0}{\makebox(0,0)[lb]{\smash{{\SetFigFont{14}{16.8}{\familydefault}{\mddefault}{\updefault}{\color[rgb]{0,0,0}$z+o_a+e_b+e_c$}%
}}}}}
\put(6436,-9433){\makebox(0,0)[lb]{\smash{{\SetFigFont{14}{16.8}{\familydefault}{\mddefault}{\updefault}{\color[rgb]{0,0,0}$x+e_a+o_b+o_c$}%
}}}}
\put(6181,-5301){\makebox(0,0)[lb]{\smash{{\SetFigFont{14}{16.8}{\familydefault}{\mddefault}{\updefault}{\color[rgb]{1,1,1}$c_1$}%
}}}}
\put(19932,-6073){\makebox(0,0)[lb]{\smash{{\SetFigFont{14}{16.8}{\familydefault}{\mddefault}{\updefault}{\color[rgb]{1,1,1}$b_2$}%
}}}}
\put(20524,-5743){\makebox(0,0)[lb]{\smash{{\SetFigFont{14}{16.8}{\familydefault}{\mddefault}{\updefault}{\color[rgb]{1,1,1}$b_1$}%
}}}}
\put(16469,-5684){\makebox(0,0)[lb]{\smash{{\SetFigFont{14}{16.8}{\familydefault}{\mddefault}{\updefault}{\color[rgb]{1,1,1}$c_1$}%
}}}}
\put(15796,-16289){\makebox(0,0)[lb]{\smash{{\SetFigFont{14}{16.8}{\familydefault}{\mddefault}{\updefault}{\color[rgb]{1,1,1}$b_1$}%
}}}}
\put(7186,-6051){\makebox(0,0)[lb]{\smash{{\SetFigFont{14}{16.8}{\familydefault}{\mddefault}{\updefault}{\color[rgb]{1,1,1}$c_2$}%
}}}}
\put(10501,-5616){\makebox(0,0)[lb]{\smash{{\SetFigFont{14}{16.8}{\familydefault}{\mddefault}{\updefault}{\color[rgb]{1,1,1}$b_1$}%
}}}}
\put(9909,-5946){\makebox(0,0)[lb]{\smash{{\SetFigFont{14}{16.8}{\familydefault}{\mddefault}{\updefault}{\color[rgb]{1,1,1}$b_2$}%
}}}}
\put(12281,-16021){\makebox(0,0)[lb]{\smash{{\SetFigFont{14}{16.8}{\familydefault}{\mddefault}{\updefault}{\color[rgb]{1,1,1}$c_1$}%
}}}}
\put(15234,-16619){\makebox(0,0)[lb]{\smash{{\SetFigFont{14}{16.8}{\familydefault}{\mddefault}{\updefault}{\color[rgb]{1,1,1}$b_2$}%
}}}}
\put(10332,-16557){\makebox(0,0)[lb]{\smash{{\SetFigFont{14}{16.8}{\familydefault}{\mddefault}{\updefault}{\color[rgb]{1,1,1}$a_3$}%
}}}}
\put(9718,-16287){\makebox(0,0)[lb]{\smash{{\SetFigFont{14}{16.8}{\familydefault}{\mddefault}{\updefault}{\color[rgb]{1,1,1}$a_2$}%
}}}}
\put(9104,-16642){\makebox(0,0)[lb]{\smash{{\SetFigFont{14}{16.8}{\familydefault}{\mddefault}{\updefault}{\color[rgb]{1,1,1}$a_1$}%
}}}}
\put(19534,-5848){\makebox(0,0)[lb]{\smash{{\SetFigFont{14}{16.8}{\familydefault}{\mddefault}{\updefault}{\color[rgb]{1,1,1}$b_3$}%
}}}}
\put(14360,-5791){\makebox(0,0)[lb]{\smash{{\SetFigFont{14}{16.8}{\familydefault}{\mddefault}{\updefault}{\color[rgb]{1,1,1}$a_2$}%
}}}}
\put(13790,-6069){\makebox(0,0)[lb]{\smash{{\SetFigFont{14}{16.8}{\familydefault}{\mddefault}{\updefault}{\color[rgb]{1,1,1}$a_1$}%
}}}}
\put(9511,-5721){\makebox(0,0)[lb]{\smash{{\SetFigFont{14}{16.8}{\familydefault}{\mddefault}{\updefault}{\color[rgb]{1,1,1}$b_3$}%
}}}}
\put(4306,-5676){\makebox(0,0)[lb]{\smash{{\SetFigFont{14}{16.8}{\familydefault}{\mddefault}{\updefault}{\color[rgb]{1,1,1}$a_2$}%
}}}}
\put(3766,-5946){\makebox(0,0)[lb]{\smash{{\SetFigFont{14}{16.8}{\familydefault}{\mddefault}{\updefault}{\color[rgb]{1,1,1}$a_1$}%
}}}}
\put(17645,-6174){\makebox(0,0)[lb]{\smash{{\SetFigFont{14}{16.8}{\familydefault}{\mddefault}{\updefault}{\color[rgb]{1,1,1}$c_2$}%
}}}}
\put(13230,-16629){\makebox(0,0)[lb]{\smash{{\SetFigFont{14}{16.8}{\familydefault}{\mddefault}{\updefault}{\color[rgb]{1,1,1}$c_2$}%
}}}}
\end{picture}%
}
\caption{(a) The region $E^{(1)}_{2,1,4}(1,2;\ 3,2;\ 2,1,1)$. (b) The region   $E^{(2)}_{2,1,4}(1,2;\ 3,2;\ 2,1,1)$. (c) The region   $E^{(6)}_{2,1,4}(1,2,1;\ 4,1;\ 2,1)$}\label{fig:off1}
\end{figure}


Before stating our enumerations, we note that one can always assume that each of our ferns has even number of triangles. Indeed, if a fern has an odd number of triangles, we can regard that the fern contains an even number of triangles by adding a triangle of side-length $0$ to its end. For the sake of simplicity, we assume this in the statements of our theorems throughout this paper.

\begin{thm}\label{off1thm1}
Assume that $\textbf{a}=(a_1,a_2,\dotsc,a_m)$, $\textbf{b}=(b_1,b_2,\dotsc,b_n)$, $\textbf{c}=(c_1,c_2,\dotsc,c_k)$ are three sequences  of a even number of nonnegative integers ($m,n,k$ are all even) and that $x,y,z$ are three nonnegative integers, such that $x$ and $z$ have the same parity.  Then
\begin{align}\label{off1eq1}
\M&(E^{(1)}_{x,y,z}(\textbf{a};\textbf{c};\textbf{b}))=\Phi_{x,2y+z+2\max(a,b),z}(c)\notag\\
&\times s\left(y+b-\min(a,b),a_1,\dotsc, a_{m},\frac{x+z}{2}-1,c_1,\dotsc,c_{k}+\frac{x+z}{2}+1+b_n,b_{n-1},\dotsc,b_1\right)\notag\\
&\times s\left(a_1,\dotsc, a_{m-1},a_{m}+\frac{x+z}{2}-1+c_1,\dotsc,c_{k},\frac{x+z}{2}+1,b_n,\dotsc,b_1,y+a-\min(a,b)\right)\notag\\
&\times\frac{\Hf(c+\frac{x+z}{2}-1)}{\Hf(c)\Hf(\frac{x+z}{2}-1)}\frac{\Hf(\max(a,b)+y+\frac{x+z}{2}-1)}{\Hf(\max(a,b)+c+y+\frac{x+z}{2}-1)}\notag\\
&\times \frac{\Hf(\max(a,b)+y+z)\Hf(\max(a,b)+c+y+z)}{\Hf(\max(a,b)-o_a+o_b+o_c+y+z)\Hf(\max(a,b)+o_a-o_b+e_c+y+z)}\notag\\
&\times \frac{\Hf(\max(a,b)-o_a+o_b+o_c+y)\Hf(\max(a,b)+o_a-o_b+e_c+y)}{\Hf(\max(a,b)+y)^2},
\end{align}
where $\Phi_{x,y,z}(m)$ is defined as in (\ref{phieq}) of Conjecture \ref{con1}.
\end{thm}

\begin{thm}\label{off1thm2}
Assume that $\textbf{a}=(a_1,a_2,\dotsc,a_m)$, $\textbf{b}=(b_1,b_2,\dotsc,b_n)$, $\textbf{c}=(c_1,c_2,\dotsc,c_k)$ are three sequences  of nonnegative integers ($m,n,k$ are all even)  and that $x,y,z$ are three integers, such that $x,z\geq 0$, $y\geq\max(a-b,-2)$, and $x$ and $z$ have the same parity.  Then
\begin{align}\label{off1eq2}
\M&(E^{(2)}_{x,y,z}(\textbf{a};\textbf{c};\textbf{b}))=\Phi_{2y+z+2\max(a,b)+2,x,z}(c)\notag\\
&\times s\left(y+b-\min(a,b),a_1,\dotsc, a_{m},\frac{x+z}{2},c_1,\dotsc,c_{k}+\frac{x+z}{2}+b_n,b_{n-1},\dotsc,b_1\right)\notag\\
&\times s\left(a_1,\dotsc, a_{m-1},a_{m}+\frac{x+z}{2}+c_1,\dotsc,c_{k},\frac{x+z}{2},b_n,\dotsc,b_1,y+a-\min(a,b)+2\right)\notag\\
&\times\frac{\Hf(c+\frac{x+z}{2})}{\Hf(c)\Hf(\frac{x+z}{2})}\frac{\Hf(\max(a,b)+y+\frac{x+z}{2})}{\Hf(\max(a,b)+c+y+\frac{x+z}{2})}\notag\\
&\times \frac{\Hf(\max(a,b)+y+z+2)\Hf(\max(a,b)+c+y+z)}{\Hf(\max(a,b)-o_a+o_b+o_c+y+z)\Hf(\max(a,b)+o_a-o_b+e_c+y+z+2)}\notag\\
&\times \frac{\Hf(\max(a,b)-o_a+o_b+o_c+y)\Hf(\max(a,b)+o_a-o_b+e_c+y+2)}{\Hf(\max(a,b)+y)\Hf(\max(a,b)+y+2)}.
\end{align}
\end{thm}

\begin{thm}\label{off1thm3}
Assume that $\textbf{a}=(a_1,a_2,\dotsc,a_m)$, $\textbf{b}=(b_1,b_2,\dotsc,b_n)$, $\textbf{c}=(c_1,c_2,\dotsc,c_k)$ are three sequences  of nonnegative integers (for $m,n,k$ are even) and that $x,y,z$ are three integers, such that $x,z\geq 0$, $y\geq\max(b-a,-2)$, and $x$ and $z$ have the same parity.  Then 
\begin{align}\label{off1eq3}
\M&(E^{(6)}_{x,y,z}(\textbf{a};\textbf{c};\textbf{b}))=\Phi_{z,2y+z+2\max(a,b)+2,x}(c)\notag\\
&\times s\left(y+b-\min(a,b)+2,a_1,\dotsc, a_{m},\frac{x+z}{2}-1,c_1,\dotsc,c_{k}+\frac{x+z}{2}+1+b_n,b_{n-1},\dotsc,b_1\right)\notag\\
&\times s\left(a_1,\dotsc, a_{m-1},a_{m}+\frac{x+z}{2}-1+c_1,\dotsc,c_{k},\frac{x+z}{2}+1,b_n,\dotsc,b_1,y+a-\min(a,b)\right)\notag\\
&\times\frac{\Hf(c+\frac{x+z}{2}-1)}{\Hf(c)\Hf(\frac{x+z}{2}-1)}\frac{\Hf(\max(a,b)+y+\frac{x+z}{2}+1)}{\Hf(\max(a,b)+c+y+\frac{x+z}{2}+1)}\notag\\
&\times \frac{\Hf(\max(a,b)+y+z)\Hf(\max(a,b)+c+y+z+2)}{\Hf(\max(a,b)-o_a+o_b+o_c+y+z+2)\Hf(\max(a,b)+o_a-o_b+e_c+y+z)}\notag\\
&\times \frac{\Hf(\max(a,b)-o_a+o_b+o_c+y+2)\Hf(\max(a,b)+o_a-o_b+e_c+y)}{\Hf(\max(a,b)+y)\Hf(\max(a,b)+y+2)}.
\end{align}
%
%
\end{thm}

One can see that the factor involving function $\Phi$ in formula (\ref{off1eq2}) (resp.,  formula (\ref{off1eq3})) is obtained from that in the formula (\ref{off1eq1}) by interchanging the $x$- and $y$-variables (resp., $x$- and $z$-variables).

\bigskip

Next, we show that Theorem \ref{off1thm1} implies Ciucu--Eisenk\"{o}lbl--Krattenthaler--Zare's Conjecture \ref{con1}.

\begin{proof}[Proof of Conjecture \ref{con1} by using Theorem \ref{off1thm1}]
Apply formula (\ref{off1eq1}) in Theorem \ref{off1thm1} to the region\\ $E^{(1)}_{x,\frac{y-z}{2},z}(\emptyset;\  m;\ \emptyset)$, we prove the conjecture for the case when $y\geq z$ (all factors in (\ref{off1eq1}), except for the first one, cancel out). The remaining case, when $y< z$, is obtained by applying Theorem \ref{off1thm1} to a horizontal reflection of the region $E^{(1)}_{x,\frac{z-y}{2},y}(\emptyset;\  0,m;\ \emptyset)$.
\end{proof}

\begin{rmk}[Geometrical Interpretation]
One can rewrite the equation (\ref{off1eq1}) in Theorem \ref{off1thm1} as
\begin{align}
\M&(E^{(1)}_{x,y,z}(\textbf{a};\textbf{c};\textbf{b}))=\M(C^{(1)}_{x,2y+z+2\max(a,b),z}(c))\notag\\
&\times s\left(y+b-\min(a,b),a_1,\dotsc, a_{m},\frac{x+z}{2}-1,c_1,\dotsc,c_{k}+\frac{x+z}{2}+1+b_n,b_{n-1},\dotsc,b_1\right)\notag\\
&\times s\left(a_1,\dotsc, a_{m-1},a_{m}+\frac{x+z}{2}-1+c_1,\dotsc,c_{k},\frac{x+z}{2}+1,b_n,\dotsc,b_1,y+a-\min(a,b)\right)\notag\\
&\frac{\PP(y+\max(a,b)-o_a+o_b+o_c,y+\max(a,b)+o_a-o_b+e_c,z)}{\PP(\frac{x+z}{2}-1,c,y+\max(a,b))\PP(y+\max(a,b)+c,y+\max(a,b),z)},
\end{align}
where $C^{(1)}_{x,y,z}(m)$ is the (off-central) cored hexagon in Conjecture \ref{con1} and \[\PP(a,b,c):=\frac{\Hf(a)\Hf(b)\Hf(c)\Hf(a+b+c)}{\Hf(a+b)\Hf(b+c)\Hf(c+a)}\] is the tiling number of the centrally symmetric hexagon of side-lengths $a,b,c,a,b,c$ (it is also the number of plane partitions fitting in an $a \times b \times c$ box \cite{Mac}). We have similar interpretation for other theorems in this section. This would be interesting to have a combinatorial proof for this.
\end{rmk}

\subsection{The $F^{(i)}$-type regions}

\begin{figure}\centering
\setlength{\unitlength}{3947sp}%
\begingroup\makeatletter\ifx\SetFigFont\undefined%
\gdef\SetFigFont#1#2#3#4#5{%
  \reset@font\fontsize{#1}{#2pt}%
  \fontfamily{#3}\fontseries{#4}\fontshape{#5}%
  \selectfont}%
\fi\endgroup%
\resizebox{15cm}{!}{
\begin{picture}(0,0)%
\includegraphics{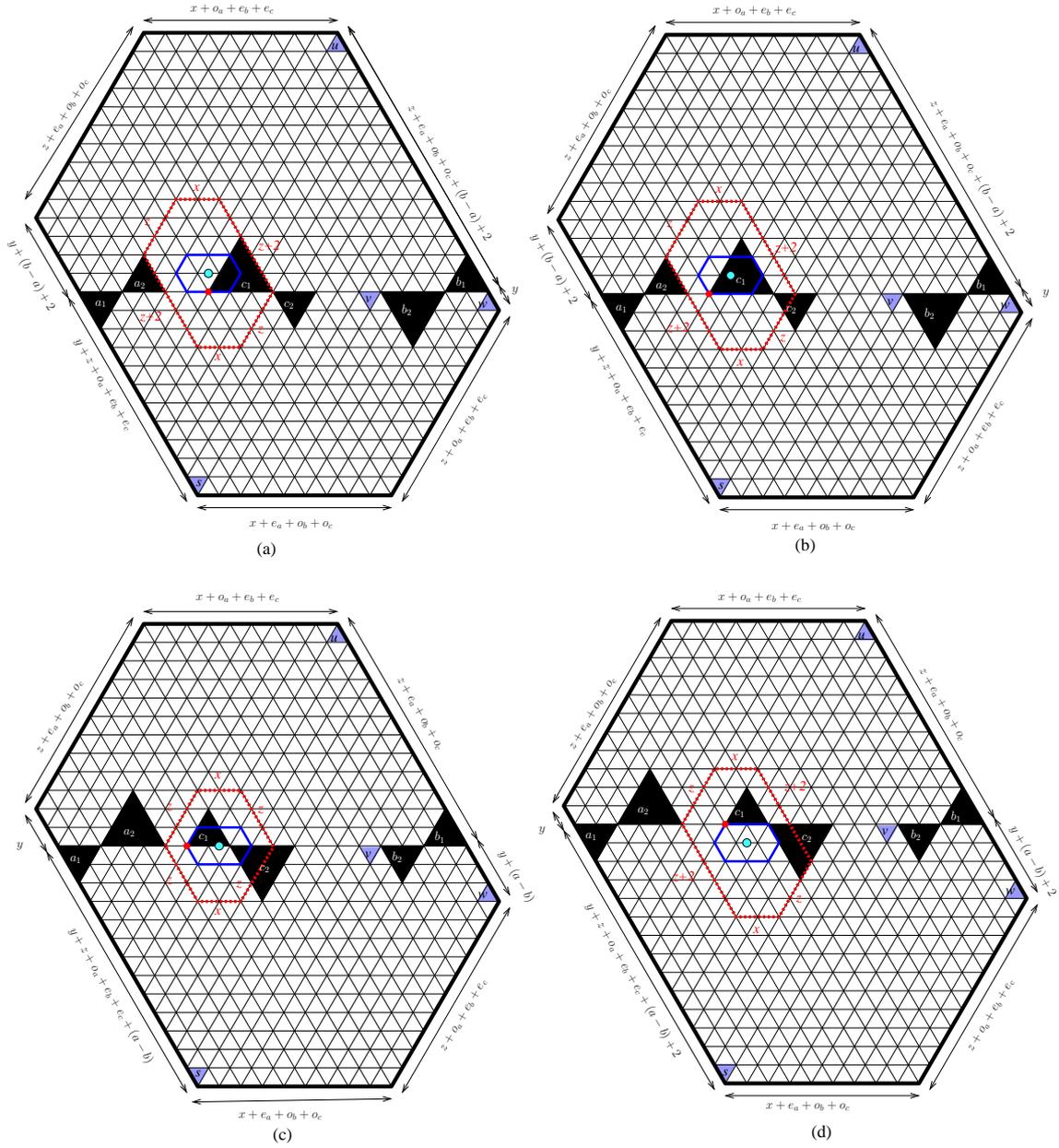}%
\end{picture}%
%
%

\begin{picture}(19911,21959)(642,-22126)
\put(4163,-410){\makebox(0,0)[lb]{\smash{{\SetFigFont{14}{16.8}{\familydefault}{\mddefault}{\updefault}{$x+o_a+e_b+e_c$}%
}}}}
\put(14086,-449){\makebox(0,0)[lb]{\smash{{\SetFigFont{14}{16.8}{\familydefault}{\mddefault}{\updefault}{$x+o_a+e_b+e_c$}%
}}}}
\put(14188,-11670){\makebox(0,0)[lb]{\smash{{\SetFigFont{14}{16.8}{\familydefault}{\mddefault}{\updefault}{$x+o_a+e_b+e_c$}%
}}}}
\put(4266,-11670){\makebox(0,0)[lb]{\smash{{\SetFigFont{14}{16.8}{\familydefault}{\mddefault}{\updefault}{$x+o_a+e_b+e_c$}%
}}}}
\put(5076,-21591){\makebox(0,0)[lb]{\smash{{\SetFigFont{14}{16.8}{\familydefault}{\mddefault}{\updefault}{$x+e_a+o_b+o_c$}%
}}}}
\put(15100,-21414){\makebox(0,0)[lb]{\smash{{\SetFigFont{14}{16.8}{\familydefault}{\mddefault}{\updefault}{$x+e_a+o_b+o_c$}%
}}}}
\put(15203,-10311){\makebox(0,0)[lb]{\smash{{\SetFigFont{14}{16.8}{\familydefault}{\mddefault}{\updefault}{$x+e_a+o_b+o_c$}%
}}}}
\put(5280,-10272){\makebox(0,0)[lb]{\smash{{\SetFigFont{14}{16.8}{\familydefault}{\mddefault}{\updefault}{$x+e_a+o_b+o_c$}%
}}}}
\put(1425,-3091){\rotatebox{60.0}{\makebox(0,0)[lb]{\smash{{\SetFigFont{14}{16.8}{\familydefault}{\mddefault}{\updefault}{$z+e_a+o_b+o_c$}%
}}}}}
\put(11347,-3268){\rotatebox{60.0}{\makebox(0,0)[lb]{\smash{{\SetFigFont{14}{16.8}{\familydefault}{\mddefault}{\updefault}{$z+e_a+o_b+o_c$}%
}}}}}
\put(11449,-14548){\rotatebox{60.0}{\makebox(0,0)[lb]{\smash{{\SetFigFont{14}{16.8}{\familydefault}{\mddefault}{\updefault}{$z+e_a+o_b+o_c$}%
}}}}}
\put(1323,-14607){\rotatebox{60.0}{\makebox(0,0)[lb]{\smash{{\SetFigFont{14}{16.8}{\familydefault}{\mddefault}{\updefault}{$z+e_a+o_b+o_c$}%
}}}}}
\put(8286,-2242){\rotatebox{300.0}{\makebox(0,0)[lb]{\smash{{\SetFigFont{14}{16.8}{\rmdefault}{\mddefault}{\itdefault}{$z+e_a+o_b+o_c+(b-a)+2$}%
}}}}}
\put(18217,-2147){\rotatebox{300.0}{\makebox(0,0)[lb]{\smash{{\SetFigFont{14}{16.8}{\rmdefault}{\mddefault}{\itdefault}{$z+e_a+o_b+o_c+(b-a)+2$}%
}}}}}
\put(18012,-13013){\rotatebox{300.0}{\makebox(0,0)[lb]{\smash{{\SetFigFont{14}{16.8}{\rmdefault}{\mddefault}{\itdefault}{$z+e_a+o_b+o_c$}%
}}}}}
\put(8086,-13181){\rotatebox{300.0}{\makebox(0,0)[lb]{\smash{{\SetFigFont{14}{16.8}{\rmdefault}{\mddefault}{\itdefault}{$z+e_a+o_b+o_c$}%
}}}}}
\put(11785,-6832){\rotatebox{300.0}{\makebox(0,0)[lb]{\smash{{\SetFigFont{14}{16.8}{\familydefault}{\mddefault}{\updefault}{$y+z+o_a+e_b+e_c$}%
}}}}}
\put(1965,-6714){\rotatebox{300.0}{\makebox(0,0)[lb]{\smash{{\SetFigFont{14}{16.8}{\familydefault}{\mddefault}{\updefault}{$y+z+o_a+e_b+e_c$}%
}}}}}
\put(11682,-17521){\rotatebox{300.0}{\makebox(0,0)[lb]{\smash{{\SetFigFont{14}{16.8}{\familydefault}{\mddefault}{\updefault}{$y+z+o_a+e_b+e_c+(a-b)+2$}%
}}}}}
\put(1863,-17993){\rotatebox{300.0}{\makebox(0,0)[lb]{\smash{{\SetFigFont{14}{16.8}{\familydefault}{\mddefault}{\updefault}{$y+z+o_a+e_b+e_c+(a-b)$}%
}}}}}
\put(20356,-5834){\makebox(0,0)[lb]{\smash{{\SetFigFont{14}{16.8}{\familydefault}{\mddefault}{\updefault}{$y$}%
}}}}
\put(10272,-5784){\makebox(0,0)[lb]{\smash{{\SetFigFont{14}{16.8}{\familydefault}{\mddefault}{\updefault}{$y$}%
}}}}
\put(862,-16414){\makebox(0,0)[lb]{\smash{{\SetFigFont{14}{16.8}{\familydefault}{\mddefault}{\updefault}{$y$}%
}}}}
\put(10783,-16119){\makebox(0,0)[lb]{\smash{{\SetFigFont{14}{16.8}{\familydefault}{\mddefault}{\updefault}{$y$}%
}}}}
\put(19742,-16045){\rotatebox{300.0}{\makebox(0,0)[lb]{\smash{{\SetFigFont{14}{16.8}{\familydefault}{\mddefault}{\updefault}{$y+(a-b)+2$}%
}}}}}
\put(10025,-16512){\rotatebox{300.0}{\makebox(0,0)[lb]{\smash{{\SetFigFont{14}{16.8}{\familydefault}{\mddefault}{\updefault}{$y+(a-b)$}%
}}}}}
\put(717,-4825){\rotatebox{300.0}{\makebox(0,0)[lb]{\smash{{\SetFigFont{14}{16.8}{\familydefault}{\mddefault}{\updefault}{$y+(b-a)+2$}%
}}}}}
\put(10639,-4706){\rotatebox{300.0}{\makebox(0,0)[lb]{\smash{{\SetFigFont{14}{16.8}{\familydefault}{\mddefault}{\updefault}{$y+(b-a)+2$}%
}}}}}
\put(8997,-9108){\rotatebox{60.0}{\makebox(0,0)[lb]{\smash{{\SetFigFont{14}{16.8}{\familydefault}{\mddefault}{\updefault}{$z+o_a+e_b+e_c$}%
}}}}}
\put(18817,-9227){\rotatebox{60.0}{\makebox(0,0)[lb]{\smash{{\SetFigFont{14}{16.8}{\familydefault}{\mddefault}{\updefault}{$z+o_a+e_b+e_c$}%
}}}}}
\put(19021,-20329){\rotatebox{60.0}{\makebox(0,0)[lb]{\smash{{\SetFigFont{14}{16.8}{\familydefault}{\mddefault}{\updefault}{$z+o_a+e_b+e_c$}%
}}}}}
\put(8997,-20329){\rotatebox{60.0}{\makebox(0,0)[lb]{\smash{{\SetFigFont{14}{16.8}{\familydefault}{\mddefault}{\updefault}{$z+o_a+e_b+e_c$}%
}}}}}
\put(14526,-5597){\makebox(0,0)[lb]{\smash{{\SetFigFont{14}{16.8}{\familydefault}{\mddefault}{\updefault}{\color[rgb]{1,1,1}$c_1$}%
}}}}
\put(4321,-16249){\makebox(0,0)[lb]{\smash{{\SetFigFont{14}{16.8}{\familydefault}{\mddefault}{\updefault}{\color[rgb]{1,1,1}$c_1$}%
}}}}
\put(5131,-5679){\makebox(0,0)[lb]{\smash{{\SetFigFont{14}{16.8}{\familydefault}{\mddefault}{\updefault}{\color[rgb]{1,1,1}$c_1$}%
}}}}
\put(14506,-15804){\makebox(0,0)[lb]{\smash{{\SetFigFont{14}{16.8}{\familydefault}{\mddefault}{\updefault}{\color[rgb]{1,1,1}$c_1$}%
}}}}
\put(8806,-16279){\makebox(0,0)[lb]{\smash{{\SetFigFont{14}{16.8}{\familydefault}{\mddefault}{\updefault}{\color[rgb]{1,1,1}$b_1$}%
}}}}
\put(9211,-5649){\makebox(0,0)[lb]{\smash{{\SetFigFont{14}{16.8}{\familydefault}{\mddefault}{\updefault}{\color[rgb]{1,1,1}$b_1$}%
}}}}
\put(18108,-6250){\makebox(0,0)[lb]{\smash{{\SetFigFont{14}{16.8}{\familydefault}{\mddefault}{\updefault}{\color[rgb]{1,1,1}$b_2$}%
}}}}
\put(17850,-16239){\makebox(0,0)[lb]{\smash{{\SetFigFont{14}{16.8}{\familydefault}{\mddefault}{\updefault}{\color[rgb]{1,1,1}$b_2$}%
}}}}
\put(7996,-16654){\makebox(0,0)[lb]{\smash{{\SetFigFont{14}{16.8}{\familydefault}{\mddefault}{\updefault}{\color[rgb]{1,1,1}$b_2$}%
}}}}
\put(8176,-6189){\makebox(0,0)[lb]{\smash{{\SetFigFont{14}{16.8}{\familydefault}{\mddefault}{\updefault}{\color[rgb]{1,1,1}$b_2$}%
}}}}
\put(3091,-5671){\makebox(0,0)[lb]{\smash{{\SetFigFont{14}{16.8}{\familydefault}{\mddefault}{\updefault}{\color[rgb]{1,1,1}$a_2$}%
}}}}
\put(15483,-6145){\makebox(0,0)[lb]{\smash{{\SetFigFont{14}{16.8}{\familydefault}{\mddefault}{\updefault}{\color[rgb]{1,1,1}$c_2$}%
}}}}
\put(1876,-16669){\makebox(0,0)[lb]{\smash{{\SetFigFont{14}{16.8}{\familydefault}{\mddefault}{\updefault}{\color[rgb]{1,1,1}$a_1$}%
}}}}
\put(11670,-16239){\makebox(0,0)[lb]{\smash{{\SetFigFont{14}{16.8}{\familydefault}{\mddefault}{\updefault}{\color[rgb]{1,1,1}$a_1$}%
}}}}
\put(12243,-6123){\makebox(0,0)[lb]{\smash{{\SetFigFont{14}{16.8}{\familydefault}{\mddefault}{\updefault}{\color[rgb]{1,1,1}$a_1$}%
}}}}
\put(2356,-6031){\makebox(0,0)[lb]{\smash{{\SetFigFont{14}{16.8}{\familydefault}{\mddefault}{\updefault}{\color[rgb]{1,1,1}$a_1$}%
}}}}
\put(5476,-16849){\makebox(0,0)[lb]{\smash{{\SetFigFont{14}{16.8}{\familydefault}{\mddefault}{\updefault}{\color[rgb]{1,1,1}$c_2$}%
}}}}
\put(15780,-16269){\makebox(0,0)[lb]{\smash{{\SetFigFont{14}{16.8}{\familydefault}{\mddefault}{\updefault}{\color[rgb]{1,1,1}$c_2$}%
}}}}
\put(5941,-6054){\makebox(0,0)[lb]{\smash{{\SetFigFont{14}{16.8}{\familydefault}{\mddefault}{\updefault}{\color[rgb]{1,1,1}$c_2$}%
}}}}
\put(2896,-16159){\makebox(0,0)[lb]{\smash{{\SetFigFont{14}{16.8}{\familydefault}{\mddefault}{\updefault}{\color[rgb]{1,1,1}$a_2$}%
}}}}
\put(13008,-5703){\makebox(0,0)[lb]{\smash{{\SetFigFont{14}{16.8}{\familydefault}{\mddefault}{\updefault}{\color[rgb]{1,1,1}$a_2$}%
}}}}
\put(12660,-15744){\makebox(0,0)[lb]{\smash{{\SetFigFont{14}{16.8}{\familydefault}{\mddefault}{\updefault}{\color[rgb]{1,1,1}$a_2$}%
}}}}
\put(19098,-5710){\makebox(0,0)[lb]{\smash{{\SetFigFont{14}{16.8}{\familydefault}{\mddefault}{\updefault}{\color[rgb]{1,1,1}$b_1$}%
}}}}
\put(18615,-15879){\makebox(0,0)[lb]{\smash{{\SetFigFont{14}{16.8}{\familydefault}{\mddefault}{\updefault}{\color[rgb]{1,1,1}$b_1$}%
}}}}
\end{picture}%
}
\caption{The four $F^{(i)}$-type regions: (a) $F^{(1)}_{2,1,3}(2,2;\ 3,2;\ 2,3)$, (b) $F^{(2)}_{2,1,3}(2,2;\ 3,2;\ 2,3)$, (c) $F^{(3)}_{2,2,3}(2,3;\ 2,3;\ 2,2)$, and (d) $F^{(4)}_{2,1,3}(2,3;\ 2,3;\ 2,2)$.  }\label{fig:offF}
\end{figure}

We now consider a similar situation to the case of $E^{(i)}$-type regions, however, $x$ and $z$ are now having opposite parities. Let $x,z$ be two nonnegative integers, $y$ an integer that may be negative, and $\textbf{a}, \textbf{b}, \textbf{c}$ three sequences of nonnegative integers that record the side-lengths of triangles in our ferns as usual. In this case, the center of our auxiliary hexagon $H_0$ is the middle point of a horizontal unit segment. There are total 8 off-central positions to remove the middle fern corresponding to the positions $1,2,3,\dots,8$ in Figure \ref{fig:offposition}(b). For $i=1,2,4,5,6,8$, we start with the auxiliary $H_0$ of side-lengths $x,z+2,z,x,z+2,z$, while $i=3,7$, we start with the auxiliary of side-lengths $x,z,z,x,z,z$. We assume in addition that $y\geq 0$ if $i=3,7$, $y\geq \max(a-b,-2)$ if $i=4,5,6$, and $y\geq \max(b-a,-2)$ if $i=2,1,8$. Next, we apply the same side-pushing procedure as in the definition of the $E^{(i)}$-type regions to obtain the base hexagon $H$. Finally, we remove similarly three collinear ferns from $H$ at the same level,  such that the root of the middle fern is at the off-central position $i$ as in Figure \ref{fig:offposition}(b) (and the two side ferns are defined uniquely such that left fern is touching the southwest side and the right fern is touching the northeast side of the base hexagon). Denote by $F^{(i)}_{x,y,z}(\textbf{a}; \textbf{c}; \textbf{b})$ the resulting region, $i=1,2,\dotsc,8$. However, by symmetry, we only need to enumerate four of them: the ones corresponding to the first, the second, the third, and the fourth off-central positions (see Figure \ref{fig:offF} for examples of these regions).

\begin{rmk}\label{rmkF}
In the definition of the $F^{(3)}$-type regions, if we remove the three ferns such that the root of the middle one is $1/2$-unit to the left of the center of the auxiliary hexagon $H_0$, then we obtain the region $R^{\leftarrow}_{x,y,z}(\textbf{a};\textbf{c};\textbf{b})$ in Theorem 2.3 of \cite{HoleDent}. In some sense, the $F^{(i)}$-type regions can be viewed as counterparts of the $R^{\leftarrow}$-type regions.
\end{rmk}

\begin{thm}\label{off32thm1}
Assume that $\textbf{a}=(a_1,a_2,\dotsc,a_m)$, $\textbf{b}=(b_1,b_2,\dotsc,b_n)$, $\textbf{c}=(c_1,c_2,\dotsc,c_k)$ are three sequences  of nonnegative integers  ($m,n,k$ are even) and that $x,y,z$ are three integers, such that $x\geq 0$, $y\geq \max(b-a,-2)$, $z\geq 0$, and $x$ has parity opposite to  $z$.
Then
%
\begin{align}\label{off32eq1}
\M&(F^{(1)}_{x,y,z}(\textbf{a};\textbf{c};\textbf{b}))=\Theta_{x,2y+z+2\max(a,b)+2,z}(c)\notag\\
&\times s\left(y+b-\min(a,b)+2,a_1,\dotsc, a_{m},\left\lfloor\frac{x+z}{2}\right\rfloor,c_1,\dotsc,c_{k}+\left\lceil\frac{x+z}{2}\right\rceil+b_n,b_{n-1},\dotsc,b_1\right)\notag\\
&\times s\left(a_1,\dotsc, a_{m-1},a_{m}+\left\lfloor\frac{x+z}{2}\right\rfloor+c_1,\dotsc,c_{k},\left\lceil\frac{x+z}{2}\right\rceil,b_n,\dotsc,b_1,y+a-\min(a,b)\right)\notag\\
&\times\frac{\Hf(c+\left\lfloor\frac{x+z}{2}\right\rfloor)}{\Hf(c)\Hf(\left\lfloor\frac{x+z}{2}\right\rfloor)}\frac{\Hf(\max(a,b)+y+\left\lceil\frac{x+z}{2}\right\rceil+1)}{\Hf(\max(a,b)+c+y+\left\lceil\frac{x+z}{2}\right\rceil+1)}\notag\\
&\times \frac{\Hf(\max(a,b)+y+z)\Hf(\max(a,b)+c+y+z+2)}{\Hf(\max(a,b)-o_a+o_b+o_c+y+z+2)\Hf(\max(a,b)+o_a-o_b+e_c+y+z)}\notag\\
&\times \frac{\Hf\max(a,b)-o_a+o_b+o_c+y+2)\Hf(\max(a,b)+o_a-o_b+e_c+y)}{\Hf(\max(a,b)+y)\Hf(\max(a,b)+y+2)},
\end{align}
where $\Theta_{x,y,z}(m)$ is given by (\ref{thetaeq1}).
\end{thm}
Unlike the tiling formulas for the $E^{(i)}$-type regions above, which has the non-linear term $P_1(x,y,z,m)$, this formula is a `\emph{simple}' product formula, in the sense the largest factor is linear in the parameters of the region. Moreover, the corresponding cored hexagons of these regions have \emph{not} appeared in the previous work of Ciucu--Eisenk\"{o}lbl--Krattenthaler--Zare \cite{CEKZ}.

Theorem \ref{off32thm1} gives an exact enumeration for a new type of cored hexagons with the triangular hole down-pointing.
\begin{cor}\label{cor1}
Let $x,y,z,m$ be nonnegative integers. The number of tilings of the cored hexagon, in which the down-pointing triangular hole of side $m$ has its leftmost at  position $1$ as in Figure \ref{fig:offposition}(b), equals $\Theta'_{x,y,z}(m)$.
\end{cor}
\begin{proof}
The case when $y\geq z$ is obtained directly from Theorem \ref{off32thm1} for the region $F^{(1)}_{x,\frac{y-z}{2},z}(\emptyset;\ 0,m;\ \emptyset)$. The case
 when $y<z$ is obtained by applying Theorem \ref{off32thm1} to a vertical reflection  of the region $F^{(1)}_{x,\frac{z-y}{2},z}(\emptyset;\ 0, m;\ \emptyset)$.
\end{proof}

\begin{thm}\label{off32thm2}
Assume that $\textbf{a}=(a_1,a_2,\dotsc,a_m)$, $\textbf{b}=(b_1,b_2,\dotsc,b_n)$, $\textbf{c}=(c_1,c_2,\dotsc,c_k)$ are three sequences  of nonnegative integers ($m,n,k$ are even) and that $x,y,z$ are three  integers, such that $x\geq 0$, $y\geq \max(b-a,-2)$, $z\geq 0$, and $x$ and $z$ have opposite parities. We have
%
\begin{align}\label{off32eq2}
\M&(F^{(2)}_{x,y,z}(\textbf{a};\textbf{c};\textbf{b}))=\Lambda_{x,2y+z+2\max(a,b)+2,z}(c)\notag\\
&\times s\left(y+b-\min(a,b)+2,a_1,\dotsc, a_{m},\left\lfloor\frac{x+z}{2}\right\rfloor-1,c_1,\dotsc,c_{k}+\left\lceil\frac{x+z}{2}\right\rceil+b_n+1,b_{n-1},\dotsc,b_1\right)\notag\\
&\times s\left(a_1,\dotsc, a_{m-1},a_{m}+\left\lfloor\frac{x+z}{2}\right\rfloor-1+c_1,\dotsc,c_{k},\left\lceil\frac{x+z}{2}\right\rceil+1,b_n,\dotsc,b_1,y+a-\min(a,b)\right)\notag\\
&\times\frac{\Hf(c+\left\lfloor\frac{x+z}{2}\right\rfloor-1)}{\Hf(c)\Hf(\left\lfloor\frac{x+z}{2}\right\rfloor-1)}
\frac{\Hf(\max(a,b)+y+\left\lfloor\frac{x+z}{2}\right\rfloor+1)}{\Hf(\max(a,b)+c+y+\left\lfloor\frac{x+z}{2}\right\rfloor+1)}\notag\\
&\times \frac{\Hf(\max(a,b)+y+z)\Hf(\max(a,b)+c+y+z+2)}{\Hf(\max(a,b)-o_a+o_b+o_c+y+z+2)\Hf(\max(a,b)+o_a-o_b+e_c+y+z)}\notag\\
&\times \frac{\Hf(\max(a,b)-o_a+o_b+o_c+y+2)\Hf(\max(a,b)+o_a-o_b+e_c+y)}{\Hf(\max(a,b)+y)\Hf(\max(a,b)+y+2)}.
\end{align}
where $\Lambda_{x,y,z}(m)$ is defined as in (\ref{lambdaeq1}).
  \end{thm}

\begin{thm}\label{off32thm3}
Assume that $\textbf{a}=(a_1,a_2,\dotsc,a_m)$, $\textbf{b}=(b_1,b_2,\dotsc,b_n)$, $\textbf{c}=(c_1,c_2,\dotsc,c_k)$ are three sequences  of nonnegative integers ($m,n,k$ are all even) and that $x,y,z$ are three nonnegative integers, such that $x$ has parity opposite to  $z$.  Then
%
\begin{align}\label{off32eq3}
\M&(F^{(3)}_{x,y,z}(\textbf{a};\textbf{c};\textbf{b}))=\Psi_{x,2y+z+2\max(a,b),z}(c)\notag\\
&\times s\left(y+b-\min(a,b),a_1,\dotsc, a_{m},\left\lfloor\frac{x+z}{2}\right\rfloor-1,c_1,\dotsc,c_{k}+\left\lceil\frac{x+z}{2}\right\rceil+b_n+1,b_{n-1},\dotsc,b_1\right)\notag\\
&\times s\left(a_1,\dotsc, a_{m-1},a_{m}+\left\lfloor\frac{x+z}{2}\right\rfloor-1+c_1,\dotsc,c_{k},\left\lceil\frac{x+z}{2}\right\rceil+1,b_n,\dotsc,b_1,y+a-\min(a,b)\right)\notag\\
&\times\frac{\Hf(c+\left\lfloor\frac{x+z}{2}\right\rfloor-1)}{\Hf(c)\Hf(\left\lfloor\frac{x+z}{2}\right\rfloor-1)}
\frac{\Hf(\max(a,b)+y+\left\lfloor\frac{x+z}{2}\right\rfloor-1)}{\Hf(\max(a,b)+c+y+\left\lfloor\frac{x+z}{2}\right\rfloor-1)}\notag\\
&\times \frac{\Hf(\max(a,b)+y+z)\Hf(\max(a,b)+c+y+z)}{\Hf(\max(a,b)-o_a+o_b+o_c+y+z)\Hf(\max(a,b)+o_a-o_b+e_c+y+z)}\notag\\
&\times \frac{\Hf(\max(a,b)-o_a+o_b+o_c+y)\Hf(\max(a,b)+o_a-o_b+e_c+y)}{\Hf(\max(a,b)+y)^2},
\end{align}
where $\Psi_{x,y,z}(m)$ is given by the formula (\ref{psieq1}) in Conjecture \ref{con2}.
%
\end{thm}

\bigskip

Next, we show that Theorem \ref{off32thm3} has Conjecture \ref{con2} as a special case.

\begin{proof}[Proof of Conjecture \ref{con2} by using Theorem \ref{off32thm3}]
Apply Theorem \ref{off32thm3} to  the region $F^{(1)}_{x,\frac{y-z}{2},z}(\emptyset;\ m \; \emptyset)$, we verify the conjecture for $y\geq z$, as all factors in (\ref{off32eq1}), except for the first one, cancel out. The case $y<z$ is obtained by applying the theorem to a horizontal reflection of the region $F^{(1)}_{x,\frac{z-y}{2},z}(\emptyset;\ 0, m \; \emptyset)$.
\end{proof}

We have an exact enumeration of a new type of cored hexagons whose triangular hole is down-pointing:
\begin{cor}\label{cor3}
Let $x,y,z,m$ be nonnegative integers. The number of tilings of the cored hexagon, in which the down-pointing triangular hole has its leftmost vertex at the third position as in Figure \ref{fig:offposition}(b), equals $\Psi'_{x,y,z}(m)$ in (\ref{psieq2}).
\end{cor}

We now have a tiling  enumeration for the $F^{(4)}$-type region:

\begin{thm}\label{off32thm4}
Assume that $\textbf{a}=(a_1,a_2,\dotsc,a_m)$, $\textbf{b}=(b_1,b_2,\dotsc,b_n)$, $\textbf{c}=(c_1,c_2,\dotsc,c_k)$ are three sequences  of nonnegative integers ($m,n,k$ are all even) and that $x,y,z$ are three integers, such that $x\geq 0$, $y\geq \max(a-b,-2)$, $z\geq 0$,  and $x$ has parity opposite to  $z$. Then
%
\begin{align}\label{off32eq4}
\M&(F^{(4)}_{x,y,z}(\textbf{a};\textbf{c};\textbf{b}))=\Lambda'_{x,z,2y+z+2\max(a,b)+2}(c)\notag\\
&\times s\left(y+b-\min(a,b),a_1,\dotsc, a_{m},\left\lfloor\frac{x+z}{2}\right\rfloor,c_1,\dotsc,c_{k}+\left\lceil\frac{x+z}{2}\right\rceil+b_n,b_{n-1},\dotsc,b_1\right)\notag\\
&\times s\left(a_1,\dotsc, a_{m-1},a_{m}+\left\lfloor\frac{x+z}{2}\right\rfloor+c_1,\dotsc,c_{k},\left\lceil\frac{x+z}{2}\right\rceil,b_n,\dotsc,b_1,y+a-\min(a,b)+2\right)\notag\\
&\times\frac{\Hf(c+\left\lfloor\frac{x+z}{2}\right\rfloor)}{\Hf(c)\Hf(\left\lfloor\frac{x+z}{2}\right\rfloor)}
\frac{\Hf(\max(a,b)+y+\left\lfloor\frac{x+z}{2}\right\rfloor)}{\Hf(\max(a,b)+c+y+\left\lfloor\frac{x+z}{2}\right\rfloor)}\notag\\
&\times \frac{\Hf(\max(a,b)+y+z+2)\Hf(\max(a,b)+c+y+z)}{\Hf(\max(a,b)-o_a+o_b+o_c+y+z)\Hf(\max(a,b)+o_a-o_b+e_c+y+z+2)}\notag\\
&\times \frac{\Hf(\max(a,b)-o_a+o_b+o_c+y)\Hf(\max(a,b)+o_a-o_b+e_c+y+2)}{\Hf(\max(a,b)+y)\Hf(\max(a,b)+y+2)},
\end{align}
where $\Lambda' _{x,y,z}(m)$ is defined as in (\ref{lambdaeq2}).
  \end{thm}

  \begin{cor}\label{cor2}
Let $x,y,z,m$ be nonnegative integers. The number of tilings of the cored hexagon, in which the down-pointing triangular hole has its leftmost vertex at the second position as in Figure \ref{fig:offposition}(b), equals $\Lambda' _{x,y,z}(m)$ in (\ref{lambdaeq2}).
\end{cor}
\begin{proof}
The case when $y\geq z$ is obtained directly from Theorem \ref{off32thm2} for the region $F^{(2)}_{x,\frac{y-z}{2},z}(\emptyset;\ 0,m;\ \emptyset)$. The case
 when $y<z$ is obtained by applying Theorem \ref{off32thm4} to a  horizontal reflection  of the region $F^{(1)}_{x,\frac{z-y}{2},z}(\emptyset;\ m;\ \emptyset)$.
\end{proof}

\medskip
\subsection{The $G^{(i)}$-type regions}
We now consider the case when the center of the auxiliary hexagon is the middle point of a southeast-to-northwest unit lattice interval. There are also $8$ off-central positions labeled by $1,2,\dots, 8$ around the center of the auxiliary hexagon $H_0$ as shown in Figure \ref{fig:offposition}(c). Let $x,z$ be two nonnegative integers with the same parity. In particular, for $i=1,2,5,6$, we consider the auxiliary hexagon $H_0$ of side-lengths $x,z+1,z,x,z+1,z$, while for $i=3,4,7,8$, the auxiliary hexagon has side-lengths $x,z+3,z,x,z+3,z$. The domain of the $y$-parameter can be defined in general as follows. If the position $i$ is $d$ units above the center of the auxiliary hexagon ($d=1/2$ or $3/2$ here), then $y\geq \max(a-b,-2d)$; if $i$ is $d$ units below the center, then $y\geq \max (b-a,-2d)$. Next, we apply the same side-pushing procedure as in the definition of the $E^{(i)}$-type region to the new auxiliary hexagon to get the base hexagon $H$. We then remove the three ferns at the same level similarly, such that the middle fern has the root at the off-central position $i$ ($i=1,2,\dots,8$). Let us denote by $G^{(i)}_{x,y,z}(\textbf{a}; \textbf{c}; \textbf{b})$ the newly defined regions (see Figure \ref{fig:offG} for examples). Again, by the symmetry, we only need to enumerate four of the eights regions above, namely $G^{(1)}$-, $G^{(2)}$-, $G^{(3)}$-, and $G^{(4)}$-type regions.

\begin{rmk}\label{rmkG}
In the definition of the $G^{(1)}$-type regions, if we remove the three ferns such that the root of the middle one is $1/2$-unit to the northwest of the center of the auxiliary hexagon $H_0$, then we obtain the region $R^{\nwarrow}_{x,y,z}(\textbf{a};\textbf{c};\textbf{b})$ in Theorem 2.4 of \cite{HoleDent}. It means that the $G^{(i)}$-type regions can be viewed as counterparts of the $R^{\nwarrow}$-type regions.
\end{rmk}

\begin{figure}\centering
\setlength{\unitlength}{3947sp}%
\begingroup\makeatletter\ifx\SetFigFont\undefined%
\gdef\SetFigFont#1#2#3#4#5{%
  \reset@font\fontsize{#1}{#2pt}%
  \fontfamily{#3}\fontseries{#4}\fontshape{#5}%
  \selectfont}%
\fi\endgroup%
\resizebox{15cm}{!}{
\begin{picture}(0,0)%
\includegraphics{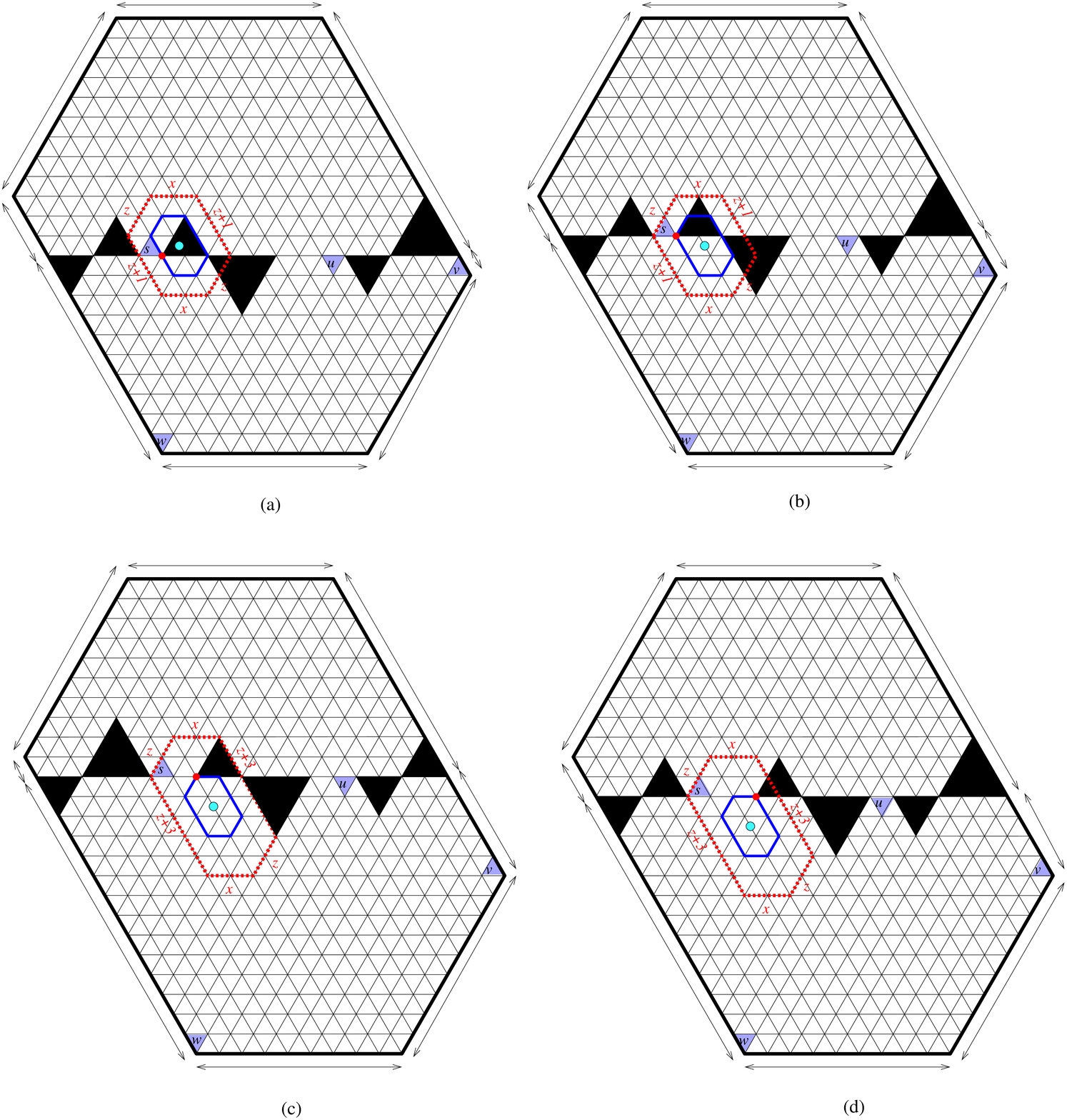}%
\end{picture}%
%
%

\begin{picture}(19430,20353)(1456,-21146)
\put(2566,-16246){\rotatebox{300.0}{\makebox(0,0)[lb]{\smash{{\SetFigFont{14}{16.8}{\familydefault}{\mddefault}{\updefault}{$y+z+o_a+e_b+e_c+(a-b)+3$}%
}}}}}
\put(12067,-6629){\rotatebox{300.0}{\makebox(0,0)[lb]{\smash{{\SetFigFont{14}{16.8}{\familydefault}{\mddefault}{\updefault}{$y+z+o_a+e_b+e_c+1$}%
}}}}}
\put(8719,-12390){\rotatebox{300.0}{\makebox(0,0)[lb]{\smash{{\SetFigFont{14}{16.8}{\rmdefault}{\mddefault}{\itdefault}{$z+o_a+o_b+o_c$}%
}}}}}
\put(19640,-5305){\rotatebox{300.0}{\makebox(0,0)[lb]{\smash{{\SetFigFont{14}{16.8}{\familydefault}{\mddefault}{\updefault}{$y+1$}%
}}}}}
\put(10534,-15384){\rotatebox{300.0}{\makebox(0,0)[lb]{\smash{{\SetFigFont{14}{16.8}{\familydefault}{\mddefault}{\updefault}{$y+(a-b)+3$}%
}}}}}
\put(10800,-4788){\rotatebox{300.0}{\makebox(0,0)[lb]{\smash{{\SetFigFont{14}{16.8}{\familydefault}{\mddefault}{\updefault}{$y+(b-a)$}%
}}}}}
\put(1846,-15076){\makebox(0,0)[lb]{\smash{{\SetFigFont{14}{16.8}{\familydefault}{\mddefault}{\updefault}{$y$}%
}}}}
\put(18693,-12670){\rotatebox{300.0}{\makebox(0,0)[lb]{\smash{{\SetFigFont{14}{16.8}{\rmdefault}{\mddefault}{\itdefault}{$z+o_a+o_b+o_c+(b-a)$}%
}}}}}
\put(20459,-15556){\rotatebox{300.0}{\makebox(0,0)[lb]{\smash{{\SetFigFont{14}{16.8}{\familydefault}{\mddefault}{\updefault}{$y+3$}%
}}}}}
\put(11543,-14804){\rotatebox{300.0}{\makebox(0,0)[lb]{\smash{{\SetFigFont{14}{16.8}{\familydefault}{\mddefault}{\updefault}{$y+(b-a)$}%
}}}}}
\put(12596,-16825){\rotatebox{300.0}{\makebox(0,0)[lb]{\smash{{\SetFigFont{14}{16.8}{\familydefault}{\mddefault}{\updefault}{$y+z+o_a+e_b+e_c+3$}%
}}}}}
\put(10026,-19153){\rotatebox{60.0}{\makebox(0,0)[lb]{\smash{{\SetFigFont{14}{16.8}{\familydefault}{\mddefault}{\updefault}{$z+o_a+e_b+e_c$}%
}}}}}
\put(19948,-19330){\rotatebox{60.0}{\makebox(0,0)[lb]{\smash{{\SetFigFont{14}{16.8}{\familydefault}{\mddefault}{\updefault}{$z+o_a+e_b+e_c$}%
}}}}}
\put(19027,-8405){\rotatebox{60.0}{\makebox(0,0)[lb]{\smash{{\SetFigFont{14}{16.8}{\familydefault}{\mddefault}{\updefault}{$z+o_a+e_b+e_c$}%
}}}}}
\put(15310,-9807){\makebox(0,0)[lb]{\smash{{\SetFigFont{14}{16.8}{\familydefault}{\mddefault}{\updefault}{$x+e_a+o_b+o_c$}%
}}}}
\put(16333,-20555){\makebox(0,0)[lb]{\smash{{\SetFigFont{14}{16.8}{\familydefault}{\mddefault}{\updefault}{$x+e_a+o_b+o_c$}%
}}}}
\put(6513,-20555){\makebox(0,0)[lb]{\smash{{\SetFigFont{14}{16.8}{\familydefault}{\mddefault}{\updefault}{$x+e_a+o_b+o_c$}%
}}}}
\put(4972,-11106){\makebox(0,0)[lb]{\smash{{\SetFigFont{14}{16.8}{\familydefault}{\mddefault}{\updefault}{$x+o_a+e_b+e_c$}%
}}}}
\put(14792,-11047){\makebox(0,0)[lb]{\smash{{\SetFigFont{14}{16.8}{\familydefault}{\mddefault}{\updefault}{$x+o_a+e_b+e_c$}%
}}}}
\put(14280,-1067){\makebox(0,0)[lb]{\smash{{\SetFigFont{14}{16.8}{\familydefault}{\mddefault}{\updefault}{$x+o_a+e_b+e_c$}%
}}}}
\put(2457,-13366){\rotatebox{60.0}{\makebox(0,0)[lb]{\smash{{\SetFigFont{14}{16.8}{\familydefault}{\mddefault}{\updefault}{$z+e_a+o_b+o_c$}%
}}}}}
\put(12276,-13366){\rotatebox{60.0}{\makebox(0,0)[lb]{\smash{{\SetFigFont{14}{16.8}{\familydefault}{\mddefault}{\updefault}{$z+e_a+o_b+o_c$}%
}}}}}
\put(11663,-3326){\rotatebox{60.0}{\makebox(0,0)[lb]{\smash{{\SetFigFont{14}{16.8}{\familydefault}{\mddefault}{\updefault}{$z+e_a+o_b+o_c$}%
}}}}}
\put(8431,-2236){\rotatebox{300.0}{\makebox(0,0)[lb]{\smash{{\SetFigFont{14}{16.8}{\rmdefault}{\mddefault}{\itdefault}{$z+o_a+o_b+o_c+(b-a)+1$}%
}}}}}
\put(10434,-5661){\makebox(0,0)[lb]{\smash{{\SetFigFont{14}{16.8}{\familydefault}{\mddefault}{\updefault}{$y$}%
}}}}
\put(9515,-8936){\rotatebox{60.0}{\makebox(0,0)[lb]{\smash{{\SetFigFont{14}{16.8}{\familydefault}{\mddefault}{\updefault}{$z+o_a+e_b+e_c$}%
}}}}}
\put(5934,-9909){\makebox(0,0)[lb]{\smash{{\SetFigFont{14}{16.8}{\familydefault}{\mddefault}{\updefault}{$x+e_a+o_b+o_c$}%
}}}}
\put(2671,-6826){\rotatebox{300.0}{\makebox(0,0)[lb]{\smash{{\SetFigFont{14}{16.8}{\familydefault}{\mddefault}{\updefault}{$y+z+o_a+e_b+e_c$}%
}}}}}
\put(1535,-4950){\rotatebox{300.0}{\makebox(0,0)[lb]{\smash{{\SetFigFont{14}{16.8}{\familydefault}{\mddefault}{\updefault}{$y+(b-a)+1$}%
}}}}}
\put(4911,-1051){\makebox(0,0)[lb]{\smash{{\SetFigFont{14}{16.8}{\familydefault}{\mddefault}{\updefault}{$x+o_a+e_b+e_c$}%
}}}}
\put(2357,-3294){\rotatebox{60.0}{\makebox(0,0)[lb]{\smash{{\SetFigFont{14}{16.8}{\familydefault}{\mddefault}{\updefault}{$z+e_a+o_b+o_c$}%
}}}}}
\put(18080,-2514){\rotatebox{300.0}{\makebox(0,0)[lb]{\smash{{\SetFigFont{14}{16.8}{\rmdefault}{\mddefault}{\itdefault}{$z+o_a+o_b+o_c+(b-a)$}%
}}}}}
\put(8261,-5968){\makebox(0,0)[lb]{\smash{{\SetFigFont{14}{16.8}{\familydefault}{\mddefault}{\updefault}{\color[rgb]{1,1,1}$b_2$}%
}}}}
\put(17467,-5614){\makebox(0,0)[lb]{\smash{{\SetFigFont{14}{16.8}{\familydefault}{\mddefault}{\updefault}{\color[rgb]{1,1,1}$b_2$}%
}}}}
\put(18081,-15771){\makebox(0,0)[lb]{\smash{{\SetFigFont{14}{16.8}{\familydefault}{\mddefault}{\updefault}{\color[rgb]{1,1,1}$b_2$}%
}}}}
\put(8466,-15417){\makebox(0,0)[lb]{\smash{{\SetFigFont{14}{16.8}{\familydefault}{\mddefault}{\updefault}{\color[rgb]{1,1,1}$b_2$}%
}}}}
\put(9346,-5356){\makebox(0,0)[lb]{\smash{{\SetFigFont{14}{16.8}{\familydefault}{\mddefault}{\updefault}{\color[rgb]{1,1,1}$b_1$}%
}}}}
\put(18474,-5101){\makebox(0,0)[lb]{\smash{{\SetFigFont{14}{16.8}{\familydefault}{\mddefault}{\updefault}{\color[rgb]{1,1,1}$b_1$}%
}}}}
\put(19119,-15166){\makebox(0,0)[lb]{\smash{{\SetFigFont{14}{16.8}{\familydefault}{\mddefault}{\updefault}{\color[rgb]{1,1,1}$b_1$}%
}}}}
\put(9294,-14836){\makebox(0,0)[lb]{\smash{{\SetFigFont{14}{16.8}{\familydefault}{\mddefault}{\updefault}{\color[rgb]{1,1,1}$b_1$}%
}}}}
\put(5056,-5341){\makebox(0,0)[lb]{\smash{{\SetFigFont{14}{16.8}{\familydefault}{\mddefault}{\updefault}{\color[rgb]{1,1,1}$c_1$}%
}}}}
\put(14236,-5236){\makebox(0,0)[lb]{\smash{{\SetFigFont{14}{16.8}{\familydefault}{\mddefault}{\updefault}{\color[rgb]{1,1,1}$c_1$}%
}}}}
\put(15691,-15166){\makebox(0,0)[lb]{\smash{{\SetFigFont{14}{16.8}{\familydefault}{\mddefault}{\updefault}{\color[rgb]{1,1,1}$c_1$}%
}}}}
\put(5611,-14851){\makebox(0,0)[lb]{\smash{{\SetFigFont{14}{16.8}{\familydefault}{\mddefault}{\updefault}{\color[rgb]{1,1,1}$c_1$}%
}}}}
\put(6061,-6061){\makebox(0,0)[lb]{\smash{{\SetFigFont{14}{16.8}{\familydefault}{\mddefault}{\updefault}{\color[rgb]{1,1,1}$c_2$}%
}}}}
\put(15436,-5641){\makebox(0,0)[lb]{\smash{{\SetFigFont{14}{16.8}{\familydefault}{\mddefault}{\updefault}{\color[rgb]{1,1,1}$c_2$}%
}}}}
\put(16621,-15796){\makebox(0,0)[lb]{\smash{{\SetFigFont{14}{16.8}{\familydefault}{\mddefault}{\updefault}{\color[rgb]{1,1,1}$c_2$}%
}}}}
\put(6616,-15391){\makebox(0,0)[lb]{\smash{{\SetFigFont{14}{16.8}{\familydefault}{\mddefault}{\updefault}{\color[rgb]{1,1,1}$c_2$}%
}}}}
\put(2746,-15346){\makebox(0,0)[lb]{\smash{{\SetFigFont{14}{16.8}{\familydefault}{\mddefault}{\updefault}{\color[rgb]{1,1,1}$a_1$}%
}}}}
\put(12781,-15661){\makebox(0,0)[lb]{\smash{{\SetFigFont{14}{16.8}{\familydefault}{\mddefault}{\updefault}{\color[rgb]{1,1,1}$a_1$}%
}}}}
\put(2941,-5941){\makebox(0,0)[lb]{\smash{{\SetFigFont{14}{16.8}{\familydefault}{\mddefault}{\updefault}{\color[rgb]{1,1,1}$a_1$}%
}}}}
\put(12196,-5626){\makebox(0,0)[lb]{\smash{{\SetFigFont{14}{16.8}{\familydefault}{\mddefault}{\updefault}{\color[rgb]{1,1,1}$a_1$}%
}}}}
\put(3781,-5491){\makebox(0,0)[lb]{\smash{{\SetFigFont{14}{16.8}{\familydefault}{\mddefault}{\updefault}{\color[rgb]{1,1,1}$a_2$}%
}}}}
\put(13006,-5176){\makebox(0,0)[lb]{\smash{{\SetFigFont{14}{16.8}{\familydefault}{\mddefault}{\updefault}{\color[rgb]{1,1,1}$a_2$}%
}}}}
\put(13606,-15226){\makebox(0,0)[lb]{\smash{{\SetFigFont{14}{16.8}{\familydefault}{\mddefault}{\updefault}{\color[rgb]{1,1,1}$a_2$}%
}}}}
\put(3781,-14731){\makebox(0,0)[lb]{\smash{{\SetFigFont{14}{16.8}{\familydefault}{\mddefault}{\updefault}{\color[rgb]{1,1,1}$a_2$}%
}}}}
\end{picture}%
}
\caption{The four $G^{(i)}$-type regions: (a) $G^{(1)}_{2,1,2}(2,2;\ 2,3;\ 3,2)$, (b) $G^{(2)}_{2,1,2}(2,2;\ 2,3;\ 3,2)$, (c)  $G^{(3)}_{2,1,2}(2,3;\ 2,3;\ 2,2)$, and (d)  $G^{(4)}_{2,1,2}(2,2;\ 2,3;\ 3,2)$.}\label{fig:offG}
\end{figure}

\begin{thm}\label{off32thm5}
Assume that $\textbf{a}=(a_1,a_2,\dotsc,a_m)$, $\textbf{b}=(b_1,b_2,\dotsc,b_n)$, $\textbf{c}=(c_1,c_2,\dotsc,c_k)$ are three sequences  of nonnegative integers ($m,n,k$ are all even) and that $x,y,z$ are three  integers, such that  $x\geq 0$, $y\geq \max(b-a,-1)$, $z\geq 0$, and $x$ and $z$ have the same parity. Then
\begin{align}\label{off32eq5}
\M&(G^{(1)}_{x,y,z}(\textbf{a};\textbf{c};\textbf{b}))=\Theta'_{2y+z+2\max(a,b)+1,z,x}(c)\notag\\
&\times s\left(y+b-\min(a,b)+1,a_1,\dotsc, a_{m},\frac{x+z}{2}-1,c_1,\dotsc,c_{k}+\frac{x+z}{2}+1+b_n,b_{n-1},\dotsc,b_1\right)\notag\\
&\times s\left(a_1,\dotsc, a_{m-1},a_{m}+\frac{x+z}{2}-1+c_1,\dotsc,c_{k},\frac{x+z}{2}+1,b_n,\dotsc,b_1,y+a-\min(a,b)\right)\notag\\
&\times\frac{\Hf(c+\frac{x+z}{2}-1)}{\Hf(c)\Hf(\frac{x+z}{2}-1)}\frac{\Hf(\max(a,b)+y+\frac{x+z}{2})}{\Hf(\max(a,b)+c+y+\frac{x+z}{2})}\notag\\
&\times \frac{\Hf(\max(a,b)+y+z)\Hf(\max(a,b)+c+y+z+1)}{\Hf(\max(a,b)-o_a+o_b+o_c+y+z+1)\Hf(\max(a,b)+o_a-o_b+e_c+y+z)}\notag\\
&\times \frac{\Hf(\max(a,b)-o_a+o_b+o_c+y+1)\Hf(\max(a,b)+o_a-o_b+e_c+y)}
{\Hf(\max(a,b)+y+1)\Hf(\max(a,b)+y)},
\end{align}
where $\Theta'_{x,y,z}(m)$ is defined as in (\ref{thetaeq2}).
\end{thm}


\begin{thm}\label{off32thm6}
Assume that $\textbf{a}=(a_1,a_2,\dotsc,a_m)$, $\textbf{b}=(b_1,b_2,\dotsc,b_n)$, $\textbf{c}=(c_1,c_2,\dotsc,c_k)$ are three sequences  of nonnegative integers ($m,n,k$ are all even) and that $x,y,z$ are three  integers, such that $x\geq 0$, $y\geq \max(a-b,-1)$, $z\geq 0$,  and $x$ and $z$ have the same parity. We have
%
\begin{align}\label{off32eq6}
\M&(G^{(2)}_{x,y,z}(\textbf{a};\textbf{c};\textbf{b}))=\Lambda'_{2y+z+2\max(a,b)+1,z,x}(c))\notag\\
&\times s\left(y+b-\min(a,b),a_1,\dotsc, a_{m},\frac{x+z}{2}-1,c_1,\dotsc,c_{k}+\frac{x+z}{2}+1+b_n,b_{n-1},\dotsc,b_1\right)\notag\\
&\times s\left(a_1,\dotsc, a_{m-1},a_{m}+\frac{x+z}{2}-1+c_1,\dotsc,c_{k},\frac{x+z}{2}+1,b_n,\dotsc,b_1,y+a-\min(a,b)+1\right)\notag\\
&\times\frac{\Hf(c+\frac{x+z}{2}-1)}{\Hf(c)\Hf(\frac{x+z}{2}-1)}\frac{\Hf(\max(a,b)+y+\frac{x+z}{2}-1)}{\Hf(\max(a,b)+c+y+\frac{x+z}{2}-1)}\notag\\
&\times \frac{\Hf(\max(a,b)+y+z+1)\Hf(\max(a,b)+c+y+z)}
{\Hf(\max(a,b)-o_a+o_b+o_c+y+z)\Hf(\max(a,b)+o_a-o_b+e_c+y+z+1)}\notag\\
&\times \frac{\Hf(\max(a,b)-o_a+o_b+o_c+y)\Hf(\max(a,b)+o_a-o_b+e_c+y+1)}
{\Hf(\max(a,b)+y+1)\Hf(\max(a,b)+y)},
\end{align}
where $\Lambda'_{x,y,z}(m)$ is defined as in (\ref{lambdaeq2}).
%
\end{thm}

\begin{thm}\label{off32thm7}
Assume that $\textbf{a}=(a_1,a_2,\dotsc,a_m)$, $\textbf{b}=(b_1,b_2,\dotsc,b_n)$, $\textbf{c}=(c_1,c_2,\dotsc,c_k)$ are three sequences  of nonnegative integers ($m,n,k$ are all even) and that $x,y,z$ are three integers, such that $x\geq 0$, $y\geq \max(a-b,-3)$, $z\geq 0$, and  $x$ and $z$ have the same parity. Then

\begin{align}\label{off32eq7}
\M&(G^{(3)}_{x,y,z}(\textbf{a};\textbf{c};\textbf{b}))=\Psi'_{2y+z+2\max(a,b)+3,z,x}(c)\notag\\
&\times s\left(y+b-\min(a,b),a_1,\dotsc, a_{m},\frac{x+z}{2},c_1,\dotsc,c_{k}+\frac{x+z}{2}+b_n,b_{n-1},\dotsc,b_1\right)\notag\\
&\times s\left(a_1,\dotsc, a_{m-1},a_{m}+\frac{x+z}{2}+c_1,\dotsc,c_{k},\frac{x+z}{2},b_n,\dotsc,b_1,y+a-\min(a,b)+3\right)\notag\\
&\times\frac{\Hf(c+\frac{x+z}{2})}{\Hf(c)\Hf(\frac{x+z}{2})}\frac{\Hf(\max(a,b)+y+\frac{x+z}{2})}{\Hf(\max(a,b)+c+y+\frac{x+z}{2})}\notag\\
&\times \frac{\Hf(\max(a,b)+y+z+3)\Hf(\max(a,b)+c+y+z)}
{\Hf(\max(a,b)-o_a+o_b+o_c+y+z)\Hf(\max(a,b)+o_a-o_b+e_c+y+z+3)}\notag\\
&\times \frac{\Hf(\max(a,b)-o_a+o_b+o_c+y)\Hf(\max(a,b)+o_a-o_b+e_c+y+3)}
{\Hf(\max(a,b)+y+3)\Hf(\max(a,b)+y)},
\end{align}
where $\Psi'_{x,y,z}(m)$ is defined as in (\ref{psieq2}).
%
\end{thm}

\begin{thm}\label{off32thm8}
Assume that $\textbf{a}=(a_1,a_2,\dotsc,a_m)$, $\textbf{b}=(b_1,b_2,\dotsc,b_n)$, $\textbf{c}=(c_1,c_2,\dotsc,c_k)$ are three sequences  of nonnegative integers ($m,n,k$ are all even) and that $x,y,z$ are three integers, such that  $x\geq 0$, $y\geq \max(a-b,-3)$, $z\geq 0$, and $x$ and $z$ have the same parity. We have

\begin{align}\label{off32eq8}
\M&(G^{(4)}_{x,y,z}(\textbf{a};\textbf{c};\textbf{b}))=\Lambda_{z,2y+z+2\max(a,b)+3,x}(c)\notag\\
&\times s\left(y+b-\min(a,b),a_1,\dotsc, a_{m},\frac{x+z}{2}+1,c_1,\dotsc,c_{k}+\frac{x+z}{2}+b_n-1,b_{n-1},\dotsc,b_1\right)\notag\\
&\times s\left(a_1,\dotsc, a_{m-1},a_{m}+\frac{x+z}{2}+c_1+1,\dotsc,c_{k},\frac{x+z}{2}-1,b_n,\dotsc,b_1,y+a-\min(a,b)+3\right)\notag\\
&\times\frac{\Hf(c+\frac{x+z}{2}+1)}{\Hf(c)\Hf(\frac{x+z}{2}+1)}\frac{\Hf(\max(a,b)+y+\frac{x+z}{2}+1)}{\Hf(\max(a,b)+c+y+\frac{x+z}{2}+1)}\notag\\
&\times \frac{\Hf(\max(a,b)+y+z+3)\Hf(\max(a,b)+c+y+z)}
{\Hf(\max(a,b)-o_a+o_b+o_c+y+z)\Hf(\max(a,b)+o_a-o_b+e_c+y+z+3)}\notag\\
&\times \frac{\Hf(\max(a,b)-o_a+o_b+o_c+y)\Hf(\max(a,b)+o_a-o_b+e_c+y+3)}
{\Hf(\max(a,b)+y+3)\Hf(\max(a,b)+y)},
\end{align}
where $\Lambda_{x,y,z}(m)$ is defined as in (\ref{lambdaeq1}).
\end{thm}

\subsection{Off-central counterparts of the $R^{\swarrow}$-type regions}
\begin{figure}\centering
\setlength{\unitlength}{3947sp}%
\begingroup\makeatletter\ifx\SetFigFont\undefined%
\gdef\SetFigFont#1#2#3#4#5{%
  \reset@font\fontsize{#1}{#2pt}%
  \fontfamily{#3}\fontseries{#4}\fontshape{#5}%
  \selectfont}%
\fi\endgroup%
\resizebox{15cm}{!}{
\begin{picture}(0,0)%
\includegraphics{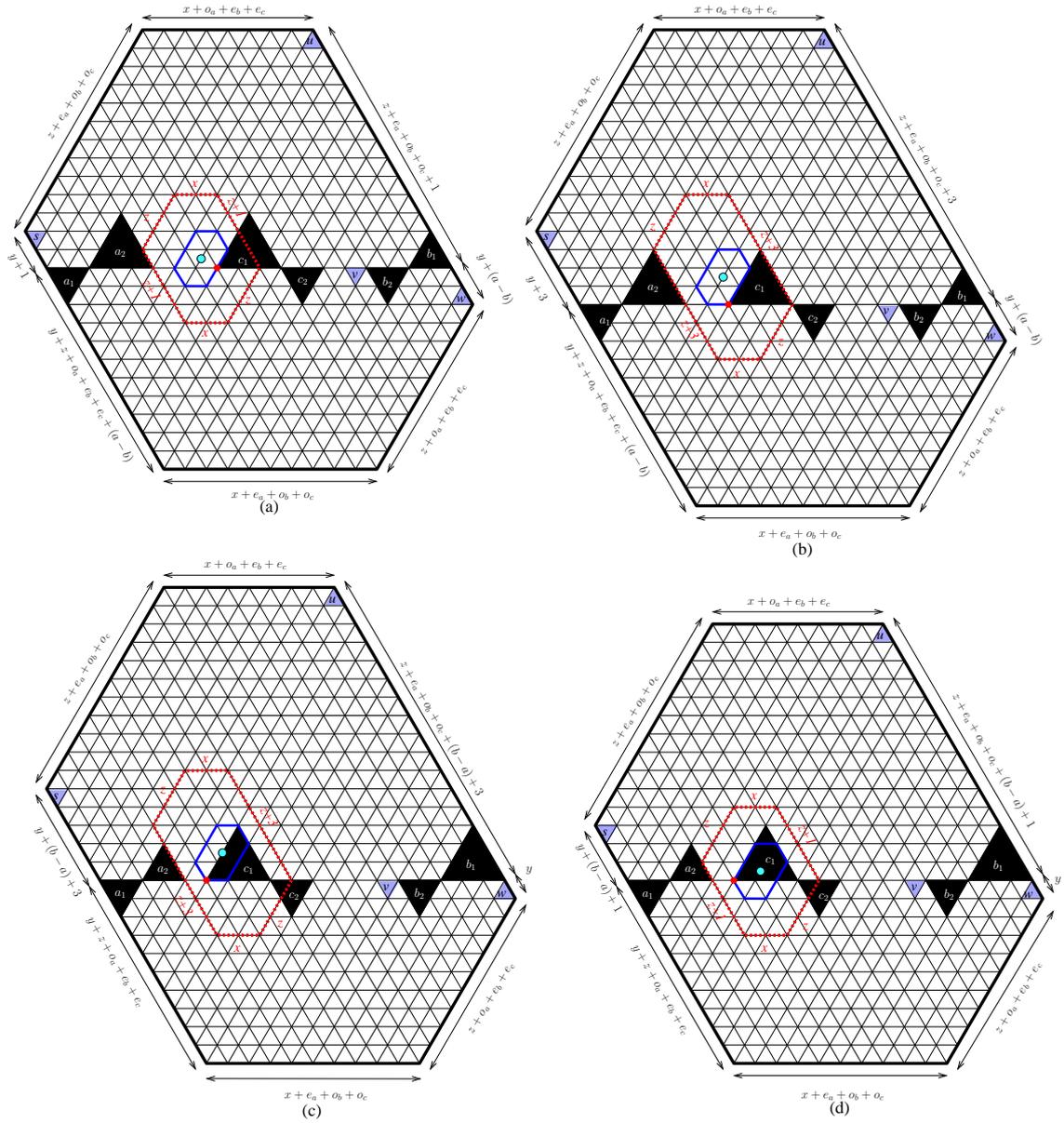}%
\end{picture}%
%
%

\begin{picture}(20308,21735)(1865,-21717)
\put(9449,-12840){\rotatebox{300.0}{\makebox(0,0)[lb]{\smash{{\SetFigFont{14}{16.8}{\rmdefault}{\mddefault}{\itdefault}{$z+e_a+o_b+o_c+(b-a)+3$}%
}}}}}
\put(20049,-13469){\rotatebox{300.0}{\makebox(0,0)[lb]{\smash{{\SetFigFont{14}{16.8}{\rmdefault}{\mddefault}{\itdefault}{$z+e_a+o_b+o_c+(b-a)+1$}%
}}}}}
\put(12846,-16362){\rotatebox{300.0}{\makebox(0,0)[lb]{\smash{{\SetFigFont{14}{16.8}{\familydefault}{\mddefault}{\updefault}{$y+(b-a)+1$}%
}}}}}
\put(13908,-18322){\rotatebox{300.0}{\makebox(0,0)[lb]{\smash{{\SetFigFont{14}{16.8}{\familydefault}{\mddefault}{\updefault}{$y+z+o_a+e_b+e_c$}%
}}}}}
\put(3440,-17787){\rotatebox{300.0}{\makebox(0,0)[lb]{\smash{{\SetFigFont{14}{16.8}{\familydefault}{\mddefault}{\updefault}{$y+z+o_a+e_b+e_c$}%
}}}}}
\put(22020,-17071){\makebox(0,0)[lb]{\smash{{\SetFigFont{14}{16.8}{\familydefault}{\mddefault}{\updefault}{$y$}%
}}}}
\put(11866,-16878){\makebox(0,0)[lb]{\smash{{\SetFigFont{14}{16.8}{\familydefault}{\mddefault}{\updefault}{$y$}%
}}}}
\put(2456,-16053){\rotatebox{300.0}{\makebox(0,0)[lb]{\smash{{\SetFigFont{14}{16.8}{\familydefault}{\mddefault}{\updefault}{$y+(b-a)+3$}%
}}}}}
\put(12646,-6685){\rotatebox{300.0}{\makebox(0,0)[lb]{\smash{{\SetFigFont{14}{16.8}{\familydefault}{\mddefault}{\updefault}{$y+z+o_a+e_b+e_c+(a-b)$}%
}}}}}
\put(19026,-2249){\rotatebox{300.0}{\makebox(0,0)[lb]{\smash{{\SetFigFont{14}{16.8}{\rmdefault}{\mddefault}{\itdefault}{$z+e_a+o_b+o_c+3$}%
}}}}}
\put(11894,-5382){\rotatebox{300.0}{\makebox(0,0)[lb]{\smash{{\SetFigFont{14}{16.8}{\familydefault}{\mddefault}{\updefault}{$y+3$}%
}}}}}
\put(21101,-5737){\rotatebox{300.0}{\makebox(0,0)[lb]{\smash{{\SetFigFont{14}{16.8}{\familydefault}{\mddefault}{\updefault}{$y+(a-b)$}%
}}}}}
\put(5688,-10988){\makebox(0,0)[lb]{\smash{{\SetFigFont{14}{16.8}{\familydefault}{\mddefault}{\updefault}{$x+o_a+e_b+e_c$}%
}}}}
\put(16121,-11697){\makebox(0,0)[lb]{\smash{{\SetFigFont{14}{16.8}{\familydefault}{\mddefault}{\updefault}{$x+o_a+e_b+e_c$}%
}}}}
\put(15098,-240){\makebox(0,0)[lb]{\smash{{\SetFigFont{14}{16.8}{\familydefault}{\mddefault}{\updefault}{$x+o_a+e_b+e_c$}%
}}}}
\put(17152,-21248){\makebox(0,0)[lb]{\smash{{\SetFigFont{14}{16.8}{\familydefault}{\mddefault}{\updefault}{$x+e_a+o_b+o_c$}%
}}}}
\put(7229,-21248){\makebox(0,0)[lb]{\smash{{\SetFigFont{14}{16.8}{\familydefault}{\mddefault}{\updefault}{$x+e_a+o_b+o_c$}%
}}}}
\put(16333,-10397){\makebox(0,0)[lb]{\smash{{\SetFigFont{14}{16.8}{\familydefault}{\mddefault}{\updefault}{$x+e_a+o_b+o_c$}%
}}}}
\put(10844,-20023){\rotatebox{60.0}{\makebox(0,0)[lb]{\smash{{\SetFigFont{14}{16.8}{\familydefault}{\mddefault}{\updefault}{$z+o_a+e_b+e_c$}%
}}}}}
\put(20971,-20023){\rotatebox{60.0}{\makebox(0,0)[lb]{\smash{{\SetFigFont{14}{16.8}{\familydefault}{\mddefault}{\updefault}{$z+o_a+e_b+e_c$}%
}}}}}
\put(20255,-9232){\rotatebox{60.0}{\makebox(0,0)[lb]{\smash{{\SetFigFont{14}{16.8}{\familydefault}{\mddefault}{\updefault}{$z+o_a+e_b+e_c$}%
}}}}}
\put(3070,-13645){\rotatebox{60.0}{\makebox(0,0)[lb]{\smash{{\SetFigFont{14}{16.8}{\familydefault}{\mddefault}{\updefault}{$z+e_a+o_b+o_c$}%
}}}}}
\put(13607,-14472){\rotatebox{60.0}{\makebox(0,0)[lb]{\smash{{\SetFigFont{14}{16.8}{\familydefault}{\mddefault}{\updefault}{$z+e_a+o_b+o_c$}%
}}}}}
\put(12481,-2854){\rotatebox{60.0}{\makebox(0,0)[lb]{\smash{{\SetFigFont{14}{16.8}{\familydefault}{\mddefault}{\updefault}{$z+e_a+o_b+o_c$}%
}}}}}
\put(10946,-4950){\rotatebox{300.0}{\makebox(0,0)[lb]{\smash{{\SetFigFont{14}{16.8}{\familydefault}{\mddefault}{\updefault}{$y+(a-b)$}%
}}}}}
\put(9104,-2012){\rotatebox{300.0}{\makebox(0,0)[lb]{\smash{{\SetFigFont{14}{16.8}{\rmdefault}{\mddefault}{\itdefault}{$z+e_a+o_b+o_c+1$}%
}}}}}
\put(9924,-8818){\rotatebox{60.0}{\makebox(0,0)[lb]{\smash{{\SetFigFont{14}{16.8}{\familydefault}{\mddefault}{\updefault}{$z+o_a+e_b+e_c$}%
}}}}}
\put(6206,-9630){\makebox(0,0)[lb]{\smash{{\SetFigFont{14}{16.8}{\familydefault}{\mddefault}{\updefault}{$x+e_a+o_b+o_c$}%
}}}}
\put(1944,-4950){\rotatebox{300.0}{\makebox(0,0)[lb]{\smash{{\SetFigFont{14}{16.8}{\familydefault}{\mddefault}{\updefault}{$y+1$}%
}}}}}
\put(2724,-6390){\rotatebox{300.0}{\makebox(0,0)[lb]{\smash{{\SetFigFont{14}{16.8}{\familydefault}{\mddefault}{\updefault}{$y+z+o_a+e_b+e_c+(a-b)$}%
}}}}}
\put(5279,-240){\makebox(0,0)[lb]{\smash{{\SetFigFont{14}{16.8}{\familydefault}{\mddefault}{\updefault}{$x+o_a+e_b+e_c$}%
}}}}
\put(2766,-2704){\rotatebox{60.0}{\makebox(0,0)[lb]{\smash{{\SetFigFont{14}{16.8}{\familydefault}{\mddefault}{\updefault}{$z+e_a+o_b+o_c$}%
}}}}}
\put(9961,-4989){\makebox(0,0)[lb]{\smash{{\SetFigFont{14}{16.8}{\familydefault}{\mddefault}{\updefault}{\color[rgb]{1,1,1}$b_1$}%
}}}}
\put(7486,-5476){\makebox(0,0)[lb]{\smash{{\SetFigFont{14}{16.8}{\familydefault}{\mddefault}{\updefault}{\color[rgb]{1,1,1}$c_2$}%
}}}}
\put(6340,-5089){\makebox(0,0)[lb]{\smash{{\SetFigFont{14}{16.8}{\familydefault}{\mddefault}{\updefault}{\color[rgb]{1,1,1}$c_1$}%
}}}}
\put(9114,-5521){\makebox(0,0)[lb]{\smash{{\SetFigFont{14}{16.8}{\familydefault}{\mddefault}{\updefault}{\color[rgb]{1,1,1}$b_2$}%
}}}}
\put(3973,-17349){\makebox(0,0)[lb]{\smash{{\SetFigFont{14}{16.8}{\familydefault}{\mddefault}{\updefault}{\color[rgb]{1,1,1}$a_1$}%
}}}}
\put(4791,-16877){\makebox(0,0)[lb]{\smash{{\SetFigFont{14}{16.8}{\familydefault}{\mddefault}{\updefault}{\color[rgb]{1,1,1}$a_2$}%
}}}}
\put(14918,-16878){\makebox(0,0)[lb]{\smash{{\SetFigFont{14}{16.8}{\familydefault}{\mddefault}{\updefault}{\color[rgb]{1,1,1}$a_2$}%
}}}}
\put(14201,-5614){\makebox(0,0)[lb]{\smash{{\SetFigFont{14}{16.8}{\familydefault}{\mddefault}{\updefault}{\color[rgb]{1,1,1}$a_2$}%
}}}}
\put(2986,-5514){\makebox(0,0)[lb]{\smash{{\SetFigFont{14}{16.8}{\familydefault}{\mddefault}{\updefault}{\color[rgb]{1,1,1}$a_1$}%
}}}}
\put(3969,-4929){\makebox(0,0)[lb]{\smash{{\SetFigFont{14}{16.8}{\familydefault}{\mddefault}{\updefault}{\color[rgb]{1,1,1}$a_2$}%
}}}}
\put(13281,-6264){\makebox(0,0)[lb]{\smash{{\SetFigFont{14}{16.8}{\familydefault}{\mddefault}{\updefault}{\color[rgb]{1,1,1}$a_1$}%
}}}}
\put(6442,-16877){\makebox(0,0)[lb]{\smash{{\SetFigFont{14}{16.8}{\familydefault}{\mddefault}{\updefault}{\color[rgb]{1,1,1}$c_1$}%
}}}}
\put(16467,-16701){\makebox(0,0)[lb]{\smash{{\SetFigFont{14}{16.8}{\familydefault}{\mddefault}{\updefault}{\color[rgb]{1,1,1}$c_1$}%
}}}}
\put(16262,-5614){\makebox(0,0)[lb]{\smash{{\SetFigFont{14}{16.8}{\familydefault}{\mddefault}{\updefault}{\color[rgb]{1,1,1}$c_1$}%
}}}}
\put(17285,-6205){\makebox(0,0)[lb]{\smash{{\SetFigFont{14}{16.8}{\familydefault}{\mddefault}{\updefault}{\color[rgb]{1,1,1}$c_2$}%
}}}}
\put(7317,-17385){\makebox(0,0)[lb]{\smash{{\SetFigFont{14}{16.8}{\familydefault}{\mddefault}{\updefault}{\color[rgb]{1,1,1}$c_2$}%
}}}}
\put(17489,-17388){\makebox(0,0)[lb]{\smash{{\SetFigFont{14}{16.8}{\familydefault}{\mddefault}{\updefault}{\color[rgb]{1,1,1}$c_2$}%
}}}}
\put(19314,-6211){\makebox(0,0)[lb]{\smash{{\SetFigFont{14}{16.8}{\familydefault}{\mddefault}{\updefault}{\color[rgb]{1,1,1}$b_2$}%
}}}}
\put(9725,-17355){\makebox(0,0)[lb]{\smash{{\SetFigFont{14}{16.8}{\familydefault}{\mddefault}{\updefault}{\color[rgb]{1,1,1}$b_2$}%
}}}}
\put(20154,-5791){\makebox(0,0)[lb]{\smash{{\SetFigFont{14}{16.8}{\familydefault}{\mddefault}{\updefault}{\color[rgb]{1,1,1}$b_1$}%
}}}}
\put(10715,-16800){\makebox(0,0)[lb]{\smash{{\SetFigFont{14}{16.8}{\familydefault}{\mddefault}{\updefault}{\color[rgb]{1,1,1}$b_1$}%
}}}}
\put(20842,-16818){\makebox(0,0)[lb]{\smash{{\SetFigFont{14}{16.8}{\familydefault}{\mddefault}{\updefault}{\color[rgb]{1,1,1}$b_1$}%
}}}}
\put(14100,-17350){\makebox(0,0)[lb]{\smash{{\SetFigFont{14}{16.8}{\familydefault}{\mddefault}{\updefault}{\color[rgb]{1,1,1}$a_1$}%
}}}}
\put(19897,-17373){\makebox(0,0)[lb]{\smash{{\SetFigFont{14}{16.8}{\familydefault}{\mddefault}{\updefault}{\color[rgb]{1,1,1}$b_2$}%
}}}}
\end{picture}%
}
\caption{Four $K^{(i)}$-type regions: (a) $K^{(1)}_{2,1,3}(2,3;\ 3,2;\ 2,2)$, (b) $K^{(2)}_{2,1,3}(2,3;\ 3,2;\ 2,2)$, (c) $K^{(3)}_{2,1,3}(2,2;\ 3,2;\ 3,2)$, and (d) $K^{(4)}_{2,1,3}(2,2;\ 3,2;\ 3,2)$.}\label{fig:offK}
\end{figure}

We consider the situation when the center of the auxiliary hexagon $H_0$ is the middle point of a southwest-to-northeast unit lattice interval. There are also $8$ off-central positions labeled by $1,2,\dots, 8$ around the center of the auxiliary hexagon $H_0$ as illustrated in Figure \ref{fig:offposition}(d). Let $x$ and $z$ be two nonnegative integers with opposite parities. For $i=1,4,5,8$, we consider the auxiliary hexagon $H_0$ of side-lengths $x,z+1,z,x,z+1,z$, while for $i=2,3,6,7$, the auxiliary hexagon has side-lengths $x,z+3,z,x,z+3,z$. The domain of the $y$-parameter is determined by the level $d$ of the position $i$ above or below the center of $H_0$ as in the definition of the $G^{(i)}$-type region ($d=1/2$ or $3/2$). Next, we apply the same side-pushing procedure as in the definition of the $E^{(i)}$-type regions and remove the three ferns at the same level, such that the middle fern has its root at the off-central position $i$ ($i=1,2,\dots,8$). Let us denote by $K^{(i)}_{x,y,z}(\textbf{a}; \textbf{c}; \textbf{b})$ the newly defined regions (see Figure \ref{fig:offK} for examples). Again, by the symmetry, we only need to enumerate four of the eights regions, namely $K^{(1)}$-, $K^{(2)}$-, $K^{(3)}$-, and $K^{(4)}$-type regions.

\begin{rmk}\label{rmkK}
In the definition of the $K^{(1)}$-type regions, if we remove the three ferns such that the root of the middle one is $1/2$-unit to the southwest of the center of the auxiliary hexagon $H_0$, then we obtain the region $R^{\swarrow}_{x,y,z}(\textbf{a};\textbf{c};\textbf{b})$ in Theorem 2.5 of \cite{HoleDent}. Therefore, we can view the $K^{(i)}$-type regions as counterparts of the $R^{\swarrow}$-type regions.
\end{rmk}

\begin{thm}\label{off32thm9}
Assume that $\textbf{a}=(a_1,a_2,\dotsc,a_m)$, $\textbf{b}=(b_1,b_2,\dotsc,b_n)$, $\textbf{c}=(c_1,c_2,\dotsc,c_k)$ are three sequences  of nonnegative integers ($m,n,k$ are all even) and that $x,y,z$ are three nonnegative integers, such that $x\geq 0$, $y\geq \max(b-a,-1)$, $z\geq 0$, and $x$ has parity opposite to  $z$. Then
%
\begin{align}\label{off32eq9}
\M&(K^{(1)}_{x,y,z}(\textbf{a};\textbf{c};\textbf{b}))=\Theta'_{z,x,2y+z+2\max(a,b)+1}(c)\notag\\
&\times s\left(y+b-\min(a,b)+1,a_1,\dotsc, a_{m},\left\lceil\frac{x+z}{2}\right\rceil,c_1,\dotsc,c_{k}+\left\lfloor\frac{x+z}{2}\right\rfloor+b_n,b_{n-1},\dotsc,b_1\right)\notag\\
&\times s\left(a_1,\dotsc, a_{m-1},a_{m}+\left\lceil\frac{x+z}{2}\right\rceil+c_1,\dotsc,c_{k},\left\lfloor\frac{x+z}{2}\right\rfloor,b_n,\dotsc,b_1,y+a-\min(a,b)\right)\notag\\
&\times\frac{\Hf(c+\left\lceil\frac{x+z}{2}\right\rceil)}{\Hf(c)\Hf(\left\lceil\frac{x+z}{2}\right\rceil)}\frac{\Hf(\max(a,b)+y+\left\lceil\frac{x+z}{2}\right\rceil+1)}{\Hf(\max(a,b)+c+y+\left\lceil\frac{x+z}{2}\right\rceil+1)}\notag\\
&\times \frac{\Hf(\max(a,b)+y+z)\Hf(\max(a,b)+c+y+z+1)}{\Hf(\max(a,b)-o_a+o_b+o_c+y+z+1)\Hf(\max(a,b)+o_a-o_b+e_c+y+z)}\notag\\
&\times \frac{\Hf(\max(a,b)-o_a+o_b+o_c+y+1)\Hf(\max(a,b)+o_a-o_b+e_c+y)}
{\Hf(\max(a,b)+y+1)\Hf(\max(a,b)+y)},
\end{align}
where the region $\Theta'_{x,y,z}(m)$ is defined as in (\ref{thetaeq2}).
%
\end{thm}

\begin{thm}\label{off32thm10}
Assume that $\textbf{a}=(a_1,a_2,\dotsc,a_m)$, $\textbf{b}=(b_1,b_2,\dotsc,b_n)$, $\textbf{c}=(c_1,c_2,\dotsc,c_k)$ are three sequences  of nonnegative integers ($m,n,k$ are all even) and that $x,y,z$ are three integers, such that $x\geq 0$, $y\geq \max(b-a,-3)$, $z\geq 0$, and $x$ has parity opposite to  $z$. Then
%
\begin{align}\label{off32eq10}
\M&(K^{(2)}_{x,y,z}(\textbf{a};\textbf{c};\textbf{b}))=\Lambda'_{z,x,2y+z+2\max(a,b)+3}(c)\notag\\
&\times s\left(y+b-\min(a,b)+3,a_1,\dotsc, a_{m},\left\lfloor\frac{x+z}{2}\right\rfloor,c_1,\dotsc,c_{k}+\left\lceil\frac{x+z}{2}\right\rceil+b_n,b_{n-1},\dotsc,b_1\right)\notag\\
&\times s\left(a_1,\dotsc, a_{m-1},a_{m}+\left\lfloor\frac{x+z}{2}\right\rfloor+c_1,\dotsc,c_{k},\left\lceil\frac{x+z}{2}\right\rceil,b_n,\dotsc,b_1,y+a-\min(a,b)\right)\notag\\
&\times\frac{\Hf(c+\left\lfloor\frac{x+z}{2}\right\rfloor)}{\Hf(c)\Hf(\left\lfloor\frac{x+z}{2}\right\rfloor)}\frac{\Hf(\max(a,b)+y+\left\lceil\frac{x+z}{2}\right\rceil+2)}{\Hf(\max(a,b)+c+y+\left\lceil\frac{x+z}{2}\right\rceil+2)}\notag\\
&\times \frac{\Hf(\max(a,b)+y+z)\Hf(\max(a,b)+c+y+z+3)}{\Hf(\max(a,b)-o_a+o_b+o_c+y+z+3)
\Hf(\max(a,b)+o_a-o_b+e_c+y+z)}\notag\\
&\times \frac{\Hf(\max(a,b)-o_a+o_b+o_c+y+3)\Hf(\max(a,b)+o_a-o_b+e_c+y)}
{\Hf(\max(a,b)+y+3)\Hf(\max(a,b)+y)},
\end{align}
where the region $\Lambda'_{x,y,z}(m)$ is defined as in (\ref{lambdaeq2}).
%
\end{thm}

\begin{thm}\label{off32thm11}
Assume that $\textbf{a}=(a_1,a_2,\dotsc,a_m)$, $\textbf{b}=(b_1,b_2,\dotsc,b_n)$, $\textbf{c}=(c_1,c_2,\dotsc,c_k)$ are three sequences  of nonnegative integers ($m,n,k$ are all even) and that $x,y,z$ are three integers, such that $x\geq 0$, $y\geq \max(b-a,-3)$, $z\geq 0$, and $x$ has parity opposite to  $x$. We get
%
\begin{align}\label{off32eq11}
\M&(K^{(3)}_{x,y,z}(\textbf{a};\textbf{c};\textbf{b}))=\Psi'_{z,x,2y+z+2\max(a,b)+3}(c)\notag\\
&\times s\left(y+b-\min(a,b)+3,a_1,\dotsc, a_{m},\left\lfloor\frac{x+z}{2}\right\rfloor-1,c_1,\dotsc,c_{k}+\left\lceil\frac{x+z}{2}\right\rceil+1+b_n,b_{n-1},\dotsc,b_1\right)\notag\\
&\times s\left(a_1,\dotsc, a_{m-1},a_{m}+\left\lfloor\frac{x+z}{2}\right\rfloor-1+c_1,\dotsc,c_{k},\left\lceil\frac{x+z}{2}\right\rceil+1,b_n,\dotsc,b_1,y+a-\min(a,b)\right)\notag\\
&\times\frac{\Hf(c+\left\lfloor\frac{x+z}{2}\right\rfloor-1)}{\Hf(c)\Hf(\left\lfloor\frac{x+z}{2}\right\rfloor-1)}\frac{\Hf(\max(a,b)+y+\left\lceil\frac{x+z}{2}\right\rceil+1)}{\Hf(\max(a,b)+c+y+\left\lceil\frac{x+z}{2}\right\rceil+1)}\notag\\
&\times \frac{\Hf(\max(a,b)+y+z)\Hf(\max(a,b)+c+y+z+3)}{\Hf(\max(a,b)-o_a+o_b+o_c+y+z+3)
\Hf(\max(a,b)+o_a-o_b+e_c+y+z)}\notag\\
&\times \frac{\Hf(\max(a,b)-o_a+o_b+o_c+y+3)\Hf(\max(a,b)+o_a-o_b+e_c+y)}
{\Hf(\max(a,b)+y+3)\Hf(\max(a,b)+y)},
\end{align}
where the region $\Psi'_{x,y,z}(m)$ is defined as in (\ref{psieq2}).
%
\end{thm}

\begin{thm}\label{off32thm12}
Assume that $\textbf{a}=(a_1,a_2,\dotsc,a_m)$, $\textbf{b}=(b_1,b_2,\dotsc,b_n)$, $\textbf{c}=(c_1,c_2,\dotsc,c_k)$ are three sequences  of nonnegative integers ($m,n,k$ are all even) and that $x,y,z$ are three integers, such that $x\geq 0$, $y\geq \max(b-a,-1)$, $z\geq 0$, and $x$ has parity opposite to  $z$. Then
%
\begin{align}\label{off32eq12}
\M&(K^{(4)}_{x,y,z}(\textbf{a};\textbf{c};\textbf{b}))=\Lambda_{2y+z+2\max(a,b)+1,x,z}(c)\notag\\
&\times s\left(y+b-\min(a,b)+1,a_1,\dotsc, a_{m},\left\lfloor\frac{x+z}{2}\right\rfloor-1,c_1,\dotsc,c_{k}+\left\lceil\frac{x+z}{2}\right\rceil+b_n+1,b_{n-1},\dotsc,b_1\right)\notag\\
&\times s\left(a_1,\dotsc, a_{m-1},a_{m}+\left\lfloor\frac{x+z}{2}\right\rfloor+c_1-1,\dotsc,c_{k},\left\lceil\frac{x+z}{2}\right\rceil+1,b_n,\dotsc,b_1,y+a-\min(a,b)\right)\notag\\
&\times\frac{\Hf(c+\left\lfloor\frac{x+z}{2}\right\rfloor-1)}{\Hf(c)\Hf(\left\lfloor\frac{x+z}{2}\right\rfloor-1)}\frac{\Hf(\max(a,b)+y+\left\lfloor\frac{x+z}{2}\right\rfloor)}{\Hf(\max(a,b)+c+y+\left\lfloor\frac{x+z}{2}\right\rfloor)}\notag\\
&\times \frac{\Hf(\max(a,b)+y+z)\Hf(\max(a,b)+c+y+z+1)}{\Hf(\max(a,b)-o_a+o_b+o_c+y+z+1)\Hf(\max(a,b)+o_a-o_b+e_c+y+z)}\notag\\
&\times \frac{\Hf(\max(a,b)-o_a+o_b+o_c+y+1)\Hf(\max(a,b)+o_a-o_b+e_c+y)}
{\Hf(\max(a,b)+y+1)\Hf(\max(a,b)+y)},
\end{align}
where the region $\Lambda_{x,y,z}(m)$ is defined as in (\ref{lambdaeq1}).
%
\end{thm}

\subsection{The $\overline{E}^{(i)}$-type regions}

In the previous four families of off-central regions, the horizontal line $l$, that contains the three removed ferns, separates the east and the west vertices of base hexagon $H$. We next consider the other situation when the line $l$ leaves the west and the east vertices of the base hexagon on the same side. Without loss of generality, we assume that these vertices are both above the line $l$.

Similar to the cases of the previous four off-central families we define our next region as follows. We (1) start with a certain auxiliary hexagon $H_0$ of certain side-lengths, (2)  perform the an edge-pushing procedure on the sides of $H_0$ to get the base hexagon $H$, and (3) remove three ferns from the resulting base hexagon at the same level, such that the leftmost point of the middle fern is at one of the off-central positions as shown in Figure \ref{fig:offposition} and that the left and the right ferns touch the \emph{northwest} and the \emph{northeast} sides of the base hexagon. The only major difference is that, in the step (2), we will perform a different pushing procedure. In particular we now push respectively the north, northeast, southeast, south, southwest, northwest sides of  $H_0$ a distance of $o_a+o_b+o_c,b+c, b+c+y+\max(0,a-b), e_a+e_b+e_c+2y+|a-b|,a+y+\max(0,b-a),a$ units. This procedure is illustrated in Figure \ref{pushing2}.

\begin{figure}
  \centering
  \setlength{\unitlength}{3947sp}%
\begingroup\makeatletter\ifx\SetFigFont\undefined%
\gdef\SetFigFont#1#2#3#4#5{%
  \reset@font\fontsize{#1}{#2pt}%
  \fontfamily{#3}\fontseries{#4}\fontshape{#5}%
  \selectfont}%
\fi\endgroup%
\resizebox{15cm}{!}{
\begin{picture}(0,0)%
\includegraphics{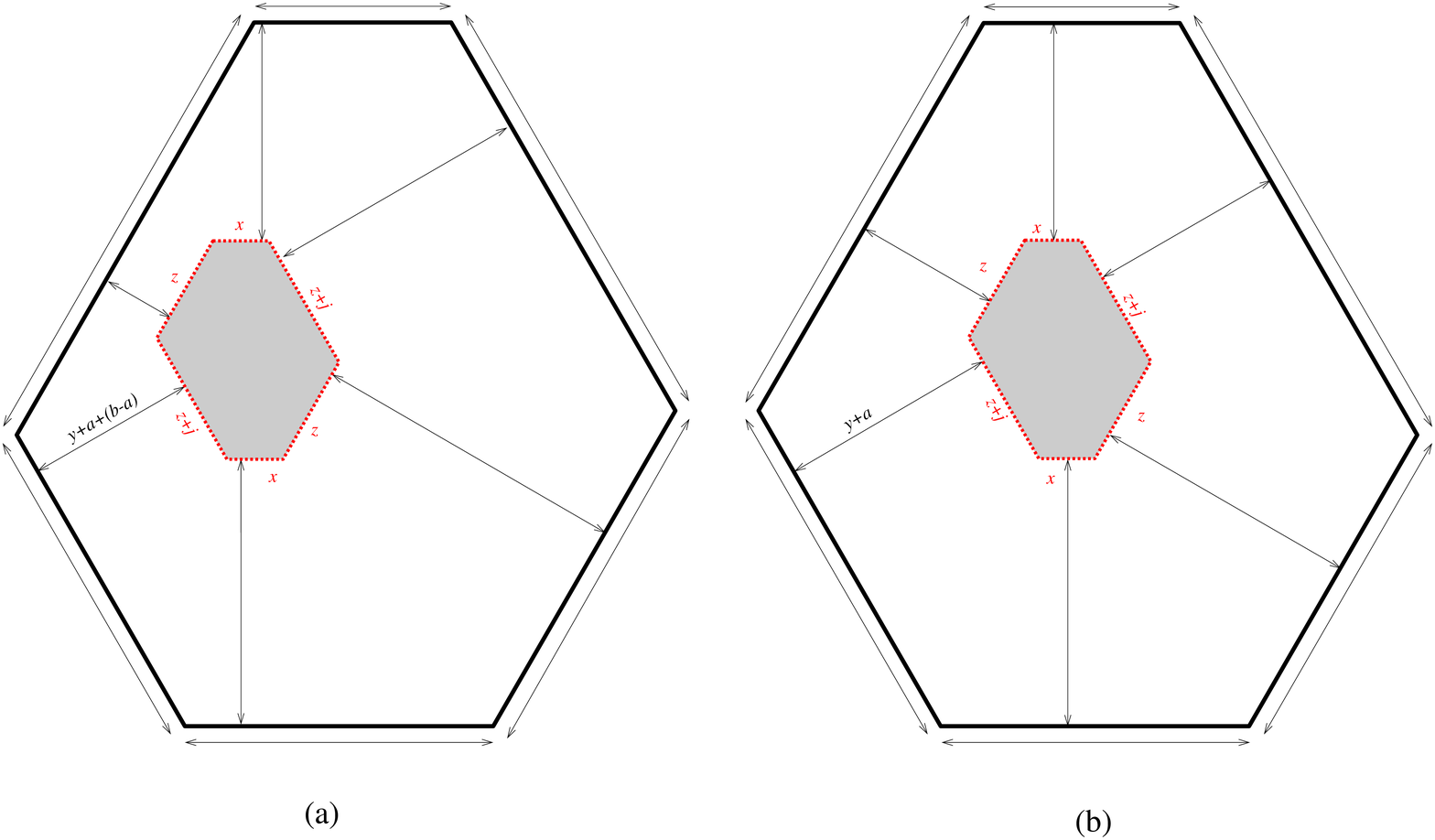}%
\end{picture}%
%
%

\begin{picture}(20891,12584)(4080,-20574)
\put(19695,-11327){\rotatebox{90.0}{\makebox(0,0)[lb]{\smash{{\SetFigFont{14}{16.8}{\rmdefault}{\mddefault}{\itdefault}{\color[rgb]{0,0,0}$o_a+o_b+o_c$}%
}}}}}
\put(23164,-18473){\rotatebox{60.0}{\makebox(0,0)[lb]{\smash{{\SetFigFont{14}{16.8}{\rmdefault}{\mddefault}{\itdefault}{\color[rgb]{0,0,0}$y+z+e_a+e_b+e_c$}%
}}}}}
\put(22662,-10085){\rotatebox{300.0}{\makebox(0,0)[lb]{\smash{{\SetFigFont{14}{16.8}{\rmdefault}{\mddefault}{\itdefault}{\color[rgb]{0,0,0}$y+z+o_a+o_b+o_c+j$}%
}}}}}
\put(7861,-18211){\rotatebox{90.0}{\makebox(0,0)[lb]{\smash{{\SetFigFont{14}{16.8}{\rmdefault}{\mddefault}{\itdefault}{\color[rgb]{0,0,0}$2y+e_a+e_b+e_c+(b-a)$}%
}}}}}
\put(10167,-14389){\rotatebox{330.0}{\makebox(0,0)[lb]{\smash{{\SetFigFont{14}{16.8}{\rmdefault}{\mddefault}{\itdefault}{\color[rgb]{0,0,0}$y+b+c$}%
}}}}}
\put(9322,-11111){\rotatebox{30.0}{\makebox(0,0)[lb]{\smash{{\SetFigFont{14}{16.8}{\rmdefault}{\mddefault}{\itdefault}{\color[rgb]{0,0,0}$b+c$}%
}}}}}
\put(8181,-11211){\rotatebox{90.0}{\makebox(0,0)[lb]{\smash{{\SetFigFont{14}{16.8}{\rmdefault}{\mddefault}{\itdefault}{\color[rgb]{0,0,0}$o_a+o_b+o_c$}%
}}}}}
\put(5628,-11535){\rotatebox{60.0}{\makebox(0,0)[lb]{\smash{{\SetFigFont{14}{16.8}{\rmdefault}{\mddefault}{\itdefault}{\color[rgb]{0,0,0}$y+z+o_a+o_b+o_c+(b-a)$}%
}}}}}
\put(11839,-9776){\rotatebox{300.0}{\makebox(0,0)[lb]{\smash{{\SetFigFont{14}{16.8}{\rmdefault}{\mddefault}{\itdefault}{\color[rgb]{0,0,0}$y+z+o_a+o_b+o_c+j$}%
}}}}}
\put(8362,-8271){\makebox(0,0)[lb]{\smash{{\SetFigFont{14}{16.8}{\rmdefault}{\mddefault}{\itdefault}{\color[rgb]{0,0,0}$x+e_a+e_b+e_c$}%
}}}}
\put(12291,-18321){\rotatebox{60.0}{\makebox(0,0)[lb]{\smash{{\SetFigFont{14}{16.8}{\rmdefault}{\mddefault}{\itdefault}{\color[rgb]{0,0,0}$y+z+e_a+e_b+e_c+(b-a)$}%
}}}}}
\put(4195,-15832){\rotatebox{300.0}{\makebox(0,0)[lb]{\smash{{\SetFigFont{14}{16.8}{\rmdefault}{\mddefault}{\itdefault}{\color[rgb]{0,0,0}$y+z+e_a+e_b+e_c+j$}%
}}}}}
\put(7945,-19492){\makebox(0,0)[lb]{\smash{{\SetFigFont{14}{16.8}{\rmdefault}{\mddefault}{\itdefault}{\color[rgb]{0,0,0}$x+o_a+o_b+o_c$}%
}}}}
\put(5934,-12541){\makebox(0,0)[lb]{\smash{{\SetFigFont{14}{16.8}{\rmdefault}{\mddefault}{\itdefault}{\color[rgb]{0,0,0}$a$}%
}}}}
\put(16264,-11544){\rotatebox{60.0}{\makebox(0,0)[lb]{\smash{{\SetFigFont{14}{16.8}{\rmdefault}{\mddefault}{\itdefault}{\color[rgb]{0,0,0}$z+o_a+o_b+o_c$}%
}}}}}
\put(18998,-8280){\makebox(0,0)[lb]{\smash{{\SetFigFont{14}{16.8}{\rmdefault}{\mddefault}{\itdefault}{\color[rgb]{0,0,0}$x+e_a+e_b+e_c$}%
}}}}
\put(18581,-19501){\makebox(0,0)[lb]{\smash{{\SetFigFont{14}{16.8}{\rmdefault}{\mddefault}{\itdefault}{\color[rgb]{0,0,0}$x+o_a+o_b+o_c$}%
}}}}
\put(19915,-18443){\rotatebox{90.0}{\makebox(0,0)[lb]{\smash{{\SetFigFont{14}{16.8}{\rmdefault}{\mddefault}{\itdefault}{\color[rgb]{0,0,0}$2y+e_a+e_b+e_c+(a-b)$}%
}}}}}
\put(21583,-15213){\rotatebox{330.0}{\makebox(0,0)[lb]{\smash{{\SetFigFont{14}{16.8}{\rmdefault}{\mddefault}{\itdefault}{\color[rgb]{0,0,0}$y+b+c+(a-b)$}%
}}}}}
\put(20984,-11695){\rotatebox{30.0}{\makebox(0,0)[lb]{\smash{{\SetFigFont{14}{16.8}{\rmdefault}{\mddefault}{\itdefault}{\color[rgb]{0,0,0}$b+c$}%
}}}}}
\put(17826,-12046){\makebox(0,0)[lb]{\smash{{\SetFigFont{14}{16.8}{\rmdefault}{\mddefault}{\itdefault}{\color[rgb]{0,0,0}$a$}%
}}}}
\put(15003,-15363){\rotatebox{300.0}{\makebox(0,0)[lb]{\smash{{\SetFigFont{14}{16.8}{\rmdefault}{\mddefault}{\itdefault}{\color[rgb]{0,0,0}$y+z+e_a+e_b+e_c+(a-b)+j$}%
}}}}}
\end{picture}}
\caption{The edge-pushing procedure used in the definitions of the $\overline{E}^{(i)}$-, $\overline{F}^{(i)}$-, $\overline{G}^{(i)}$-, $\overline{K}^{(i)}$-type regions: (a) the case $a\leq b$, (b) the case $a\geq b$.}\label{pushing2}
\end{figure}

 Let $x,z$ be two nonnegative integers of the same parity, and let $y$ be an integer that may be negative as indicated particularly in the statements of the theorems below. For $i=1$ or $4$, we perform the above pushing procedure on the auxiliary hexagon $H_0$ of side-lengths $x,z,z,x,z,z$.  This way we get the base hexagon $H$ of side-lengths $x+e_a+e_b+e_c, y+z+o_a+o_b+o_c+\max(a-b,0),y+z+e_a+e_b+e_c+\max(b-a,0),x+o_a+o_b+o_c,y+z+e_a+e_b+e_c+\max(a-b,0),y+z+o_a+o_b+o_c+\max(b-a,0)$. Denote by $\overline{E}^{(i)}$, the region corresponding to the choice of the middle fern whose leftmost is at the off-central position $i$ in Figure \ref{fig:offposition}(a), for $i=1,4$.  We do the same in the case $i=2,3$, the only difference is that we start with an auxiliary hexagon of side-lengths $x,z+2,z,x,z+2,z$ (as opposed to the one of side-lengths $x,z,z,x,z,z$). See Figure \ref{fig:offE2} for examples.

We note that in this family (and in the next three families) we will only allow the root of the middle fern to be above or on the same level as the center of the auxiliary hexagon $H_0$. In particularly, our region here is not defined for the case $i=5,6$.

 Again, by symmetries, it is enough for us to enumerate only three of the four new regions, namely $\overline{E}^{(1)}$-, $\overline{E}^{(2)}$-,$\overline{E}^{(3)}$-type regions, as any $\overline{E}^{(4)}$-type region is simply a reflection of some $\overline{E}^{(1)}$-type one over a vertical line.

 \begin{rmk}\label{rmkQE}
In the definition of the $\overline{E}^{(1)}$-type regions, if we remove the three ferns such that the root of the middle one is at the center of the auxiliary hexagon $H_0$, then we obtain the region $Q^{\odot}_{x,y,z}(\textbf{a};\textbf{c};\textbf{b})$ in Theorem 2.5 of \cite{HoleDent}. In some sense, the $\overline{E}^{(i)}$-type regions can be viewed as counterparts of the $Q^{\odot}$-type regions.
\end{rmk}

\begin{figure}\centering
\setlength{\unitlength}{3947sp}%
\begingroup\makeatletter\ifx\SetFigFont\undefined%
\gdef\SetFigFont#1#2#3#4#5{%
  \reset@font\fontsize{#1}{#2pt}%
  \fontfamily{#3}\fontseries{#4}\fontshape{#5}%
  \selectfont}%
\fi\endgroup%
\resizebox{15cm}{!}{
\begin{picture}(0,0)%
\includegraphics{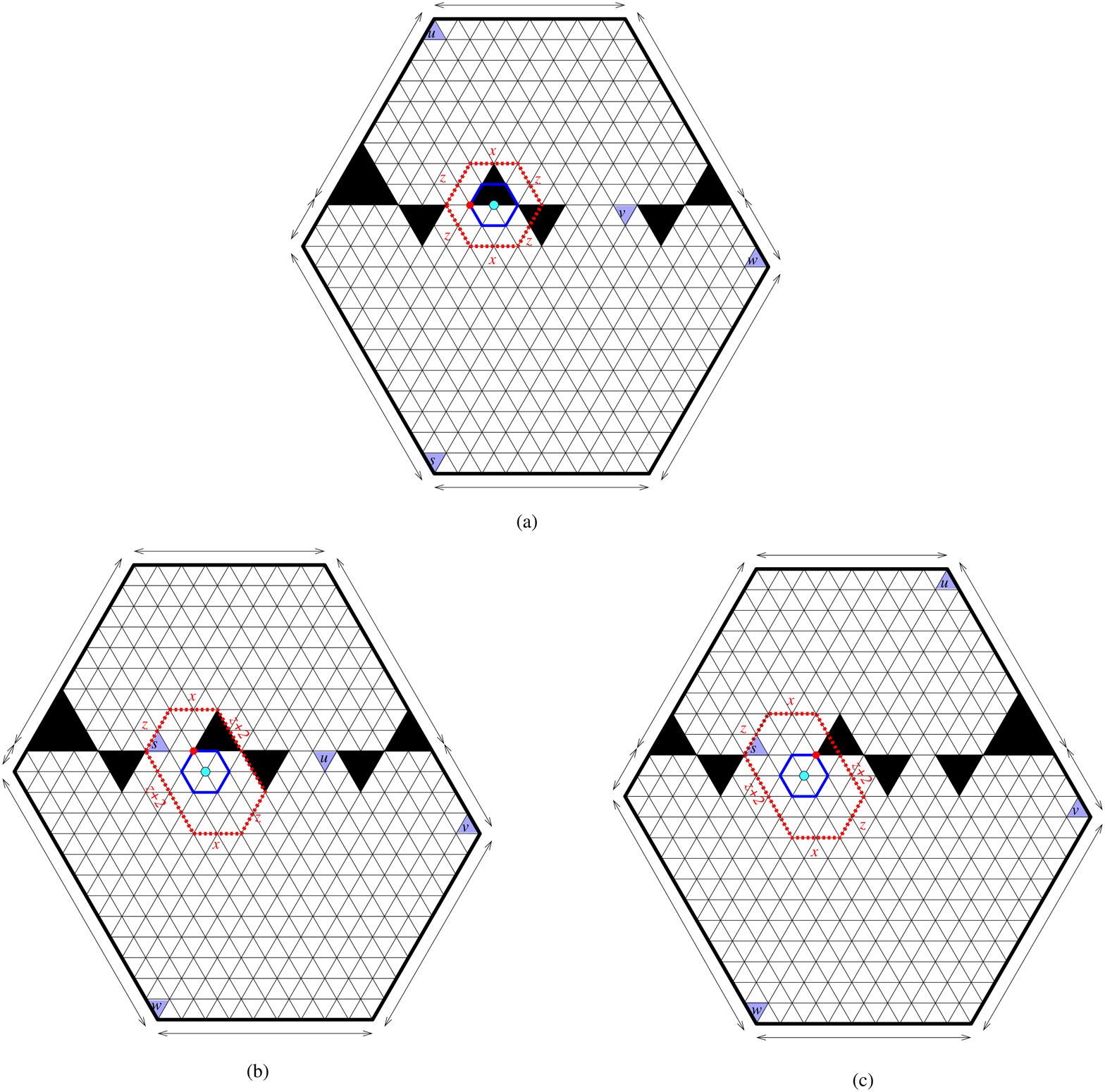}%
\end{picture}%

\begin{picture}(19730,18980)(758,-18838)
\put(11641,-14793){\rotatebox{300.0}{\makebox(0,0)[lb]{\smash{{\SetFigFont{14}{16.8}{\rmdefault}{\mddefault}{\itdefault}{$y+z+e_a+e_b+e_c+2$}%
}}}}}
\put(19644,-13316){\rotatebox{300.0}{\makebox(0,0)[lb]{\smash{{\SetFigFont{14}{16.8}{\rmdefault}{\mddefault}{\itdefault}{$y+2$}%
}}}}}
\put(11401,-13426){\rotatebox{60.0}{\makebox(0,0)[lb]{\smash{{\SetFigFont{14}{16.8}{\rmdefault}{\mddefault}{\itdefault}{$y+b-a$}%
}}}}}
\put(18093,-10656){\rotatebox{300.0}{\makebox(0,0)[lb]{\smash{{\SetFigFont{14}{16.8}{\rmdefault}{\mddefault}{\itdefault}{$z+o_a+o_b+o_c$}%
}}}}}
\put(12250,-12105){\rotatebox{60.0}{\makebox(0,0)[lb]{\smash{{\SetFigFont{14}{16.8}{\rmdefault}{\mddefault}{\itdefault}{$z+o_a+o_b+o_c$}%
}}}}}
\put(1593,-12051){\rotatebox{60.0}{\makebox(0,0)[lb]{\smash{{\SetFigFont{14}{16.8}{\rmdefault}{\mddefault}{\itdefault}{$z+o_a+o_b+o_c$}%
}}}}}
\put(7511,-10654){\rotatebox{300.0}{\makebox(0,0)[lb]{\smash{{\SetFigFont{14}{16.8}{\rmdefault}{\mddefault}{\itdefault}{$z+o_a+o_b+o_c$}%
}}}}}
\put(8992,-13144){\rotatebox{300.0}{\makebox(0,0)[lb]{\smash{{\SetFigFont{14}{16.8}{\rmdefault}{\mddefault}{\itdefault}{$y+a-b+2$}%
}}}}}
\put(963,-13019){\rotatebox{60.0}{\makebox(0,0)[lb]{\smash{{\SetFigFont{14}{16.8}{\rmdefault}{\mddefault}{\itdefault}{$y$}%
}}}}}
\put(1345,-14747){\rotatebox{300.0}{\makebox(0,0)[lb]{\smash{{\SetFigFont{14}{16.8}{\rmdefault}{\mddefault}{\itdefault}{$y+z+e_a+e_b+e_c+a-b+2$}%
}}}}}
\put(18687,-16856){\rotatebox{60.0}{\makebox(0,0)[lb]{\smash{{\SetFigFont{14}{16.8}{\rmdefault}{\mddefault}{\itdefault}{$y+z+e_a+e_b+e_c+b-a$}%
}}}}}
\put(8255,-17049){\rotatebox{60.0}{\makebox(0,0)[lb]{\smash{{\SetFigFont{14}{16.8}{\rmdefault}{\mddefault}{\itdefault}{$y+z+e_a+e_b+e_c$}%
}}}}}
\put(14697,-9573){\makebox(0,0)[lb]{\smash{{\SetFigFont{14}{16.8}{\rmdefault}{\mddefault}{\itdefault}{$x+e_a+e_b+e_c$}%
}}}}
\put(4062,-9481){\makebox(0,0)[lb]{\smash{{\SetFigFont{14}{16.8}{\rmdefault}{\mddefault}{\itdefault}{$x+e_a+e_b+e_c$}%
}}}}
\put(15001,-18183){\makebox(0,0)[lb]{\smash{{\SetFigFont{14}{16.8}{\rmdefault}{\mddefault}{\itdefault}{$x+o_a+o_b+o_c$}%
}}}}
\put(4689,-18128){\makebox(0,0)[lb]{\smash{{\SetFigFont{14}{16.8}{\rmdefault}{\mddefault}{\itdefault}{$x+o_a+o_b+o_c$}%
}}}}
\put(13014,-7595){\rotatebox{60.0}{\makebox(0,0)[lb]{\smash{{\SetFigFont{14}{16.8}{\rmdefault}{\mddefault}{\itdefault}{$y+z+e_a+e_b+e_c$}%
}}}}}
\put(9354,-8750){\makebox(0,0)[lb]{\smash{{\SetFigFont{14}{16.8}{\rmdefault}{\mddefault}{\itdefault}{$x+o_a+o_b+o_c$}%
}}}}
\put(6197,-5277){\rotatebox{300.0}{\makebox(0,0)[lb]{\smash{{\SetFigFont{14}{16.8}{\rmdefault}{\mddefault}{\itdefault}{$y+z+e_a+e_b+e_c+a-b$}%
}}}}}
\put(6024,-3890){\rotatebox{60.0}{\makebox(0,0)[lb]{\smash{{\SetFigFont{14}{16.8}{\rmdefault}{\mddefault}{\itdefault}{$y$}%
}}}}}
\put(6789,-2600){\rotatebox{60.0}{\makebox(0,0)[lb]{\smash{{\SetFigFont{14}{16.8}{\rmdefault}{\mddefault}{\itdefault}{$z+o_a+o_b+o_c$}%
}}}}}
\put(14109,-3747){\rotatebox{300.0}{\makebox(0,0)[lb]{\smash{{\SetFigFont{14}{16.8}{\rmdefault}{\mddefault}{\itdefault}{$y+a-b$}%
}}}}}
\put(12542,-1085){\rotatebox{300.0}{\makebox(0,0)[lb]{\smash{{\SetFigFont{14}{16.8}{\rmdefault}{\mddefault}{\itdefault}{$z+o_a+o_b+o_c$}%
}}}}}
\put(9212,-147){\makebox(0,0)[lb]{\smash{{\SetFigFont{14}{16.8}{\rmdefault}{\mddefault}{\itdefault}{$x+e_a+e_b+e_c$}%
}}}}
\put(8132,-3912){\makebox(0,0)[lb]{\smash{{\SetFigFont{14}{16.8}{\familydefault}{\mddefault}{\updefault}{\color[rgb]{1,1,1}$a_2$}%
}}}}
\put(7119,-3342){\makebox(0,0)[lb]{\smash{{\SetFigFont{14}{16.8}{\familydefault}{\mddefault}{\updefault}{\color[rgb]{1,1,1}$a_1$}%
}}}}
\put(13007,-3492){\makebox(0,0)[lb]{\smash{{\SetFigFont{14}{16.8}{\familydefault}{\mddefault}{\updefault}{\color[rgb]{1,1,1}$b_1$}%
}}}}
\put(2985,-13269){\makebox(0,0)[lb]{\smash{{\SetFigFont{14}{16.8}{\familydefault}{\mddefault}{\updefault}{\color[rgb]{1,1,1}$a_2$}%
}}}}
\put(18316,-12836){\makebox(0,0)[lb]{\smash{{\SetFigFont{14}{16.8}{\familydefault}{\mddefault}{\updefault}{\color[rgb]{1,1,1}$b_1$}%
}}}}
\put(12390,-12954){\makebox(0,0)[lb]{\smash{{\SetFigFont{14}{16.8}{\familydefault}{\mddefault}{\updefault}{\color[rgb]{1,1,1}$a_1$}%
}}}}
\put(17353,-13373){\makebox(0,0)[lb]{\smash{{\SetFigFont{14}{16.8}{\familydefault}{\mddefault}{\updefault}{\color[rgb]{1,1,1}$b_2$}%
}}}}
\put(10232,-3979){\makebox(0,0)[lb]{\smash{{\SetFigFont{14}{16.8}{\familydefault}{\mddefault}{\updefault}{\color[rgb]{1,1,1}$c_2$}%
}}}}
\put(9347,-3514){\makebox(0,0)[lb]{\smash{{\SetFigFont{14}{16.8}{\familydefault}{\mddefault}{\updefault}{\color[rgb]{1,1,1}$c_1$}%
}}}}
\put(16161,-13410){\makebox(0,0)[lb]{\smash{{\SetFigFont{14}{16.8}{\familydefault}{\mddefault}{\updefault}{\color[rgb]{1,1,1}$c_2$}%
}}}}
\put(15269,-12898){\makebox(0,0)[lb]{\smash{{\SetFigFont{14}{16.8}{\familydefault}{\mddefault}{\updefault}{\color[rgb]{1,1,1}$c_1$}%
}}}}
\put(13243,-13339){\makebox(0,0)[lb]{\smash{{\SetFigFont{14}{16.8}{\familydefault}{\mddefault}{\updefault}{\color[rgb]{1,1,1}$a_2$}%
}}}}
\put(1927,-12826){\makebox(0,0)[lb]{\smash{{\SetFigFont{14}{16.8}{\familydefault}{\mddefault}{\updefault}{\color[rgb]{1,1,1}$a_1$}%
}}}}
\put(4601,-12885){\makebox(0,0)[lb]{\smash{{\SetFigFont{14}{16.8}{\familydefault}{\mddefault}{\updefault}{\color[rgb]{1,1,1}$c_1$}%
}}}}
\put(5493,-13332){\makebox(0,0)[lb]{\smash{{\SetFigFont{14}{16.8}{\familydefault}{\mddefault}{\updefault}{\color[rgb]{1,1,1}$c_2$}%
}}}}
\put(7094,-13303){\makebox(0,0)[lb]{\smash{{\SetFigFont{14}{16.8}{\familydefault}{\mddefault}{\updefault}{\color[rgb]{1,1,1}$b_2$}%
}}}}
\put(12242,-3942){\makebox(0,0)[lb]{\smash{{\SetFigFont{14}{16.8}{\familydefault}{\mddefault}{\updefault}{\color[rgb]{1,1,1}$b_2$}%
}}}}
\put(7859,-12853){\makebox(0,0)[lb]{\smash{{\SetFigFont{14}{16.8}{\familydefault}{\mddefault}{\updefault}{\color[rgb]{1,1,1}$b_1$}%
}}}}
\end{picture}%
}
\caption{Three $\overline{E}^{(i)}$-type regions: (a) $\overline{E}^{(1)}_{2,2,2}(3,2 ;\ 2,2;\ 2,2)$, (b) $\overline{E}^{(2)}_{2,1,2}(3,2 ;\ 2,2;\ 2,2)$, and (c) $\overline{E}^{(3)}_{2,1,2}(2,2 ;\ 2,2;\ 3,2)$. }\label{fig:offE2}
\end{figure}

\begin{thm}\label{off1thmQ1}
Assume that $\textbf{a}=(a_1,a_2,\dotsc,a_m)$, $\textbf{b}=(b_1,b_2,\dotsc,b_n)$, $\textbf{c}=(c_1,c_2,\dotsc,c_k)$ are three sequences  of nonnegative integers ($m,n,k$ are all even) and that $x,y,z$ are three nonnegative integers, such that $x$  and $z$ have the same parity. Then
\begin{align}\label{off1eqQ1}
\M&(\overline{E}^{(1)}_{x,y,z}(\textbf{a};\textbf{c};\textbf{b}))=\Phi_{x,2y+z+2\max(a,b),z}(c)\notag\\
&\times s\left(a_1,\dotsc, a_{m}+\frac{x+z}{2}-1,c_1,\dotsc,c_{k}+\frac{x+z}{2}+b_n+1,b_{n-1},\dotsc,b_1\right)\notag\\
&\times s\left(y+b-\min(a,b), a_1,\dotsc,a_{m},\frac{x+z}{2}+c_1-1,\dotsc,c_{k},\frac{x+z}{2}+1,b_n,\dotsc,b_1,y+a-\min(a,b)\right)\notag\\
&\times\frac{\Hf(c+\frac{x+z}{2}-1)}{\Hf(c)\Hf(\frac{x+z}{2}-1)}\frac{\Hf(\max(a,b)+y+\frac{x+z}{2}-1)}{\Hf(\max(a,b)+c+y+\frac{x+z}{2}-1)}\notag\\
&\times \frac{\Hf(\max(a,b)+y+z)\Hf(\max(a,b)+c+y+z)}{\Hf(o_a+o_b+o_c+z)\Hf(|a-b|+e_a+e_b+e_c+2y+z)}\notag\\
&\times \frac{\Hf(o_a+o_b+o_c)\Hf(|a-b|+e_a+e_b+e_c+2y)}{\Hf(\max(a,b)+y)^2},
\end{align}
where $\Phi_{x,y,z}(m)$ denotes the formula in (\ref{phieq}) in Conjecture \ref{con1}.
\end{thm}

\begin{thm}\label{off1thmQ2}
Assume that $\textbf{a}=(a_1,a_2,\dotsc,a_m)$, $\textbf{b}=(b_1,b_2,\dotsc,b_n)$, $\textbf{c}=(c_1,c_2,\dotsc,c_k)$ are three sequences  of nonnegative integers ($m,n,k$ are all even) and that $x,y,z$ are three  integers, such that $x\geq 0$, $y\geq \max(a-b,-2)$, $z\geq0$, and $x$ and $z$ have the same parity.  Then
\begin{align}\label{off1eqQ2}
\M&(\overline{E}^{(2)}_{x,y,z}(\textbf{a};\textbf{c};\textbf{b}))=\Phi_{2y+z+2\max(a,b)+2,x,z}(c)\notag\\
&\times s\left(a_1,\dotsc, a_{m}+\frac{x+z}{2},c_1,\dotsc,c_{k}+\frac{x+z}{2}+b_n,b_{n-1},\dotsc,b_1\right)\notag\\
&\times s\left(y+b-\min(a,b), a_1,\dotsc,a_{m},\frac{x+z}{2}+c_1,\dotsc,c_{k},\frac{x+z}{2},b_n,\dotsc,b_1,y+a-\min(a,b)+2\right)\notag\\
&\times\frac{\Hf(c+\frac{x+z}{2})}{\Hf(c)\Hf(\frac{x+z}{2})}\frac{\Hf(\max(a,b)+y+\frac{x+z}{2})}{\Hf(\max(a,b)+c+y+\frac{x+z}{2})}\notag\\
&\times \frac{\Hf(\max(a,b)+y+z+2)\Hf(\max(a,b)+c+y+z)}{\Hf(o_a+o_b+o_c+z)\Hf(|a-b|+e_a+e_b+e_c+2y+z+2)}\notag\\
&\times \frac{\Hf(o_a+o_b+o_c)\Hf(|a-b|+e_a+e_b+e_c+2y+2)}{\Hf(\max(a,b)+y)\Hf(\max(a,b)+y+2)}.
\end{align}
\end{thm}

\begin{thm}\label{off1thmQ3}
Assume that $\textbf{a}=(a_1,a_2,\dotsc,a_m)$, $\textbf{b}=(b_1,b_2,\dotsc,b_n)$, $\textbf{c}=(c_1,c_2,\dotsc,c_k)$ are three sequences  of nonnegative integers ($m,n,k$ are all even) and that $x,y,z$ are three integers, such that $x\geq 0$, $y\geq \max(a-b,-2)$, $z\geq0$, and $x$ and $z$ have the same parity. Then
\begin{align}\label{off1eqQ3}
\M&(\overline{E}^{(3)}_{x,y,z}(\textbf{a};\textbf{c};\textbf{b}))=\Phi_{z,x,2y+z+2\max(a,b)+2}(c)\notag\\
&\times s\left(a_1,\dotsc, a_{m}+\frac{x+z}{2}+1,c_1,\dotsc,c_{k}+\frac{x+z}{2}+b_n-1,b_{n-1},\dotsc,b_1\right)\notag\\
&\times s\left(y+b-\min(a,b), a_1,\dotsc,a_{m},\frac{x+z}{2}+c_1+1,\dotsc,c_{k},\frac{x+z}{2}-1,b_n,\dotsc,b_1,y+a-\min(a,b)+2\right)\notag\\
&\times\frac{\Hf(c+\frac{x+z}{2}-1)}{\Hf(c)\Hf(\frac{x+z}{2}-1)}\frac{\Hf(\max(a,b)+y+\frac{x+z}{2}+1)}{\Hf(\max(a,b)+c+y+\frac{x+z}{2}+1)}\notag\\
&\times \frac{\Hf(\max(a,b)+y+z)\Hf(\max(a,b)+c+y+z+2)}{\Hf(o_a+o_b+o_c+z)\Hf(|a-b|+e_a+e_b+e_c+2y+z+2)}\notag\\
&\times \frac{\Hf(o_a+o_b+o_c)\Hf(|a-b|+e_a+e_b+e_c+2y+2)}{\Hf(\max(a,b)+y)\Hf(\max(a,b)+y+2)}.
\end{align}
\end{thm}

\begin{figure}
  \centering
  \setlength{\unitlength}{3947sp}%
\begingroup\makeatletter\ifx\SetFigFont\undefined%
\gdef\SetFigFont#1#2#3#4#5{%
  \reset@font\fontsize{#1}{#2pt}%
  \fontfamily{#3}\fontseries{#4}\fontshape{#5}%
  \selectfont}%
\fi\endgroup%
\resizebox{15cm}{!}{
\begin{picture}(0,0)%
\includegraphics{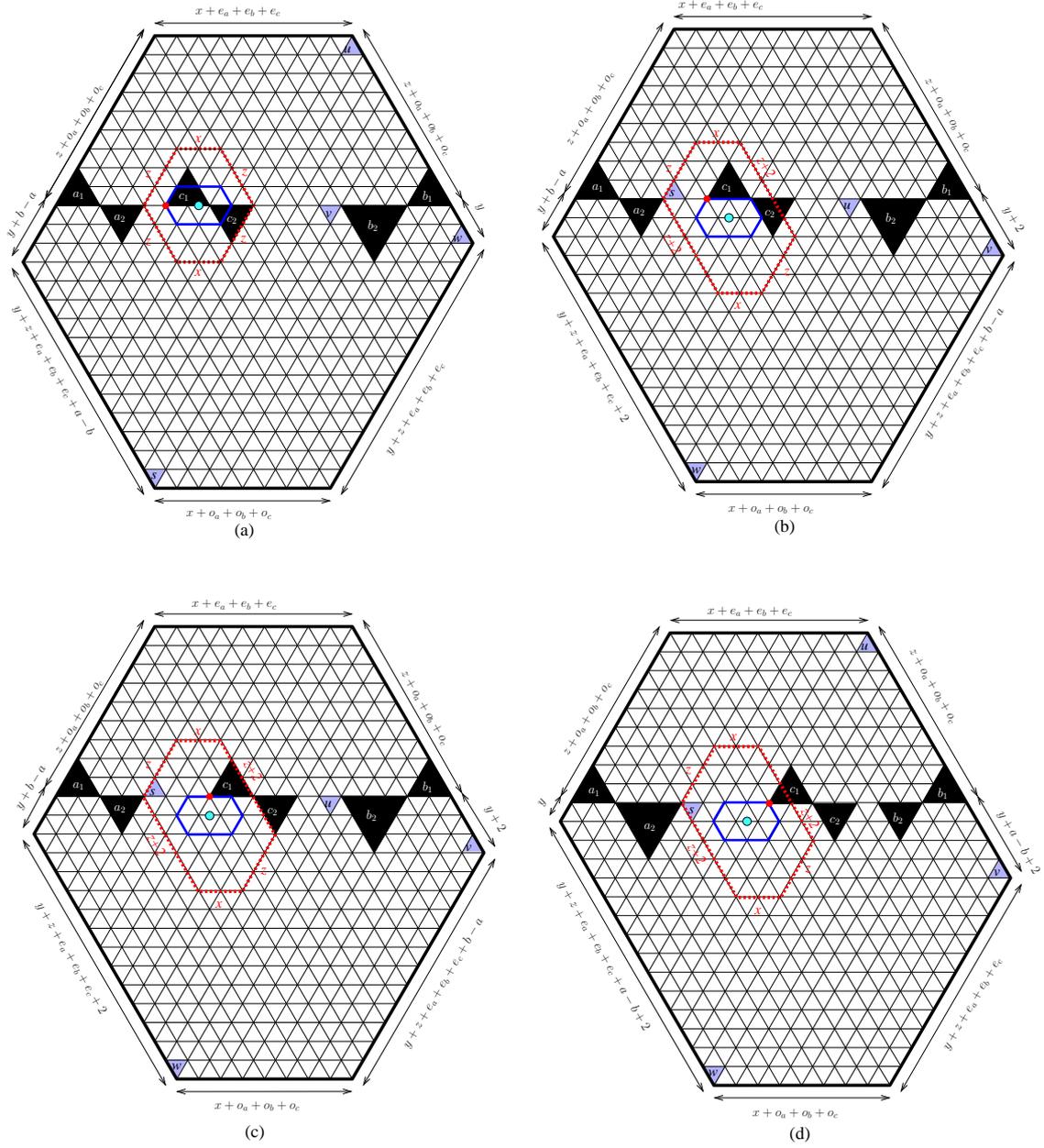}%
\end{picture}%
%
%

\begin{picture}(19432,21655)(604,-25796)
\put(9646,-19471){\rotatebox{300.0}{\makebox(0,0)[lb]{\smash{{\SetFigFont{14}{16.8}{\rmdefault}{\mddefault}{\itdefault}{$y+2$}%
}}}}}
\put(8071,-16921){\rotatebox{300.0}{\makebox(0,0)[lb]{\smash{{\SetFigFont{14}{16.8}{\rmdefault}{\mddefault}{\itdefault}{$z+o_a+o_b+o_c$}%
}}}}}
\put(19201,-19561){\rotatebox{300.0}{\makebox(0,0)[lb]{\smash{{\SetFigFont{14}{16.8}{\rmdefault}{\mddefault}{\itdefault}{$y+a-b+2$}%
}}}}}
\put(9478,-8143){\rotatebox{300.0}{\makebox(0,0)[lb]{\smash{{\SetFigFont{14}{16.8}{\rmdefault}{\mddefault}{\itdefault}{$y$}%
}}}}}
\put(17546,-16512){\rotatebox{300.0}{\makebox(0,0)[lb]{\smash{{\SetFigFont{14}{16.8}{\rmdefault}{\mddefault}{\itdefault}{$z+o_a+o_b+o_c$}%
}}}}}
\put(17836,-5776){\rotatebox{300.0}{\makebox(0,0)[lb]{\smash{{\SetFigFont{14}{16.8}{\rmdefault}{\mddefault}{\itdefault}{$z+o_a+o_b+o_c$}%
}}}}}
\put(19304,-8197){\rotatebox{300.0}{\makebox(0,0)[lb]{\smash{{\SetFigFont{14}{16.8}{\rmdefault}{\mddefault}{\itdefault}{$y+2$}%
}}}}}
\put(8116,-5836){\rotatebox{300.0}{\makebox(0,0)[lb]{\smash{{\SetFigFont{14}{16.8}{\rmdefault}{\mddefault}{\itdefault}{$z+o_a+o_b+o_c$}%
}}}}}
\put(11190,-18490){\rotatebox{60.0}{\makebox(0,0)[lb]{\smash{{\SetFigFont{14}{16.8}{\rmdefault}{\mddefault}{\itdefault}{$z+o_a+o_b+o_c$}%
}}}}}
\put(1709,-18451){\rotatebox{60.0}{\makebox(0,0)[lb]{\smash{{\SetFigFont{14}{16.8}{\rmdefault}{\mddefault}{\itdefault}{$z+o_a+o_b+o_c$}%
}}}}}
\put(11243,-7024){\rotatebox{60.0}{\makebox(0,0)[lb]{\smash{{\SetFigFont{14}{16.8}{\rmdefault}{\mddefault}{\itdefault}{$z+o_a+o_b+o_c$}%
}}}}}
\put(1709,-7261){\rotatebox{60.0}{\makebox(0,0)[lb]{\smash{{\SetFigFont{14}{16.8}{\rmdefault}{\mddefault}{\itdefault}{$z+o_a+o_b+o_c$}%
}}}}}
\put(10706,-19430){\rotatebox{60.0}{\makebox(0,0)[lb]{\smash{{\SetFigFont{14}{16.8}{\rmdefault}{\mddefault}{\itdefault}{$y$}%
}}}}}
\put(989,-19501){\rotatebox{60.0}{\makebox(0,0)[lb]{\smash{{\SetFigFont{14}{16.8}{\rmdefault}{\mddefault}{\itdefault}{$y+b-a$}%
}}}}}
\put(10694,-8296){\rotatebox{60.0}{\makebox(0,0)[lb]{\smash{{\SetFigFont{14}{16.8}{\rmdefault}{\mddefault}{\itdefault}{$y+b-a$}%
}}}}}
\put(809,-8731){\rotatebox{60.0}{\makebox(0,0)[lb]{\smash{{\SetFigFont{14}{16.8}{\rmdefault}{\mddefault}{\itdefault}{$y+b-a$}%
}}}}}
\put(11011,-10156){\rotatebox{300.0}{\makebox(0,0)[lb]{\smash{{\SetFigFont{14}{16.8}{\rmdefault}{\mddefault}{\itdefault}{$y+z+e_a+e_b+e_c+2$}%
}}}}}
\put(844,-10058){\rotatebox{300.0}{\makebox(0,0)[lb]{\smash{{\SetFigFont{14}{16.8}{\rmdefault}{\mddefault}{\itdefault}{$y+z+e_a+e_b+e_c+a-b$}%
}}}}}
\put(1276,-21286){\rotatebox{300.0}{\makebox(0,0)[lb]{\smash{{\SetFigFont{14}{16.8}{\rmdefault}{\mddefault}{\itdefault}{$y+z+e_a+e_b+e_c+2$}%
}}}}}
\put(11006,-20896){\rotatebox{300.0}{\makebox(0,0)[lb]{\smash{{\SetFigFont{14}{16.8}{\rmdefault}{\mddefault}{\itdefault}{$y+z+e_a+e_b+e_c+a-b+2$}%
}}}}}
\put(18278,-24018){\rotatebox{60.0}{\makebox(0,0)[lb]{\smash{{\SetFigFont{14}{16.8}{\rmdefault}{\mddefault}{\itdefault}{$y+z+e_a+e_b+e_c$}%
}}}}}
\put(8213,-23980){\rotatebox{60.0}{\makebox(0,0)[lb]{\smash{{\SetFigFont{14}{16.8}{\rmdefault}{\mddefault}{\itdefault}{$y+z+e_a+e_b+e_c+b-a$}%
}}}}}
\put(17926,-12511){\rotatebox{60.0}{\makebox(0,0)[lb]{\smash{{\SetFigFont{14}{16.8}{\rmdefault}{\mddefault}{\itdefault}{$y+z+e_a+e_b+e_c+b-a$}%
}}}}}
\put(7883,-12828){\rotatebox{60.0}{\makebox(0,0)[lb]{\smash{{\SetFigFont{14}{16.8}{\rmdefault}{\mddefault}{\itdefault}{$y+z+e_a+e_b+e_c$}%
}}}}}
\put(14627,-25251){\makebox(0,0)[lb]{\smash{{\SetFigFont{14}{16.8}{\rmdefault}{\mddefault}{\itdefault}{$x+o_a+o_b+o_c$}%
}}}}
\put(4592,-25123){\makebox(0,0)[lb]{\smash{{\SetFigFont{14}{16.8}{\rmdefault}{\mddefault}{\itdefault}{$x+o_a+o_b+o_c$}%
}}}}
\put(14170,-13897){\makebox(0,0)[lb]{\smash{{\SetFigFont{14}{16.8}{\rmdefault}{\mddefault}{\itdefault}{$x+o_a+o_b+o_c$}%
}}}}
\put(4067,-13986){\makebox(0,0)[lb]{\smash{{\SetFigFont{14}{16.8}{\rmdefault}{\mddefault}{\itdefault}{$x+o_a+o_b+o_c$}%
}}}}
\put(13786,-15823){\makebox(0,0)[lb]{\smash{{\SetFigFont{14}{16.8}{\rmdefault}{\mddefault}{\itdefault}{$x+e_a+e_b+e_c$}%
}}}}
\put(4171,-15688){\makebox(0,0)[lb]{\smash{{\SetFigFont{14}{16.8}{\rmdefault}{\mddefault}{\itdefault}{$x+e_a+e_b+e_c$}%
}}}}
\put(13241,-4430){\makebox(0,0)[lb]{\smash{{\SetFigFont{14}{16.8}{\rmdefault}{\mddefault}{\itdefault}{$x+e_a+e_b+e_c$}%
}}}}
\put(4240,-4548){\makebox(0,0)[lb]{\smash{{\SetFigFont{14}{16.8}{\rmdefault}{\mddefault}{\itdefault}{$x+e_a+e_b+e_c$}%
}}}}
\put(3928,-8008){\makebox(0,0)[lb]{\smash{{\SetFigFont{14}{16.8}{\familydefault}{\mddefault}{\updefault}{\color[rgb]{1,1,1}$c_1$}%
}}}}
\put(4741,-19081){\makebox(0,0)[lb]{\smash{{\SetFigFont{14}{16.8}{\familydefault}{\mddefault}{\updefault}{\color[rgb]{1,1,1}$c_1$}%
}}}}
\put(14049,-7847){\makebox(0,0)[lb]{\smash{{\SetFigFont{14}{16.8}{\familydefault}{\mddefault}{\updefault}{\color[rgb]{1,1,1}$c_1$}%
}}}}
\put(15316,-19171){\makebox(0,0)[lb]{\smash{{\SetFigFont{14}{16.8}{\familydefault}{\mddefault}{\updefault}{\color[rgb]{1,1,1}$c_1$}%
}}}}
\put(8469,-8034){\makebox(0,0)[lb]{\smash{{\SetFigFont{14}{16.8}{\familydefault}{\mddefault}{\updefault}{\color[rgb]{1,1,1}$b_1$}%
}}}}
\put(11633,-7850){\makebox(0,0)[lb]{\smash{{\SetFigFont{14}{16.8}{\familydefault}{\mddefault}{\updefault}{\color[rgb]{1,1,1}$a_1$}%
}}}}
\put(12403,-8286){\makebox(0,0)[lb]{\smash{{\SetFigFont{14}{16.8}{\familydefault}{\mddefault}{\updefault}{\color[rgb]{1,1,1}$a_2$}%
}}}}
\put(7449,-8589){\makebox(0,0)[lb]{\smash{{\SetFigFont{14}{16.8}{\familydefault}{\mddefault}{\updefault}{\color[rgb]{1,1,1}$b_2$}%
}}}}
\put(14852,-8329){\makebox(0,0)[lb]{\smash{{\SetFigFont{14}{16.8}{\familydefault}{\mddefault}{\updefault}{\color[rgb]{1,1,1}$c_2$}%
}}}}
\put(1951,-7974){\makebox(0,0)[lb]{\smash{{\SetFigFont{14}{16.8}{\familydefault}{\mddefault}{\updefault}{\color[rgb]{1,1,1}$a_1$}%
}}}}
\put(12601,-19861){\makebox(0,0)[lb]{\smash{{\SetFigFont{14}{16.8}{\familydefault}{\mddefault}{\updefault}{\color[rgb]{1,1,1}$a_2$}%
}}}}
\put(2720,-8413){\makebox(0,0)[lb]{\smash{{\SetFigFont{14}{16.8}{\familydefault}{\mddefault}{\updefault}{\color[rgb]{1,1,1}$a_2$}%
}}}}
\put(17108,-8447){\makebox(0,0)[lb]{\smash{{\SetFigFont{14}{16.8}{\familydefault}{\mddefault}{\updefault}{\color[rgb]{1,1,1}$b_2$}%
}}}}
\put(16066,-19741){\makebox(0,0)[lb]{\smash{{\SetFigFont{14}{16.8}{\familydefault}{\mddefault}{\updefault}{\color[rgb]{1,1,1}$c_2$}%
}}}}
\put(18152,-7910){\makebox(0,0)[lb]{\smash{{\SetFigFont{14}{16.8}{\familydefault}{\mddefault}{\updefault}{\color[rgb]{1,1,1}$b_1$}%
}}}}
\put(1951,-19077){\makebox(0,0)[lb]{\smash{{\SetFigFont{14}{16.8}{\familydefault}{\mddefault}{\updefault}{\color[rgb]{1,1,1}$a_1$}%
}}}}
\put(2721,-19513){\makebox(0,0)[lb]{\smash{{\SetFigFont{14}{16.8}{\familydefault}{\mddefault}{\updefault}{\color[rgb]{1,1,1}$a_2$}%
}}}}
\put(5575,-19566){\makebox(0,0)[lb]{\smash{{\SetFigFont{14}{16.8}{\familydefault}{\mddefault}{\updefault}{\color[rgb]{1,1,1}$c_2$}%
}}}}
\put(7426,-19674){\makebox(0,0)[lb]{\smash{{\SetFigFont{14}{16.8}{\familydefault}{\mddefault}{\updefault}{\color[rgb]{1,1,1}$b_2$}%
}}}}
\put(8470,-19137){\makebox(0,0)[lb]{\smash{{\SetFigFont{14}{16.8}{\familydefault}{\mddefault}{\updefault}{\color[rgb]{1,1,1}$b_1$}%
}}}}
\put(18078,-19249){\makebox(0,0)[lb]{\smash{{\SetFigFont{14}{16.8}{\familydefault}{\mddefault}{\updefault}{\color[rgb]{1,1,1}$b_1$}%
}}}}
\put(17244,-19726){\makebox(0,0)[lb]{\smash{{\SetFigFont{14}{16.8}{\familydefault}{\mddefault}{\updefault}{\color[rgb]{1,1,1}$b_2$}%
}}}}
\put(4820,-8451){\makebox(0,0)[lb]{\smash{{\SetFigFont{14}{16.8}{\familydefault}{\mddefault}{\updefault}{\color[rgb]{1,1,1}$c_2$}%
}}}}
\put(11566,-19195){\makebox(0,0)[lb]{\smash{{\SetFigFont{14}{16.8}{\familydefault}{\mddefault}{\updefault}{\color[rgb]{1,1,1}$a_1$}%
}}}}
\end{picture}%
}
  \caption{Four $\overline{F}^{(i)}$-type regions: (a) $\overline{F}^{(3)}_{2,2,3}(2,2;\ 2,2;\ 2,3)$, (b) $\overline{F}^{(4)}_{2,1,3}(2,2;\ 2,2;\ 2,3)$, (c) $\overline{F}^{(5)}_{2,1,3}(2,2;\ 2,2;\ 2,3)$, and (d) $\overline{F}^{(6)}_{2,1,3}(2,3;\ 2,2;\ 2,2)$.}\label{fig:offF2}
\end{figure}

\subsection{The $\overline{F}^{(i)}$-type regions}
 We now consider the case when $x$ and $z$ have opposite parities. For $i=4,5,6$, we start with an auxiliary hexagon $H_0$ of side-lengths $x,z+2,z,x,z+2,z$, and for $i=3,7$, we start with the auxiliary hexagon of side-lengths $x,z,z,x,z,z$. We still push all six sides of the auxiliary hexagon as in the definition of the $\overline{E}^{(i)}$-type regions to get the base hexagon $H$, and remove the three ferns from the resulting hexagon $H$, such that the root of the middle fern is now at the off-central position $i$ as shown in Figure \ref{fig:offposition}(b) (the left and right ferns are also determined by the position of the middle one). Similar to the case of the $\overline{E}^{(i)}$-type regions, there are no counterparts for the case $i=1,2,8$. Denote by $\overline{F}^{(i)}_{x,y,z}(\textbf{a}; \ \textbf{c};\ \textbf{b})$ the resulting regions (illustrated in Figure \ref{fig:offF2}). By the symmetry, we only need to enumerate four of the five regions (as any $\overline{F}^{(7)}$-type region is a vertical refection of an $\overline{F}^{(3)}$-type one).

 \begin{rmk}\label{rmkQF}
In the definition of the $\overline{F}^{(3)}$-type regions, if we remove the three ferns such that the root of the middle one is $1/2$-unit to the left of the center of the auxiliary hexagon $H_0$, then we obtain the region $Q^{\leftarrow}_{x,y,z}(\textbf{a};\textbf{c};\textbf{b})$ in Theorem 2.6 of \cite{HoleDent}.
\end{rmk}

\begin{thm}\label{off1thmQF3}
Assume that $\textbf{a}=(a_1,a_2,\dotsc,a_m)$, $\textbf{b}=(b_1,b_2,\dotsc,b_n)$, $\textbf{c}=(c_1,c_2,\dotsc,c_k)$ are three sequences  of nonnegative integers  ($m,n,k$ are all even) and that $x,y,z$ are three nonnegative integers, such that $x$ and $z$ have the different parities. Then

\begin{align}\label{off1eqQF3}
\M&(\overline{F}^{(3)}_{x,y,z}(\textbf{a};\textbf{c};\textbf{b}))=\Psi_{z,2y+z+2\max(a,b),x}(c)\notag\\
&\times s\left(a_1,\dotsc, a_{m}+\left\lfloor\frac{x+z}{2}\right\rfloor-1,c_1,\dotsc,c_{k}+\left\lceil\frac{x+z}{2}\right\rceil+b_n+1,b_{n-1},\dotsc,b_1\right)\notag\\
&\times s\left(y+b-\min(a,b), a_1,\dotsc,a_{m},\left\lfloor\frac{x+z}{2}\right\rfloor+c_1-1,\dotsc,c_{k},\left\lceil\frac{x+z}{2}\right\rceil+1,b_n,\dotsc,b_1,y+a-\min(a,b)\right)\notag\\
&\times\frac{\Hf(c+\left\lfloor\frac{x+z}{2}\right\rfloor-1)}{\Hf(c)\Hf(\left\lfloor\frac{x+z}{2}\right\rfloor-1)}\frac{\Hf(\max(a,b)+y+\left\lfloor\frac{x+z}{2}\right\rfloor-1)}{\Hf(\max(a,b)+c+y+\left\lfloor\frac{x+z}{2}\right\rfloor-1)}\notag\\
&\times \frac{\Hf(\max(a,b)+y+z)\Hf(\max(a,b)+c+y+z)}{\Hf(o_a+o_b+o_c+z)\Hf(|a-b|+e_a+e_b+e_c+2y+z)}\notag\\
&\times \frac{\Hf(o_a+o_b+o_c)\Hf(|a-b|+e_a+e_b+e_c+2y)}{\Hf(\max(a,b)+y)^2},
\end{align}
where $\Psi_{x,y,z}(m)$ denotes the formula in (\ref{psieq1}).
\end{thm}

\begin{thm}\label{off1thmQF4}
Assume that $\textbf{a}=(a_1,a_2,\dotsc,a_m)$, $\textbf{b}=(b_1,b_2,\dotsc,b_n)$, $\textbf{c}=(c_1,c_2,\dotsc,c_k)$ are three sequences  of nonnegative integers ($m,n,k$ are all even) and that $x,y,z$ are three  integers, such that $x\geq 0$, $y\geq \max(a-b,-2)$, $z\geq0$, and  $x$ and $z$ have different parities. Then

\begin{align}\label{off1eqQF4}
\M&(\overline{F}^{(4)}_{x,y,z}(\textbf{a};\textbf{c};\textbf{b}))=\Lambda'_{x,z,2y+z+2\max(a,b)+2}(c)\notag\\
&\times s\left(a_1,\dotsc, a_{m}+\left\lfloor\frac{x+z}{2}\right\rfloor,c_1,\dotsc,c_{k}+\left\lceil\frac{x+z}{2}\right\rceil+b_n,b_{n-1},\dotsc,b_1\right)\notag\\
&\times s\left(y+b-\min(a,b), a_1,\dotsc,a_{m},\left\lfloor\frac{x+z}{2}\right\rfloor+c_1,\dotsc,c_{k},\left\lceil\frac{x+z}{2}\right\rceil,b_n,\dotsc,b_1,y+a-\min(a,b)+2 \right) \notag\\
&\times\frac{\Hf(c+\left\lfloor\frac{x+z}{2}\right\rfloor)}{\Hf(c)\Hf(\left\lfloor\frac{x+z}{2}\right\rfloor)}\frac{\Hf(\max(a,b)+y+\left\lfloor\frac{x+z}{2}\right\rfloor)}{\Hf(\max(a,b)+c+y+\left\lfloor\frac{x+z}{2}\right\rfloor)}\notag\\
&\times \frac{\Hf(\max(a,b)+y+z+2)\Hf(\max(a,b)+c+y+z)}{\Hf(o_a+o_b+o_c+z)\Hf(|a-b|+e_a+e_b+e_c+2y+z+2)}\notag\\
&\times \frac{\Hf(o_a+o_b+o_c)\Hf(|a-b|+e_a+e_b+e_c+2y+2)}{\Hf(\max(a,b)+y)\Hf(\max(a,b)+y+2)},
\end{align}
where $\Lambda'_{x,y,z}(m)$ is defined as in  (\ref{lambdaeq2}).
\end{thm}

\begin{thm}\label{off1thmQF5}
Assume that $\textbf{a}=(a_1,a_2,\dotsc,a_m)$, $\textbf{b}=(b_1,b_2,\dotsc,b_n)$, $\textbf{c}=(c_1,c_2,\dotsc,c_k)$ are three sequences  of nonnegative integers ($m,n,k$ are all even) and that $x,y,z$ are three integers, such that $x\geq 0$, $y\geq \max(a-b,-2)$, $z\geq0$, and $x$ and $z$ have different parities. Then
\begin{align}\label{off1eqQF5}
\M&(\overline{F}^{(5)}_{x,y,z}(\textbf{a};\textbf{c};\textbf{b}))=\Theta'_{x,2y+z+2\max(a,b)+2,z}(c)\notag\\
&\times s\left(a_1,\dotsc, a_{m}+\left\lceil\frac{x+z}{2}\right\rceil,c_1,\dotsc,c_{k}+\left\lfloor\frac{x+z}{2}\right\rfloor+b_n,b_{n-1},\dotsc,b_1\right)\notag\\
&\times s\left(y+b-\min(a,b), a_1,\dotsc,a_{m},\left\lceil\frac{x+z}{2}\right\rceil+c_1,\dotsc,c_{k},\left\lfloor\frac{x+z}{2}\right\rfloor,b_n,\dotsc,b_1,y+a-\min(a,b)+2 \right) \notag\\
&\times\frac{\Hf(c+\frac{x+z}{2})}{\Hf(c)\Hf(\frac{x+z}{2})}\frac{\Hf(\max(a,b)+y+\frac{x+z}{2}+1)}{\Hf(\max(a,b)+c+y+\frac{x+z}{2}+1)}\notag\\
&\times \frac{\Hf(\max(a,b)+y+z)\Hf(\max(a,b)+c+y+z+2)}{\Hf(o_a+o_b+o_c+z)\Hf(|a-b|+e_a+e_b+e_c+2y+z+2)}\notag\\
&\times \frac{\Hf(o_a+o_b+o_c)\Hf(a-b|+e_a+e_b+e_c+2y+2)}{\Hf(\max(a,b)+y)\Hf(\max(a,b)+y+2)},
\end{align}
where $\Theta'_{x,y,z}(m)$ is defined as in (\ref{thetaeq2}).
\end{thm}

\begin{thm}\label{off1thmQF6}
Assume that $\textbf{a}=(a_1,a_2,\dotsc,a_m)$, $\textbf{b}=(b_1,b_2,\dotsc,b_n)$, $\textbf{c}=(c_1,c_2,\dotsc,c_k)$ are three sequences  of nonnegative integers ($m,n,k$ are all even) and that $x,y,z$ are three  integers, such that $x\geq 0$, $y\geq \max(a-b,-2)$, $z\geq0$, and $x$ and $z$ have different parities. Then
\begin{align}\label{off1eqQF6}
\M&(\overline{F}^{(6)}_{x,y,z}(\textbf{a};\textbf{c};\textbf{b}))=\Lambda'_{x,2y+z+2\max(a,b)+2,z}(c)\notag\\
&\times s\left(a_1,\dotsc, a_{m}+\left\lceil\frac{x+z}{2}\right\rceil+1,c_1,\dotsc,c_{k}+\left\lfloor\frac{x+z}{2}\right\rfloor+b_n-1,b_{n-1},\dotsc,b_1\right)\notag\\
&\times s\left(y+b-\min(a,b), a_1,\dotsc,a_{m},\left\lceil\frac{x+z}{2}\right\rceil+c_1+1,\dotsc,c_{k},\left\lfloor\frac{x+z}{2}\right\rfloor-1,b_n,\dotsc,b_1,y+a-\min(a,b)+2 \right) \notag\\
&\times\frac{\Hf(c+\frac{x+z}{2}-1)}{\Hf(c)\Hf(\frac{x+z}{2}-1)}\frac{\Hf(\max(a,b)+y+\frac{x+z}{2}+1)}{\Hf(\max(a,b)+c+y+\frac{x+z}{2}+1)}\notag\\
&\times \frac{\Hf(\max(a,b)+y+z)\Hf(\max(a,b)+c+y+z+2)}{\Hf(o_a+o_b+o_c+z)\Hf(|a-b|+e_a+e_b+e_c+2y+z+2)}\notag\\
&\times \frac{\Hf(o_a+o_b+o_c)\Hf(|a-b|+e_a+e_b+e_c+2y+2)}{\Hf(\max(a,b)+y)\Hf(\max(a,b)+y+2)},
\end{align}
where $\Lambda'_{x,y,z}(m)$ is defined as in (\ref{lambdaeq2}).
\end{thm}

\begin{figure}\centering
\setlength{\unitlength}{3947sp}%
\begingroup\makeatletter\ifx\SetFigFont\undefined%
\gdef\SetFigFont#1#2#3#4#5{%
  \reset@font\fontsize{#1}{#2pt}%
  \fontfamily{#3}\fontseries{#4}\fontshape{#5}%
  \selectfont}%
\fi\endgroup%
\resizebox{15cm}{!}{
\begin{picture}(0,0)%
\includegraphics{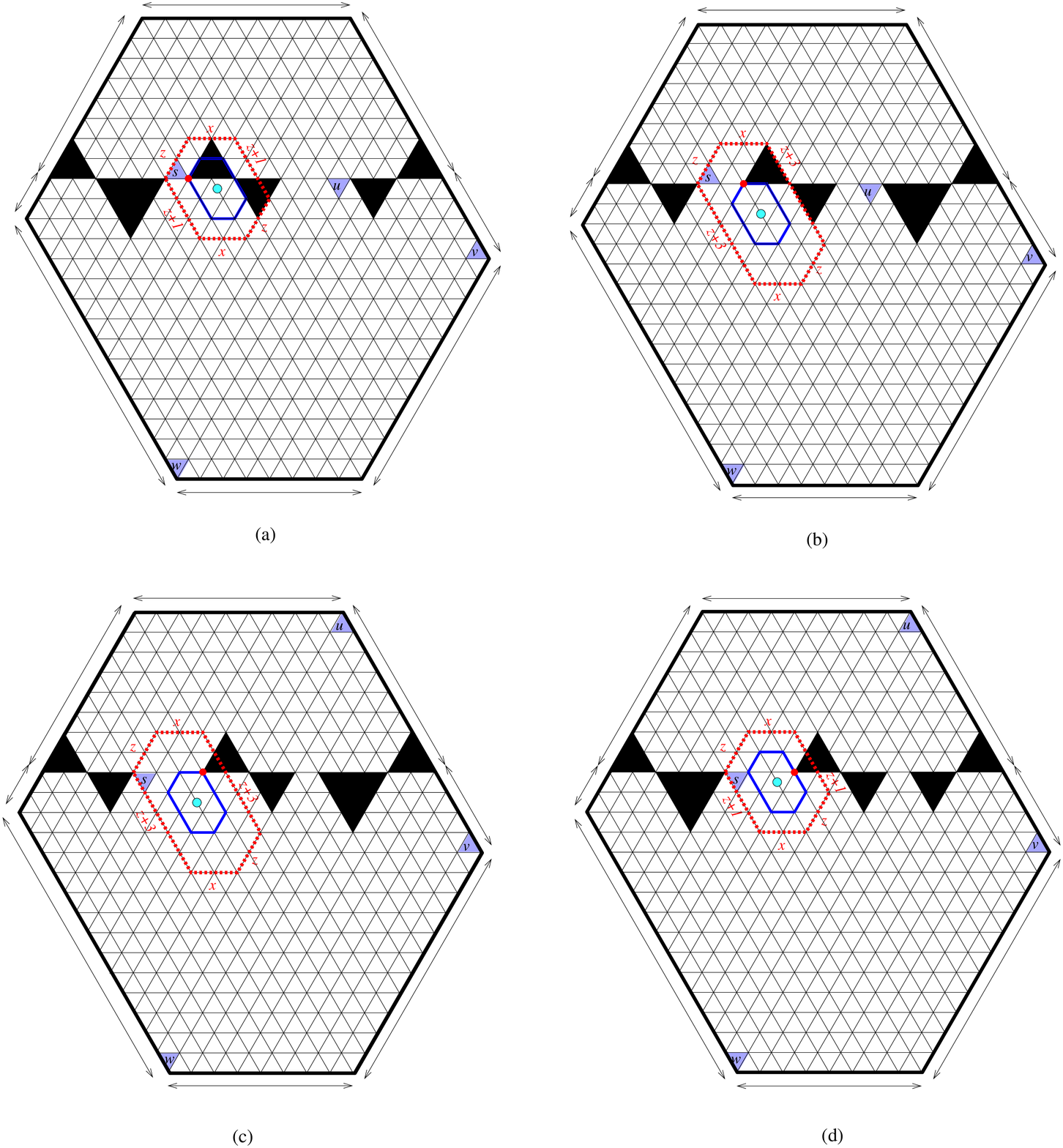}%
\end{picture}%
%
%

\begin{picture}(19276,20702)(794,-22102)
\put(1383,-6916){\rotatebox{300.0}{\makebox(0,0)[lb]{\smash{{\SetFigFont{14}{16.8}{\rmdefault}{\mddefault}{\itdefault}{$y+z+e_a+e_b+e_c+a-b+1$}%
}}}}}
\put(10934,-5536){\rotatebox{60.0}{\makebox(0,0)[lb]{\smash{{\SetFigFont{14}{16.8}{\rmdefault}{\mddefault}{\itdefault}{$y+b-a$}%
}}}}}
\put(19154,-5093){\rotatebox{300.0}{\makebox(0,0)[lb]{\smash{{\SetFigFont{14}{16.8}{\rmdefault}{\mddefault}{\itdefault}{$y+3$}%
}}}}}
\put(7777,-13000){\rotatebox{300.0}{\makebox(0,0)[lb]{\smash{{\SetFigFont{14}{16.8}{\rmdefault}{\mddefault}{\itdefault}{$z+o_a+o_b+o_c$}%
}}}}}
\put(19124,-15458){\rotatebox{300.0}{\makebox(0,0)[lb]{\smash{{\SetFigFont{14}{16.8}{\rmdefault}{\mddefault}{\itdefault}{$y+a-b+1$}%
}}}}}
\put(9164,-15593){\rotatebox{300.0}{\makebox(0,0)[lb]{\smash{{\SetFigFont{14}{16.8}{\rmdefault}{\mddefault}{\itdefault}{$y+3$}%
}}}}}
\put(989,-15901){\rotatebox{60.0}{\makebox(0,0)[lb]{\smash{{\SetFigFont{14}{16.8}{\rmdefault}{\mddefault}{\itdefault}{$y+b-a$}%
}}}}}
\put(11227,-15640){\rotatebox{60.0}{\makebox(0,0)[lb]{\smash{{\SetFigFont{14}{16.8}{\rmdefault}{\mddefault}{\itdefault}{$y$}%
}}}}}
\put(1053,-17155){\rotatebox{300.0}{\makebox(0,0)[lb]{\smash{{\SetFigFont{14}{16.8}{\rmdefault}{\mddefault}{\itdefault}{$y+z+e_a+e_b+e_c+3$}%
}}}}}
\put(11206,-7036){\rotatebox{300.0}{\makebox(0,0)[lb]{\smash{{\SetFigFont{14}{16.8}{\rmdefault}{\mddefault}{\itdefault}{$y+z+e_a+e_b+e_c+3$}%
}}}}}
\put(11161,-17101){\rotatebox{300.0}{\makebox(0,0)[lb]{\smash{{\SetFigFont{14}{16.8}{\rmdefault}{\mddefault}{\itdefault}{$y+z+e_a+e_b+e_c+a-b+1$}%
}}}}}
\put(18311,-9713){\rotatebox{60.0}{\makebox(0,0)[lb]{\smash{{\SetFigFont{14}{16.8}{\rmdefault}{\mddefault}{\itdefault}{$y+z+e_a+e_b+e_c+b-a$}%
}}}}}
\put(8126,-20258){\rotatebox{60.0}{\makebox(0,0)[lb]{\smash{{\SetFigFont{14}{16.8}{\rmdefault}{\mddefault}{\itdefault}{$y+z+e_a+e_b+e_c+b-a$}%
}}}}}
\put(18206,-20348){\rotatebox{60.0}{\makebox(0,0)[lb]{\smash{{\SetFigFont{14}{16.8}{\rmdefault}{\mddefault}{\itdefault}{$y+z+e_a+e_b+e_c$}%
}}}}}
\put(17707,-2500){\rotatebox{300.0}{\makebox(0,0)[lb]{\smash{{\SetFigFont{14}{16.8}{\rmdefault}{\mddefault}{\itdefault}{$z+o_a+o_b+o_c$}%
}}}}}
\put(17692,-12880){\rotatebox{300.0}{\makebox(0,0)[lb]{\smash{{\SetFigFont{14}{16.8}{\rmdefault}{\mddefault}{\itdefault}{$z+o_a+o_b+o_c$}%
}}}}}
\put(11554,-14790){\rotatebox{60.0}{\makebox(0,0)[lb]{\smash{{\SetFigFont{14}{16.8}{\rmdefault}{\mddefault}{\itdefault}{$z+o_a+o_b+o_c$}%
}}}}}
\put(1534,-14790){\rotatebox{60.0}{\makebox(0,0)[lb]{\smash{{\SetFigFont{14}{16.8}{\rmdefault}{\mddefault}{\itdefault}{$z+o_a+o_b+o_c$}%
}}}}}
\put(11594,-4456){\rotatebox{60.0}{\makebox(0,0)[lb]{\smash{{\SetFigFont{14}{16.8}{\rmdefault}{\mddefault}{\itdefault}{$z+o_a+o_b+o_c$}%
}}}}}
\put(14240,-1766){\makebox(0,0)[lb]{\smash{{\SetFigFont{14}{16.8}{\rmdefault}{\mddefault}{\itdefault}{$x+e_a+e_b+e_c$}%
}}}}
\put(14401,-12189){\makebox(0,0)[lb]{\smash{{\SetFigFont{14}{16.8}{\rmdefault}{\mddefault}{\itdefault}{$x+e_a+e_b+e_c$}%
}}}}
\put(4081,-12219){\makebox(0,0)[lb]{\smash{{\SetFigFont{14}{16.8}{\rmdefault}{\mddefault}{\itdefault}{$x+e_a+e_b+e_c$}%
}}}}
\put(14341,-10936){\makebox(0,0)[lb]{\smash{{\SetFigFont{14}{16.8}{\rmdefault}{\mddefault}{\itdefault}{$x+o_a+o_b+o_c$}%
}}}}
\put(4374,-21444){\makebox(0,0)[lb]{\smash{{\SetFigFont{14}{16.8}{\rmdefault}{\mddefault}{\itdefault}{$x+o_a+o_b+o_c$}%
}}}}
\put(14416,-21481){\makebox(0,0)[lb]{\smash{{\SetFigFont{14}{16.8}{\rmdefault}{\mddefault}{\itdefault}{$x+o_a+o_b+o_c$}%
}}}}
\put(1184,-5221){\rotatebox{60.0}{\makebox(0,0)[lb]{\smash{{\SetFigFont{14}{16.8}{\rmdefault}{\mddefault}{\itdefault}{$y$}%
}}}}}
\put(1684,-4335){\rotatebox{60.0}{\makebox(0,0)[lb]{\smash{{\SetFigFont{14}{16.8}{\rmdefault}{\mddefault}{\itdefault}{$z+o_a+o_b+o_c$}%
}}}}}
\put(4036,-1674){\makebox(0,0)[lb]{\smash{{\SetFigFont{14}{16.8}{\rmdefault}{\mddefault}{\itdefault}{$x+e_a+e_b+e_c$}%
}}}}
\put(7846,-2431){\rotatebox{300.0}{\makebox(0,0)[lb]{\smash{{\SetFigFont{14}{16.8}{\rmdefault}{\mddefault}{\itdefault}{$z+o_a+o_b+o_c$}%
}}}}}
\put(9391,-5011){\rotatebox{300.0}{\makebox(0,0)[lb]{\smash{{\SetFigFont{14}{16.8}{\rmdefault}{\mddefault}{\itdefault}{$y+a-b+1$}%
}}}}}
\put(8281,-9766){\rotatebox{60.0}{\makebox(0,0)[lb]{\smash{{\SetFigFont{14}{16.8}{\rmdefault}{\mddefault}{\itdefault}{$y+z+e_a+e_b+e_c$}%
}}}}}
\put(4576,-10876){\makebox(0,0)[lb]{\smash{{\SetFigFont{14}{16.8}{\rmdefault}{\mddefault}{\itdefault}{$x+o_a+o_b+o_c$}%
}}}}
\put(4636,-4838){\makebox(0,0)[lb]{\smash{{\SetFigFont{14}{16.8}{\familydefault}{\mddefault}{\updefault}{\color[rgb]{1,1,1}$c_1$}%
}}}}
\put(14435,-4870){\makebox(0,0)[lb]{\smash{{\SetFigFont{14}{16.8}{\familydefault}{\mddefault}{\updefault}{\color[rgb]{1,1,1}$c_1$}%
}}}}
\put(15388,-15250){\makebox(0,0)[lb]{\smash{{\SetFigFont{14}{16.8}{\familydefault}{\mddefault}{\updefault}{\color[rgb]{1,1,1}$c_1$}%
}}}}
\put(4918,-15261){\makebox(0,0)[lb]{\smash{{\SetFigFont{14}{16.8}{\familydefault}{\mddefault}{\updefault}{\color[rgb]{1,1,1}$c_1$}%
}}}}
\put(2799,-15786){\makebox(0,0)[lb]{\smash{{\SetFigFont{14}{16.8}{\familydefault}{\mddefault}{\updefault}{\color[rgb]{1,1,1}$a_2$}%
}}}}
\put(13015,-15887){\makebox(0,0)[lb]{\smash{{\SetFigFont{14}{16.8}{\familydefault}{\mddefault}{\updefault}{\color[rgb]{1,1,1}$a_2$}%
}}}}
\put(12782,-5350){\makebox(0,0)[lb]{\smash{{\SetFigFont{14}{16.8}{\familydefault}{\mddefault}{\updefault}{\color[rgb]{1,1,1}$a_2$}%
}}}}
\put(3181,-5394){\makebox(0,0)[lb]{\smash{{\SetFigFont{14}{16.8}{\familydefault}{\mddefault}{\updefault}{\color[rgb]{1,1,1}$a_2$}%
}}}}
\put(12040,-15280){\makebox(0,0)[lb]{\smash{{\SetFigFont{14}{16.8}{\familydefault}{\mddefault}{\updefault}{\color[rgb]{1,1,1}$a_1$}%
}}}}
\put(11942,-4878){\makebox(0,0)[lb]{\smash{{\SetFigFont{14}{16.8}{\familydefault}{\mddefault}{\updefault}{\color[rgb]{1,1,1}$a_1$}%
}}}}
\put(7078,-15914){\makebox(0,0)[lb]{\smash{{\SetFigFont{14}{16.8}{\familydefault}{\mddefault}{\updefault}{\color[rgb]{1,1,1}$b_2$}%
}}}}
\put(2000,-15292){\makebox(0,0)[lb]{\smash{{\SetFigFont{14}{16.8}{\familydefault}{\mddefault}{\updefault}{\color[rgb]{1,1,1}$a_1$}%
}}}}
\put(5683,-15786){\makebox(0,0)[lb]{\smash{{\SetFigFont{14}{16.8}{\familydefault}{\mddefault}{\updefault}{\color[rgb]{1,1,1}$c_2$}%
}}}}
\put(16138,-15775){\makebox(0,0)[lb]{\smash{{\SetFigFont{14}{16.8}{\familydefault}{\mddefault}{\updefault}{\color[rgb]{1,1,1}$c_2$}%
}}}}
\put(15253,-5373){\makebox(0,0)[lb]{\smash{{\SetFigFont{14}{16.8}{\familydefault}{\mddefault}{\updefault}{\color[rgb]{1,1,1}$c_2$}%
}}}}
\put(5394,-5311){\makebox(0,0)[lb]{\smash{{\SetFigFont{14}{16.8}{\familydefault}{\mddefault}{\updefault}{\color[rgb]{1,1,1}$c_2$}%
}}}}
\put(2128,-4793){\makebox(0,0)[lb]{\smash{{\SetFigFont{14}{16.8}{\familydefault}{\mddefault}{\updefault}{\color[rgb]{1,1,1}$a_1$}%
}}}}
\put(17398,-15797){\makebox(0,0)[lb]{\smash{{\SetFigFont{14}{16.8}{\familydefault}{\mddefault}{\updefault}{\color[rgb]{1,1,1}$b_2$}%
}}}}
\put(17068,-5485){\makebox(0,0)[lb]{\smash{{\SetFigFont{14}{16.8}{\familydefault}{\mddefault}{\updefault}{\color[rgb]{1,1,1}$b_2$}%
}}}}
\put(7441,-5274){\makebox(0,0)[lb]{\smash{{\SetFigFont{14}{16.8}{\familydefault}{\mddefault}{\updefault}{\color[rgb]{1,1,1}$b_2$}%
}}}}
\put(18163,-15302){\makebox(0,0)[lb]{\smash{{\SetFigFont{14}{16.8}{\familydefault}{\mddefault}{\updefault}{\color[rgb]{1,1,1}$b_1$}%
}}}}
\put(18110,-4893){\makebox(0,0)[lb]{\smash{{\SetFigFont{14}{16.8}{\familydefault}{\mddefault}{\updefault}{\color[rgb]{1,1,1}$b_1$}%
}}}}
\put(8289,-4801){\makebox(0,0)[lb]{\smash{{\SetFigFont{14}{16.8}{\familydefault}{\mddefault}{\updefault}{\color[rgb]{1,1,1}$b_1$}%
}}}}
\put(8143,-15299){\makebox(0,0)[lb]{\smash{{\SetFigFont{14}{16.8}{\familydefault}{\mddefault}{\updefault}{\color[rgb]{1,1,1}$b_1$}%
}}}}
\end{picture}%
}
\caption{Four $\overline{G}^{(i)}$-type regions: (a)  $\overline{G}^{(2)}_{2,2,2}(2,3;\ 2,2;\ 2,2)$, (b) $\overline{G}^{(3)}_{2,1,2}(2,2;\ 2,2;\ 2,3)$, (c) $\overline{G}^{(4)}_{2,1,2}(2,2;\ 2,2;\ 2,3)$, and (d) $\overline{G}^{(5)}_{2,2,2}(2,3;\ 2,2;\ 2,2)$.}\label{fig:offG2}
\end{figure}

\subsection{The $\overline{G}^{(i)}$-type regions}
Our regions in this section is defined by applying the pushing procedure in the definition of the $\overline{E}^{(i)}$-type region to the auxiliary hexagon as in the case of the $G^{(i)}$-type region. In particular, we are assuming that $x$ and $z$ are two nonnegative integers  having the same parity, and that $y$ is an integer with domain defined particularly in the theorems below.  For $i=2,5$, we start with an auxiliary hexagon of side-lengths $x,z+1,z,x,z+1,z$, for $i=3,4$, we have the auxiliary hexagon of side-lengths $x,z+3,z,x,z+3,z$. Next, we perform the 2-stage pushing to all six sides of the auxiliary hexagons above in the same way as in the definition of the $\overline{E}^{(i)}$-type regions to obtain the base hexagon $H$. Finally, the three ferns are removed from the base hexagon $H$, such that the leftmost of the middle fern is at the off-central positions $i$ as indicated in Figure \ref{fig:offposition}(c). There are no counterparts for $i=1,6,7,8$. Denote by $\overline{G}^{(i)}_{x,y,z}(\textbf{a};\ \textbf{c};\ \textbf{b})$ the resulting regions, for $i=2,3,4,5$.

 \begin{rmk}\label{rmkQG}
In the definition of the $\overline{G}^{(2)}$-type regions, if we remove the three ferns such that the root of the middle one is $1/2$-unit to the northwest of the center of the auxiliary hexagon $H_0$, then we obtain the region $Q^{\nwarrow}_{x,y,z}(\textbf{a};\textbf{c};\textbf{b})$ in Theorem 2.7 of \cite{HoleDent}. This way, the $\overline{G}^{(i)}$-type regions can be viewed as counterparts of the $Q^{\nwarrow}$-type regions.
\end{rmk}

\begin{thm}\label{off1thmQG2}
Assume that $\textbf{a}=(a_1,a_2,\dotsc,a_m)$, $\textbf{b}=(b_1,b_2,\dotsc,b_n)$, $\textbf{c}=(c_1,c_2,\dotsc,c_k)$ are three sequences  of nonnegative integers ($m,n,k$ are all even) and that $x,y,z$ are three integers, such that $x\geq 0$, $y\geq \max(a-b,-1)$, $z\geq0$, and $x$ and $z$ have the same parity. Then
\begin{align}\label{off1eqQG2}
\M&(\overline{G}^{(2)}_{x,y,z}(\textbf{a};\textbf{c};\textbf{b}))=\Lambda'_{2y+z+2\max(a,b)+1,z,x}(c)\notag\\
&\times s\left(a_1,\dotsc, a_{m}+\frac{x+z}{2}-1,c_1,\dotsc,c_{k}+\frac{x+z}{2}+b_n+1,b_{n-1},\dotsc,b_1\right)\notag\\
&\times s\left(y+b-\min(a,b), a_1,\dotsc,a_{m},\frac{x+z}{2}+c_1-1,\dotsc,c_{k},\frac{x+z}{2}+1,b_n,\dotsc,b_1,y+a-\min(a,b)+1 \right) \notag\\
&\times\frac{\Hf(c+\frac{x+z}{2}-1)}{\Hf(c)\Hf(\frac{x+z}{2}-1)}\frac{\Hf(\max(a,b)+y+\frac{x+z}{2}-1)}{\Hf(\max(a,b)+c+y+\frac{x+z}{2}-1)}\notag\\
&\times \frac{\Hf(\max(a,b)+y+z+1)\Hf(\max(a,b)+c+y+z)}{\Hf(o_a+o_b+o_c+z)\Hf(|a-b|+e_a+e_b+e_c+2y+z+1)}\notag\\
&\times \frac{\Hf(o_a+o_b+o_c)\Hf(|a-b|+e_a+e_b+e_c+2y)}{\Hf(\max(a,b)+y)\Hf(\max(a,b)+y+1)},
\end{align}
where $\Lambda'_{x,y,z}(m)$ is defined as in (\ref{lambdaeq2}).
\end{thm}

\begin{thm}\label{off1thmQG3}
Assume that $\textbf{a}=(a_1,a_2,\dotsc,a_m)$, $\textbf{b}=(b_1,b_2,\dotsc,b_n)$, $\textbf{c}=(c_1,c_2,\dotsc,c_k)$ are three sequences  of nonnegative integers  ($m,n,k$ are all even) and that $x,y,z$ are three  integers, such that $x\geq 0$, $y\geq \max(a-b,-3)$, $z\geq0$, and $x$ and $z$ have the same parity.  Then
\begin{align}\label{off1eqQG3}
\M&(\overline{G}^{(3)}_{x,y,z}(\textbf{a};\textbf{c};\textbf{b}))=\Psi'_{2y+z+2\max(a,b)+3,z,x}(c)\notag\\
&\times s\left(a_1,\dotsc, a_{m}+\frac{x+z}{2},c_1,\dotsc,c_{k}+\frac{x+z}{2}+b_n,b_{n-1},\dotsc,b_1\right)\notag\\
&\times s\left(y+b-\min(a,b), a_1,\dotsc,a_{m},\frac{x+z}{2}+c_1,\dotsc,c_{k},\frac{x+z}{2},b_n,\dotsc,b_1,y+a-\min(a,b)+3 \right) \notag\\
&\times\frac{\Hf(c+\frac{x+z}{2})}{\Hf(c)\Hf(\frac{x+z}{2})}\frac{\Hf(\max(a,b)+y+\frac{x+z}{2})}{\Hf(\max(a,b)+c+y+\frac{x+z}{2})}\notag\\
&\times \frac{\Hf(\max(a,b)+y+z+3)\Hf(\max(a,b)+c+y+z)}{\Hf(o_a+o_b+o_c+z)\Hf(|a-b|+e_a+e_b+e_c+2y+z+3)}\notag\\
&\times \frac{\Hf(o_a+o_b+o_c)\Hf(|a-b|+e_a+e_b+e_c+2y+3)}{\Hf(\max(a,b)+y)\Hf(\max(a,b)+y+3)},
\end{align}
where $\Psi'_{x,y,z}(m)$ is the formula in  (\ref{psieq2}).
\end{thm}

\begin{thm}\label{off1thmQG4}
Assume that $\textbf{a}=(a_1,a_2,\dotsc,a_m)$, $\textbf{b}=(b_1,b_2,\dotsc,b_n)$, $\textbf{c}=(c_1,c_2,\dotsc,c_k)$ are three sequences  of nonnegative integers ($m,n,k$ are all even) and that $x,y,z$ are three integers, such that $x\geq 0$, $y\geq \max(a-b,-3)$, $z\geq0$, and $x$ and $z$ have the same parity. Then
\begin{align}\label{off1eqQG4}
\M&(\overline{G}^{(4)}_{x,y,z}(\textbf{a};\textbf{c};\textbf{b}))=\Lambda_{z,2y+z+2\max(a,b)+3,x}(c)\notag\\
&\times s\left(a_1,\dotsc, a_{m}+\frac{x+z}{2}+1,c_1,\dotsc,c_{k}+\frac{x+z}{2}+b_n,b_{n-1}-1,\dotsc,b_1\right)\notag\\
&\times s\left(y+b-\min(a,b), a_1,\dotsc,a_{m},\frac{x+z}{2}+c_1+1,\dotsc,c_{k},\frac{x+z}{2}-1,b_n,\dotsc,b_1,y+a-\min(a,b)+3 \right) \notag\\
&\times\frac{\Hf(c+\frac{x+z}{2}+1)}{\Hf(c)\Hf(\frac{x+z}{2}+1)}\frac{\Hf(\max(a,b)+y+\frac{x+z}{2}+1)}{\Hf(\max(a,b)+c+y+\frac{x+z}{2}+1)}\notag\\
&\times \frac{\Hf(\max(a,b)+y+z+3)\Hf(\max(a,b)+c+y+z)}{\Hf(o_a+o_b+o_c+z))\Hf(|a-b|+e_a+e_b+e_c+2y+z+3)}\notag\\
&\times \frac{\Hf(o_a+o_b+o_c)\Hf(|a-b|+e_a+e_b+e_c+2y+3)}{\Hf(\max(a,b)+y)\Hf(\max(a,b)+y+3)},
\end{align}
where $\Lambda_{x,y,z}(m)$ is defined as in (\ref{lambdaeq1}).
\end{thm}

\begin{thm}\label{off1thmQG5}
Assume that $\textbf{a}=(a_1,a_2,\dotsc,a_m)$, $\textbf{b}=(b_1,b_2,\dotsc,b_n)$, $\textbf{c}=(c_1,c_2,\dotsc,c_k)$ are three sequences  of nonnegative integers ($m,n,k$ are all even) and that $x,y,z$ are three integers, such that $x\geq 0$, $y\geq \max(a-b,-1)$, $z\geq0$, and $x$ and $z$ have the same parity.  Then
\begin{align}\label{off1eqQG5}
\M&(\overline{G}^{(5)}_{x,y,z}(\textbf{a};\textbf{c};\textbf{b}))=\Theta_{2y+z+2\max(a,b)+1,z,x}(c)\notag\\
&\times s\left(a_1,\dotsc, a_{m}+\frac{x+z}{2}+1,c_1,\dotsc,c_{k}+\frac{x+z}{2}+b_n,b_{n-1}-1,\dotsc,b_1\right)\notag\\
&\times s\left(y+b-\min(a,b), a_1,\dotsc,a_{m},\frac{x+z}{2}+c_1+1,\dotsc,c_{k},\frac{x+z}{2}-1,b_n,\dotsc,b_1,y+a-\min(a,b)+1 \right) \notag\\
&\times\frac{\Hf(c+\frac{x+z}{2}-1)}{\Hf(c)\Hf(\frac{x+z}{2}-1)}\frac{\Hf(\max(a,b)+y+\frac{x+z}{2})}{\Hf(\max(a,b)+c+y+\frac{x+z}{2})}\notag\\
&\times \frac{\Hf(\max(a,b)+y+z)\Hf(\max(a,b)+c+y+z+1)}{\Hf(o_a+o_b+o_c+z))\Hf(|a-b|+e_a+e_b+e_c+2y+z+1)}\notag\\
&\times \frac{\Hf(o_a+o_b+o_c)\Hf(|a-b|+e_a+e_b+e_c+2y+1)}{\Hf(\max(a,b)+y)\Hf(\max(a,b)+y+1)},
\end{align}
where $\Theta_{x,y,z}(m)$ is defined as in (\ref{thetaeq1}).
\end{thm}

\begin{figure}\centering
\setlength{\unitlength}{3947sp}%
\begingroup\makeatletter\ifx\SetFigFont\undefined%
\gdef\SetFigFont#1#2#3#4#5{%
  \reset@font\fontsize{#1}{#2pt}%
  \fontfamily{#3}\fontseries{#4}\fontshape{#5}%
  \selectfont}%
\fi\endgroup%
\resizebox{15cm}{!}{
\begin{picture}(0,0)%
\includegraphics{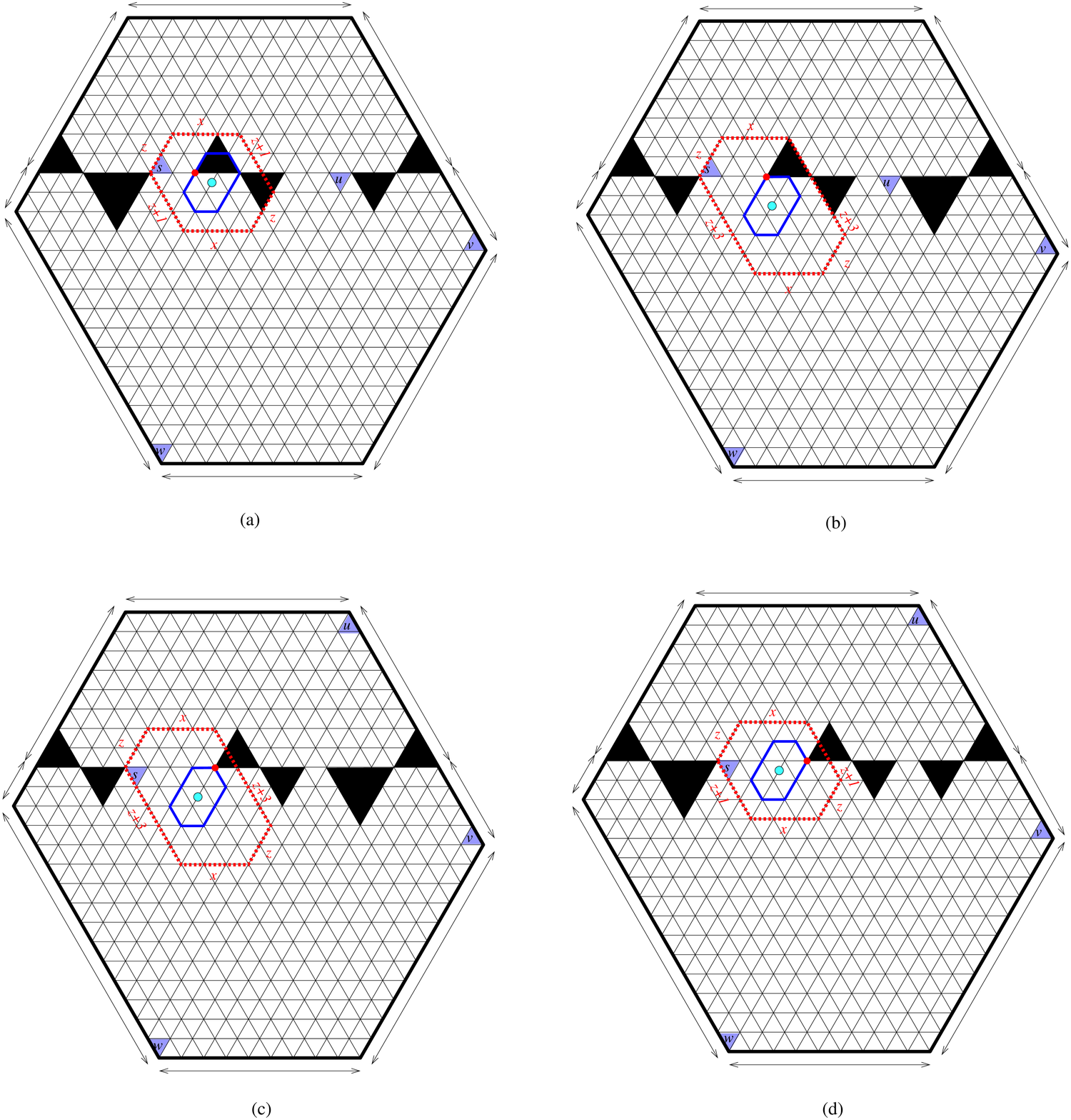}%
\end{picture}%
%
%

\begin{picture}(20276,20822)(996,-24277)
\put(11725,-7585){\rotatebox{60.0}{\makebox(0,0)[lb]{\smash{{\SetFigFont{14}{16.8}{\rmdefault}{\mddefault}{\itdefault}{$y+b-a$}%
}}}}}
\put(1191,-18308){\rotatebox{60.0}{\makebox(0,0)[lb]{\smash{{\SetFigFont{14}{16.8}{\rmdefault}{\mddefault}{\itdefault}{$y+b-a$}%
}}}}}
\put(12373,-9371){\rotatebox{300.0}{\makebox(0,0)[lb]{\smash{{\SetFigFont{14}{16.8}{\rmdefault}{\mddefault}{\itdefault}{$y+z+e_a+e_b+e_c+3$}%
}}}}}
\put(1539,-19944){\rotatebox{300.0}{\makebox(0,0)[lb]{\smash{{\SetFigFont{14}{16.8}{\rmdefault}{\mddefault}{\itdefault}{$y+z+e_a+e_b+e_c+3$}%
}}}}}
\put(15570,-13056){\makebox(0,0)[lb]{\smash{{\SetFigFont{14}{16.8}{\rmdefault}{\mddefault}{\itdefault}{$x+o_a+o_b+o_c$}%
}}}}
\put(4779,-12939){\makebox(0,0)[lb]{\smash{{\SetFigFont{14}{16.8}{\rmdefault}{\mddefault}{\itdefault}{$x+o_a+o_b+o_c$}%
}}}}
\put(15252,-23696){\makebox(0,0)[lb]{\smash{{\SetFigFont{14}{16.8}{\rmdefault}{\mddefault}{\itdefault}{$x+o_a+o_b+o_c$}%
}}}}
\put(4751,-23659){\makebox(0,0)[lb]{\smash{{\SetFigFont{14}{16.8}{\rmdefault}{\mddefault}{\itdefault}{$x+o_a+o_b+o_c$}%
}}}}
\put(15034,-14516){\makebox(0,0)[lb]{\smash{{\SetFigFont{14}{16.8}{\rmdefault}{\mddefault}{\itdefault}{$x+e_a+e_b+e_c$}%
}}}}
\put(4668,-14599){\makebox(0,0)[lb]{\smash{{\SetFigFont{14}{16.8}{\rmdefault}{\mddefault}{\itdefault}{$x+e_a+e_b+e_c$}%
}}}}
\put(4501,-3729){\makebox(0,0)[lb]{\smash{{\SetFigFont{14}{16.8}{\rmdefault}{\mddefault}{\itdefault}{$x+e_a+e_b+e_c$}%
}}}}
\put(14947,-3801){\makebox(0,0)[lb]{\smash{{\SetFigFont{14}{16.8}{\rmdefault}{\mddefault}{\itdefault}{$x+e_a+e_b+e_c$}%
}}}}
\put(18973,-4642){\rotatebox{300.0}{\makebox(0,0)[lb]{\smash{{\SetFigFont{14}{16.8}{\rmdefault}{\mddefault}{\itdefault}{$z+o_a+o_b+o_c$}%
}}}}}
\put(8422,-4435){\rotatebox{300.0}{\makebox(0,0)[lb]{\smash{{\SetFigFont{14}{16.8}{\rmdefault}{\mddefault}{\itdefault}{$z+o_a+o_b+o_c$}%
}}}}}
\put(8469,-15350){\rotatebox{300.0}{\makebox(0,0)[lb]{\smash{{\SetFigFont{14}{16.8}{\rmdefault}{\mddefault}{\itdefault}{$z+o_a+o_b+o_c$}%
}}}}}
\put(19030,-15402){\rotatebox{300.0}{\makebox(0,0)[lb]{\smash{{\SetFigFont{14}{16.8}{\rmdefault}{\mddefault}{\itdefault}{$z+o_a+o_b+o_c$}%
}}}}}
\put(12277,-17192){\rotatebox{60.0}{\makebox(0,0)[lb]{\smash{{\SetFigFont{14}{16.8}{\rmdefault}{\mddefault}{\itdefault}{$z+o_a+o_b+o_c$}%
}}}}}
\put(1896,-17155){\rotatebox{60.0}{\makebox(0,0)[lb]{\smash{{\SetFigFont{14}{16.8}{\rmdefault}{\mddefault}{\itdefault}{$z+o_a+o_b+o_c$}%
}}}}}
\put(1939,-6270){\rotatebox{60.0}{\makebox(0,0)[lb]{\smash{{\SetFigFont{14}{16.8}{\rmdefault}{\mddefault}{\itdefault}{$z+o_a+o_b+o_c$}%
}}}}}
\put(12355,-6402){\rotatebox{60.0}{\makebox(0,0)[lb]{\smash{{\SetFigFont{14}{16.8}{\rmdefault}{\mddefault}{\itdefault}{$z+o_a+o_b+o_c$}%
}}}}}
\put(8788,-23039){\rotatebox{60.0}{\makebox(0,0)[lb]{\smash{{\SetFigFont{14}{16.8}{\rmdefault}{\mddefault}{\itdefault}{$y+z+e_a+e_b+e_c+b-a$}%
}}}}}
\put(19382,-12136){\rotatebox{60.0}{\makebox(0,0)[lb]{\smash{{\SetFigFont{14}{16.8}{\rmdefault}{\mddefault}{\itdefault}{$y+z+e_a+e_b+e_c+b-a$}%
}}}}}
\put(20435,-7172){\rotatebox{300.0}{\makebox(0,0)[lb]{\smash{{\SetFigFont{14}{16.8}{\rmdefault}{\mddefault}{\itdefault}{$y+3$}%
}}}}}
\put(9901,-18000){\rotatebox{300.0}{\makebox(0,0)[lb]{\smash{{\SetFigFont{14}{16.8}{\rmdefault}{\mddefault}{\itdefault}{$y+3$}%
}}}}}
\put(19184,-22791){\rotatebox{60.0}{\makebox(0,0)[lb]{\smash{{\SetFigFont{14}{16.8}{\rmdefault}{\mddefault}{\itdefault}{$y+z+e_a+e_b+e_c$}%
}}}}}
\put(20417,-17947){\rotatebox{300.0}{\makebox(0,0)[lb]{\smash{{\SetFigFont{14}{16.8}{\rmdefault}{\mddefault}{\itdefault}{$y+a-b+1$}%
}}}}}
\put(9914,-7073){\rotatebox{300.0}{\makebox(0,0)[lb]{\smash{{\SetFigFont{14}{16.8}{\rmdefault}{\mddefault}{\itdefault}{$y+a-b+1$}%
}}}}}
\put(11976,-19608){\rotatebox{300.0}{\makebox(0,0)[lb]{\smash{{\SetFigFont{14}{16.8}{\rmdefault}{\mddefault}{\itdefault}{$y+z+e_a+e_b+e_c+a-b+1$}%
}}}}}
\put(11884,-17808){\rotatebox{60.0}{\makebox(0,0)[lb]{\smash{{\SetFigFont{14}{16.8}{\rmdefault}{\mddefault}{\itdefault}{$y$}%
}}}}}
\put(1507,-7165){\rotatebox{60.0}{\makebox(0,0)[lb]{\smash{{\SetFigFont{14}{16.8}{\rmdefault}{\mddefault}{\itdefault}{$y$}%
}}}}}
\put(1458,-8560){\rotatebox{300.0}{\makebox(0,0)[lb]{\smash{{\SetFigFont{14}{16.8}{\rmdefault}{\mddefault}{\itdefault}{$y+z+e_a+e_b+e_c+a-b+1$}%
}}}}}
\put(8996,-11873){\rotatebox{60.0}{\makebox(0,0)[lb]{\smash{{\SetFigFont{14}{16.8}{\rmdefault}{\mddefault}{\itdefault}{$y+z+e_a+e_b+e_c$}%
}}}}}
\put(16414,-17516){\makebox(0,0)[lb]{\smash{{\SetFigFont{14}{16.8}{\familydefault}{\mddefault}{\updefault}{\color[rgb]{1,1,1}$c_1$}%
}}}}
\put(5176,-6849){\makebox(0,0)[lb]{\smash{{\SetFigFont{14}{16.8}{\familydefault}{\mddefault}{\updefault}{\color[rgb]{1,1,1}$c_1$}%
}}}}
\put(15637,-6861){\makebox(0,0)[lb]{\smash{{\SetFigFont{14}{16.8}{\familydefault}{\mddefault}{\updefault}{\color[rgb]{1,1,1}$c_1$}%
}}}}
\put(5598,-17629){\makebox(0,0)[lb]{\smash{{\SetFigFont{14}{16.8}{\familydefault}{\mddefault}{\updefault}{\color[rgb]{1,1,1}$c_1$}%
}}}}
\put(13684,-18086){\makebox(0,0)[lb]{\smash{{\SetFigFont{14}{16.8}{\familydefault}{\mddefault}{\updefault}{\color[rgb]{1,1,1}$a_2$}%
}}}}
\put(7698,-18244){\makebox(0,0)[lb]{\smash{{\SetFigFont{14}{16.8}{\familydefault}{\mddefault}{\updefault}{\color[rgb]{1,1,1}$b_2$}%
}}}}
\put(8821,-6834){\makebox(0,0)[lb]{\smash{{\SetFigFont{14}{16.8}{\familydefault}{\mddefault}{\updefault}{\color[rgb]{1,1,1}$b_1$}%
}}}}
\put(5956,-7254){\makebox(0,0)[lb]{\smash{{\SetFigFont{14}{16.8}{\familydefault}{\mddefault}{\updefault}{\color[rgb]{1,1,1}$c_2$}%
}}}}
\put(2296,-6819){\makebox(0,0)[lb]{\smash{{\SetFigFont{14}{16.8}{\familydefault}{\mddefault}{\updefault}{\color[rgb]{1,1,1}$a_1$}%
}}}}
\put(8808,-17719){\makebox(0,0)[lb]{\smash{{\SetFigFont{14}{16.8}{\familydefault}{\mddefault}{\updefault}{\color[rgb]{1,1,1}$b_1$}%
}}}}
\put(12757,-6891){\makebox(0,0)[lb]{\smash{{\SetFigFont{14}{16.8}{\familydefault}{\mddefault}{\updefault}{\color[rgb]{1,1,1}$a_1$}%
}}}}
\put(2283,-17689){\makebox(0,0)[lb]{\smash{{\SetFigFont{14}{16.8}{\familydefault}{\mddefault}{\updefault}{\color[rgb]{1,1,1}$a_1$}%
}}}}
\put(12664,-17561){\makebox(0,0)[lb]{\smash{{\SetFigFont{14}{16.8}{\familydefault}{\mddefault}{\updefault}{\color[rgb]{1,1,1}$a_1$}%
}}}}
\put(3286,-7389){\makebox(0,0)[lb]{\smash{{\SetFigFont{14}{16.8}{\familydefault}{\mddefault}{\updefault}{\color[rgb]{1,1,1}$a_2$}%
}}}}
\put(19204,-17621){\makebox(0,0)[lb]{\smash{{\SetFigFont{14}{16.8}{\familydefault}{\mddefault}{\updefault}{\color[rgb]{1,1,1}$b_1$}%
}}}}
\put(18409,-18011){\makebox(0,0)[lb]{\smash{{\SetFigFont{14}{16.8}{\familydefault}{\mddefault}{\updefault}{\color[rgb]{1,1,1}$b_2$}%
}}}}
\put(19327,-6861){\makebox(0,0)[lb]{\smash{{\SetFigFont{14}{16.8}{\familydefault}{\mddefault}{\updefault}{\color[rgb]{1,1,1}$b_1$}%
}}}}
\put(18262,-7446){\makebox(0,0)[lb]{\smash{{\SetFigFont{14}{16.8}{\familydefault}{\mddefault}{\updefault}{\color[rgb]{1,1,1}$b_2$}%
}}}}
\put(8071,-7254){\makebox(0,0)[lb]{\smash{{\SetFigFont{14}{16.8}{\familydefault}{\mddefault}{\updefault}{\color[rgb]{1,1,1}$b_2$}%
}}}}
\put(16447,-7326){\makebox(0,0)[lb]{\smash{{\SetFigFont{14}{16.8}{\familydefault}{\mddefault}{\updefault}{\color[rgb]{1,1,1}$c_2$}%
}}}}
\put(6363,-18124){\makebox(0,0)[lb]{\smash{{\SetFigFont{14}{16.8}{\familydefault}{\mddefault}{\updefault}{\color[rgb]{1,1,1}$c_2$}%
}}}}
\put(17179,-18011){\makebox(0,0)[lb]{\smash{{\SetFigFont{14}{16.8}{\familydefault}{\mddefault}{\updefault}{\color[rgb]{1,1,1}$c_2$}%
}}}}
\put(13552,-7326){\makebox(0,0)[lb]{\smash{{\SetFigFont{14}{16.8}{\familydefault}{\mddefault}{\updefault}{\color[rgb]{1,1,1}$a_2$}%
}}}}
\put(3048,-18139){\makebox(0,0)[lb]{\smash{{\SetFigFont{14}{16.8}{\familydefault}{\mddefault}{\updefault}{\color[rgb]{1,1,1}$a_2$}%
}}}}
\end{picture}%
}
\caption{Four $\overline{K}^{(i)}$-type regions: (a) $\overline{K}^{(5)}_{3,2,2}(2,3;\ 2,2;\ 2,2)$, (b) $\overline{K}^{(6)}_{3,2,2}(2,2;\ 2,2;\ 2,3)$, (c) $\overline{K}^{(7)}_{3,2,2}(2,2;\ 2,2;\ 2,3)$, and (d) $\overline{K}^{(8)}_{3,2,2}(2,3;\ 2,2;\ 2,2)$.} \label{fig:offK2}
\end{figure}

\subsection{The $\overline{K}^{(i)}$-type regions}
We now define the final group of the off-central regions considered in this paper. We also start with an auxiliary hexagon $H_0$ whose side-lengths are $x,z+1,z,x,z+1,z$ for $i=5,8$, and are $x,z+3,z,x,z+3,z$, for $i=6,7$ ($x$ and $z$ have different parities in this case). After pushing the sides of the auxiliary hexagon $H_0$ in the same way as in the definition of the $\overline{E}^{(i)}$-type regions, we remove three ferns such that the root of the middle fern is at the off-central positions $i$ as shown in Figure \ref{fig:offposition}(d). There are no counterparts for the case $i=1,2,3,4$ here. Let $\overline{K}^{(i)}$ denote the newly defined region, $i=5,6,7,8$.

 \begin{rmk}\label{rmkQK}
In the definition of the $\overline{K}^{(5)}$-type regions, if we remove the three ferns such that the root of the middle one is $1/2$-unit to the southwest of the center of the auxiliary hexagon $H_0$, then we obtain the region $Q^{\swarrow}_{x,y,z}(\textbf{a};\textbf{c};\textbf{b})$ in Theorem 2.8 of \cite{HoleDent}. One can vew the $\overline{K}^{(i)}$-type regions as counterparts of the $Q^{\swarrow}$-type regions.
\end{rmk}

\begin{thm}\label{off1thmQK5}
Assume that $\textbf{a}=(a_1,a_2,\dotsc,a_m)$, $\textbf{b}=(b_1,b_2,\dotsc,b_n)$, $\textbf{c}=(c_1,c_2,\dotsc,c_k)$ are three sequences  of nonnegative integers ($m,n,k$ are all even) and that $x,y,z$ are three  integers, such that $x\geq 0$, $y\geq \max(a-b,-1)$, $z\geq0$, and $x$  and $z$ have different parities. Then
\begin{align}\label{off1eqQK5}
\M&(\overline{K}^{(5)}_{x,y,z}(\textbf{a};\textbf{c};\textbf{b}))=\Theta_{z,x,2y+z+2\max(a,b)+1}(c)\notag\\
&\times s\left(a_1,\dotsc, a_{m}+\left\lfloor\frac{x+z}{2}\right\rfloor,c_1,\dotsc,c_{k}+\left\lceil\frac{x+z}{2}\right\rceil+b_n,b_{n-1},\dotsc,b_1\right)\notag\\
&\times s\left(y+b-\min(a,b), a_1,\dotsc,a_{m},\left\lfloor\frac{x+z}{2}\right\rfloor+c_1,\dotsc,c_{k},\left\lceil\frac{x+z}{2}\right\rceil,b_n,\dotsc,b_1,y+a-\min(a,b)+1\right) \notag\\
&\times\frac{\Hf(c+\left\lceil\frac{x+z}{2}\right\rceil)}{\Hf(c)\Hf(\left\lceil\frac{x+z}{2}\right\rceil)}\frac{\Hf(\max(a,b)+y+\left\lceil\frac{x+z}{2}\right\rceil+1)}{\Hf(\max(a,b)+c+y+\left\lceil\frac{x+z}{2}\right\rceil+1}\notag\\
&\times \frac{\Hf(\max(a,b)+y+z)\Hf(\max(a,b)+c+y+z+1)}{\Hf(o_a+o_b+o_c+z))\Hf(|a-b|+e_a+e_b+e_c+2y+z+1)}\notag\\
&\times \frac{\Hf(o_a+o_b+o_c)\Hf(|a-b|+e_a+e_b+e_c+2y+1)}{\Hf(\max(a,b)+y)\Hf(\max(a,b)+y+1)},
\end{align}
where $\Theta_{x,y,z}(m)$ is defined as in (\ref{thetaeq1}).
\end{thm}

\begin{thm}\label{off1thmQK6}
Assume that $\textbf{a}=(a_1,a_2,\dotsc,a_m)$, $\textbf{b}=(b_1,b_2,\dotsc,b_n)$, $\textbf{c}=(c_1,c_2,\dotsc,c_k)$ are three sequences  of nonnegative integers ($m,n,k$ are all even) and that $x,y,z$ are three  integers, such that  $x\geq 0$, $y\geq \max(a-b,-3)$, $z\geq0$, and $x$ and $z$ have different parities. Then
\begin{align}\label{off1eqQK6}
\M&(\overline{K}^{(6)}_{x,y,z}(\textbf{a};\textbf{c};\textbf{b}))=\Lambda_{z,x,2y+z+2\max(a,b)+3}(c)\notag\\
&\times s\left(a_1,\dotsc, a_{m}+\left\lceil\frac{x+z}{2}\right\rceil,c_1,\dotsc,c_{k}+\left\lfloor\frac{x+z}{2}\right\rfloor,+b_n,b_{n-1},\dotsc,b_1\right)\notag\\
&\times s\left(y+b-\min(a,b), a_1,\dotsc,a_{m},\left\lceil\frac{x+z}{2}\right\rceil+c_1,\dotsc,c_{k},\left\lfloor\frac{x+z}{2}\right\rfloor,b_n,\dotsc,b_1,y+a-\min(a,b)+3\right) \notag\\
&\times\frac{\Hf(c+\left\lfloor\frac{x+z}{2}\right\rfloor)}{\Hf(c)\Hf(\left\lfloor\frac{x+z}{2}\right\rfloor)}\frac{\Hf(\max(a,b)+y+\left\lceil\frac{x+z}{2}\right\rceil+2)}{\Hf(\max(a,b)+c+y+\left\lceil\frac{x+z}{2}\right\rceil+2)}\notag\\
&\times \frac{\Hf(\max(a,b)+y+z)\Hf(\max(a,b)+c+y+z+3)}{\Hf(o_a+o_b+o_c+z))\Hf(|a-b|+e_a+e_b+e_c+2y+z+3)}\notag\\
&\times \frac{\Hf(o_a+o_b+o_c)\Hf(|a-b|+e_a+e_b+e_c+2y+3)}{\Hf(\max(a,b)+y)\Hf(\max(a,b)+y+3)},
\end{align}
where $\Lambda_{x,y,z}(m)$ is defined as in (\ref{lambdaeq1}).
\end{thm}

\begin{thm}\label{off1thmQK7}
Assume that $\textbf{a}=(a_1,a_2,\dotsc,a_m)$, $\textbf{b}=(b_1,b_2,\dotsc,b_n)$, $\textbf{c}=(c_1,c_2,\dotsc,c_k)$ are three sequences  of nonnegative integers ($m,n,k$ are all even) and that $x,y,z$ are three  integers, such that $x\geq 0$, $y\geq \max(a-b,-3)$, $z\geq0$, and $x$ and $z$ have different parities.  Then
\begin{align}\label{off1eqQK7}
\M&(\overline{K}^{(7)}_{x,y,z}(\textbf{a};\textbf{c};\textbf{b}))=\Psi_{z,x,2y+z+2\max(a,b)+3}(c)\notag\\
&\times s\left(a_1,\dotsc, a_{m}+\left\lceil\frac{x+z}{2}\right\rceil+1,c_1,\dotsc,c_{k}+\left\lfloor\frac{x+z}{2}\right\rfloor-1,+b_n,b_{n-1},\dotsc,b_1\right)\notag\\
&\times s\left(y+b-\min(a,b), a_1,\dotsc,a_{m},\left\lceil\frac{x+z}{2}\right\rceil+c_1+1,\dotsc,c_{k},\left\lfloor\frac{x+z}{2}\right\rfloor-1,b_n,\dotsc,b_1,y+a-\min(a,b)+3\right) \notag\\
&\times\frac{\Hf(c+\left\lfloor\frac{x+z}{2}\right\rfloor-1)}{\Hf(c)\Hf(\left\lfloor\frac{x+z}{2}\right\rfloor-1)}\frac{\Hf(\max(a,b)+y+\left\lceil\frac{x+z}{2}\right\rceil+1)}{\Hf(\max(a,b)+c+y+\left\lceil\frac{x+z}{2}\right\rceil+1)}\notag\\
&\times \frac{\Hf(\max(a,b)+y+z)\Hf(\max(a,b)+c+y+z+3)}{\Hf(o_a+o_b+o_c+z))\Hf(|a-b|+e_a+e_b+e_c+2y+z+3)}\notag\\
&\times \frac{\Hf(o_a+o_b+o_c)\Hf(|a-b|+e_a+e_b+e_c+2y+3)}{\Hf(\max(a,b)+y)\Hf(\max(a,b)+y+3)},
\end{align}
where $\Psi_{x,y,z}(m)$ is defined as in (\ref{psieq1}).
\end{thm}

\begin{thm}\label{off1thmQK8}
Assume that $\textbf{a}=(a_1,a_2,\dotsc,a_m)$, $\textbf{b}=(b_1,b_2,\dotsc,b_n)$, $\textbf{c}=(c_1,c_2,\dotsc,c_k)$ are three sequences  of nonnegative integers  ($m,n,k$ are all even)and that $x,y,z$ are three  integers, such that $x\geq 0$, $y\geq \max(a-b,-1)$, $z\geq0$, and $x$ and $z$ have the same parity. Then
\begin{align}\label{off1eqQK8}
\M&(\overline{K}^{(8)}_{x,y,z}(\textbf{a};\textbf{c};\textbf{b}))=\Lambda'_{2y+z+2\max(a,b)+1,x,z}(c)\notag\\
&\times s\left(a_1,\dotsc, a_{m}+\left\lceil\frac{x+z}{2}\right\rceil+1,c_1,\dotsc,c_{k}+\left\lfloor\frac{x+z}{2}\right\rfloor-1,+b_n,b_{n-1},\dotsc,b_1\right)\notag\\
&\times s\left(y+b-\min(a,b), a_1,\dotsc,a_{m},\left\lceil\frac{x+z}{2}\right\rceil+c_1+1,\dotsc,c_{k},\left\lfloor\frac{x+z}{2}\right\rfloor-1,b_n,\dotsc,b_1,y+a-\min(a,b)+1\right) \notag\\
&\times\frac{\Hf(c+\left\lfloor\frac{x+z}{2}\right\rfloor-1)}{\Hf(c)\Hf(\left\lfloor\frac{x+z}{2}\right\rfloor-1)}\frac{\Hf(\max(a,b)+y+\left\lceil\frac{x+z}{2}\right\rceil)}{\Hf(\max(a,b)+c+y+\left\lceil\frac{x+z}{2}\right\rceil)}\notag\\
&\times \frac{\Hf(\max(a,b)+y+z)\Hf(\max(a,b)+c+y+z+1)}{\Hf(o_a+o_b+o_c+z))\Hf(|a-b|+e_a+e_b+e_c+2y+z+1)}\notag\\
&\times \frac{\Hf(o_a+o_b+o_c)\Hf(|a-b|+e_a+e_b+e_c+2y+1)}{\Hf(\max(a,b)+y)\Hf(\max(a,b)+y+1)},
\end{align}
where $\Lambda'_{x,y,z}(m)$ is defined as in (\ref{lambdaeq2}).
\end{thm}

\section{Proofs of the main results for off-central regions}\label{sec:proofoff}

\subsection{Organization of the proof}\label{subsec:organize}

Recall that our 8 families of regions, $E^{(i)}$, $F^{(i)}$, $G^{(i)}$,  $K^{(i)}$, $\overline{E}^{(i)}$,  $\overline{F}^{(i)}$, $\overline{G}^{(i)}$  and $\overline{K}^{(i)}$,  are all obtained from a  certain base hexagon $H$ by removing three ferns along a common lattice line $\ell$. We call the perimeter  of the base hexagon the \emph{quasi-perimeter} of the regions, denoted by $p$ in the rest of the proof. One readily sees that
\begin{lem}\label{claimp} We always have for all regions in the eight families
\[p \geq 2x+4z.\]
\end{lem}
\begin{proof}
If we start with an auxiliary hexagon $H_0$ of side-lengths $x,z+j,z,x,z+j,z$ (for $j=0,1,2,3$) in the definition of our region, then the quasi perimeter is always
\begin{equation}
p=2x+4y+4z+3a+3b+2|a-b|+2j.
\end{equation}
In this case, we always have $y\geq\max(-|a-b|,-j)$. It means that $4y+2|a-b|+2j\geq 0$, so $p\geq 2x+4z+3a+3b\geq2x+4z$.
\end{proof}

We aim to prove \emph{all} 30 tiling formulas in Theorems \ref{off1thm1}--\ref{off1thmQK8} at once by induction on $h:=p+x+z$, where $p$ is the quasi-perimeter of the region. Our proof is organized as follows. In Section 3.2, we quote the particular versions the Kuo condensation that will be employed in our proofs. Next, in Sections 3.3--3.10, we will present carefully 66 recurrences for our 8 families of regions obtained by applying Kuo condensation. Each family of regions will have more than one recurrences, depending on whether $a> b$, $a=b$, or $a< b$.  We would like to emphasize that, due to the difference in the structures of our regions, the universal recurrence seems \emph{not} to exist.  In Section 3.11, we investigate the two extremal cases of our proof. Section 3.12  is devoted to the main arguments of the inductive proof. 

\subsection{Kuo condensation and other preliminary results}\label{sec:kuo}

A \emph{forced lozenge} in a region $R$ on the triangular lattice is a lozenge contained in any tilings of $R$. Removal of forced lozenges does not change the tiling number of a region.

A region on the triangular lattice is said to be \emph{balanced} if it has the same number of up- and down-pointing unit triangles (this is a necessary condition that a region admits a tiling). The following useful lemma allows us to decompose a large region into several smaller ones.

\begin{lem}[Region-splitting Lemma \cite{Tri1, Tri2}]\label{RS}
Let $R$ be a balanced region on the triangular lattice. Assume that a blanched sub-region $Q$ of $R$ satisfies the condition that there is only one type of unit triangles running along each side of the border between $Q$ and $R-Q$.
Then
$\M(R)=\M(Q)\, \M(R-Q).$
\end{lem}

Let $G$ be a finite graph without loops, however the multiple-edges are allowed. A \emph{perfect matching} of $G$ is a collection of disjoint edges covering all vertices of $G$. The \emph{(planar) dual graph} of a region $R$ on the triangular lattice  is the graph whose vertices are unit triangles in $R$ and whose edges connects precisely two unit triangles sharing an edge.  We can identify the tilings of a region and perfect matchings of its dual graph.  In this point of view, we use the notation $\M(G)$ for the number of perfect matchings of the graph $G$.

The following two theorems of  Kuo are the key of our proofs in this paper.
\begin{thm}[Theorem 5.1 \cite{Kuo}]\label{kuothm1}
Let $G=(V_1,V_2,E)$ be a bipartite planar graph in which $|V_1|=|V_2|$. Assume that  $u, v, w, s$ are four vertices appearing in a cyclic order on a face of $G$ so that $u,w \in V_1$ and $v,s \in V_2$. Then
\begin{equation}\label{kuoeq1}
\M(G)\M(G-\{u, v, w, s\})=\M(G-\{u, v\})\M(G-\{ w, s\})+\M(G-\{u, s\})\M(G-\{v, w\}).
\end{equation}
\end{thm}

\begin{thm}[Theorem 5.2 \cite{Kuo}]\label{kuothm2}
Let $G=(V_1,V_2,E)$ be a bipartite planar graph in which $|V_1|=|V_2|$. Assume that  $u, v, w, s$ are four vertices appearing in a cyclic order on a face of $G$ so that $u,v \in V_1$ and $w,s \in V_2$. Then
\begin{equation}\label{kuoeq2}
\M(G-\{u, s\})\M(G-\{v, w\})=\M(G)\M(G-\{u, v, w, s\})+\M(G-\{u,w\})\M(G-\{v, s\}).
\end{equation}
\end{thm}

Theorems \ref{kuothm1} and \ref{kuothm2} are usually mentioned two variants of  \emph{Kuo condensation}. Kuo condensation (or \emph{graphical condensation} as called in \cite{Kuo}) can be considered as a combinatorial interpretation of the well-known \emph{Dodgson condensation} in linear algebra (which is based on the Jacobi--Desnanot identity, see e.g. \cite{Abeles}, \cite{Dod} and \cite{Mui}, pp. 136--148, and \cite{Zeil} for a bijective proof). The Dodgson condensation was named after Charles Lutwidge Dodgson (1832--1898), better known by his pen name Lewis Carroll, an English writer, mathematician, and photographer.

The preliminary version of Kuo condensation (when the for vertices $u,v,w,s$ in Theorem \ref{kuothm1} form a $4$-cycle in the graph $G$) was originally conjectured by Alexandru Ionescu in context of Aztec diamond graphs, and was proved by Propp in 1993 (see e.g. \cite{Propp2}). Eric H. Kuo introduced Kuo condensation  in his 2004 paper \cite{Kuo} with four different versions, two of them are Theorems \ref{kuothm1} and \ref{kuothm2} stated above. Kuo condensation has become a powerful tool in the enumeration of tilings with a number of applications. We refer the reader to \cite{Ciucu, Ful, Knuth, Kuo06, speyer, YYZ, YZ} for various aspects and generalizations of Kuo condensation, and e.g. \cite{CF, CK, CL, KW, LMNT, Lai15a, Tri1, Tri2, Halfhex1, Halfhex2, Halfhex3, Minor, LM, LR, Ranjan1, Ranjan2} for recent applications of the method.


\subsection{Recurrences for  $E^{(i)}$-type regions}\label{subsec:offrecurE}
We apply Kuo condensation in Theorem \ref{kuothm1} to the dual graph $G$ of the region $E^{(1)}_{x,y,z}(\textbf{a};\ \textbf{c};\ \textbf{b})$ with the four vertices $u,v,w,s$ chosen as in Figure \ref{fig:off1}. In particular, the unit triangle corresponding to $u$ is the shaded one on the northeast corner of the region, the $v$-triangle is appended to the left of the right fern, the $w$-triangle is at the east corner of the region, and the $s$-triangle is at the southwest corner.


Let us consider the region corresponding to the graph $G-\{u,v,w,s\}$ (see Figure \ref{fig:kuooff1}(b)). The removal of the $u$-, $w$-, $s$-triangles yields several forced lozenges on the boundary of the region. The removal of the $v$-triangle yields forced lozenges on the side of the last down-pointing triangle of the $b$-fern. After removing these forced lozenges, we get a $R^{\leftarrow}$-type region considered in Theorem 2.3 of \cite{HoleDent}. We note that Figure \ref{fig:kuooff1}(c) illustrates the case when the $b$-fern has an even number of triangles, i.e. it ends with a down-pointing triangle. In the case $b$-fern has an odd number of triangles, we can regard that the $b$-fern ends with a down-pointing triangle of side-length $0$. In particular, we obtain the region $R^{\leftarrow}_{x,y-1,z-1}(\textbf{a};\ \textbf{c}; \ \textbf{b}^{+1})$, where $\textbf{b}^{+1}$ denotes the sequence obtained from $\textbf{b}$ by adding 1 to the last term if $\textbf{b}$ has an even number of terms, in the case $\textbf{b}$ has an odd number of terms, we include a new term $1$ to the end of the sequence.

We note that, besides changing the side-length of the bases hexagon and the $b$-fern, the removal of forced lozenges may change the center of the region. As a consequence, it may change the type of our region, in particular, from $E^{(1)}$ to $R^{\leftarrow}$ in this case. Since the removal of the forced lozenges does not change the tiling number, we get
\begin{equation}\label{kuothm1eq1}
\M(G-\{u,v,w,s\})=\M(R^{\leftarrow}_{x,y-1,z-1}(\textbf{a};\ \textbf{c}; \ \textbf{b}^{+1})).
\end{equation}

By considering forced lozenges yields by the removal of $u$-, $v$-, $w$-, $s$-triangles as in Figures \ref{fig:kuooff1}(d), (e) and (f), we have
\begin{equation}\label{kuothm1eq2}
\M(G-\{w,s\})=\M(G^{(1)}_{x,y-1,z}(\textbf{a};\ \textbf{c};\ \textbf{b})),
\end{equation}
\begin{equation}\label{kuothm1eq3}
\M(G-\{u,s\})=\M(E^{(1)}_{x+1,y,z-1}(\textbf{a};\ \textbf{c};\ \textbf{b}))
\end{equation}
and
\begin{equation}\label{kuothm1eq4}
\M(G-\{v,w\})=\M(R^{\leftarrow}_{x-1,y,z}(\textbf{a};\ \textbf{c};\ \textbf{b}^{+1})).
\end{equation}

Finally, we remove forced lozenges from the region corresponding to $G-\{u,v\}$, $180^{\circ}$-rotate the resulting region, and obtain
\begin{equation}\label{kuothm1eq5}
\M(G-\{u,v\})=\M(K^{(1)}_{x,y-1,z-1}(\textbf{b}^{+1};\ \overline{\textbf{c}};\ \textbf{a})),
\end{equation}
where $\overline{\textbf{c}}$ is the sequence obtained from $\textbf{c}$ by reverting the order of terms its if $\textbf{c}$ has an even number of terms, otherwise, we revert $\textbf{c}$ and include a new 0 term in the beginning (see Figure \ref{fig:kuooff1}(c)).

Plugging equations (\ref{kuothm1eq1})--(\ref{kuothm1eq5}) into the equation in Kuo's Theorem \ref{kuothm1}, we have the following $E^{(1)}$-recurrence for $a\leq b$:
\begin{align}\label{offcenterrecurE1a}
\M(E^{(1)}_{x,y,z}(\textbf{a};\ \textbf{c};\ \textbf{b}))\M(R^{\leftarrow}_{x,y-1,z-1}(\textbf{a};\ \textbf{c}; \ \textbf{b}^{+1}))&=\M(K^{(1)}_{x,y-1,z-1}(\textbf{b}^{+1};\ \overline{\textbf{c}};\ \textbf{a})) \M(G^{(1)}_{x,y-1,z}(\textbf{a};\ \textbf{c};\ \textbf{b}))\notag\\
&+
\M(E^{(1)}_{x+1,y,z-1}(\textbf{a};\ \textbf{c};\ \textbf{b})) \M(R^{\leftarrow}_{x-1,y-1,z}(\textbf{a};\ \textbf{c};\ \textbf{b}^{+1})).
\end{align}

Working similarly for the case $a>b$, we get a slightly different $E^{(1)}$-recurrence:
\begin{align}\label{offcenterrecurE1b}
\M(E^{(1)}_{x,y,z}(\textbf{a};\ \textbf{c};\ \textbf{b}))\M(R^{\leftarrow}_{x,y,z-1}(\textbf{a};\ \textbf{c}; \ \textbf{b}^{+1}))&=\M(K^{(1)}_{x,y,z-1}(\textbf{b}^{+1};\ \overline{\textbf{c}};\ \textbf{a})) \M(G^{(1)}_{x,y-1,z}(\textbf{a};\ \textbf{c};\ \textbf{b}))\notag\\
&+
\M(E^{(1)}_{x+1,y,z-1}(\textbf{a};\ \textbf{c};\ \textbf{b})) \M(R^{\leftarrow}_{x-1,y,z}(\textbf{a};\ \textbf{c};\ \textbf{b}^{+1})).
\end{align}

\begin{figure}\centering
\setlength{\unitlength}{3947sp}%
\begingroup\makeatletter\ifx\SetFigFont\undefined%
\gdef\SetFigFont#1#2#3#4#5{%
  \reset@font\fontsize{#1}{#2pt}%
  \fontfamily{#3}\fontseries{#4}\fontshape{#5}%
  \selectfont}%
\fi\endgroup%
\resizebox{15cm}{!}{
\begin{picture}(0,0)%
\includegraphics{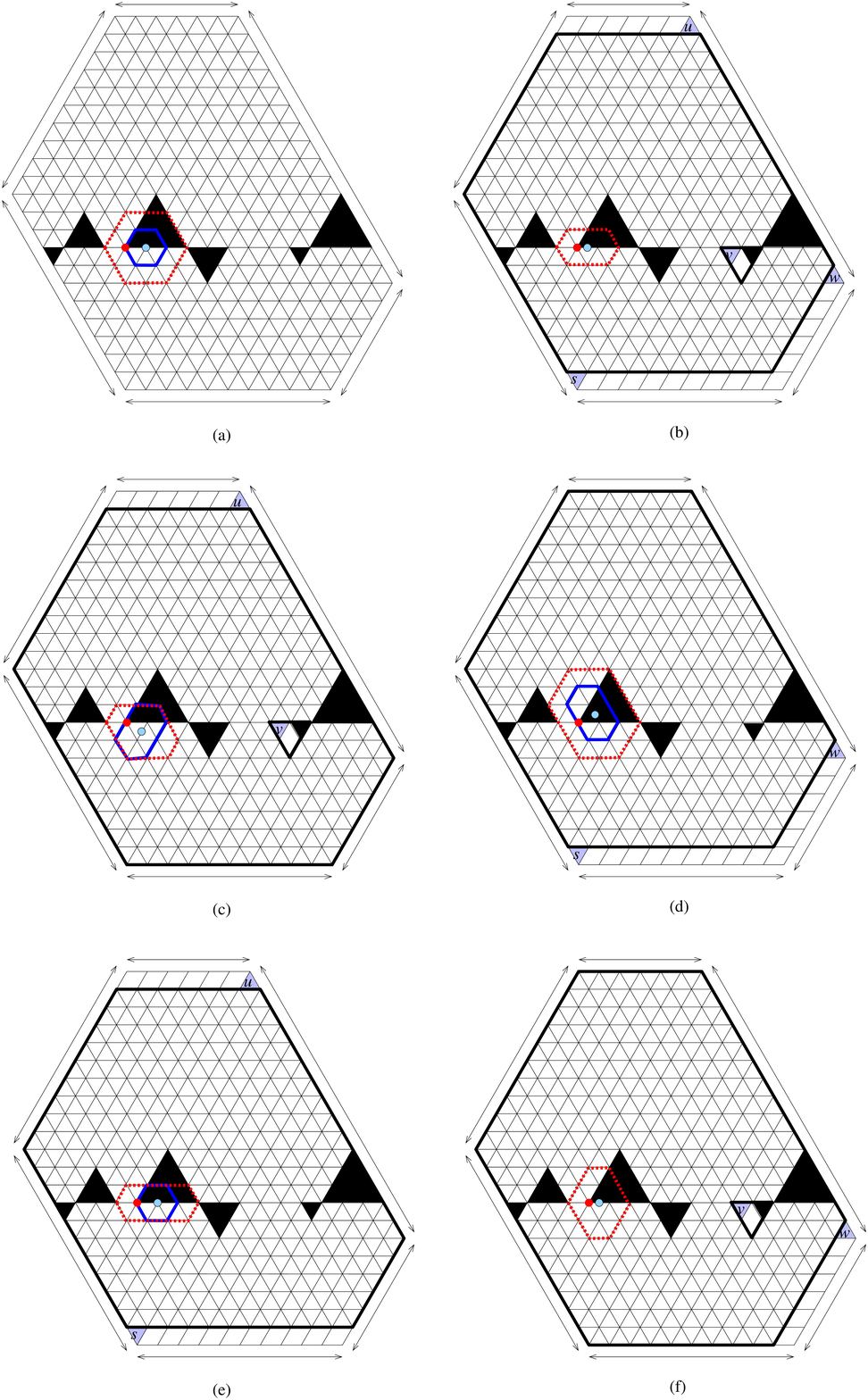}%
\end{picture}%
%
%

\begin{picture}(17695,28209)(1962,-29687)
\put(18545,-28405){\rotatebox{60.0}{\makebox(0,0)[lb]{\smash{{\SetFigFont{14}{16.8}{\familydefault}{\mddefault}{\updefault}{\color[rgb]{0,0,0}$z+o_a+e_b+e_c$}%
}}}}}
\put(9544,-28405){\rotatebox{60.0}{\makebox(0,0)[lb]{\smash{{\SetFigFont{14}{16.8}{\familydefault}{\mddefault}{\updefault}{\color[rgb]{0,0,0}$z+o_a+e_b+e_c$}%
}}}}}
\put(9309,-9378){\rotatebox{60.0}{\makebox(0,0)[lb]{\smash{{\SetFigFont{14}{16.8}{\familydefault}{\mddefault}{\updefault}{\color[rgb]{0,0,0}$z+o_a+e_b+e_c$}%
}}}}}
\put(18340,-18838){\rotatebox{60.0}{\makebox(0,0)[lb]{\smash{{\SetFigFont{14}{16.8}{\familydefault}{\mddefault}{\updefault}{\color[rgb]{0,0,0}$z+o_a+e_b+e_c$}%
}}}}}
\put(9339,-18838){\rotatebox{60.0}{\makebox(0,0)[lb]{\smash{{\SetFigFont{14}{16.8}{\familydefault}{\mddefault}{\updefault}{\color[rgb]{0,0,0}$z+o_a+e_b+e_c$}%
}}}}}
\put(18310,-9378){\rotatebox{60.0}{\makebox(0,0)[lb]{\smash{{\SetFigFont{14}{16.8}{\familydefault}{\mddefault}{\updefault}{\color[rgb]{0,0,0}$z+o_a+e_b+e_c$}%
}}}}}
\put(4808,-1759){\makebox(0,0)[lb]{\smash{{\SetFigFont{14}{16.8}{\familydefault}{\mddefault}{\updefault}{\color[rgb]{0,0,0}$x+o_a+e_b+e_c$}%
}}}}
\put(7672,-2952){\rotatebox{300.0}{\makebox(0,0)[lb]{\smash{{\SetFigFont{14}{16.8}{\familydefault}{\mddefault}{\updefault}{\color[rgb]{0,0,0}$2y+z+e_a+o_b+o_c+b-a$}%
}}}}}
\put(5320,-10027){\makebox(0,0)[lb]{\smash{{\SetFigFont{14}{16.8}{\familydefault}{\mddefault}{\updefault}{\color[rgb]{0,0,0}$x+e_a+o_b+o_c$}%
}}}}
\put(2047,-6377){\rotatebox{300.0}{\makebox(0,0)[lb]{\smash{{\SetFigFont{14}{16.8}{\familydefault}{\mddefault}{\updefault}{\color[rgb]{0,0,0}$2y+z+o_a+e_b+e_c+b-a$}%
}}}}}
\put(2354,-4417){\rotatebox{60.0}{\makebox(0,0)[lb]{\smash{{\SetFigFont{14}{16.8}{\familydefault}{\mddefault}{\updefault}{\color[rgb]{0,0,0}$z+e_a+o_b+o_c$}%
}}}}}
\put(13809,-1759){\makebox(0,0)[lb]{\smash{{\SetFigFont{14}{16.8}{\familydefault}{\mddefault}{\updefault}{\color[rgb]{0,0,0}$x+o_a+e_b+e_c$}%
}}}}
\put(16673,-2952){\rotatebox{300.0}{\makebox(0,0)[lb]{\smash{{\SetFigFont{14}{16.8}{\familydefault}{\mddefault}{\updefault}{\color[rgb]{0,0,0}$2y+z+e_a+o_b+o_c+b-a$}%
}}}}}
\put(14321,-10027){\makebox(0,0)[lb]{\smash{{\SetFigFont{14}{16.8}{\familydefault}{\mddefault}{\updefault}{\color[rgb]{0,0,0}$x+e_a+o_b+o_c$}%
}}}}
\put(11048,-6377){\rotatebox{300.0}{\makebox(0,0)[lb]{\smash{{\SetFigFont{14}{16.8}{\familydefault}{\mddefault}{\updefault}{\color[rgb]{0,0,0}$2y+z+o_a+e_b+e_c+b-a$}%
}}}}}
\put(11355,-4417){\rotatebox{60.0}{\makebox(0,0)[lb]{\smash{{\SetFigFont{14}{16.8}{\familydefault}{\mddefault}{\updefault}{\color[rgb]{0,0,0}$z+e_a+o_b+o_c$}%
}}}}}
\put(4838,-11219){\makebox(0,0)[lb]{\smash{{\SetFigFont{14}{16.8}{\familydefault}{\mddefault}{\updefault}{\color[rgb]{0,0,0}$x+o_a+e_b+e_c$}%
}}}}
\put(7702,-12412){\rotatebox{300.0}{\makebox(0,0)[lb]{\smash{{\SetFigFont{14}{16.8}{\familydefault}{\mddefault}{\updefault}{\color[rgb]{0,0,0}$2y+z+e_a+o_b+o_c+b-a$}%
}}}}}
\put(5350,-19487){\makebox(0,0)[lb]{\smash{{\SetFigFont{14}{16.8}{\familydefault}{\mddefault}{\updefault}{\color[rgb]{0,0,0}$x+e_a+o_b+o_c$}%
}}}}
\put(2077,-15837){\rotatebox{300.0}{\makebox(0,0)[lb]{\smash{{\SetFigFont{14}{16.8}{\familydefault}{\mddefault}{\updefault}{\color[rgb]{0,0,0}$2y+z+o_a+e_b+e_c+b-a$}%
}}}}}
\put(2384,-13877){\rotatebox{60.0}{\makebox(0,0)[lb]{\smash{{\SetFigFont{14}{16.8}{\familydefault}{\mddefault}{\updefault}{\color[rgb]{0,0,0}$z+e_a+o_b+o_c$}%
}}}}}
\put(13839,-11219){\makebox(0,0)[lb]{\smash{{\SetFigFont{14}{16.8}{\familydefault}{\mddefault}{\updefault}{\color[rgb]{0,0,0}$x+o_a+e_b+e_c$}%
}}}}
\put(16703,-12412){\rotatebox{300.0}{\makebox(0,0)[lb]{\smash{{\SetFigFont{14}{16.8}{\familydefault}{\mddefault}{\updefault}{\color[rgb]{0,0,0}$2y+z+e_a+o_b+o_c+b-a$}%
}}}}}
\put(14351,-19487){\makebox(0,0)[lb]{\smash{{\SetFigFont{14}{16.8}{\familydefault}{\mddefault}{\updefault}{\color[rgb]{0,0,0}$x+e_a+o_b+o_c$}%
}}}}
\put(11078,-15837){\rotatebox{300.0}{\makebox(0,0)[lb]{\smash{{\SetFigFont{14}{16.8}{\familydefault}{\mddefault}{\updefault}{\color[rgb]{0,0,0}$2y+z+o_a+e_b+e_c+b-a$}%
}}}}}
\put(11385,-13877){\rotatebox{60.0}{\makebox(0,0)[lb]{\smash{{\SetFigFont{14}{16.8}{\familydefault}{\mddefault}{\updefault}{\color[rgb]{0,0,0}$z+e_a+o_b+o_c$}%
}}}}}
\put(5043,-20786){\makebox(0,0)[lb]{\smash{{\SetFigFont{14}{16.8}{\familydefault}{\mddefault}{\updefault}{\color[rgb]{0,0,0}$x+o_a+e_b+e_c$}%
}}}}
\put(7907,-21979){\rotatebox{300.0}{\makebox(0,0)[lb]{\smash{{\SetFigFont{14}{16.8}{\familydefault}{\mddefault}{\updefault}{\color[rgb]{0,0,0}$2y+z+e_a+o_b+o_c+b-a$}%
}}}}}
\put(5555,-29054){\makebox(0,0)[lb]{\smash{{\SetFigFont{14}{16.8}{\familydefault}{\mddefault}{\updefault}{\color[rgb]{0,0,0}$x+e_a+o_b+o_c$}%
}}}}
\put(2282,-25404){\rotatebox{300.0}{\makebox(0,0)[lb]{\smash{{\SetFigFont{14}{16.8}{\familydefault}{\mddefault}{\updefault}{\color[rgb]{0,0,0}$2y+z+o_a+e_b+e_c+b-a$}%
}}}}}
\put(2589,-23444){\rotatebox{60.0}{\makebox(0,0)[lb]{\smash{{\SetFigFont{14}{16.8}{\familydefault}{\mddefault}{\updefault}{\color[rgb]{0,0,0}$z+e_a+o_b+o_c$}%
}}}}}
\put(14044,-20786){\makebox(0,0)[lb]{\smash{{\SetFigFont{14}{16.8}{\familydefault}{\mddefault}{\updefault}{\color[rgb]{0,0,0}$x+o_a+e_b+e_c$}%
}}}}
\put(16908,-21979){\rotatebox{300.0}{\makebox(0,0)[lb]{\smash{{\SetFigFont{14}{16.8}{\familydefault}{\mddefault}{\updefault}{\color[rgb]{0,0,0}$2y+z+e_a+o_b+o_c+b-a$}%
}}}}}
\put(14556,-29054){\makebox(0,0)[lb]{\smash{{\SetFigFont{14}{16.8}{\familydefault}{\mddefault}{\updefault}{\color[rgb]{0,0,0}$x+e_a+o_b+o_c$}%
}}}}
\put(11283,-25404){\rotatebox{300.0}{\makebox(0,0)[lb]{\smash{{\SetFigFont{14}{16.8}{\familydefault}{\mddefault}{\updefault}{\color[rgb]{0,0,0}$2y+z+o_a+e_b+e_c+b-a$}%
}}}}}
\put(11590,-23444){\rotatebox{60.0}{\makebox(0,0)[lb]{\smash{{\SetFigFont{14}{16.8}{\familydefault}{\mddefault}{\updefault}{\color[rgb]{0,0,0}$z+e_a+o_b+o_c$}%
}}}}}
\end{picture}%
}
\caption{Obtaining a recurrence for $E^{(1)}$-type regions with $a\leq b$. Kuo condensation is applied to the region $E^{(1)}_{2,2,2}(1,2 ;\ 3,2 ;\ 3,1)$ (picture (a)) as shown on the picture (b).}\label{fig:kuooff1}
\end{figure}

Apply Kuo condensation to the dual graph $G$ of the region $E^{(2)}_{x,y,z}(\textbf{a};\ \textbf{c};\ \textbf{b})$ with the same choice of the four vertices  $u,v,w,s$ as above, we get the $E^{(2)}$-recurrence for $a\leq b$ as
\begin{align}\label{offcenterrecurE2a}
\M(E^{(2)}_{x,y,z}(\textbf{a};\ \textbf{c};\ \textbf{b}))\M(F^{(1)}_{x,y-1,z-1}(\textbf{b}^{+1};\ \overline{\textbf{c}}; \ \textbf{a}))&=\M(K^{(2)}_{x,y-1,z-1}(\textbf{b}^{+1};\ \overline{\textbf{c}};\ \textbf{a})) \M(R^{\nwarrow}_{x,y-1,z}(\textbf{a};\ \textbf{c};\ \textbf{b}))\notag\\
&+
\M(E^{(2)}_{x+1,y,z-1}(\textbf{a};\ \textbf{c};\ \textbf{b}))
 \M(F^{(1)}_{x-1,y-1,z}(\textbf{b}^{+1};\ \overline{\textbf{c}};\ \textbf{a})),
\end{align}
 and the $E^{(2)}$-recurrence for $a> b$:
\begin{align}\label{offcenterrecurE2b}
\M(E^{(2)}_{x,y,z}(\textbf{a};\ \textbf{c};\ \textbf{b}))\M(F^{(1)}_{x,y,z-1}(\textbf{b}^{+1};\ \overline{\textbf{c}}; \ \textbf{a}))&=\M(K^{(2)}_{x,y,z-1}(\textbf{b}^{+1};\ \overline{\textbf{c}};\ \textbf{a})) \M(R^{\nwarrow}_{x,y-1,z}(\textbf{a};\ \textbf{c};\ \textbf{b}))\notag\\
&+
\M(E^{(2)}_{x+1,y,z-1}(\textbf{a};\ \textbf{c};\ \textbf{b})) \M(F^{(1)}_{x-1,y,z}(\textbf{b}^{+1};\ \overline{\textbf{c}};\ \textbf{a})).
\end{align}
We note that the removal of forced lozenges from the region corresponding with $G-\{w,s\}$ yield a $R^{\nwarrow}$-type region in Theorem 2.4 in \cite{HoleDent}.

Finally, we have the recurrence for $E^{(6)}$-type region by the same application of Kuo condensation. For $a\leq b$, we have
\begin{align}\label{offcenterrecurE6a}
\M(E^{(6)}_{x,y,z}(\textbf{a};\ \textbf{c};\ \textbf{b}))\M(F^{(1)}_{x,y-1,z-1}(\textbf{a};\ \textbf{c}; \ \textbf{b}^{+1}))&=\M(R^{\swarrow}_{x,y-1,z-1}(\textbf{a};\ \textbf{c}; \ \textbf{b}^{+1})) \M(G^{(4)}_{x,y-1,z}(\textbf{b};\ \overline{\textbf{c}};\ \textbf{a}))\notag\\
&+
\M(E^{(6)}_{x+1,y,z-1}(\textbf{a};\ \textbf{c};\ \textbf{b})) \M(F^{(1)}_{x-1,y-1,z}(\textbf{a};\ \textbf{c}; \ \textbf{b}^{+1})),
\end{align}
and for $a> b$
\begin{align}\label{offcenterrecurE6b}
\M(E^{(6)}_{x,y,z}(\textbf{a};\ \textbf{c};\ \textbf{b}))\M(F^{(1)}_{x,y,z-1}(\textbf{a};\ \textbf{c}; \ \textbf{b}^{+1}))&=\M(R^{\swarrow}_{x,y,z-1}(\textbf{a};\ \textbf{c}; \ \textbf{b}^{+1})) \M(G^{(4)}_{x,y-1,z}(\textbf{b};\ \overline{\textbf{c}};\ \textbf{a}))\notag\\
&+
\M(E^{(6)}_{x+1,y,z-1}(\textbf{a};\ \textbf{c};\ \textbf{b})) \M(F^{(1)}_{x-1,y,z}(\textbf{a};\ \textbf{c}; \ \textbf{b}^{+1})).
\end{align}
The region corresponding with $G-\{u,v\}$ is congruent with, after removing forced lozenges, a region of type $R^{\swarrow}$ in Theorem 2.5 of \cite{HoleDent}.

\subsection{Recurrences for $F^{(i)}$-type regions} \label{subsec:offrecurF}



Applying Kuo Theorem \ref{kuothm1} to the dual graph $G$ of region $F^{(1)}_{x,y,z}(\textbf{a};\ \textbf{c};\ \textbf{b})$, for $a>b$, with the same choice of the four vertices $u,v,w,s$ as in the case for $E^{(i)}$-type regions in the previous subsection (see Figure \ref{fig:offF}). By considering forced lozenges yielded by the removal of the $u$-, $v$-, $w$-,$s$-triangles as in Figures \ref{fig:kuooff2}(b)--(f), we get respectively
\begin{equation}\label{kuothm2eq1}
\M(G-\{u,v,w,s\})=\M(E^{(2)}_{x,y,z-1}(\textbf{b}^{+1};\ \overline{ \textbf{c}}; \ \textbf{a}))
\end{equation}
\begin{equation}\label{kuothm2eq2}
\M(G-\{u,v\})=\M(R^{\nwarrow}_{x,y+1,z-1}(\textbf{b}^{+1};\ \overline{ \textbf{c}}; \ \textbf{a}))
\end{equation}
\begin{equation}\label{kuothm2eq3}
\M(G-\{w,s\})=\M(K^{(2)}_{x,y-1,z}(\textbf{a};\ \textbf{c};\ \textbf{b}))
\end{equation}
\begin{equation}\label{kuothm2eq4}
\M(G-\{u,s\})=\M(F^{(1)}_{x+1,y,z-1}(\textbf{a};\ \textbf{c};\ \textbf{b}))
\end{equation}
\begin{equation}\label{kuothm2eq5}
\M(G-\{v,w\})=\M(E^{(2)}_{x-1,y,z}(\textbf{b}^{+1};\ \overline{ \textbf{c}}; \ \textbf{a})).
\end{equation}
We note that the regions on the right-hand sides of (\ref{kuothm2eq1}), (\ref{kuothm2eq2}) and (\ref{kuothm2eq5}) obtained by $180^{\circ}$-rotating the leftover regions (i.e. the remaining ones after removing forced lozenges; these regions are illustrated as the ones with bold contour).

Plugging the above  5 equations into the equation in Kuo's Theorem \ref{kuothm1}, we get the $F^{(1)}$-recurrence for $a>b$
\begin{align}\label{offcenterrecurF1b}
\M(F^{(1)}_{x,y,z}(\textbf{a};\ \textbf{c};\ \textbf{b}))\M(E^{(2)}_{x,y,z-1}(\textbf{b}^{+1};\ \overline{ \textbf{c}}; \ \textbf{a}))&=\M(R^{\nwarrow}_{x,y+1,z-1}(\textbf{b}^{+1};\ \overline{ \textbf{c}}; \ \textbf{a})) \M(K^{(2)}_{x,y-1,z}(\textbf{a};\ \textbf{c};\ \textbf{b}))\notag\\
&+
\M(F^{(1)}_{x+1,y,z-1}(\textbf{a};\ \textbf{c};\ \textbf{b})) \M(E^{(2)}_{x-1,y,z}(\textbf{b}^{+1};\ \overline{ \textbf{c}}; \ \textbf{a})).
\end{align}
Working similarly for the case $a\leq b$, we have the `sibling' recurrence for $a\leq b$
\begin{align}\label{offcenterrecurF1a}
\M(F^{(1)}_{x,y,z}(\textbf{a};\ \textbf{c};\ \textbf{b}))\M(E^{(2)}_{x,y-1,z-1}(\textbf{b}^{+1};\ \overline{ \textbf{c}}; \ \textbf{a}))&=\M(R^{\nwarrow}_{x,y,z-1}(\textbf{b}^{+1};\ \overline{ \textbf{c}}; \ \textbf{a})) \M(K^{(2)}_{x,y-1,z}(\textbf{a};\ \textbf{c};\ \textbf{b}))\notag\\
&+
\M(F^{(1)}_{x+1,y,z-1}(\textbf{a};\ \textbf{c};\ \textbf{b}))
\M(E^{(2)}_{x-1,y-1,z}(\textbf{b}^{+1};\ \overline{ \textbf{c}}; \ \textbf{a})).
\end{align}

\begin{figure}\centering
\setlength{\unitlength}{3947sp}%
\begingroup\makeatletter\ifx\SetFigFont\undefined%
\gdef\SetFigFont#1#2#3#4#5{%
  \reset@font\fontsize{#1}{#2pt}%
  \fontfamily{#3}\fontseries{#4}\fontshape{#5}%
  \selectfont}%
\fi\endgroup%
\resizebox{15cm}{!}{
\begin{picture}(0,0)%
\includegraphics{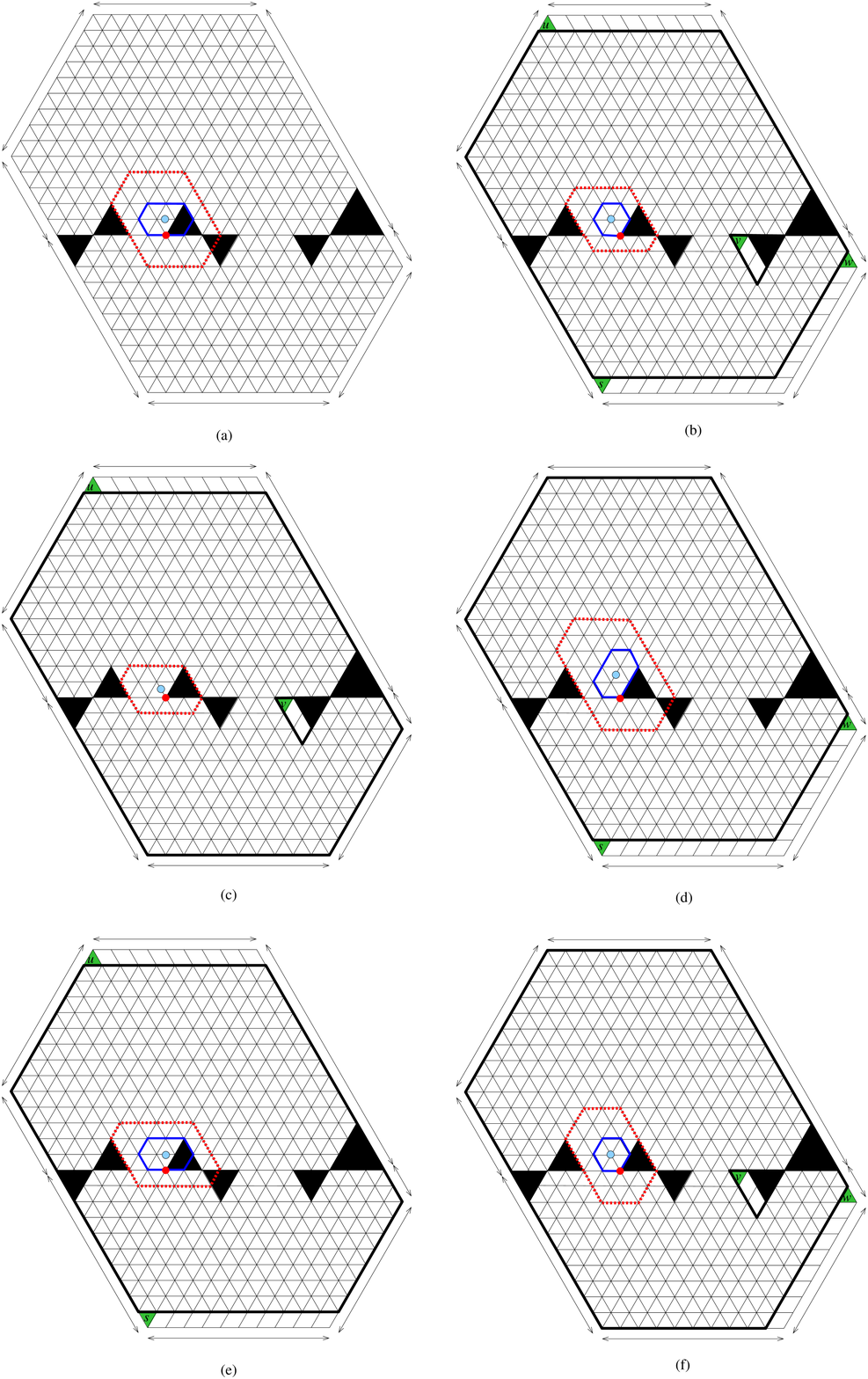}%
\end{picture}%
%
%

\begin{picture}(19498,31352)(1415,-31167)
\put(20749,-26739){\makebox(0,0)[lb]{\smash{{\SetFigFont{14}{16.8}{\rmdefault}{\mddefault}{\updefault}{\color[rgb]{0,0,0}$y$}%
}}}}
\put(18089,-21803){\rotatebox{300.0}{\makebox(0,0)[lb]{\smash{{\SetFigFont{14}{16.8}{\rmdefault}{\mddefault}{\updefault}{\color[rgb]{0,0,0}$y+z+e_a+o_b+o_c+b-a+2$}%
}}}}}
\put(14509,-21154){\makebox(0,0)[lb]{\smash{{\SetFigFont{14}{16.8}{\rmdefault}{\mddefault}{\updefault}{\color[rgb]{0,0,0}$x+o_a+e_b+e_c$}%
}}}}
\put(19726,-29717){\rotatebox{60.0}{\makebox(0,0)[lb]{\smash{{\SetFigFont{14}{16.8}{\rmdefault}{\mddefault}{\updefault}{\color[rgb]{0,0,0}$z+o_a+e_b+e_c$}%
}}}}}
\put(15839,-30603){\makebox(0,0)[lb]{\smash{{\SetFigFont{14}{16.8}{\rmdefault}{\mddefault}{\updefault}{\color[rgb]{0,0,0}$x+e_a+o_b+o_c$}%
}}}}
\put(12872,-27473){\rotatebox{300.0}{\makebox(0,0)[lb]{\smash{{\SetFigFont{14}{16.8}{\rmdefault}{\mddefault}{\updefault}{\color[rgb]{0,0,0}$y+z+o_a+e_b+e_c$}%
}}}}}
\put(11747,-25541){\rotatebox{300.0}{\makebox(0,0)[lb]{\smash{{\SetFigFont{14}{16.8}{\rmdefault}{\mddefault}{\updefault}{\color[rgb]{0,0,0}$y+b-a+2$}%
}}}}}
\put(12084,-23516){\rotatebox{60.0}{\makebox(0,0)[lb]{\smash{{\SetFigFont{14}{16.8}{\rmdefault}{\mddefault}{\updefault}{\color[rgb]{0,0,0}$z+e_a+o_b+o_c$}%
}}}}}
\put(1866,-23500){\rotatebox{60.0}{\makebox(0,0)[lb]{\smash{{\SetFigFont{14}{16.8}{\rmdefault}{\mddefault}{\updefault}{\color[rgb]{0,0,0}$z+e_a+o_b+o_c$}%
}}}}}
\put(1529,-25525){\rotatebox{300.0}{\makebox(0,0)[lb]{\smash{{\SetFigFont{14}{16.8}{\rmdefault}{\mddefault}{\updefault}{\color[rgb]{0,0,0}$y+b-a+2$}%
}}}}}
\put(2654,-27457){\rotatebox{300.0}{\makebox(0,0)[lb]{\smash{{\SetFigFont{14}{16.8}{\rmdefault}{\mddefault}{\updefault}{\color[rgb]{0,0,0}$y+z+o_a+e_b+e_c$}%
}}}}}
\put(5621,-30587){\makebox(0,0)[lb]{\smash{{\SetFigFont{14}{16.8}{\rmdefault}{\mddefault}{\updefault}{\color[rgb]{0,0,0}$x+e_a+o_b+o_c$}%
}}}}
\put(9508,-29701){\rotatebox{60.0}{\makebox(0,0)[lb]{\smash{{\SetFigFont{14}{16.8}{\rmdefault}{\mddefault}{\updefault}{\color[rgb]{0,0,0}$z+o_a+e_b+e_c$}%
}}}}}
\put(10531,-26723){\makebox(0,0)[lb]{\smash{{\SetFigFont{14}{16.8}{\rmdefault}{\mddefault}{\updefault}{\color[rgb]{0,0,0}$y$}%
}}}}
\put(7871,-21787){\rotatebox{300.0}{\makebox(0,0)[lb]{\smash{{\SetFigFont{14}{16.8}{\rmdefault}{\mddefault}{\updefault}{\color[rgb]{0,0,0}$y+z+e_a+o_b+o_c+b-a+2$}%
}}}}}
\put(4291,-21138){\makebox(0,0)[lb]{\smash{{\SetFigFont{14}{16.8}{\rmdefault}{\mddefault}{\updefault}{\color[rgb]{0,0,0}$x+o_a+e_b+e_c$}%
}}}}
\put(12084,-12886){\rotatebox{60.0}{\makebox(0,0)[lb]{\smash{{\SetFigFont{14}{16.8}{\rmdefault}{\mddefault}{\updefault}{\color[rgb]{0,0,0}$z+e_a+o_b+o_c$}%
}}}}}
\put(11747,-14911){\rotatebox{300.0}{\makebox(0,0)[lb]{\smash{{\SetFigFont{14}{16.8}{\rmdefault}{\mddefault}{\updefault}{\color[rgb]{0,0,0}$y+b-a+2$}%
}}}}}
\put(12872,-16843){\rotatebox{300.0}{\makebox(0,0)[lb]{\smash{{\SetFigFont{14}{16.8}{\rmdefault}{\mddefault}{\updefault}{\color[rgb]{0,0,0}$y+z+o_a+e_b+e_c$}%
}}}}}
\put(15839,-19973){\makebox(0,0)[lb]{\smash{{\SetFigFont{14}{16.8}{\rmdefault}{\mddefault}{\updefault}{\color[rgb]{0,0,0}$x+e_a+o_b+o_c$}%
}}}}
\put(19726,-19087){\rotatebox{60.0}{\makebox(0,0)[lb]{\smash{{\SetFigFont{14}{16.8}{\rmdefault}{\mddefault}{\updefault}{\color[rgb]{0,0,0}$z+o_a+e_b+e_c$}%
}}}}}
\put(20749,-16109){\makebox(0,0)[lb]{\smash{{\SetFigFont{14}{16.8}{\rmdefault}{\mddefault}{\updefault}{\color[rgb]{0,0,0}$y$}%
}}}}
\put(18089,-11173){\rotatebox{300.0}{\makebox(0,0)[lb]{\smash{{\SetFigFont{14}{16.8}{\rmdefault}{\mddefault}{\updefault}{\color[rgb]{0,0,0}$y+z+e_a+o_b+o_c+b-a+2$}%
}}}}}
\put(14509,-10524){\makebox(0,0)[lb]{\smash{{\SetFigFont{14}{16.8}{\rmdefault}{\mddefault}{\updefault}{\color[rgb]{0,0,0}$x+o_a+e_b+e_c$}%
}}}}
\put(1866,-12870){\rotatebox{60.0}{\makebox(0,0)[lb]{\smash{{\SetFigFont{14}{16.8}{\rmdefault}{\mddefault}{\updefault}{\color[rgb]{0,0,0}$z+e_a+o_b+o_c$}%
}}}}}
\put(1529,-14895){\rotatebox{300.0}{\makebox(0,0)[lb]{\smash{{\SetFigFont{14}{16.8}{\rmdefault}{\mddefault}{\updefault}{\color[rgb]{0,0,0}$y+b-a+2$}%
}}}}}
\put(4297,-106){\makebox(0,0)[lb]{\smash{{\SetFigFont{14}{16.8}{\rmdefault}{\mddefault}{\updefault}{\color[rgb]{0,0,0}$x+o_a+e_b+e_c$}%
}}}}
\put(7877,-755){\rotatebox{300.0}{\makebox(0,0)[lb]{\smash{{\SetFigFont{14}{16.8}{\rmdefault}{\mddefault}{\updefault}{\color[rgb]{0,0,0}$y+z+e_a+o_b+o_c+b-a+2$}%
}}}}}
\put(10537,-5691){\makebox(0,0)[lb]{\smash{{\SetFigFont{14}{16.8}{\rmdefault}{\mddefault}{\updefault}{\color[rgb]{0,0,0}$y$}%
}}}}
\put(9514,-8669){\rotatebox{60.0}{\makebox(0,0)[lb]{\smash{{\SetFigFont{14}{16.8}{\rmdefault}{\mddefault}{\updefault}{\color[rgb]{0,0,0}$z+o_a+e_b+e_c$}%
}}}}}
\put(5627,-9555){\makebox(0,0)[lb]{\smash{{\SetFigFont{14}{16.8}{\rmdefault}{\mddefault}{\updefault}{\color[rgb]{0,0,0}$x+e_a+o_b+o_c$}%
}}}}
\put(2660,-6425){\rotatebox{300.0}{\makebox(0,0)[lb]{\smash{{\SetFigFont{14}{16.8}{\rmdefault}{\mddefault}{\updefault}{\color[rgb]{0,0,0}$y+z+o_a+e_b+e_c$}%
}}}}}
\put(1535,-4493){\rotatebox{300.0}{\makebox(0,0)[lb]{\smash{{\SetFigFont{14}{16.8}{\rmdefault}{\mddefault}{\updefault}{\color[rgb]{0,0,0}$y+b-a+2$}%
}}}}}
\put(1872,-2468){\rotatebox{60.0}{\makebox(0,0)[lb]{\smash{{\SetFigFont{14}{16.8}{\rmdefault}{\mddefault}{\updefault}{\color[rgb]{0,0,0}$z+e_a+o_b+o_c$}%
}}}}}
\put(14515,-122){\makebox(0,0)[lb]{\smash{{\SetFigFont{14}{16.8}{\rmdefault}{\mddefault}{\updefault}{\color[rgb]{0,0,0}$x+o_a+e_b+e_c$}%
}}}}
\put(18095,-771){\rotatebox{300.0}{\makebox(0,0)[lb]{\smash{{\SetFigFont{14}{16.8}{\rmdefault}{\mddefault}{\updefault}{\color[rgb]{0,0,0}$y+z+e_a+o_b+o_c+b-a+2$}%
}}}}}
\put(20755,-5707){\makebox(0,0)[lb]{\smash{{\SetFigFont{14}{16.8}{\rmdefault}{\mddefault}{\updefault}{\color[rgb]{0,0,0}$y$}%
}}}}
\put(19732,-8685){\rotatebox{60.0}{\makebox(0,0)[lb]{\smash{{\SetFigFont{14}{16.8}{\rmdefault}{\mddefault}{\updefault}{\color[rgb]{0,0,0}$z+o_a+e_b+e_c$}%
}}}}}
\put(15845,-9571){\makebox(0,0)[lb]{\smash{{\SetFigFont{14}{16.8}{\rmdefault}{\mddefault}{\updefault}{\color[rgb]{0,0,0}$x+e_a+o_b+o_c$}%
}}}}
\put(12878,-6441){\rotatebox{300.0}{\makebox(0,0)[lb]{\smash{{\SetFigFont{14}{16.8}{\rmdefault}{\mddefault}{\updefault}{\color[rgb]{0,0,0}$y+z+o_a+e_b+e_c$}%
}}}}}
\put(11753,-4509){\rotatebox{300.0}{\makebox(0,0)[lb]{\smash{{\SetFigFont{14}{16.8}{\rmdefault}{\mddefault}{\updefault}{\color[rgb]{0,0,0}$y+b-a+2$}%
}}}}}
\put(12090,-2484){\rotatebox{60.0}{\makebox(0,0)[lb]{\smash{{\SetFigFont{14}{16.8}{\rmdefault}{\mddefault}{\updefault}{\color[rgb]{0,0,0}$z+e_a+o_b+o_c$}%
}}}}}
\put(4291,-10508){\makebox(0,0)[lb]{\smash{{\SetFigFont{14}{16.8}{\rmdefault}{\mddefault}{\updefault}{\color[rgb]{0,0,0}$x+o_a+e_b+e_c$}%
}}}}
\put(7871,-11157){\rotatebox{300.0}{\makebox(0,0)[lb]{\smash{{\SetFigFont{14}{16.8}{\rmdefault}{\mddefault}{\updefault}{\color[rgb]{0,0,0}$y+z+e_a+o_b+o_c+b-a+2$}%
}}}}}
\put(10531,-16093){\makebox(0,0)[lb]{\smash{{\SetFigFont{14}{16.8}{\rmdefault}{\mddefault}{\updefault}{\color[rgb]{0,0,0}$y$}%
}}}}
\put(9508,-19071){\rotatebox{60.0}{\makebox(0,0)[lb]{\smash{{\SetFigFont{14}{16.8}{\rmdefault}{\mddefault}{\updefault}{\color[rgb]{0,0,0}$z+o_a+e_b+e_c$}%
}}}}}
\put(5621,-19957){\makebox(0,0)[lb]{\smash{{\SetFigFont{14}{16.8}{\rmdefault}{\mddefault}{\updefault}{\color[rgb]{0,0,0}$x+e_a+o_b+o_c$}%
}}}}
\put(2654,-16827){\rotatebox{300.0}{\makebox(0,0)[lb]{\smash{{\SetFigFont{14}{16.8}{\rmdefault}{\mddefault}{\updefault}{\color[rgb]{0,0,0}$y+z+o_a+e_b+e_c$}%
}}}}}
\end{picture}%
}
\caption{Obtaining a recurrence for $F^{(1)}$-type regions with $a\leq b$. Kuo condensation is applied to the region $F^{(1)}_{3,2,2}(2,2;\ 2,2 ;\ 3,2)$ (picture (a)) as shown on the picture (b).}\label{fig:kuooff2}
\end{figure}
Similar application of Kuo condensation we have the following recurrences for other $F^{(i)}$-type regions.
The $F^{(2)}$-recurrence for $a\leq b$ is
\begin{align}\label{offcenterrecurF2a}
\M(F^{(2)}_{x,y,z}(\textbf{a};\ \textbf{c};\ \textbf{b}))\M(E^{(6)}_{x,y-1,z-1}(\textbf{a};\ \textbf{c}; \ \textbf{b}^{+1}))&=\M(G^{(1)}_{x,y,z-1}(\textbf{a};\ \textbf{c};\ \textbf{b}^{+1})) \M(K^{(3)}_{x,y-1,z}(\textbf{a};\ \textbf{c};\ \textbf{b}))\notag\\
&+
\M(F^{(2)}_{x+1,y,z-1}(\textbf{a};\ \textbf{c};\ \textbf{b})) \M(E^{(6)}_{x-1,y-1,z}(\textbf{a};\ \textbf{c};\ \textbf{b}^{+1})),
\end{align}
and its $a>b$ counterpart is
\begin{align}\label{offcenterrecurF2b}
\M(F^{(2)}_{x,y,z}(\textbf{a};\ \textbf{c};\ \textbf{b}))\M(E^{(6)}_{x,y,z-1}(\textbf{a};\ \textbf{c}; \ \textbf{b}^{+1}))&=\M(G^{(1)}_{x,y+1,z-1}(\textbf{a};\ \textbf{c};\ \textbf{b}^{+1})) \M(K^{(3)}_{x,y-1,z}(\textbf{a};\ \textbf{c};\ \textbf{b}))\notag\\
&+
\M(F^{(2)}_{x+1,y,z-1}(\textbf{a};\ \textbf{c};\ \textbf{b})) \M(E^{(6)}_{x-1,y,z}(\textbf{a};\ \textbf{c};\ \textbf{b}^{+1})).
\end{align}

Next, we have the $F^{(3)}$-recurrences:
\begin{align}\label{offcenterrecurF3a}
\M(F^{(3)}_{x,y,z}(\textbf{a};\ \textbf{c};\ \textbf{b}))\M(E^{(1)}_{x,y-1,z-1}(\textbf{a};\ \textbf{c}; \ \textbf{b}^{+1}))&=\M(G^{(2)}_{x,y-1,z-1}(\textbf{a};\ \textbf{c};\ \textbf{b}^{+1})) \M(K^{(4)}_{x,y-1,z}(\textbf{a};\ \textbf{c};\ \textbf{b}))\notag\\
&+
\M(F^{(3)}_{x+1,y,z-1}(\textbf{a};\ \textbf{c};\ \textbf{b})) \M(E^{(1)}_{x-1,y-1,z}(\textbf{a};\ \textbf{c};\ \textbf{b}^{+1}))
\end{align}
if $a\leq b$, and
\begin{align}\label{offcenterrecurF3b}
\M(F^{(3)}_{x,y,z}(\textbf{a};\ \textbf{c};\ \textbf{b}))\M(E^{(1)}_{x,y,z-1}(\textbf{a};\ \textbf{c}; \ \textbf{b}^{+1}))&=\M(G^{(2)}_{x,y,z-1}(\textbf{a};\ \textbf{c};\ \textbf{b}^{+1})) \M(K^{(4)}_{x,y-1,z}(\textbf{a};\ \textbf{c};\ \textbf{b}))\notag\\
&+
\M(F^{(3)}_{x+1,y,z-1}(\textbf{a};\ \textbf{c};\ \textbf{b})) \M(E^{(1)}_{x-1,y,z}(\textbf{a};\ \textbf{c};\ \textbf{b}^{+1}))
\end{align}
if $a>b$.

The last recurrence in this subsection is the $F^{(4)}$-recurrences below.
 For $a\leq b$, we have
\begin{align}\label{offcenterrecurF4a}
\M(F^{(4)}_{x,y,z}(\textbf{a};\ \textbf{c};\ \textbf{b}))\M(E^{(2)}_{x,y-1,z-1}(\textbf{a};\ \textbf{c}; \ \textbf{b}^{+1}))&=\M(G^{(3)}_{x,y-1,z-1}(\textbf{a};\ \textbf{c};\ \textbf{b}^{+1})) \M(K^{(1)}_{x,y,z}(\textbf{b};\ \overline{\textbf{c}};\ \textbf{a}))\notag\\
&+
\M(F^{(4)}_{x+1,y,z-1}(\textbf{a};\ \textbf{c};\ \textbf{b})) \M(E^{(2)}_{x-1,y-1,z}(\textbf{a};\ \textbf{c};\ \textbf{b}^{+1})),
\end{align}
and for $a>b$
\begin{align}\label{offcenterrecurF4b}
\M(F^{(4)}_{x,y,z}(\textbf{a};\ \textbf{c};\ \textbf{b}))\M(E^{(2)}_{x,y,z-1}(\textbf{a};\ \textbf{c}; \ \textbf{b}^{+1}))&=\M(G^{(3)}_{x,y,z-1}(\textbf{a};\ \textbf{c};\ \textbf{b}^{+1})) \M(K^{(1)}_{x,y,z}(\textbf{b};\ \overline{\textbf{c}};\ \textbf{a}))\notag\\
&+
\M(F^{(4)}_{x+1,y,z-1}(\textbf{a};\ \textbf{c};\ \textbf{b})) \M(E^{(2)}_{x-1,y,z}(\textbf{a};\ \textbf{c};\ \textbf{b}^{+1})).
\end{align}



\subsection{Recurrences for $G^{(i)}$-type regions} \label{subsec:offrecurG}


For the $G^{(i)}$-type regions, we still apply Kuo's Theorem \ref{kuothm1}, however the four vertices $u,v,w,s$ are chosen differently as in Figure \ref{fig:offG}. In particular, the $u$- and $s$-triangles are the shaded one appended to the ends of the right and
left ferns, respectively. The $v$ and $w$-triangles are located at the east and southwest corners of the region, respectively.

Let us consider the dual graph $G$ of the region $G^{(2)}_{x,y,z}(\textbf{a};\ \textbf{c};\ \textbf{b})$ in the case $a>b$. Working on forced lozenges formed from the removal of the $u$-, $v$-, $w$-,$s$-triangles as in Figures \ref{fig:kuooff3}(b)--(f), we have respectively
\begin{equation}\label{kuothm3eq1}
\M(G-\{u,v,w,s\})=\M(E^{(1)}_{x-1,y,z-1}(\textbf{a}^{+1};\ \textbf{c}; \ \textbf{b}^{+1}))
\end{equation}
\begin{equation}\label{kuothm3eq2}
\M(G-\{u,v\})=\M(K^{(1)}_{x-1,y,z}(\textbf{b}^{+1};\ \overline{\textbf{c}}; \ \textbf{a}))
\end{equation}
\begin{equation}\label{kuothm3eq3}
\M(G-\{w,s\})=\M(F^{(3)}_{x,y,z-1}(\textbf{a}^{+1};\ \textbf{c};\ \textbf{b}))
\end{equation}
\begin{equation}\label{kuothm3eq4}
\M(G-\{u,s\})=\M(G^{(2)}_{x-1,y,z-1}(\textbf{a}^{+1};\ \textbf{c};\ \textbf{b}^{+1}))
\end{equation}
\begin{equation}\label{kuothm3eq5}
\M(G-\{v,w\})=\M(E^{(1)}_{x,y,z}(\textbf{a};\ \textbf{c};\ \textbf{b})).
\end{equation}
We note that the region on the right-hand side of (\ref{kuothm3eq2}) is obtained from the leftover region (the one restricted by the bold contour in Figure \ref{fig:kuooff3}(c)) by a $180^{\circ}$-rotation.

Plugging the above 5 equalities into the equation in Kuo's Theorem \ref{kuothm1}, we get the $G^{(2)}$-recurrences for $a>b$:
\begin{align}\label{offcenterrecurG2b}
\M(G^{(2)}_{x,y,z}(\textbf{a};\ \textbf{c};\ \textbf{b}))\M(E^{(1)}_{x-1,y,z-1}(\textbf{a}^{+1};\ \textbf{c}; \ \textbf{b}^{+1}))&=\M(K^{(1)}_{x-1,y,z}(\textbf{b}^{+1};\ \overline{\textbf{c}}; \ \textbf{a})) \M(F^{(3)}_{x,y,z-1}(\textbf{a}^{+1};\ \textbf{c};\ \textbf{b}))\notag\\
&+
\M(G^{(2)}_{x-1,y,z-1}(\textbf{a}^{+1};\ \textbf{c};\ \textbf{b}^{+1})) \M(E^{(1)}_{x,y,z}(\textbf{a};\ \textbf{c};\ \textbf{b})).
\end{align}
Similarly, we get the $G^{(2)}$-recurrence for $a< b$
\begin{align}\label{offcenterrecurG2a}
\M(G^{(2)}_{x,y,z}(\textbf{a};\ \textbf{c};\ \textbf{b}))\M(E^{(1)}_{x-1,y,z-1}(\textbf{a}^{+1};\ \textbf{c}; \ \textbf{b}^{+1}))&=\M(K^{(1)}_{x-1,y-1,z}(\textbf{b}^{+1};\ \overline{\textbf{c}}; \ \textbf{a})) \M(F^{(3)}_{x,y+1,z-1}(\textbf{a}^{+1};\ \textbf{c};\ \textbf{b}))\notag\\
&+
\M(G^{(2)}_{x-1,y,z-1}(\textbf{a}^{+1};\ \textbf{c};\ \textbf{b}^{+1})) \M(E^{(1)}_{x,y,z}(\textbf{a};\ \textbf{c};\ \textbf{b})),
\end{align}
 and for for $a=b$
\begin{align}\label{offcenterrecurG2c}
\M(G^{(2)}_{x,y,z}(\textbf{a};\ \textbf{c};\ \textbf{b}))\M(E^{(1)}_{x-1,y,z-1}(\textbf{a}^{+1};\ \textbf{c}; \ \textbf{b}^{+1}))&=\M(K^{(1)}_{x-1,y-1,z}(\textbf{b}^{+1};\ \overline{\textbf{c}}; \ \textbf{a})) \M(F^{(3)}_{x,y,z-1}(\textbf{a}^{+1};\ \textbf{c};\ \textbf{b}))\notag\\
&+
\M(G^{(2)}_{x-1,y,z-1}(\textbf{a}^{+1};\ \textbf{c};\ \textbf{b}^{+1})) \M(E^{(1)}_{x,y,z}(\textbf{a};\ \textbf{c};\ \textbf{b})).
\end{align}

The same application of Kuo condensation gives us the recurrence of other $G^{(i)}$-type regions.

The $G^{(1)}$-recurrence for $a<b$
\begin{align}\label{offcenterrecurG1a}
\M(G^{(1)}_{x,y,z}(\textbf{a};\ \textbf{c};\ \textbf{b}))\M(E^{(6)}_{x-1,y-1,z-1}(\textbf{a}^{+1};\ \textbf{c}; \ \textbf{b}^{+1}))&=\M(R^{\swarrow}_{x-1,y-1,z}(\textbf{a};\ \textbf{c}; \ \textbf{b}^{+1})) \M(F^{(2)}_{x,y,z-1}(\textbf{a}^{+1};\ \textbf{c};\ \textbf{b}))\notag\\
&+
\M(G^{(1)}_{x-1,y,z-1}(\textbf{a}^{+1};\ \textbf{c};\ \textbf{b}^{+1})) \M(E^{(6)}_{x,y-1,z}(\textbf{a};\ \textbf{c};\ \textbf{b}))
\end{align}, for $a> b$
\begin{align}\label{offcenterrecurG1b}
\M(G^{(1)}_{x,y,z}(\textbf{a};\ \textbf{c};\ \textbf{b}))\M(E^{(6)}_{x-1,y-1,z-1}(\textbf{a}^{+1};\ \textbf{c}; \ \textbf{b}^{+1}))&=\M(R^{\swarrow}_{x-1,y,z}(\textbf{a};\ \textbf{c}; \ \textbf{b}^{+1})) \M(F^{(2)}_{x,y-1,z-1}(\textbf{a}^{+1};\ \textbf{c};\ \textbf{b}))\notag\\
&+
\M(G^{(1)}_{x-1,y,z-1}(\textbf{a}^{+1};\ \textbf{c};\ \textbf{b}^{+1})) \M(E^{(6)}_{x,y-1,z}(\textbf{a};\ \textbf{c};\ \textbf{b})),
\end{align}
and for $a=b$
\begin{align}\label{offcenterrecurG1c}
\M(G^{(1)}_{x,y,z}(\textbf{a};\ \textbf{c};\ \textbf{b}))\M(E^{(6)}_{x-1,y-1,z-1}(\textbf{a}^{+1};\ \textbf{c}; \ \textbf{b}^{+1}))&=\M(R^{\swarrow}_{x-1,y-1,z}(\textbf{a};\ \textbf{c}; \ \textbf{b}^{+1})) \M(F^{(2)}_{x,y-1,z-1}(\textbf{a}^{+1};\ \textbf{c};\ \textbf{b}))\notag\\
&+
\M(G^{(1)}_{x-1,y,z-1}(\textbf{a}^{+1};\ \textbf{c};\ \textbf{b}^{+1})) \M(E^{(6)}_{x,y-1,z}(\textbf{a};\ \textbf{c};\ \textbf{b})).
\end{align}

\begin{figure}\centering
\setlength{\unitlength}{3947sp}%
\begingroup\makeatletter\ifx\SetFigFont\undefined%
\gdef\SetFigFont#1#2#3#4#5{%
  \reset@font\fontsize{#1}{#2pt}%
  \fontfamily{#3}\fontseries{#4}\fontshape{#5}%
  \selectfont}%
\fi\endgroup%
\resizebox{15cm}{!}{
\begin{picture}(0,0)%
\includegraphics{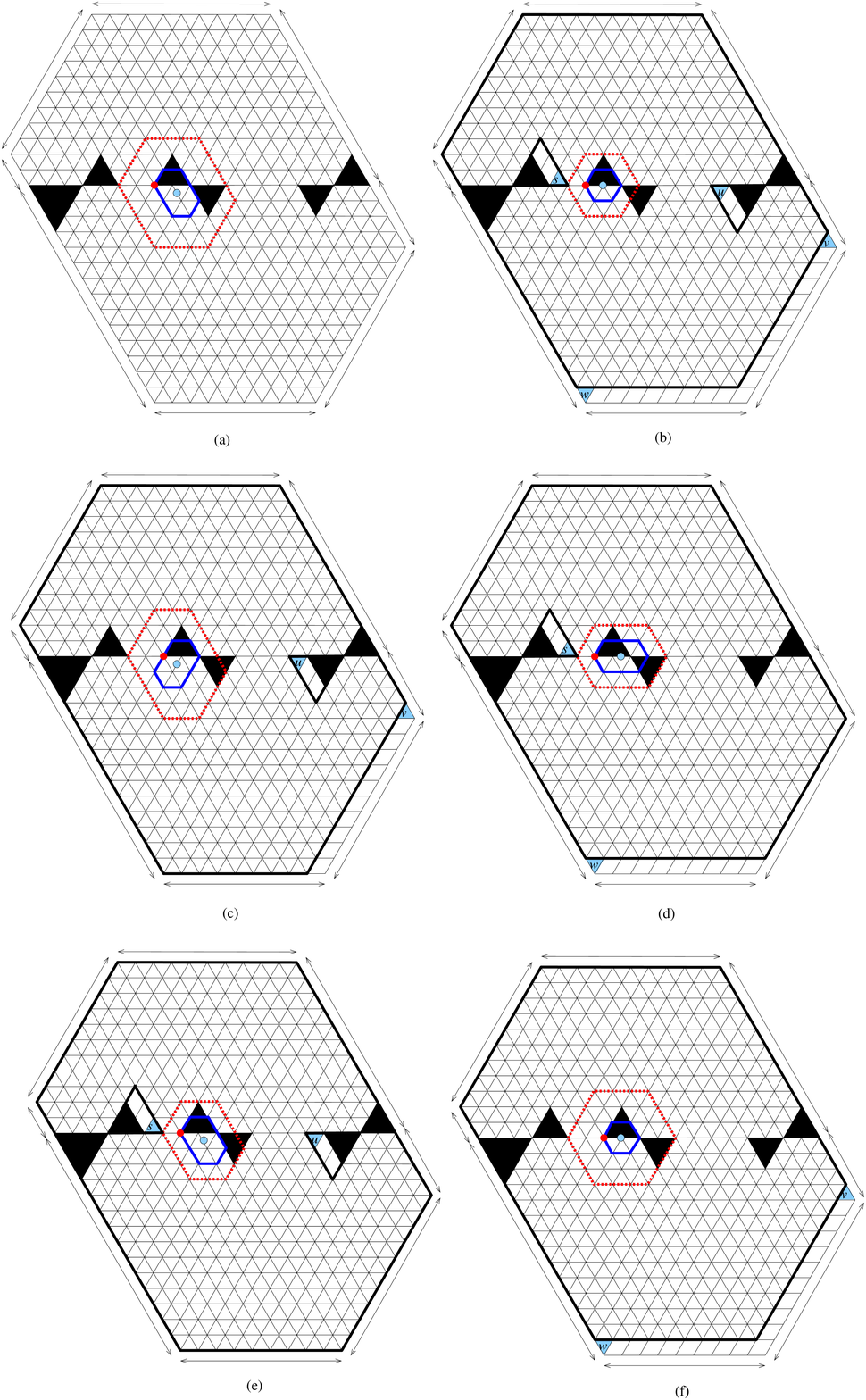}%
\end{picture}%
%
%

\begin{picture}(20105,32232)(2946,-34218)
\put(13793,-26307){\rotatebox{60.0}{\makebox(0,0)[lb]{\smash{{\SetFigFont{14}{16.8}{\rmdefault}{\mddefault}{\itdefault}{\color[rgb]{0,0,0}$z+e_a+o_b+o_c$}%
}}}}}
\put(13213,-28137){\makebox(0,0)[lb]{\smash{{\SetFigFont{14}{16.8}{\rmdefault}{\mddefault}{\itdefault}{\color[rgb]{0,0,0}$y$}%
}}}}
\put(21423,-32637){\rotatebox{60.0}{\makebox(0,0)[lb]{\smash{{\SetFigFont{14}{16.8}{\rmdefault}{\mddefault}{\itdefault}{\color[rgb]{0,0,0}$z+o_a+e_b+e_c$}%
}}}}}
\put(17663,-33787){\makebox(0,0)[lb]{\smash{{\SetFigFont{14}{16.8}{\rmdefault}{\mddefault}{\itdefault}{\color[rgb]{0,0,0}$x+e_a+o_b+o_c$}%
}}}}
\put(14195,-29777){\rotatebox{300.0}{\makebox(0,0)[lb]{\smash{{\SetFigFont{14}{16.8}{\rmdefault}{\mddefault}{\itdefault}{\color[rgb]{0,0,0}$y+z+o_a+e_b+e_c+a-b+1$}%
}}}}}
\put(6381,-2271){\makebox(0,0)[lb]{\smash{{\SetFigFont{14}{16.8}{\rmdefault}{\mddefault}{\itdefault}{\color[rgb]{0,0,0}$x+o_a+e_b+e_c$}%
}}}}
\put(10171,-3381){\rotatebox{300.0}{\makebox(0,0)[lb]{\smash{{\SetFigFont{14}{16.8}{\rmdefault}{\mddefault}{\itdefault}{\color[rgb]{0,0,0}$y+z+e_a+o_b+o_c$}%
}}}}}
\put(11981,-6571){\rotatebox{300.0}{\makebox(0,0)[lb]{\smash{{\SetFigFont{14}{16.8}{\rmdefault}{\mddefault}{\itdefault}{\color[rgb]{0,0,0}$y+a-b+1$}%
}}}}}
\put(3943,-8041){\rotatebox{300.0}{\makebox(0,0)[lb]{\smash{{\SetFigFont{14}{16.8}{\rmdefault}{\mddefault}{\itdefault}{\color[rgb]{0,0,0}$y+z+o_a+e_b+e_c+a-b+1$}%
}}}}}
\put(7411,-12051){\makebox(0,0)[lb]{\smash{{\SetFigFont{14}{16.8}{\rmdefault}{\mddefault}{\itdefault}{\color[rgb]{0,0,0}$x+e_a+o_b+o_c$}%
}}}}
\put(11171,-10901){\rotatebox{60.0}{\makebox(0,0)[lb]{\smash{{\SetFigFont{14}{16.8}{\rmdefault}{\mddefault}{\itdefault}{\color[rgb]{0,0,0}$z+o_a+e_b+e_c$}%
}}}}}
\put(2961,-6401){\makebox(0,0)[lb]{\smash{{\SetFigFont{14}{16.8}{\rmdefault}{\mddefault}{\itdefault}{\color[rgb]{0,0,0}$y$}%
}}}}
\put(3541,-4571){\rotatebox{60.0}{\makebox(0,0)[lb]{\smash{{\SetFigFont{14}{16.8}{\rmdefault}{\mddefault}{\itdefault}{\color[rgb]{0,0,0}$z+e_a+o_b+o_c$}%
}}}}}
\put(16215,-2276){\makebox(0,0)[lb]{\smash{{\SetFigFont{14}{16.8}{\rmdefault}{\mddefault}{\itdefault}{\color[rgb]{0,0,0}$x+o_a+e_b+e_c$}%
}}}}
\put(20005,-3386){\rotatebox{300.0}{\makebox(0,0)[lb]{\smash{{\SetFigFont{14}{16.8}{\rmdefault}{\mddefault}{\itdefault}{\color[rgb]{0,0,0}$y+z+e_a+o_b+o_c$}%
}}}}}
\put(21815,-6576){\rotatebox{300.0}{\makebox(0,0)[lb]{\smash{{\SetFigFont{14}{16.8}{\rmdefault}{\mddefault}{\itdefault}{\color[rgb]{0,0,0}$y+a-b+1$}%
}}}}}
\put(13777,-8046){\rotatebox{300.0}{\makebox(0,0)[lb]{\smash{{\SetFigFont{14}{16.8}{\rmdefault}{\mddefault}{\itdefault}{\color[rgb]{0,0,0}$y+z+o_a+e_b+e_c+a-b+1$}%
}}}}}
\put(17245,-12056){\makebox(0,0)[lb]{\smash{{\SetFigFont{14}{16.8}{\rmdefault}{\mddefault}{\itdefault}{\color[rgb]{0,0,0}$x+e_a+o_b+o_c$}%
}}}}
\put(21005,-10906){\rotatebox{60.0}{\makebox(0,0)[lb]{\smash{{\SetFigFont{14}{16.8}{\rmdefault}{\mddefault}{\itdefault}{\color[rgb]{0,0,0}$z+o_a+e_b+e_c$}%
}}}}}
\put(12795,-6406){\makebox(0,0)[lb]{\smash{{\SetFigFont{14}{16.8}{\rmdefault}{\mddefault}{\itdefault}{\color[rgb]{0,0,0}$y$}%
}}}}
\put(13375,-4576){\rotatebox{60.0}{\makebox(0,0)[lb]{\smash{{\SetFigFont{14}{16.8}{\rmdefault}{\mddefault}{\itdefault}{\color[rgb]{0,0,0}$z+e_a+o_b+o_c$}%
}}}}}
\put(6591,-13019){\makebox(0,0)[lb]{\smash{{\SetFigFont{14}{16.8}{\rmdefault}{\mddefault}{\itdefault}{\color[rgb]{0,0,0}$x+o_a+e_b+e_c$}%
}}}}
\put(10381,-14129){\rotatebox{300.0}{\makebox(0,0)[lb]{\smash{{\SetFigFont{14}{16.8}{\rmdefault}{\mddefault}{\itdefault}{\color[rgb]{0,0,0}$y+z+e_a+o_b+o_c$}%
}}}}}
\put(12191,-17319){\rotatebox{300.0}{\makebox(0,0)[lb]{\smash{{\SetFigFont{14}{16.8}{\rmdefault}{\mddefault}{\itdefault}{\color[rgb]{0,0,0}$y+a-b+1$}%
}}}}}
\put(4153,-18789){\rotatebox{300.0}{\makebox(0,0)[lb]{\smash{{\SetFigFont{14}{16.8}{\rmdefault}{\mddefault}{\itdefault}{\color[rgb]{0,0,0}$y+z+o_a+e_b+e_c+a-b+1$}%
}}}}}
\put(7621,-22799){\makebox(0,0)[lb]{\smash{{\SetFigFont{14}{16.8}{\rmdefault}{\mddefault}{\itdefault}{\color[rgb]{0,0,0}$x+e_a+o_b+o_c$}%
}}}}
\put(11381,-21649){\rotatebox{60.0}{\makebox(0,0)[lb]{\smash{{\SetFigFont{14}{16.8}{\rmdefault}{\mddefault}{\itdefault}{\color[rgb]{0,0,0}$z+o_a+e_b+e_c$}%
}}}}}
\put(3171,-17149){\makebox(0,0)[lb]{\smash{{\SetFigFont{14}{16.8}{\rmdefault}{\mddefault}{\itdefault}{\color[rgb]{0,0,0}$y$}%
}}}}
\put(3751,-15319){\rotatebox{60.0}{\makebox(0,0)[lb]{\smash{{\SetFigFont{14}{16.8}{\rmdefault}{\mddefault}{\itdefault}{\color[rgb]{0,0,0}$z+e_a+o_b+o_c$}%
}}}}}
\put(16425,-13024){\makebox(0,0)[lb]{\smash{{\SetFigFont{14}{16.8}{\rmdefault}{\mddefault}{\itdefault}{\color[rgb]{0,0,0}$x+o_a+e_b+e_c$}%
}}}}
\put(20215,-14134){\rotatebox{300.0}{\makebox(0,0)[lb]{\smash{{\SetFigFont{14}{16.8}{\rmdefault}{\mddefault}{\itdefault}{\color[rgb]{0,0,0}$y+z+e_a+o_b+o_c$}%
}}}}}
\put(22025,-17324){\rotatebox{300.0}{\makebox(0,0)[lb]{\smash{{\SetFigFont{14}{16.8}{\rmdefault}{\mddefault}{\itdefault}{\color[rgb]{0,0,0}$y+a-b+1$}%
}}}}}
\put(13987,-18794){\rotatebox{300.0}{\makebox(0,0)[lb]{\smash{{\SetFigFont{14}{16.8}{\rmdefault}{\mddefault}{\itdefault}{\color[rgb]{0,0,0}$y+z+o_a+e_b+e_c+a-b+1$}%
}}}}}
\put(17455,-22804){\makebox(0,0)[lb]{\smash{{\SetFigFont{14}{16.8}{\rmdefault}{\mddefault}{\itdefault}{\color[rgb]{0,0,0}$x+e_a+o_b+o_c$}%
}}}}
\put(21215,-21654){\rotatebox{60.0}{\makebox(0,0)[lb]{\smash{{\SetFigFont{14}{16.8}{\rmdefault}{\mddefault}{\itdefault}{\color[rgb]{0,0,0}$z+o_a+e_b+e_c$}%
}}}}}
\put(13005,-17154){\makebox(0,0)[lb]{\smash{{\SetFigFont{14}{16.8}{\rmdefault}{\mddefault}{\itdefault}{\color[rgb]{0,0,0}$y$}%
}}}}
\put(13585,-15324){\rotatebox{60.0}{\makebox(0,0)[lb]{\smash{{\SetFigFont{14}{16.8}{\rmdefault}{\mddefault}{\itdefault}{\color[rgb]{0,0,0}$z+e_a+o_b+o_c$}%
}}}}}
\put(6972,-23897){\makebox(0,0)[lb]{\smash{{\SetFigFont{14}{16.8}{\rmdefault}{\mddefault}{\itdefault}{\color[rgb]{0,0,0}$x+o_a+e_b+e_c$}%
}}}}
\put(10762,-25007){\rotatebox{300.0}{\makebox(0,0)[lb]{\smash{{\SetFigFont{14}{16.8}{\rmdefault}{\mddefault}{\itdefault}{\color[rgb]{0,0,0}$y+z+e_a+o_b+o_c$}%
}}}}}
\put(12572,-28197){\rotatebox{300.0}{\makebox(0,0)[lb]{\smash{{\SetFigFont{14}{16.8}{\rmdefault}{\mddefault}{\itdefault}{\color[rgb]{0,0,0}$y+a-b+1$}%
}}}}}
\put(4534,-29667){\rotatebox{300.0}{\makebox(0,0)[lb]{\smash{{\SetFigFont{14}{16.8}{\rmdefault}{\mddefault}{\itdefault}{\color[rgb]{0,0,0}$y+z+o_a+e_b+e_c+a-b+1$}%
}}}}}
\put(8002,-33677){\makebox(0,0)[lb]{\smash{{\SetFigFont{14}{16.8}{\rmdefault}{\mddefault}{\itdefault}{\color[rgb]{0,0,0}$x+e_a+o_b+o_c$}%
}}}}
\put(11762,-32527){\rotatebox{60.0}{\makebox(0,0)[lb]{\smash{{\SetFigFont{14}{16.8}{\rmdefault}{\mddefault}{\itdefault}{\color[rgb]{0,0,0}$z+o_a+e_b+e_c$}%
}}}}}
\put(3552,-28027){\makebox(0,0)[lb]{\smash{{\SetFigFont{14}{16.8}{\rmdefault}{\mddefault}{\itdefault}{\color[rgb]{0,0,0}$y$}%
}}}}
\put(4132,-26197){\rotatebox{60.0}{\makebox(0,0)[lb]{\smash{{\SetFigFont{14}{16.8}{\rmdefault}{\mddefault}{\itdefault}{\color[rgb]{0,0,0}$z+e_a+o_b+o_c$}%
}}}}}
\put(16633,-24007){\makebox(0,0)[lb]{\smash{{\SetFigFont{14}{16.8}{\rmdefault}{\mddefault}{\itdefault}{\color[rgb]{0,0,0}$x+o_a+e_b+e_c$}%
}}}}
\put(20423,-25117){\rotatebox{300.0}{\makebox(0,0)[lb]{\smash{{\SetFigFont{14}{16.8}{\rmdefault}{\mddefault}{\itdefault}{\color[rgb]{0,0,0}$y+z+e_a+o_b+o_c$}%
}}}}}
\put(22233,-28307){\rotatebox{300.0}{\makebox(0,0)[lb]{\smash{{\SetFigFont{14}{16.8}{\rmdefault}{\mddefault}{\itdefault}{\color[rgb]{0,0,0}$y+a-b+1$}%
}}}}}
\end{picture}%
}
\caption{Obtaining a recurrence for $G^{(2)}$-type regions with $a> b$. Kuo condensation is applied to the region $G^{(2)}_{3,3,2}(3,2 ;\ 2,2 ;\ 2,2)$ (picture (a)) as shown on the picture (b).}\label{fig:kuooff3}
\end{figure}

The $G^{(3)}$-recurrence for $a> b$ is
\begin{align}\label{offcenterrecurG3a}
\M(G^{(3)}_{x,y,z}(\textbf{a};\ \textbf{c};\ \textbf{b}))\M(E^{(2)}_{x-1,y,z-1}(\textbf{a}^{+1};\ \textbf{c}; \ \textbf{b}^{+1}))&=\M(K^{(2)}_{x-1,y-1,z}(\textbf{b}^{+1};\ \overline{\textbf{c}}; \ \textbf{a})) \M(F^{(4)}_{x,y+1,z-1}(\textbf{a}^{+1};\ \textbf{c};\ \textbf{b}))\notag\\
&+
\M(G^{(3)}_{x-1,y,z-1}(\textbf{a}^{+1};\ \textbf{c};\ \textbf{b}^{+1})) \M(E^{(2)}_{x,y,z}(\textbf{a};\ \textbf{c};\ \textbf{b})).
\end{align}
When $a> b$, this recurrence becomes
\begin{align}\label{offcenterrecurG3b}
\M(G^{(3)}_{x,y,z}(\textbf{a};\ \textbf{c};\ \textbf{b}))\M(E^{(2)}_{x-1,y,z-1}(\textbf{a}^{+1};\ \textbf{c}; \ \textbf{b}^{+1}))&=\M(K^{(2)}_{x-1,y,z}(\textbf{b}^{+1};\ \overline{\textbf{c}}; \ \textbf{a})) \M(F^{(4)}_{x,y,z-1}(\textbf{a}^{+1};\ \textbf{c};\ \textbf{b}))\notag\\
&+
\M(G^{(3)}_{x-1,y,z-1}(\textbf{a}^{+1};\ \textbf{c};\ \textbf{b}^{+1})) \M(E^{(2)}_{x,y,z}(\textbf{a};\ \textbf{c};\ \textbf{b})),
\end{align}
and when $a=b$ we have
\begin{align}\label{offcenterrecurG3c}
\M(G^{(3)}_{x,y,z}(\textbf{a};\ \textbf{c};\ \textbf{b}))\M(E^{(2)}_{x-1,y,z-1}(\textbf{a}^{+1};\ \textbf{c}; \ \textbf{b}^{+1}))&=\M(K^{(2)}_{x-1,y-1,z}(\textbf{b}^{+1};\ \overline{\textbf{c}}; \ \textbf{a})) \M(F^{(4)}_{x,y,z-1}(\textbf{a}^{+1};\ \textbf{c};\ \textbf{b}))\notag\\
&+
\M(G^{(3)}_{x-1,y,z-1}(\textbf{a}^{+1};\ \textbf{c};\ \textbf{b}^{+1})) \M(E^{(2)}_{x,y,z}(\textbf{a};\ \textbf{c};\ \textbf{b})).
\end{align}

Finally, the application of Kuo condensation gives the $G^{(4)}$-recurrence as
\begin{align}\label{offcenterrecurG4a}
\M(G^{(4)}_{x,y,z}(\textbf{a};\ \textbf{c};\ \textbf{b}))\M(E^{(6)}_{x-1,y,z-1}(\textbf{b}^{+1};\ \overline{\textbf{c}}; \ \textbf{a}^{+1}))&=\M(K^{(3)}_{x-1,y-1,z}(\textbf{b}^{+1};\ \overline{\textbf{c}}; \ \textbf{a})) \M(F^{(1)}_{x,y+1,z-1}(\textbf{b};\ \textbf{c};\ \textbf{a}^{+1}))\notag\\
&+
\M(G^{(4)}_{x-1,y,z-1}(\textbf{a}^{+1};\ \textbf{c};\ \textbf{b}^{+1})) \M(E^{(6)}_{x,y,z}(\textbf{b};\ \overline{\textbf{c}};\ \textbf{a}))
\end{align}
for $a< b$,
\begin{align}\label{offcenterrecurG4b}
\M(G^{(4)}_{x,y,z}(\textbf{a};\ \textbf{c};\ \textbf{b}))\M(E^{(6)}_{x-1,y,z-1}(\textbf{b}^{+1};\ \overline{\textbf{c}}; \ \textbf{a}^{+1}))&=\M(K^{(3)}_{x-1,y,z}(\textbf{b}^{+1};\ \overline{\textbf{c}}; \ \textbf{a})) \M(F^{(1)}_{x,y,z-1}(\textbf{b};\ \textbf{c};\ \textbf{a}^{+1}))\notag\\
&+
\M(G^{(4)}_{x-1,y,z-1}(\textbf{a}^{+1};\ \textbf{c};\ \textbf{b}^{+1})) \M(E^{(6)}_{x,y,z}(\textbf{b};\ \overline{\textbf{c}};\ \textbf{a}))
\end{align}
for $a>b$, and
\begin{align}\label{offcenterrecurG4c}
\M(G^{(4)}_{x,y,z}(\textbf{a};\ \textbf{c};\ \textbf{b}))\M(E^{(6)}_{x-1,y,z-1}(\textbf{b}^{+1};\ \overline{\textbf{c}}; \ \textbf{a}^{+1}))&=\M(K^{(3)}_{x-1,y-1,z}(\textbf{b}^{+1};\ \overline{\textbf{c}}; \ \textbf{a})) \M(F^{(1)}_{x,y,z-1}(\textbf{b};\ \textbf{c};\ \textbf{a}^{+1}))\notag\\
&+
\M(G^{(4)}_{x-1,y,z-1}(\textbf{a}^{+1};\ \textbf{c};\ \textbf{b}^{+1})) \M(E^{(6)}_{x,y,z}(\textbf{b};\ \overline{\textbf{c}};\ \textbf{a}))
\end{align}
for $a=b$.

\subsection{Recurrences for $K^{(i)}$-type regions}\label{subsec:offrecurK}



For the case of $K^{(i)}$-regions, we still apply Kuo's Theorem \ref{kuothm1} with the three vertices $u,v,w$ chosen similarly to that in the cases of $E^{(i)}$- and $F^{(i)}$-type regions, the only difference is that the $s$-triangle is now located at the west corner (see Figure \ref{fig:offK}). Let us work in detail for the case of $K^{(1)}$-type regions with $a>b$.

Let $G$ be the dual graph of the region $K^{(1)}_{x,y,z}(\textbf{a};\ \textbf{c};\ \textbf{b})$ (for $a<b$). Consider forced lozenges as in Figures \ref{fig:kuooff4}(b)--(f), we get
\begin{equation}\label{kuothm4eq1}
\M(G-\{u,v,w,s\})=\M(E^{(1)}_{x-1,y,z-1}(E^{(1)}_{x-1,y,z}(\textbf{b}^{+1};\ \overline{\textbf{c}}; \ \textbf{a}))
\end{equation}
\begin{equation}\label{kuothm4eq2}
\M(G-\{u,v\})=\M(E^{(1)}_{x,y+1,z-1}(\textbf{b}^{+1};\ \overline{\textbf{c}}; \ \textbf{a}))
\end{equation}
\begin{equation}\label{kuothm4eq3}
\M(G-\{w,s\})=\M(K^{(1)}_{x-1,y-1,z+1}(\textbf{a};\ \textbf{c};\ \textbf{b}))
\end{equation}
\begin{equation}\label{kuothm4eq4}
\M(G-\{u,s\})=\M(R^{\leftarrow}_{x,y,z}(\textbf{a};\ \textbf{c};\ \textbf{b}))
\end{equation}
\begin{equation}\label{kuothm4eq5}
\M(G-\{v,w\})=\M(G^{(2)}_{x-1,y,z}(\textbf{b}^{+1};\ \overline{\textbf{c}};\ \textbf{a})).
\end{equation}
We need to $180^{\circ}$-rotate the leftover regions in the case of (\ref{kuothm4eq1}), (\ref{kuothm4eq2}), (\ref{kuothm4eq5}).
By Eqs. (\ref{kuothm4eq1})--(\ref{kuothm4eq5}) and Theorem \ref{kuothm1}, we have the $K^{(1)}$-recurrence for $a>b$:
\begin{align}\label{offcenterrecurK1b}
\M(K^{(1)}_{x,y,z}(\textbf{a};\ \textbf{c};\ \textbf{b}))\M(E^{(1)}_{x-1,y,z}(\textbf{b}^{+1};\ \overline{\textbf{c}}; \ \textbf{a}))&=\M(E^{(1)}_{x,y+1,z-1}(\textbf{b}^{+1};\ \overline{\textbf{c}}; \ \textbf{a})) \M(K^{(1)}_{x-1,y-1,z+1}(\textbf{a};\ \textbf{c};\ \textbf{b}))\notag\\
&+
\M(R^{\leftarrow}_{x,y,z}(\textbf{a};\ \textbf{c};\ \textbf{b})) \M(G^{(2)}_{x-1,y,z}(\textbf{b}^{+1};\ \overline{\textbf{c}};\ \textbf{a})).
\end{align}
By applying Kuo condensation similarly, we get other recurrences:

 The $K^{(1)}$-recurrence for $a\leq b$:
\begin{align}\label{offcenterrecurK1a}
\M(K^{(1)}_{x,y,z}(\textbf{a};\ \textbf{c};\ \textbf{b}))\M(E^{(1)}_{x-1,y-1,z}(\textbf{b}^{+1};\ \overline{\textbf{c}}; \ \textbf{a}))&=\M(E^{(1)}_{x,y,z-1}(\textbf{b}^{+1};\ \overline{\textbf{c}}; \ \textbf{a})) \M(K^{(1)}_{x-1,y-1,z+1}(\textbf{a};\ \textbf{c};\ \textbf{b}))\notag\\
&+
\M(R^{\leftarrow}_{x,y,z}(\textbf{a};\ \textbf{c};\ \textbf{b})) \M(G^{(2)}_{x-1,y-1,z}(\textbf{b}^{+1};\ \overline{\textbf{c}};\ \textbf{a})).
\end{align}

\begin{figure}\centering
\setlength{\unitlength}{3947sp}%
\begingroup\makeatletter\ifx\SetFigFont\undefined%
\gdef\SetFigFont#1#2#3#4#5{%
  \reset@font\fontsize{#1}{#2pt}%
  \fontfamily{#3}\fontseries{#4}\fontshape{#5}%
  \selectfont}%
\fi\endgroup%
\resizebox{14cm}{!}{
\begin{picture}(0,0)%
\includegraphics{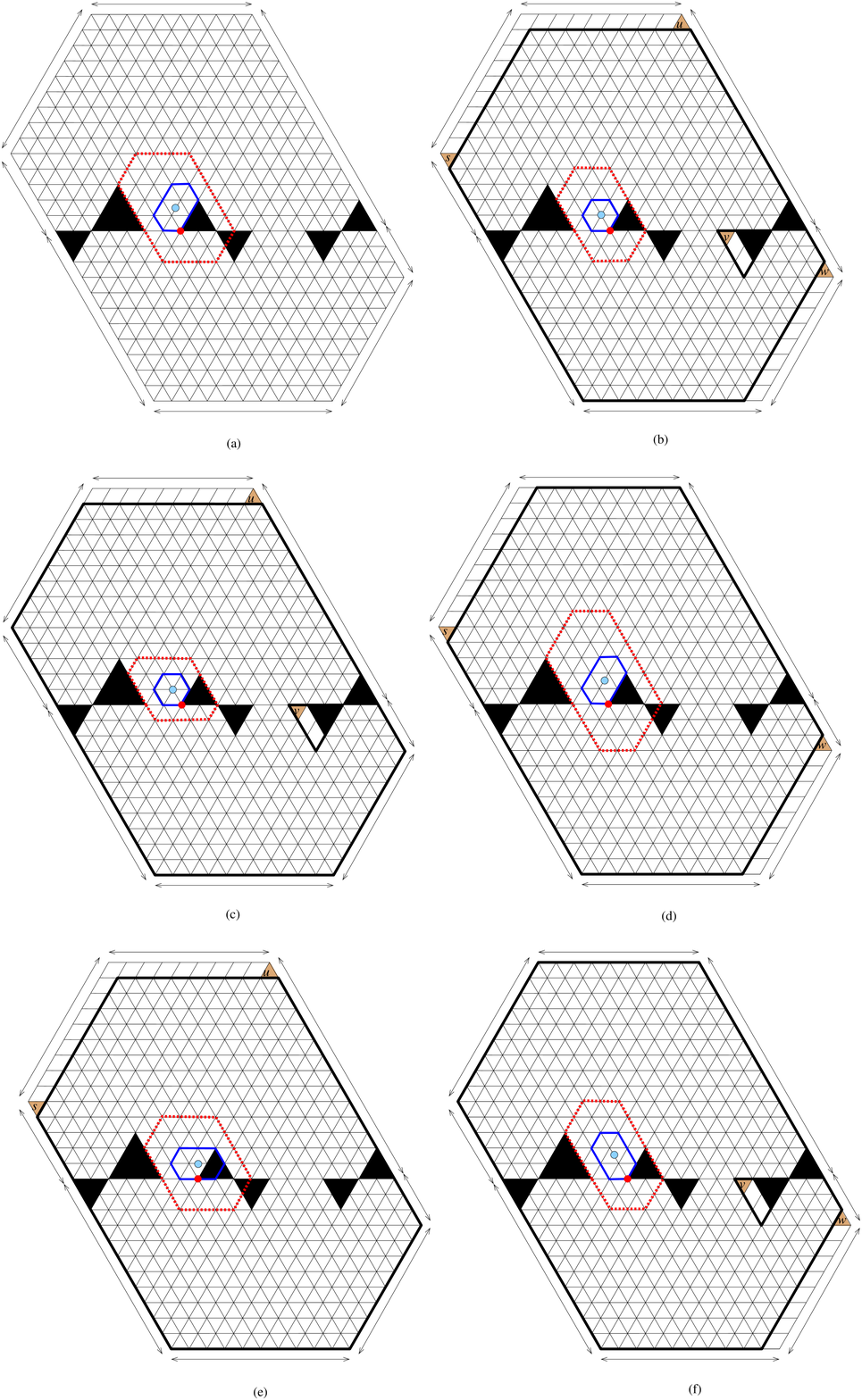}%
\end{picture}%
%
%

\begin{picture}(20084,32450)(2035,-33527)
\put(21416,-28409){\rotatebox{300.0}{\makebox(0,0)[lb]{\smash{{\SetFigFont{17}{20.4}{\rmdefault}{\mddefault}{\itdefault}{\color[rgb]{0,0,0}$y+a-b$}%
}}}}}
\put(19136,-24669){\rotatebox{300.0}{\makebox(0,0)[lb]{\smash{{\SetFigFont{17}{20.4}{\rmdefault}{\mddefault}{\itdefault}{\color[rgb]{0,0,0}$y+z+e_a+o_b+o_c+3$}%
}}}}}
\put(20690,-31984){\rotatebox{60.0}{\makebox(0,0)[lb]{\smash{{\SetFigFont{17}{20.4}{\rmdefault}{\mddefault}{\itdefault}{\color[rgb]{0,0,0}$z+o_a+e_b+e_c$}%
}}}}}
\put(16950,-32934){\makebox(0,0)[lb]{\smash{{\SetFigFont{17}{20.4}{\rmdefault}{\mddefault}{\itdefault}{\color[rgb]{0,0,0}$x+e_a+o_b+o_c$}%
}}}}
\put(13220,-29014){\rotatebox{300.0}{\makebox(0,0)[lb]{\smash{{\SetFigFont{17}{20.4}{\rmdefault}{\mddefault}{\itdefault}{\color[rgb]{0,0,0}$y+z+o_a+e_b+e_c+a-b$}%
}}}}}
\put(12388,-27557){\rotatebox{300.0}{\makebox(0,0)[lb]{\smash{{\SetFigFont{17}{20.4}{\rmdefault}{\mddefault}{\itdefault}{\color[rgb]{0,0,0}$y+3$}%
}}}}}
\put(12670,-25594){\rotatebox{60.0}{\makebox(0,0)[lb]{\smash{{\SetFigFont{17}{20.4}{\rmdefault}{\mddefault}{\itdefault}{\color[rgb]{0,0,0}$z+e_a+o_b+o_c$}%
}}}}}
\put(15150,-23128){\makebox(0,0)[lb]{\smash{{\SetFigFont{17}{20.4}{\rmdefault}{\mddefault}{\itdefault}{\color[rgb]{0,0,0}$x+o_a+e_b+e_c$}%
}}}}
\put(2830,-25594){\rotatebox{60.0}{\makebox(0,0)[lb]{\smash{{\SetFigFont{17}{20.4}{\rmdefault}{\mddefault}{\itdefault}{\color[rgb]{0,0,0}$z+e_a+o_b+o_c$}%
}}}}}
\put(2548,-27557){\rotatebox{300.0}{\makebox(0,0)[lb]{\smash{{\SetFigFont{17}{20.4}{\rmdefault}{\mddefault}{\itdefault}{\color[rgb]{0,0,0}$y+3$}%
}}}}}
\put(3380,-29014){\rotatebox{300.0}{\makebox(0,0)[lb]{\smash{{\SetFigFont{17}{20.4}{\rmdefault}{\mddefault}{\itdefault}{\color[rgb]{0,0,0}$y+z+o_a+e_b+e_c+a-b$}%
}}}}}
\put(7110,-32934){\makebox(0,0)[lb]{\smash{{\SetFigFont{17}{20.4}{\rmdefault}{\mddefault}{\itdefault}{\color[rgb]{0,0,0}$x+e_a+o_b+o_c$}%
}}}}
\put(10850,-31984){\rotatebox{60.0}{\makebox(0,0)[lb]{\smash{{\SetFigFont{17}{20.4}{\rmdefault}{\mddefault}{\itdefault}{\color[rgb]{0,0,0}$z+o_a+e_b+e_c$}%
}}}}}
\put(11576,-28409){\rotatebox{300.0}{\makebox(0,0)[lb]{\smash{{\SetFigFont{17}{20.4}{\rmdefault}{\mddefault}{\itdefault}{\color[rgb]{0,0,0}$y+a-b$}%
}}}}}
\put(9296,-24669){\rotatebox{300.0}{\makebox(0,0)[lb]{\smash{{\SetFigFont{17}{20.4}{\rmdefault}{\mddefault}{\itdefault}{\color[rgb]{0,0,0}$y+z+e_a+o_b+o_c+3$}%
}}}}}
\put(5310,-23128){\makebox(0,0)[lb]{\smash{{\SetFigFont{17}{20.4}{\rmdefault}{\mddefault}{\itdefault}{\color[rgb]{0,0,0}$x+o_a+e_b+e_c$}%
}}}}
\put(12230,-14714){\rotatebox{60.0}{\makebox(0,0)[lb]{\smash{{\SetFigFont{17}{20.4}{\rmdefault}{\mddefault}{\itdefault}{\color[rgb]{0,0,0}$z+e_a+o_b+o_c$}%
}}}}}
\put(11948,-16677){\rotatebox{300.0}{\makebox(0,0)[lb]{\smash{{\SetFigFont{17}{20.4}{\rmdefault}{\mddefault}{\itdefault}{\color[rgb]{0,0,0}$y+3$}%
}}}}}
\put(12780,-18134){\rotatebox{300.0}{\makebox(0,0)[lb]{\smash{{\SetFigFont{17}{20.4}{\rmdefault}{\mddefault}{\itdefault}{\color[rgb]{0,0,0}$y+z+o_a+e_b+e_c+a-b$}%
}}}}}
\put(16510,-22054){\makebox(0,0)[lb]{\smash{{\SetFigFont{17}{20.4}{\rmdefault}{\mddefault}{\itdefault}{\color[rgb]{0,0,0}$x+e_a+o_b+o_c$}%
}}}}
\put(20250,-21104){\rotatebox{60.0}{\makebox(0,0)[lb]{\smash{{\SetFigFont{17}{20.4}{\rmdefault}{\mddefault}{\itdefault}{\color[rgb]{0,0,0}$z+o_a+e_b+e_c$}%
}}}}}
\put(20976,-17529){\rotatebox{300.0}{\makebox(0,0)[lb]{\smash{{\SetFigFont{17}{20.4}{\rmdefault}{\mddefault}{\itdefault}{\color[rgb]{0,0,0}$y+a-b$}%
}}}}}
\put(18696,-13789){\rotatebox{300.0}{\makebox(0,0)[lb]{\smash{{\SetFigFont{17}{20.4}{\rmdefault}{\mddefault}{\itdefault}{\color[rgb]{0,0,0}$y+z+e_a+o_b+o_c+3$}%
}}}}}
\put(14710,-12248){\makebox(0,0)[lb]{\smash{{\SetFigFont{17}{20.4}{\rmdefault}{\mddefault}{\itdefault}{\color[rgb]{0,0,0}$x+o_a+e_b+e_c$}%
}}}}
\put(2460,-14734){\rotatebox{60.0}{\makebox(0,0)[lb]{\smash{{\SetFigFont{17}{20.4}{\rmdefault}{\mddefault}{\itdefault}{\color[rgb]{0,0,0}$z+e_a+o_b+o_c$}%
}}}}}
\put(2178,-16697){\rotatebox{300.0}{\makebox(0,0)[lb]{\smash{{\SetFigFont{17}{20.4}{\rmdefault}{\mddefault}{\itdefault}{\color[rgb]{0,0,0}$y+3$}%
}}}}}
\put(3010,-18154){\rotatebox{300.0}{\makebox(0,0)[lb]{\smash{{\SetFigFont{17}{20.4}{\rmdefault}{\mddefault}{\itdefault}{\color[rgb]{0,0,0}$y+z+o_a+e_b+e_c+a-b$}%
}}}}}
\put(4911,-1405){\makebox(0,0)[lb]{\smash{{\SetFigFont{17}{20.4}{\rmdefault}{\mddefault}{\itdefault}{\color[rgb]{0,0,0}$x+o_a+e_b+e_c$}%
}}}}
\put(8897,-2946){\rotatebox{300.0}{\makebox(0,0)[lb]{\smash{{\SetFigFont{17}{20.4}{\rmdefault}{\mddefault}{\itdefault}{\color[rgb]{0,0,0}$y+z+e_a+o_b+o_c+3$}%
}}}}}
\put(11177,-6686){\rotatebox{300.0}{\makebox(0,0)[lb]{\smash{{\SetFigFont{17}{20.4}{\rmdefault}{\mddefault}{\itdefault}{\color[rgb]{0,0,0}$y+a-b$}%
}}}}}
\put(10451,-10261){\rotatebox{60.0}{\makebox(0,0)[lb]{\smash{{\SetFigFont{17}{20.4}{\rmdefault}{\mddefault}{\itdefault}{\color[rgb]{0,0,0}$z+o_a+e_b+e_c$}%
}}}}}
\put(6711,-11211){\makebox(0,0)[lb]{\smash{{\SetFigFont{17}{20.4}{\rmdefault}{\mddefault}{\itdefault}{\color[rgb]{0,0,0}$x+e_a+o_b+o_c$}%
}}}}
\put(2981,-7291){\rotatebox{300.0}{\makebox(0,0)[lb]{\smash{{\SetFigFont{17}{20.4}{\rmdefault}{\mddefault}{\itdefault}{\color[rgb]{0,0,0}$y+z+o_a+e_b+e_c+a-b$}%
}}}}}
\put(2149,-5834){\rotatebox{300.0}{\makebox(0,0)[lb]{\smash{{\SetFigFont{17}{20.4}{\rmdefault}{\mddefault}{\itdefault}{\color[rgb]{0,0,0}$y+3$}%
}}}}}
\put(2431,-3871){\rotatebox{60.0}{\makebox(0,0)[lb]{\smash{{\SetFigFont{17}{20.4}{\rmdefault}{\mddefault}{\itdefault}{\color[rgb]{0,0,0}$z+e_a+o_b+o_c$}%
}}}}}
\put(14750,-1398){\makebox(0,0)[lb]{\smash{{\SetFigFont{17}{20.4}{\rmdefault}{\mddefault}{\itdefault}{\color[rgb]{0,0,0}$x+o_a+e_b+e_c$}%
}}}}
\put(18736,-2939){\rotatebox{300.0}{\makebox(0,0)[lb]{\smash{{\SetFigFont{17}{20.4}{\rmdefault}{\mddefault}{\itdefault}{\color[rgb]{0,0,0}$y+z+e_a+o_b+o_c+3$}%
}}}}}
\put(21016,-6679){\rotatebox{300.0}{\makebox(0,0)[lb]{\smash{{\SetFigFont{17}{20.4}{\rmdefault}{\mddefault}{\itdefault}{\color[rgb]{0,0,0}$y+a-b$}%
}}}}}
\put(20290,-10254){\rotatebox{60.0}{\makebox(0,0)[lb]{\smash{{\SetFigFont{17}{20.4}{\rmdefault}{\mddefault}{\itdefault}{\color[rgb]{0,0,0}$z+o_a+e_b+e_c$}%
}}}}}
\put(16550,-11204){\makebox(0,0)[lb]{\smash{{\SetFigFont{17}{20.4}{\rmdefault}{\mddefault}{\itdefault}{\color[rgb]{0,0,0}$x+e_a+o_b+o_c$}%
}}}}
\put(12820,-7284){\rotatebox{300.0}{\makebox(0,0)[lb]{\smash{{\SetFigFont{17}{20.4}{\rmdefault}{\mddefault}{\itdefault}{\color[rgb]{0,0,0}$y+z+o_a+e_b+e_c+a-b$}%
}}}}}
\put(11988,-5827){\rotatebox{300.0}{\makebox(0,0)[lb]{\smash{{\SetFigFont{17}{20.4}{\rmdefault}{\mddefault}{\itdefault}{\color[rgb]{0,0,0}$y+3$}%
}}}}}
\put(12270,-3864){\rotatebox{60.0}{\makebox(0,0)[lb]{\smash{{\SetFigFont{17}{20.4}{\rmdefault}{\mddefault}{\itdefault}{\color[rgb]{0,0,0}$z+e_a+o_b+o_c$}%
}}}}}
\put(4940,-12268){\makebox(0,0)[lb]{\smash{{\SetFigFont{17}{20.4}{\rmdefault}{\mddefault}{\itdefault}{\color[rgb]{0,0,0}$x+o_a+e_b+e_c$}%
}}}}
\put(8926,-13809){\rotatebox{300.0}{\makebox(0,0)[lb]{\smash{{\SetFigFont{17}{20.4}{\rmdefault}{\mddefault}{\itdefault}{\color[rgb]{0,0,0}$y+z+e_a+o_b+o_c+3$}%
}}}}}
\put(11206,-17549){\rotatebox{300.0}{\makebox(0,0)[lb]{\smash{{\SetFigFont{17}{20.4}{\rmdefault}{\mddefault}{\itdefault}{\color[rgb]{0,0,0}$y+a-b$}%
}}}}}
\put(10480,-21124){\rotatebox{60.0}{\makebox(0,0)[lb]{\smash{{\SetFigFont{17}{20.4}{\rmdefault}{\mddefault}{\itdefault}{\color[rgb]{0,0,0}$z+o_a+e_b+e_c$}%
}}}}}
\put(6740,-22074){\makebox(0,0)[lb]{\smash{{\SetFigFont{17}{20.4}{\rmdefault}{\mddefault}{\itdefault}{\color[rgb]{0,0,0}$x+e_a+o_b+o_c$}%
}}}}
\end{picture}%
}
\caption{Obtaining a recurrence for $K^{(2)}$-type regions with $a> b$. Kuo condensation is applied to the region $K^{(2)}_{3,2,2}(2,3 ;\ 2,2;\ 2,2)$ (picture (a)) as shown on the picture (b).}\label{fig:kuooff4}
\end{figure}

The $K^{(2)}$-recurrence for $a\leq b$:
\begin{align}\label{offcenterrecurK2a}
\M(K^{(2)}_{x,y,z}(\textbf{a};\ \textbf{c};\ \textbf{b}))\M(E^{(2)}_{x-1,y-1,z}(\textbf{b}^{+1};\ \overline{\textbf{c}}; \ \textbf{a}))&=\M(E^{(2)}_{x,y,z-1}(\textbf{b}^{+1};\ \overline{\textbf{c}}; \ \textbf{a})) \M(K^{(2)}_{x-1,y-1,z+1}(\textbf{a};\ \textbf{c};\ \textbf{b}))\notag\\
&+
\M(F^{(1)}_{x,y,z}(\textbf{a};\ \textbf{c};\ \textbf{b})) \M(G^{(3)}_{x-1,y-1,z}(\textbf{b}^{+1};\ \overline{\textbf{c}};\ \textbf{a})).
\end{align}

The $K^{(2)}$-recurrence for $a> b$:
\begin{align}\label{offcenterrecurK2b}
\M(K^{(2)}_{x,y,z}(\textbf{a};\ \textbf{c};\ \textbf{b}))\M(E^{(2)}_{x-1,y,z}(\textbf{b}^{+1};\ \overline{\textbf{c}}; \ \textbf{a}))&=\M(E^{(2)}_{x,y+1,z-1}(\textbf{b}^{+1};\ \overline{\textbf{c}}; \ \textbf{a})) \M(K^{(2)}_{x-1,y-1,z+1}(\textbf{a};\ \textbf{c};\ \textbf{b}))\notag\\
&+
\M(F^{(1)}_{x,y,z}(\textbf{a};\ \textbf{c};\ \textbf{b})) \M(G^{(3)}_{x-1,y,z}(\textbf{b}^{+1};\ \overline{\textbf{c}};\ \textbf{a})).
\end{align}

The $K^{(3)}$-recurrence for $a\leq b$:
\begin{align}\label{offcenterrecurK3a}
\M(K^{(3)}_{x,y,z}(\textbf{a};\ \textbf{c};\ \textbf{b}))\M(E^{(6)}_{x-1,y-1,z}(\textbf{a};\ \textbf{c}; \ \textbf{b}^{+1}))&=\M(E^{(6)}_{x,y,z-1}(\textbf{a};\ \textbf{c}; \ \textbf{b}^{+1})) \M(K^{(3)}_{x-1,y-1,z+1}(\textbf{a};\ \textbf{c};\ \textbf{b}))\notag\\
&+
\M(F^{(2)}_{x,y,z}(\textbf{a};\ \textbf{c};\ \textbf{b})) \M(G^{(4)}_{x-1,y-1,z}(\textbf{b}^{+1};\ \overline{\textbf{c}};\ \textbf{a})).
\end{align}

The $K^{(3)}$-recurrence for $a> b$:
\begin{align}\label{offcenterrecurK3b}
\M(K^{(3)}_{x,y,z}(\textbf{a};\ \textbf{c};\ \textbf{b}))\M(E^{(6)}_{x-1,y,z}(\textbf{a};\ \textbf{c}; \ \textbf{b}^{+1}))&=\M(E^{(6)}_{x,y+1,z-1}(\textbf{a};\ \textbf{c}; \ \textbf{b}^{+1})) \M(K^{(3)}_{x-1,y-1,z+1}(\textbf{a};\ \textbf{c};\ \textbf{b}))\notag\\
&+
\M(F^{(2)}_{x,y,z}(\textbf{a};\ \textbf{c};\ \textbf{b})) \M(G^{(4)}_{x-1,y,z}(\textbf{b}^{+1};\ \overline{\textbf{c}};\ \textbf{a})).
\end{align}

The $K^{(4)}$-recurrence for $a\leq b$:
\begin{align}\label{offcenterrecurK4a}
\M(K^{(4)}_{x,y,z}(\textbf{a};\ \textbf{c};\ \textbf{b}))\M(E^{(1)}_{x-1,y-1,z}(\textbf{a};\ \textbf{c}; \ \textbf{b}^{+1}))&=\M(E^{(1)}_{x,y,z-1}(\textbf{a};\ \textbf{c}; \ \textbf{b}^{+1})) \M(K^{(4)}_{x-1,y-1,z+1}(\textbf{a};\ \textbf{c};\ \textbf{b}))\notag\\
&+
\M(F^{(3)}_{x,y,z}(\textbf{a};\ \textbf{c};\ \textbf{b})) \M(G^{(1)}_{x-1,y-1,z}(\textbf{a};\ \textbf{c};\ \textbf{b}^{+1})).
\end{align}

The $K^{(4)}$-recurrence for $a> b$:
\begin{align}\label{offcenterrecurK4b}
\M(K^{(4)}_{x,y,z}(\textbf{a};\ \textbf{c};\ \textbf{b}))\M(E^{(1)}_{x-1,y,z}(\textbf{a};\ \textbf{c}; \ \textbf{b}^{+1}))&=\M(E^{(1)}_{x,y+1,z-1}(\textbf{a};\ \textbf{c}; \ \textbf{b}^{+1})) \M(K^{(4)}_{x-1,y-1,z+1}(\textbf{a};\ \textbf{c};\ \textbf{b}))\notag\\
&+
\M(F^{(3)}_{x,y,z}(\textbf{a};\ \textbf{c};\ \textbf{b})) \M(G^{(1)}_{x-1,y,z}(\textbf{a};\ \textbf{c};\ \textbf{b}^{+1})).
\end{align}

\subsection{Recurrences for  $\overline{E}^{(i)}$-type regions}\label{subsec:offrecurE2}

We apply Theorem \ref{kuothm1}  to the $\overline{E}^{(1)}$- and  $\overline{E}^{(2)}$-type regions with the particular choice of
the four vertices $u,v,w,s$ as shown in Figures \ref{fig:offE2}(a) and (b), respectively.  For the  $\overline{E}^{(3)}$-type region, we apply Theorem \ref{kuothm2} as in Figure  \ref{fig:offE2}(c).

Let us work out in detail for case of $\overline{E}^{(1)}$-type regions when $a\leq b$. We also apply Kuo's Theorem \ref{kuothm1} to the dual graph $G$ of the region $\overline{E}^{(1)}_{x,y,z}(\textbf{a};\ \textbf{c};\ \textbf{b})$. Consider forced lozenges as in Figures \ref{fig:kuooff5}(b)--(f), we have
\begin{equation}\label{kuothm5eq1}
\M(G-\{u,v,w,s\})=\M(Q^{\leftarrow}_{x,y,z-1}(\textbf{a};\ \textbf{c}; \ \textbf{b}^{+1}))
\end{equation}
\begin{equation}\label{kuothm5eq2}
\M(G-\{u,v\})=\M(\overline{K}^{(5)}_{x,y,z-1}(\textbf{a};\ \textbf{c}; \ \textbf{b}^{+1}))
\end{equation}
\begin{equation}\label{kuothm5eq3}
\M(G-\{w,s\})=\M(\overline{G}^{(5)}_{x,y-1,z}(\textbf{b};\ \textbf{c}^{\leftrightarrow};\ \textbf{a}))
\end{equation}
\begin{equation}\label{kuothm5eq4}
\M(G-\{u,s\})=\M(\overline{E}^{(1)}_{x+1,y,z-1}(\textbf{a};\ \textbf{c};\ \textbf{b}))
\end{equation}
\begin{equation}\label{kuothm5eq5}
\M(G-\{v,w\})=\M(Q^{\leftarrow}_{x-1,y,z}(\textbf{a};\ \textbf{c};\ \textbf{b}^{+1})).
\end{equation}
The region of type $Q^{\leftarrow}$ was enumerated in Theorem 2.7 of \cite{HoleDent}. 
The sequence $\textbf{c}^{\leftrightarrow}$ obtained by reverting the sequence $\textbf{c}$ when it has an odd number of terms, otherwise, we revert and add a new $0$ term in the beginning. More precisely, the leftover region after the removal of forced lozenges in Figure  (\ref{kuothm5eq3}) is not of our interests. We need to reflect the this region to get $\overline{G}^{(5)}_{x,y-1,z}(\textbf{b};\ \textbf{c}^{\leftrightarrow};\ \textbf{a})$.

Plugging the above five equations into that of Theorem \ref{kuothm1}, we get the $\overline{E}^{(1)}$-recurrence for $a\leq b$:
\begin{align}\label{offcenterrecurE1qb}
\M(\overline{E}^{(1)}_{x,y,z}(\textbf{a};\ \textbf{c};\ \textbf{b}))\M(Q^{\leftarrow}_{x,y,z-1}(\textbf{a};\ \textbf{c}; \ \textbf{b}^{+1}))&=\M(\overline{K}^{(5)}_{x,y,z-1}(\textbf{a};\ \textbf{c}; \ \textbf{b}^{+1})) \M(\overline{G}^{(5)}_{x,y-1,z}(\textbf{b};\ \textbf{c}^{\leftrightarrow};\ \textbf{a}))\notag\\
&+
\M(\overline{E}^{(1)}_{x+1,y,z-1}(\textbf{a};\ \textbf{c};\ \textbf{b})) \M(Q^{\leftarrow}_{x-1,y,z}(\textbf{a};\ \textbf{c};\ \textbf{b}^{+1})).
\end{align}

Working out similarly, we get other recurrences:

The $\overline{E}^{(1)}$-recurrence for $a< b$
\begin{align}\label{offcenterrecurE1qa}
\M(\overline{E}^{(1)}_{x,y,z}(\textbf{a};\ \textbf{c};\ \textbf{b}))\M(Q^{\leftarrow}_{x,y-1,z-1}(\textbf{a};\ \textbf{c}; \ \textbf{b}^{+1}))&=\M(\overline{K}^{(5)}_{x,y-1,z-1}(\textbf{a};\ \textbf{c}; \ \textbf{b}^{+1})) \M(\overline{G}^{(5)}_{x,y-1,z}(\textbf{b};\ \textbf{c}^{\leftrightarrow};\ \textbf{a}))\notag\\
&+
\M(\overline{E}^{(1)}_{x+1,y,z-1}(\textbf{a};\ \textbf{c};\ \textbf{b})) \M(Q^{\leftarrow}_{x-1,y-1,z}(\textbf{a};\ \textbf{c};\ \textbf{b}^{+1})).
\end{align}

\begin{figure}\centering
\setlength{\unitlength}{3947sp}%
\begingroup\makeatletter\ifx\SetFigFont\undefined%
\gdef\SetFigFont#1#2#3#4#5{%
  \reset@font\fontsize{#1}{#2pt}%
  \fontfamily{#3}\fontseries{#4}\fontshape{#5}%
  \selectfont}%
\fi\endgroup%
\resizebox{15cm}{!}{
\begin{picture}(0,0)%
\includegraphics{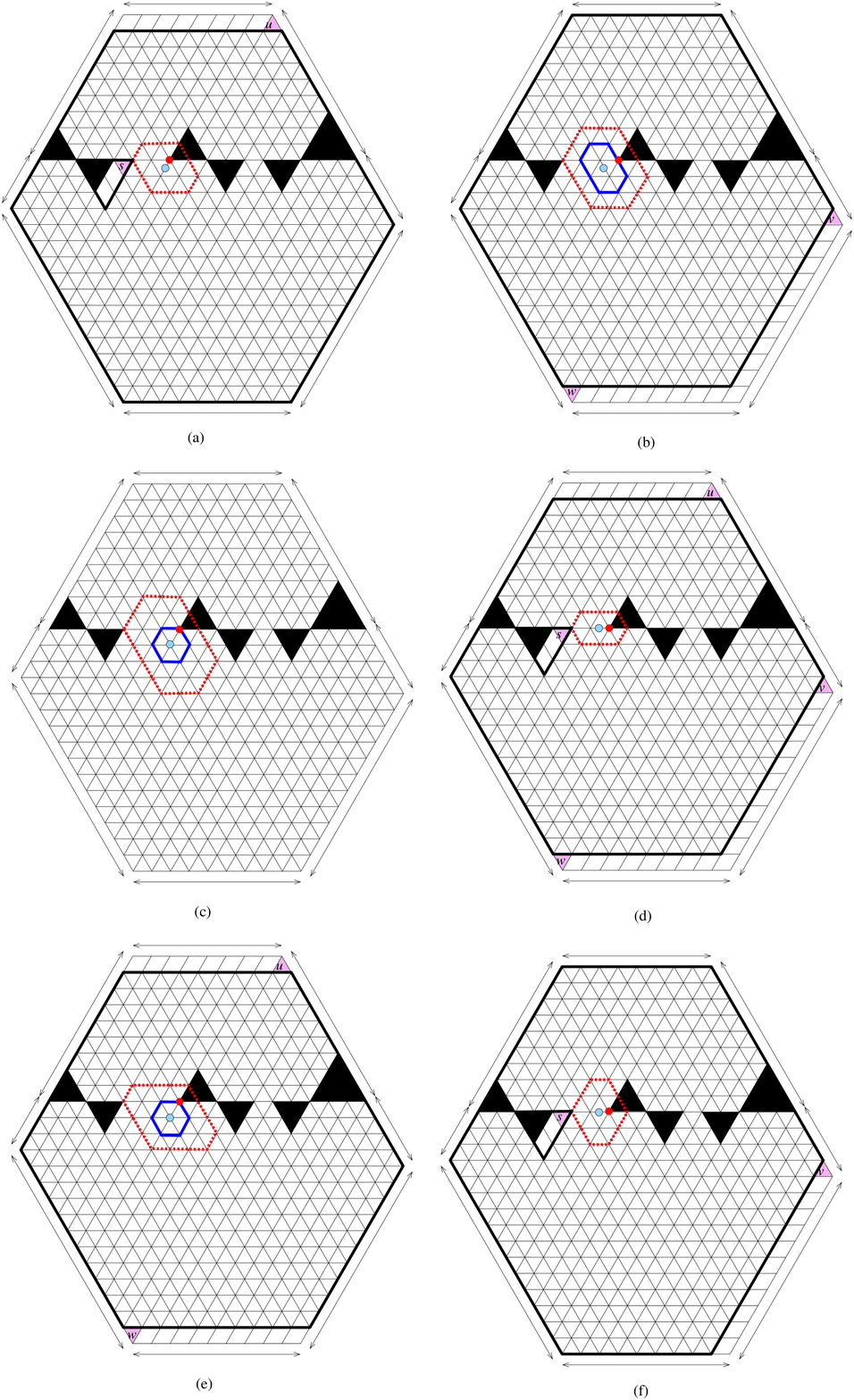}%
\end{picture}%
%
%

\begin{picture}(19377,31096)(3818,-31100)
\put(11160,-11597){\rotatebox{300.0}{\makebox(0,0)[lb]{\smash{{\SetFigFont{17}{20.4}{\rmdefault}{\mddefault}{\itdefault}{\color[rgb]{0,0,0}$z+o_a+o_b+o_c$}%
}}}}}
\put(10944,-1299){\rotatebox{300.0}{\makebox(0,0)[lb]{\smash{{\SetFigFont{17}{20.4}{\rmdefault}{\mddefault}{\itdefault}{\color[rgb]{0,0,0}$z+o_a+o_b+o_c$}%
}}}}}
\put(20580,-11577){\rotatebox{300.0}{\makebox(0,0)[lb]{\smash{{\SetFigFont{17}{20.4}{\rmdefault}{\mddefault}{\itdefault}{\color[rgb]{0,0,0}$z+o_a+o_b+o_c$}%
}}}}}
\put(11150,-21967){\rotatebox{300.0}{\makebox(0,0)[lb]{\smash{{\SetFigFont{17}{20.4}{\rmdefault}{\mddefault}{\itdefault}{\color[rgb]{0,0,0}$z+o_a+o_b+o_c$}%
}}}}}
\put(20570,-22197){\rotatebox{300.0}{\makebox(0,0)[lb]{\smash{{\SetFigFont{17}{20.4}{\rmdefault}{\mddefault}{\itdefault}{\color[rgb]{0,0,0}$z+o_a+o_b+o_c$}%
}}}}}
\put(20790,-1307){\rotatebox{300.0}{\makebox(0,0)[lb]{\smash{{\SetFigFont{17}{20.4}{\rmdefault}{\mddefault}{\itdefault}{\color[rgb]{0,0,0}$z+o_a+o_b+o_c$}%
}}}}}
\put(22116,-25053){\rotatebox{300.0}{\makebox(0,0)[lb]{\smash{{\SetFigFont{17}{20.4}{\rmdefault}{\mddefault}{\itdefault}{\color[rgb]{0,0,0}$y+2$}%
}}}}}
\put(20966,-29483){\rotatebox{60.0}{\makebox(0,0)[lb]{\smash{{\SetFigFont{17}{20.4}{\rmdefault}{\mddefault}{\itdefault}{\color[rgb]{0,0,0}$y+z+e_a+e_b+e_c+b-a$}%
}}}}}
\put(14693,-23701){\rotatebox{60.0}{\makebox(0,0)[lb]{\smash{{\SetFigFont{17}{20.4}{\rmdefault}{\mddefault}{\itdefault}{\color[rgb]{0,0,0}$z+o_a+o_b+o_c$}%
}}}}}
\put(13670,-25354){\rotatebox{60.0}{\makebox(0,0)[lb]{\smash{{\SetFigFont{17}{20.4}{\rmdefault}{\mddefault}{\itdefault}{\color[rgb]{0,0,0}$y+b-a$}%
}}}}}
\put(13892,-26978){\rotatebox{300.0}{\makebox(0,0)[lb]{\smash{{\SetFigFont{17}{20.4}{\rmdefault}{\mddefault}{\itdefault}{\color[rgb]{0,0,0}$y+z+e_a+e_b+e_c+2$}%
}}}}}
\put(17195,-30669){\makebox(0,0)[lb]{\smash{{\SetFigFont{17}{20.4}{\rmdefault}{\mddefault}{\itdefault}{\color[rgb]{0,0,0}$x+o_a+o_b+o_c$}%
}}}}
\put(17093,-21221){\makebox(0,0)[lb]{\smash{{\SetFigFont{17}{20.4}{\rmdefault}{\mddefault}{\itdefault}{\color[rgb]{0,0,0}$x+e_a+e_b+e_c$}%
}}}}
\put(7775,-30439){\makebox(0,0)[lb]{\smash{{\SetFigFont{17}{20.4}{\rmdefault}{\mddefault}{\itdefault}{\color[rgb]{0,0,0}$x+o_a+o_b+o_c$}%
}}}}
\put(4472,-26748){\rotatebox{300.0}{\makebox(0,0)[lb]{\smash{{\SetFigFont{17}{20.4}{\rmdefault}{\mddefault}{\itdefault}{\color[rgb]{0,0,0}$y+z+e_a+e_b+e_c+2$}%
}}}}}
\put(4250,-25124){\rotatebox{60.0}{\makebox(0,0)[lb]{\smash{{\SetFigFont{17}{20.4}{\rmdefault}{\mddefault}{\itdefault}{\color[rgb]{0,0,0}$y+b-a$}%
}}}}}
\put(5273,-23471){\rotatebox{60.0}{\makebox(0,0)[lb]{\smash{{\SetFigFont{17}{20.4}{\rmdefault}{\mddefault}{\itdefault}{\color[rgb]{0,0,0}$z+o_a+o_b+o_c$}%
}}}}}
\put(11546,-29253){\rotatebox{60.0}{\makebox(0,0)[lb]{\smash{{\SetFigFont{17}{20.4}{\rmdefault}{\mddefault}{\itdefault}{\color[rgb]{0,0,0}$y+z+e_a+e_b+e_c+b-a$}%
}}}}}
\put(12696,-24823){\rotatebox{300.0}{\makebox(0,0)[lb]{\smash{{\SetFigFont{17}{20.4}{\rmdefault}{\mddefault}{\itdefault}{\color[rgb]{0,0,0}$y+2$}%
}}}}}
\put(7673,-20991){\makebox(0,0)[lb]{\smash{{\SetFigFont{17}{20.4}{\rmdefault}{\mddefault}{\itdefault}{\color[rgb]{0,0,0}$x+e_a+e_b+e_c$}%
}}}}
\put(17205,-20049){\makebox(0,0)[lb]{\smash{{\SetFigFont{17}{20.4}{\rmdefault}{\mddefault}{\itdefault}{\color[rgb]{0,0,0}$x+o_a+o_b+o_c$}%
}}}}
\put(13902,-16358){\rotatebox{300.0}{\makebox(0,0)[lb]{\smash{{\SetFigFont{17}{20.4}{\rmdefault}{\mddefault}{\itdefault}{\color[rgb]{0,0,0}$y+z+e_a+e_b+e_c+2$}%
}}}}}
\put(13680,-14734){\rotatebox{60.0}{\makebox(0,0)[lb]{\smash{{\SetFigFont{17}{20.4}{\rmdefault}{\mddefault}{\itdefault}{\color[rgb]{0,0,0}$y+b-a$}%
}}}}}
\put(14703,-13081){\rotatebox{60.0}{\makebox(0,0)[lb]{\smash{{\SetFigFont{17}{20.4}{\rmdefault}{\mddefault}{\itdefault}{\color[rgb]{0,0,0}$z+o_a+o_b+o_c$}%
}}}}}
\put(20976,-18863){\rotatebox{60.0}{\makebox(0,0)[lb]{\smash{{\SetFigFont{17}{20.4}{\rmdefault}{\mddefault}{\itdefault}{\color[rgb]{0,0,0}$y+z+e_a+e_b+e_c+b-a$}%
}}}}}
\put(22126,-14433){\rotatebox{300.0}{\makebox(0,0)[lb]{\smash{{\SetFigFont{17}{20.4}{\rmdefault}{\mddefault}{\itdefault}{\color[rgb]{0,0,0}$y+2$}%
}}}}}
\put(17103,-10601){\makebox(0,0)[lb]{\smash{{\SetFigFont{17}{20.4}{\rmdefault}{\mddefault}{\itdefault}{\color[rgb]{0,0,0}$x+e_a+e_b+e_c$}%
}}}}
\put(7467,-323){\makebox(0,0)[lb]{\smash{{\SetFigFont{17}{20.4}{\rmdefault}{\mddefault}{\itdefault}{\color[rgb]{0,0,0}$x+e_a+e_b+e_c$}%
}}}}
\put(12490,-4155){\rotatebox{300.0}{\makebox(0,0)[lb]{\smash{{\SetFigFont{17}{20.4}{\rmdefault}{\mddefault}{\itdefault}{\color[rgb]{0,0,0}$y+2$}%
}}}}}
\put(11340,-8585){\rotatebox{60.0}{\makebox(0,0)[lb]{\smash{{\SetFigFont{17}{20.4}{\rmdefault}{\mddefault}{\itdefault}{\color[rgb]{0,0,0}$y+z+e_a+e_b+e_c+b-a$}%
}}}}}
\put(5067,-2803){\rotatebox{60.0}{\makebox(0,0)[lb]{\smash{{\SetFigFont{17}{20.4}{\rmdefault}{\mddefault}{\itdefault}{\color[rgb]{0,0,0}$z+o_a+o_b+o_c$}%
}}}}}
\put(4044,-4456){\rotatebox{60.0}{\makebox(0,0)[lb]{\smash{{\SetFigFont{17}{20.4}{\rmdefault}{\mddefault}{\itdefault}{\color[rgb]{0,0,0}$y+b-a$}%
}}}}}
\put(4266,-6080){\rotatebox{300.0}{\makebox(0,0)[lb]{\smash{{\SetFigFont{17}{20.4}{\rmdefault}{\mddefault}{\itdefault}{\color[rgb]{0,0,0}$y+z+e_a+e_b+e_c+2$}%
}}}}}
\put(7569,-9771){\makebox(0,0)[lb]{\smash{{\SetFigFont{17}{20.4}{\rmdefault}{\mddefault}{\itdefault}{\color[rgb]{0,0,0}$x+o_a+o_b+o_c$}%
}}}}
\put(17313,-331){\makebox(0,0)[lb]{\smash{{\SetFigFont{17}{20.4}{\rmdefault}{\mddefault}{\itdefault}{\color[rgb]{0,0,0}$x+e_a+e_b+e_c$}%
}}}}
\put(22336,-4163){\rotatebox{300.0}{\makebox(0,0)[lb]{\smash{{\SetFigFont{17}{20.4}{\rmdefault}{\mddefault}{\itdefault}{\color[rgb]{0,0,0}$y+2$}%
}}}}}
\put(21186,-8593){\rotatebox{60.0}{\makebox(0,0)[lb]{\smash{{\SetFigFont{17}{20.4}{\rmdefault}{\mddefault}{\itdefault}{\color[rgb]{0,0,0}$y+z+e_a+e_b+e_c+b-a$}%
}}}}}
\put(14913,-2811){\rotatebox{60.0}{\makebox(0,0)[lb]{\smash{{\SetFigFont{17}{20.4}{\rmdefault}{\mddefault}{\itdefault}{\color[rgb]{0,0,0}$z+o_a+o_b+o_c$}%
}}}}}
\put(13890,-4464){\rotatebox{60.0}{\makebox(0,0)[lb]{\smash{{\SetFigFont{17}{20.4}{\rmdefault}{\mddefault}{\itdefault}{\color[rgb]{0,0,0}$y+b-a$}%
}}}}}
\put(14112,-6088){\rotatebox{300.0}{\makebox(0,0)[lb]{\smash{{\SetFigFont{17}{20.4}{\rmdefault}{\mddefault}{\itdefault}{\color[rgb]{0,0,0}$y+z+e_a+e_b+e_c+$}%
}}}}}
\put(17415,-9779){\makebox(0,0)[lb]{\smash{{\SetFigFont{17}{20.4}{\rmdefault}{\mddefault}{\itdefault}{\color[rgb]{0,0,0}$x+o_a+o_b+o_c$}%
}}}}
\put(7683,-10621){\makebox(0,0)[lb]{\smash{{\SetFigFont{17}{20.4}{\rmdefault}{\mddefault}{\itdefault}{\color[rgb]{0,0,0}$x+e_a+e_b+e_c$}%
}}}}
\put(12706,-14453){\rotatebox{300.0}{\makebox(0,0)[lb]{\smash{{\SetFigFont{17}{20.4}{\rmdefault}{\mddefault}{\itdefault}{\color[rgb]{0,0,0}$y+2$}%
}}}}}
\put(11556,-18883){\rotatebox{60.0}{\makebox(0,0)[lb]{\smash{{\SetFigFont{17}{20.4}{\rmdefault}{\mddefault}{\itdefault}{\color[rgb]{0,0,0}$y+z+e_a+e_b+e_c+b-a$}%
}}}}}
\put(5283,-13101){\rotatebox{60.0}{\makebox(0,0)[lb]{\smash{{\SetFigFont{17}{20.4}{\rmdefault}{\mddefault}{\itdefault}{\color[rgb]{0,0,0}$z+o_a+o_b+o_c$}%
}}}}}
\put(4260,-14754){\rotatebox{60.0}{\makebox(0,0)[lb]{\smash{{\SetFigFont{17}{20.4}{\rmdefault}{\mddefault}{\itdefault}{\color[rgb]{0,0,0}$y+b-a$}%
}}}}}
\put(4482,-16378){\rotatebox{300.0}{\makebox(0,0)[lb]{\smash{{\SetFigFont{17}{20.4}{\rmdefault}{\mddefault}{\itdefault}{\color[rgb]{0,0,0}$y+z+e_a+e_b+e_c+2$}%
}}}}}
\put(7785,-20069){\makebox(0,0)[lb]{\smash{{\SetFigFont{17}{20.4}{\rmdefault}{\mddefault}{\itdefault}{\color[rgb]{0,0,0}$x+o_a+o_b+o_c$}%
}}}}
\end{picture}%
}
\caption{Obtaining a recurrence for $\overline{E}^{(3)}$-type regions with $a\leq b$. Kuo condensation is applied to the region $\overline{E}^{(3)}_{2,2,2}(2,2 ;\ 2,2 ;\ 3,2)$ (picture (a)) as shown on the picture (b).}\label{fig:kuooff5}
\end{figure}

The $\overline{E}^{(2)}$-recurrence for $a< b$
\begin{align}\label{offcenterrecurE2qa}
\M(\overline{E}^{(2)}_{x,y,z}(\textbf{a};\ \textbf{c};\ \textbf{b}))\M(Q^{\nwarrow}_{x-1,y,z-1}(\textbf{a}^{+1};\ \textbf{c}; \ \textbf{b}^{+1}))&=\M(\overline{F}^{(5)}_{x-1,y-1,z}(\textbf{a};\ \textbf{c}; \ \textbf{b}^{+1})) \M(\overline{K}^{(5)}_{x,y+1,z-1}(\textbf{a}^{+1};\ \textbf{c};\ \textbf{b}))\notag\\
&+
\M(\overline{E}^{(2)}_{x-1,y,z-1}(\textbf{a}^{+1};\ \textbf{c};\ \textbf{b}^{+1})) \M(Q^{\nwarrow}_{x,y,z}(\textbf{a};\ \textbf{c};\ \textbf{b})).
\end{align}

The $\overline{E}^{(2)}$-recurrence for $a\geq b$
\begin{align}\label{offcenterrecurE2qb}
\M(\overline{E}^{(2)}_{x,y,z}(\textbf{a};\ \textbf{c};\ \textbf{b}))\M(Q^{\nwarrow}_{x-1,y,z-1}(\textbf{a}^{+1};\ \textbf{c}; \ \textbf{b}^{+1}))&=\M(\overline{F}^{(5)}_{x-1,y,z}(\textbf{a};\ \textbf{c}; \ \textbf{b}^{+1})) \M(\overline{K}^{(5)}_{x,y,z-1}(\textbf{a}^{+1};\ \textbf{c};\ \textbf{b}))\notag\\
&+
\M(\overline{E}^{(2)}_{x-1,y,z-1}(\textbf{a}^{+1};\ \textbf{c};\ \textbf{b}^{+1})) \M(Q^{\nwarrow}_{x,y,z}(\textbf{a};\ \textbf{c};\ \textbf{b})).
\end{align}

The $\overline{E}^{(3)}$-recurrence for $a< b$
\begin{align}\label{offcenterrecurE3qa}
\M(Q^{\nearrow}_{x,y+1,z-1}(\textbf{a}^{+1};\ \textbf{c}; \ \textbf{b})) \M(\overline{G}^{(5)}_{x,y,z}(\textbf{a};\ \textbf{c};\ \textbf{b}))&=\M(\overline{E}^{(3)}_{x,y,z}(\textbf{a};\ \textbf{c};\ \textbf{b}))\M(Q^{\leftarrow}_{x,y+1,z-1}(\textbf{b};\ \textbf{c}^{\leftrightarrow}; \ \textbf{a}^{+1}))\notag\\
&+
\M(\overline{E}^{(3)}_{x+1,y,z-1}(\textbf{a};\ \textbf{c};\ \textbf{b})) \M(Q^{\leftarrow}_{x-1,y+1,z}(\textbf{b};\ \textbf{c}^{\leftrightarrow};\ \textbf{a}^{+1})).
\end{align}

The $\overline{E}^{(3)}$-recurrence for $a\geq b$
\begin{align}\label{offcenterrecurE3qb}
\M(Q^{\nearrow}_{x,y,z-1}(\textbf{a}^{+1};\ \textbf{c}; \ \textbf{b})) \M(\overline{G}^{(5)}_{x,y,z}(\textbf{a};\ \textbf{c};\ \textbf{b}))&=\M(\overline{E}^{(3)}_{x,y,z}(\textbf{a};\ \textbf{c};\ \textbf{b}))\M(Q^{\leftarrow}_{x,y,z-1}(\textbf{b};\ \textbf{c}^{\leftrightarrow}; \ \textbf{a}^{+1}))\notag\\
&+
\M(\overline{E}^{(3)}_{x+1,y,z-1}(\textbf{a};\ \textbf{c};\ \textbf{b})) \M(Q^{\leftarrow}_{x-1,y,z}(\textbf{b};\ \textbf{c}^{\leftrightarrow};\ \textbf{a}^{+1})).
\end{align}

\subsection{Recurrences for $\overline{F}^{(i)}$-type regions}\label{subsec:offrecurF2}

We apply Kuo's Theorem \ref{kuothm1} to the $\overline{F}^{(3)}$-, $\overline{F}^{(4)}$- and $\overline{F}^{(5)}$-type regions as shown in Figures \ref{fig:offF2}(a), (b) and (c).
The second Kuo's theorem, Theorem \ref{kuothm2},
 is used for the $\overline{F}^{(6)}$-type regions as in Figure \ref{fig:offF2}(d).

Let us demonstrate in detail for the case of $\overline{F}^{(6)}$-type region when $a\geq b$. We apply Kuo's Theorem \ref{kuothm2} to the dual graph $G$ of the  region $\overline{F}^{(6)}_{x,y,z}(\textbf{a};\ \textbf{c};\ \textbf{b})$ with the choice of the four vertices $u,v,w,s$ shown  in Figure \ref{fig:kuooff6}(d). By considering forced lozenges as in Figures \ref{fig:kuooff6}(a),(b),(d),(e),(f), we get
\begin{equation}\label{kuothm6eq1}
\M(G-\{u,s\})=\M(\overline{G}^{(5)}_{x,y,z-1}(\textbf{a}^{+1};\ \textbf{c}; \ \textbf{b}))
\end{equation}
\begin{equation}\label{kuothm6eq2}
\M(G-\{v,w\})=\M(\overline{K}^{(8)}_{x,y,z}(\textbf{a};\ \textbf{c};\ \textbf{b}))
\end{equation}
\begin{equation}\label{kuothm6eq3}
\M(G-\{u,v,w,s\})=\M(\overline{E}^{(1)}_{x,y,z-1}(\textbf{b};\ \textbf{c}^{\leftrightarrow}; \ \textbf{a}^{+1}))
\end{equation}
\begin{equation}\label{kuothm6eq4}
\M(G-\{u,w\})=\M(\overline{F}^{(6)}_{x+1,y,z-1}(\textbf{a};\ \textbf{c};\ \textbf{b}))
\end{equation}
\begin{equation}\label{kuothm6eq5}
\M(G-\{v,s\})=\M(\overline{E}^{(1)}_{x-1,y,z}(\textbf{b};\ \textbf{c}^{\leftrightarrow};\ \textbf{a}^{+1})).
\end{equation}
We note that we reflected  the leftover regions by a vertical line in (\ref{kuothm6eq3}) and (\ref{kuothm6eq5}) to get respectively $\overline{E}^{(1)}_{x,y,z-1}(\textbf{b};\ \textbf{c}^{\leftrightarrow}; \ \textbf{a}^{+1})$ and $\overline{E}^{(1)}_{x-1,y,z}(\textbf{b};\ \textbf{c}^{\leftrightarrow};\ \textbf{a}^{+1})$.

Plugging the above 5 equalities into that of Kuo's Theorem \ref{kuothm2}, we have the $\overline{F}^{(6)}$-recurrence for $a\geq b$:
\begin{align}\label{offcenterrecurF6qb}
\M(\overline{G}^{(5)}_{x,y,z-1}(\textbf{a}^{+1};\ \textbf{c}; \ \textbf{b})) \M(\overline{K}^{(8)}_{x,y,z}(\textbf{a};\ \textbf{c};\ \textbf{b}))&=\M(\overline{F}^{(6)}_{x,y,z}(\textbf{a};\ \textbf{c};\ \textbf{b}))\M(\overline{E}^{(1)}_{x,y,z-1}(\textbf{b};\ \textbf{c}^{\leftrightarrow}; \ \textbf{a}^{+1}))\notag\\
&+
\M(\overline{F}^{(6)}_{x+1,y,z-1}(\textbf{a};\ \textbf{c};\ \textbf{b})) \M(\overline{E}^{(1)}_{x-1,y,z}(\textbf{b};\ \textbf{c}^{\leftrightarrow};\ \textbf{a}^{+1})).
\end{align}

Similar application of Kuo condensation give the other 6 recurrences:

The $\overline{F}^{(3)}$-recurrence for $a< b$
\begin{align}\label{offcenterrecurF3qa}
\M(\overline{F}^{(3)}_{x,y,z}(\textbf{a};\ \textbf{c};\ \textbf{b}))\M(\overline{E}^{(1)}_{x,y-1,z-1}(\textbf{a};\ \textbf{c}; \ \textbf{b}^{+1}))&=\M(\overline{G}^{(2)}_{x,y-1,z-1}(\textbf{a};\ \textbf{c}; \ \textbf{b}^{+1})) \M(\overline{K}^{(8)}_{x,y-1,z}(\textbf{b};\ \textbf{c}^{\leftrightarrow};\ \textbf{a}))\notag\\
&+
\M(\overline{F}^{(3)}_{x+1,y,z-1}(\textbf{a};\ \textbf{c};\ \textbf{b})) \M(\overline{E}^{(1)}_{x-1,y-1,z}(\textbf{a};\ \textbf{c};\ \textbf{b}^{+1})).
\end{align}

The $\overline{F}^{(3)}$-recurrence for $a\geq b$
\begin{align}\label{offcenterrecurF3qb}
\M(\overline{F}^{(3)}_{x,y,z}(\textbf{a};\ \textbf{c};\ \textbf{b}))\M(\overline{E}^{(1)}_{x,y,z-1}(\textbf{a};\ \textbf{c}; \ \textbf{b}^{+1}))&=\M(\overline{G}^{(2)}_{x,y,z-1}(\textbf{a};\ \textbf{c}; \ \textbf{b}^{+1})) \M(\overline{K}^{(8)}_{x,y-1,z}(\textbf{b};\ \textbf{c}^{\leftrightarrow};\ \textbf{a}))\notag\\
&+
\M(\overline{F}^{(3)}_{x+1,y,z-1}(\textbf{a};\ \textbf{c};\ \textbf{b})) \M(\overline{E}^{(1)}_{x-1,y,z}(\textbf{a};\ \textbf{c};\ \textbf{b}^{+1})).
\end{align}

The $\overline{F}^{(4)}$-recurrence for $a< b$
\begin{align}\label{offcenterrecurF4qa}
\M(\overline{F}^{(4)}_{x,y,z}(\textbf{a};\ \textbf{c};\ \textbf{b}))\M(\overline{K}^{(5)}_{x-1,y,z-1}(\textbf{a}^{+1};\ \textbf{c}; \ \textbf{b}^{+1}))&=\M(\overline{E}^{(2)}_{x-1,y-1,z}(\textbf{a};\ \textbf{c}; \ \textbf{b}^{+1})) \M(\overline{G}^{(2)}_{x,y+1,z-1}(\textbf{a}^{+1};\ \textbf{c};\ \textbf{b}))\notag\\
&+
\M(\overline{F}^{(4)}_{x-1,y,z-1}(\textbf{a}^{+1};\ \textbf{c};\ \textbf{b}^{+1})) \M(\overline{K}^{(5)}_{x,y,z}(\textbf{a};\ \textbf{c};\ \textbf{b})).
\end{align}

The $\overline{F}^{(4)}$-recurrence for $a\geq b$
\begin{align}\label{offcenterrecurF4qb}
\M(\overline{F}^{(4)}_{x,y,z}(\textbf{a};\ \textbf{c};\ \textbf{b}))\M(\overline{K}^{(5)}_{x-1,y,z-1}(\textbf{a}^{+1};\ \textbf{c}; \ \textbf{b}^{+1}))&=\M(\overline{E}^{(2)}_{x-1,y,z}(\textbf{a};\ \textbf{c}; \ \textbf{b}^{+1})) \M(\overline{G}^{(2)}_{x,y,z-1}(\textbf{a}^{+1};\ \textbf{c};\ \textbf{b}))\notag\\
&+
\M(\overline{F}^{(4)}_{x-1,y,z-1}(\textbf{a}^{+1};\ \textbf{c};\ \textbf{b}^{+1})) \M(\overline{K}^{(5)}_{x,y,z}(\textbf{a};\ \textbf{c};\ \textbf{b})).
\end{align}

The $\overline{F}^{(5)}$-recurrence for $a< b$
\begin{align}\label{offcenterrecurF5qa}
\M(\overline{F}^{(5)}_{x,y,z}(\textbf{a};\ \textbf{c};\ \textbf{b}))\M(Q^{\nearrow}_{x-1,y,z-1}(\textbf{a}^{+1};\ \textbf{c}; \ \textbf{b}^{+1}))&=\M(\overline{E}^{(3)}_{x-1,y-1,z}(\textbf{a};\ \textbf{c}; \ \textbf{b}^{+1})) \M(Q^{\nwarrow}_{x,y+1,z-1}(\textbf{a}^{+1};\ \textbf{c};\ \textbf{b}))\notag\\
&+
\M(\overline{F}^{(5)}_{x-1,y,z-1}(\textbf{a}^{+1};\ \textbf{c};\ \textbf{b}^{+1})) \M(Q^{\nearrow}_{x,y,z}(\textbf{a};\ \textbf{c};\ \textbf{b})).
\end{align}

The $\overline{F}^{(5)}$-recurrence for $a\geq b$
\begin{align}\label{offcenterrecurF5qb}
\M(\overline{F}^{(5)}_{x,y,z}(\textbf{a};\ \textbf{c};\ \textbf{b}))\M(Q^{\nearrow}_{x-1,y,z-1}(\textbf{a}^{+1};\ \textbf{c}; \ \textbf{b}^{+1}))&=\M(\overline{E}^{(3)}_{x-1,y,z}(\textbf{a};\ \textbf{c}; \ \textbf{b}^{+1})) \M(Q^{\nwarrow}_{x,y,z-1}(\textbf{a}^{+1};\ \textbf{c};\ \textbf{b}))\notag\\
&+
\M(\overline{F}^{(5)}_{x-1,y,z-1}(\textbf{a}^{+1};\ \textbf{c};\ \textbf{b}^{+1})) \M(Q^{\nearrow}_{x,y,z}(\textbf{a};\ \textbf{c};\ \textbf{b})).
\end{align}

The $\overline{F}^{(6)}$-recurrence for $a< b$
\begin{align}\label{offcenterrecurF6qa}
\M(\overline{G}^{(5)}_{x,y+1,z-1}(\textbf{a}^{+1};\ \textbf{c}; \ \textbf{b})) \M(\overline{K}^{(8)}_{x,y,z}(\textbf{a};\ \textbf{c};\ \textbf{b}))&=\M(\overline{F}^{(6)}_{x,y,z}(\textbf{a};\ \textbf{c};\ \textbf{b}))\M(\overline{E}^{(1)}_{x,y+1,z-1}(\textbf{b};\ \textbf{c}^{\leftrightarrow}; \ \textbf{a}^{+1}))\notag\\
&+
\M(\overline{F}^{(6)}_{x+1,y,z-1}(\textbf{a};\ \textbf{c};\ \textbf{b})) \M(\overline{E}^{(1)}_{x-1,y+1,z}(\textbf{b};\ \textbf{c}^{\leftrightarrow};\ \textbf{a}^{+1})).
\end{align}

\begin{figure}\centering
\setlength{\unitlength}{3947sp}%
\begingroup\makeatletter\ifx\SetFigFont\undefined%
\gdef\SetFigFont#1#2#3#4#5{%
  \reset@font\fontsize{#1}{#2pt}%
  \fontfamily{#3}\fontseries{#4}\fontshape{#5}%
  \selectfont}%
\fi\endgroup%
\resizebox{15cm}{!}{
\begin{picture}(0,0)%
\includegraphics{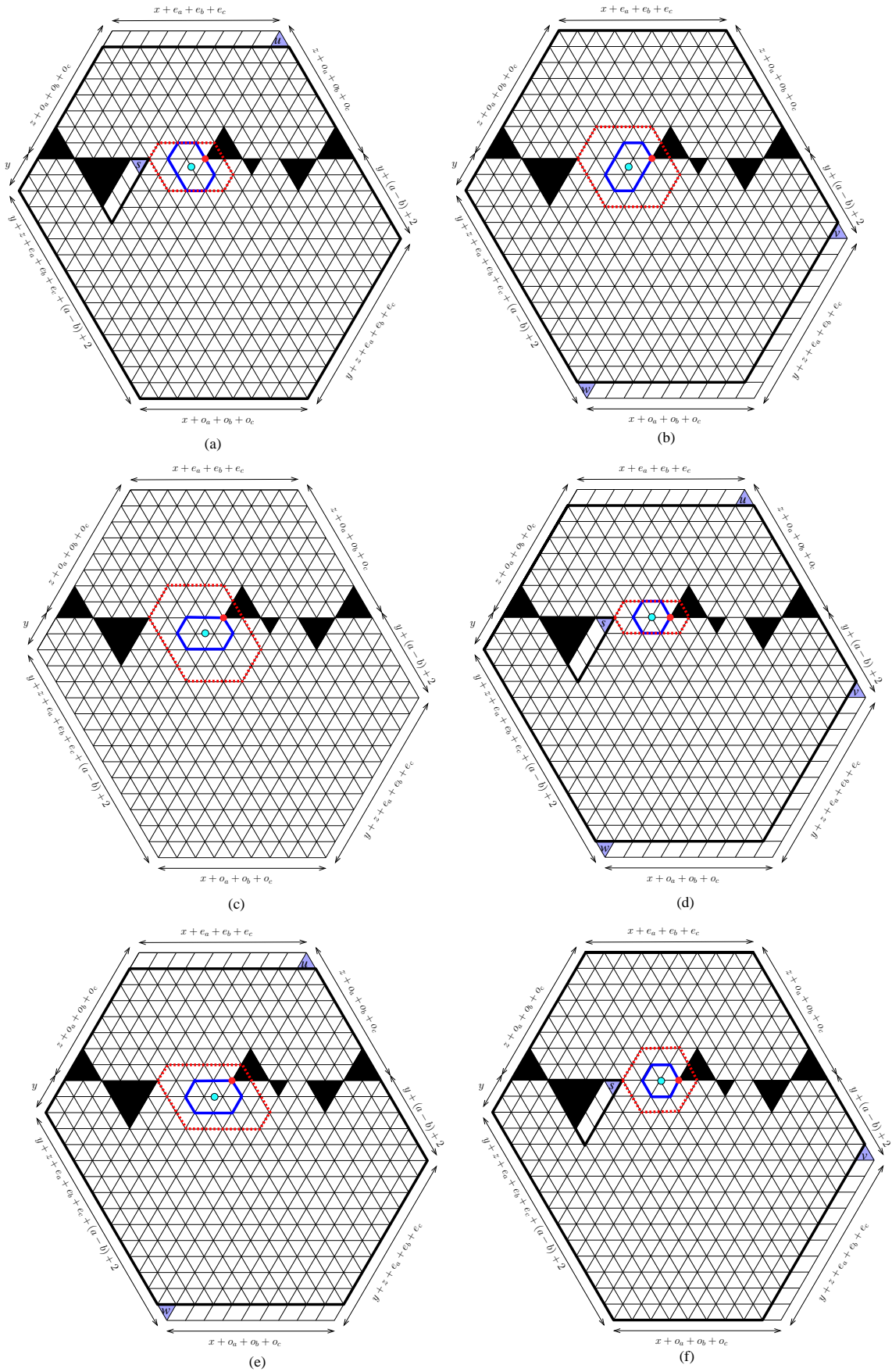}%
\end{picture}%
%

\begin{picture}(19525,30407)(3361,-34534)
\put(4596,-27384){\rotatebox{60.0}{\makebox(0,0)[lb]{\smash{{\SetFigFont{14}{16.8}{\rmdefault}{\mddefault}{\updefault}{\color[rgb]{0,0,0}$z+o_a+o_b+o_c$}%
}}}}}
\put(17060,-24836){\makebox(0,0)[lb]{\smash{{\SetFigFont{14}{16.8}{\rmdefault}{\mddefault}{\updefault}{\color[rgb]{0,0,0}$x+e_a+e_b+e_c$}%
}}}}
\put(20537,-25663){\rotatebox{300.0}{\makebox(0,0)[lb]{\smash{{\SetFigFont{14}{16.8}{\rmdefault}{\mddefault}{\updefault}{\color[rgb]{0,0,0}$z+o_a+o_b+o_c$}%
}}}}}
\put(21969,-28261){\rotatebox{300.0}{\makebox(0,0)[lb]{\smash{{\SetFigFont{14}{16.8}{\rmdefault}{\mddefault}{\updefault}{\color[rgb]{0,0,0}$y+(a-b)+2$}%
}}}}}
\put(21356,-32986){\rotatebox{60.0}{\makebox(0,0)[lb]{\smash{{\SetFigFont{14}{16.8}{\rmdefault}{\mddefault}{\updefault}{\color[rgb]{0,0,0}$y+z+e_a+e_b+e_c$}%
}}}}}
\put(17673,-33930){\makebox(0,0)[lb]{\smash{{\SetFigFont{14}{16.8}{\rmdefault}{\mddefault}{\updefault}{\color[rgb]{0,0,0}$x+o_a+o_b+o_c$}%
}}}}
\put(13889,-29383){\rotatebox{300.0}{\makebox(0,0)[lb]{\smash{{\SetFigFont{14}{16.8}{\rmdefault}{\mddefault}{\updefault}{\color[rgb]{0,0,0}$y+z+e_a+e_b+e_c+(a-b)+2$}%
}}}}}
\put(13786,-28261){\makebox(0,0)[lb]{\smash{{\SetFigFont{14}{16.8}{\rmdefault}{\mddefault}{\updefault}{\color[rgb]{0,0,0}$y$}%
}}}}
\put(14417,-27375){\rotatebox{60.0}{\makebox(0,0)[lb]{\smash{{\SetFigFont{14}{16.8}{\rmdefault}{\mddefault}{\updefault}{\color[rgb]{0,0,0}$z+o_a+o_b+o_c$}%
}}}}}
\put(6650,-4417){\makebox(0,0)[lb]{\smash{{\SetFigFont{14}{16.8}{\rmdefault}{\mddefault}{\updefault}{\color[rgb]{0,0,0}$x+e_a+e_b+e_c$}%
}}}}
\put(10127,-5244){\rotatebox{300.0}{\makebox(0,0)[lb]{\smash{{\SetFigFont{14}{16.8}{\rmdefault}{\mddefault}{\updefault}{\color[rgb]{0,0,0}$z+o_a+o_b+o_c$}%
}}}}}
\put(11559,-7842){\rotatebox{300.0}{\makebox(0,0)[lb]{\smash{{\SetFigFont{14}{16.8}{\rmdefault}{\mddefault}{\updefault}{\color[rgb]{0,0,0}$y+(a-b)+2$}%
}}}}}
\put(10946,-12567){\rotatebox{60.0}{\makebox(0,0)[lb]{\smash{{\SetFigFont{14}{16.8}{\rmdefault}{\mddefault}{\updefault}{\color[rgb]{0,0,0}$y+z+e_a+e_b+e_c$}%
}}}}}
\put(7263,-13511){\makebox(0,0)[lb]{\smash{{\SetFigFont{14}{16.8}{\rmdefault}{\mddefault}{\updefault}{\color[rgb]{0,0,0}$x+o_a+o_b+o_c$}%
}}}}
\put(3479,-8964){\rotatebox{300.0}{\makebox(0,0)[lb]{\smash{{\SetFigFont{14}{16.8}{\rmdefault}{\mddefault}{\updefault}{\color[rgb]{0,0,0}$y+z+e_a+e_b+e_c+(a-b)+2$}%
}}}}}
\put(3376,-7842){\makebox(0,0)[lb]{\smash{{\SetFigFont{14}{16.8}{\rmdefault}{\mddefault}{\updefault}{\color[rgb]{0,0,0}$y$}%
}}}}
\put(4007,-6956){\rotatebox{60.0}{\makebox(0,0)[lb]{\smash{{\SetFigFont{14}{16.8}{\rmdefault}{\mddefault}{\updefault}{\color[rgb]{0,0,0}$z+o_a+o_b+o_c$}%
}}}}}
\put(16471,-4408){\makebox(0,0)[lb]{\smash{{\SetFigFont{14}{16.8}{\rmdefault}{\mddefault}{\updefault}{\color[rgb]{0,0,0}$x+e_a+e_b+e_c$}%
}}}}
\put(19948,-5235){\rotatebox{300.0}{\makebox(0,0)[lb]{\smash{{\SetFigFont{14}{16.8}{\rmdefault}{\mddefault}{\updefault}{\color[rgb]{0,0,0}$z+o_a+o_b+o_c$}%
}}}}}
\put(21380,-7833){\rotatebox{300.0}{\makebox(0,0)[lb]{\smash{{\SetFigFont{14}{16.8}{\rmdefault}{\mddefault}{\updefault}{\color[rgb]{0,0,0}$y+(a-b)+2$}%
}}}}}
\put(20767,-12558){\rotatebox{60.0}{\makebox(0,0)[lb]{\smash{{\SetFigFont{14}{16.8}{\rmdefault}{\mddefault}{\updefault}{\color[rgb]{0,0,0}$y+z+e_a+e_b+e_c$}%
}}}}}
\put(17084,-13502){\makebox(0,0)[lb]{\smash{{\SetFigFont{14}{16.8}{\rmdefault}{\mddefault}{\updefault}{\color[rgb]{0,0,0}$x+o_a+o_b+o_c$}%
}}}}
\put(13300,-8955){\rotatebox{300.0}{\makebox(0,0)[lb]{\smash{{\SetFigFont{14}{16.8}{\rmdefault}{\mddefault}{\updefault}{\color[rgb]{0,0,0}$y+z+e_a+e_b+e_c+(a-b)+2$}%
}}}}}
\put(13197,-7833){\makebox(0,0)[lb]{\smash{{\SetFigFont{14}{16.8}{\rmdefault}{\mddefault}{\updefault}{\color[rgb]{0,0,0}$y$}%
}}}}
\put(13828,-6947){\rotatebox{60.0}{\makebox(0,0)[lb]{\smash{{\SetFigFont{14}{16.8}{\rmdefault}{\mddefault}{\updefault}{\color[rgb]{0,0,0}$z+o_a+o_b+o_c$}%
}}}}}
\put(7052,-14589){\makebox(0,0)[lb]{\smash{{\SetFigFont{14}{16.8}{\rmdefault}{\mddefault}{\updefault}{\color[rgb]{0,0,0}$x+e_a+e_b+e_c$}%
}}}}
\put(10529,-15416){\rotatebox{300.0}{\makebox(0,0)[lb]{\smash{{\SetFigFont{14}{16.8}{\rmdefault}{\mddefault}{\updefault}{\color[rgb]{0,0,0}$z+o_a+o_b+o_c$}%
}}}}}
\put(11961,-18014){\rotatebox{300.0}{\makebox(0,0)[lb]{\smash{{\SetFigFont{14}{16.8}{\rmdefault}{\mddefault}{\updefault}{\color[rgb]{0,0,0}$y+(a-b)+2$}%
}}}}}
\put(11348,-22739){\rotatebox{60.0}{\makebox(0,0)[lb]{\smash{{\SetFigFont{14}{16.8}{\rmdefault}{\mddefault}{\updefault}{\color[rgb]{0,0,0}$y+z+e_a+e_b+e_c$}%
}}}}}
\put(7665,-23683){\makebox(0,0)[lb]{\smash{{\SetFigFont{14}{16.8}{\rmdefault}{\mddefault}{\updefault}{\color[rgb]{0,0,0}$x+o_a+o_b+o_c$}%
}}}}
\put(3881,-19136){\rotatebox{300.0}{\makebox(0,0)[lb]{\smash{{\SetFigFont{14}{16.8}{\rmdefault}{\mddefault}{\updefault}{\color[rgb]{0,0,0}$y+z+e_a+e_b+e_c+(a-b)+2$}%
}}}}}
\put(3778,-18014){\makebox(0,0)[lb]{\smash{{\SetFigFont{14}{16.8}{\rmdefault}{\mddefault}{\updefault}{\color[rgb]{0,0,0}$y$}%
}}}}
\put(4409,-17128){\rotatebox{60.0}{\makebox(0,0)[lb]{\smash{{\SetFigFont{14}{16.8}{\rmdefault}{\mddefault}{\updefault}{\color[rgb]{0,0,0}$z+o_a+o_b+o_c$}%
}}}}}
\put(16873,-14580){\makebox(0,0)[lb]{\smash{{\SetFigFont{14}{16.8}{\rmdefault}{\mddefault}{\updefault}{\color[rgb]{0,0,0}$x+e_a+e_b+e_c$}%
}}}}
\put(20350,-15407){\rotatebox{300.0}{\makebox(0,0)[lb]{\smash{{\SetFigFont{14}{16.8}{\rmdefault}{\mddefault}{\updefault}{\color[rgb]{0,0,0}$z+o_a+o_b+o_c$}%
}}}}}
\put(21782,-18005){\rotatebox{300.0}{\makebox(0,0)[lb]{\smash{{\SetFigFont{14}{16.8}{\rmdefault}{\mddefault}{\updefault}{\color[rgb]{0,0,0}$y+(a-b)+2$}%
}}}}}
\put(21169,-22730){\rotatebox{60.0}{\makebox(0,0)[lb]{\smash{{\SetFigFont{14}{16.8}{\rmdefault}{\mddefault}{\updefault}{\color[rgb]{0,0,0}$y+z+e_a+e_b+e_c$}%
}}}}}
\put(17486,-23674){\makebox(0,0)[lb]{\smash{{\SetFigFont{14}{16.8}{\rmdefault}{\mddefault}{\updefault}{\color[rgb]{0,0,0}$x+o_a+o_b+o_c$}%
}}}}
\put(13702,-19127){\rotatebox{300.0}{\makebox(0,0)[lb]{\smash{{\SetFigFont{14}{16.8}{\rmdefault}{\mddefault}{\updefault}{\color[rgb]{0,0,0}$y+z+e_a+e_b+e_c+(a-b)+2$}%
}}}}}
\put(13599,-18005){\makebox(0,0)[lb]{\smash{{\SetFigFont{14}{16.8}{\rmdefault}{\mddefault}{\updefault}{\color[rgb]{0,0,0}$y$}%
}}}}
\put(14230,-17119){\rotatebox{60.0}{\makebox(0,0)[lb]{\smash{{\SetFigFont{14}{16.8}{\rmdefault}{\mddefault}{\updefault}{\color[rgb]{0,0,0}$z+o_a+o_b+o_c$}%
}}}}}
\put(7239,-24845){\makebox(0,0)[lb]{\smash{{\SetFigFont{14}{16.8}{\rmdefault}{\mddefault}{\updefault}{\color[rgb]{0,0,0}$x+e_a+e_b+e_c$}%
}}}}
\put(10716,-25672){\rotatebox{300.0}{\makebox(0,0)[lb]{\smash{{\SetFigFont{14}{16.8}{\rmdefault}{\mddefault}{\updefault}{\color[rgb]{0,0,0}$z+o_a+o_b+o_c$}%
}}}}}
\put(12148,-28270){\rotatebox{300.0}{\makebox(0,0)[lb]{\smash{{\SetFigFont{14}{16.8}{\rmdefault}{\mddefault}{\updefault}{\color[rgb]{0,0,0}$y+(a-b)+2$}%
}}}}}
\put(11535,-32995){\rotatebox{60.0}{\makebox(0,0)[lb]{\smash{{\SetFigFont{14}{16.8}{\rmdefault}{\mddefault}{\updefault}{\color[rgb]{0,0,0}$y+z+e_a+e_b+e_c$}%
}}}}}
\put(7852,-33939){\makebox(0,0)[lb]{\smash{{\SetFigFont{14}{16.8}{\rmdefault}{\mddefault}{\updefault}{\color[rgb]{0,0,0}$x+o_a+o_b+o_c$}%
}}}}
\put(4068,-29392){\rotatebox{300.0}{\makebox(0,0)[lb]{\smash{{\SetFigFont{14}{16.8}{\rmdefault}{\mddefault}{\updefault}{\color[rgb]{0,0,0}$y+z+e_a+e_b+e_c+(a-b)+2$}%
}}}}}
\put(3965,-28270){\makebox(0,0)[lb]{\smash{{\SetFigFont{14}{16.8}{\rmdefault}{\mddefault}{\updefault}{\color[rgb]{0,0,0}$y$}%
}}}}
\end{picture}%
}
\caption{Obtaining a recurrence for $\overline{F}^{(6)}$-type regions with $a> b$. Kuo condensation is applied to the region $\overline{F}^{(6)}_{3,2,2}(2,3;\ 2,1 ;\ 2,2)$ (picture (c)) as shown on the picture (d).}\label{fig:kuooff6}
\end{figure}

\subsection{Recurrences for $\overline{G}^{(i)}$-type regions}\label{subsec:offrecurG2}

We apply Theorem \ref{kuothm1} to the $\overline{G}^{(2)}$- and  $\overline{G}^{(3)}$-type regions as in Figures \ref{fig:offG2}(a) and (b). For the case of $\overline{G}^{(4)}$- and $\overline{G}^{(5)}$-type regions, we
employ Theorem \ref{kuothm2} with the choice of the vertices $u,v,w,s$ shown in Figures \ref{fig:offG2}(c) and (d).

Let us work out in detail for the case of $\overline{G}^{(3)}$-type region when $a< b$. We apply Kuo's Theorem \ref{kuothm1} to the dual graph $G$ of the  region $\overline{G}^{(3)}_{x,y,z}(\textbf{a};\ \textbf{c};\ \textbf{b})$. By considering forced lozenges as in Figures \ref{fig:kuooff7}(b)--(f), we get
\begin{equation}\label{kuothm7eq1}
\M(G-\{u,v,w,s\})=\M(\overline{E}^{(2)}_{x-1,y,z-1}(\textbf{a}^{+1};\ \textbf{c}; \ \textbf{b}^{+1}))
\end{equation}
\begin{equation}\label{kuothm7eq2}
\M(G-\{u,v\})=\M(\overline{K}^{(6)}_{x-1,y-1,z}(\textbf{a};\ \textbf{c}; \ \textbf{b}^{+1}))
\end{equation}
\begin{equation}\label{kuothm7eq3}
\M(G-\{w,s\})=\M(\overline{F}^{(4)}_{x,y+1,z-1}(\textbf{a}^{+1};\ \textbf{c};\ \textbf{b}))
\end{equation}
\begin{equation}\label{kuothm7eq4}
\M(G-\{u,s\})=\M(\overline{G}^{(3)}_{x-1,y,z-1}(\textbf{a}^{+1};\ \textbf{c};\ \textbf{b}^{+1}))
\end{equation}
\begin{equation}\label{kuothm7eq5}
\M(G-\{v,w\})=\M(\overline{E}^{(2)}_{x,y,z}(\textbf{a};\ \textbf{c};\ \textbf{b})).
\end{equation}
Plugging the above 5 recurrences into the equation in Kuo's Theorem \ref{kuothm1}, we obtain the $\overline{G}^{(3)}$-recurrence for $a< b$:
\begin{align}\label{offcenterrecurG3qa}
\M(\overline{G}^{(3)}_{x,y,z}(\textbf{a};\ \textbf{c};\ \textbf{b}))\M(\overline{E}^{(2)}_{x-1,y,z-1}(\textbf{a}^{+1};\ \textbf{c}; \ \textbf{b}^{+1}))&=\M(\overline{K}^{(6)}_{x-1,y-1,z}(\textbf{a};\ \textbf{c}; \ \textbf{b}^{+1})) \M(\overline{F}^{(4)}_{x,y+1,z-1}(\textbf{a}^{+1};\ \textbf{c};\ \textbf{b}))\notag\\
&+
\M(\overline{G}^{(3)}_{x-1,y,z-1}(\textbf{a}^{+1};\ \textbf{c};\ \textbf{b}^{+1})) \M(\overline{E}^{(2)}_{x,y,z}(\textbf{a};\ \textbf{c};\ \textbf{b})).
\end{align}

Similarly, we have the other 7 recurrences:

The $\overline{G}^{(2)}$-recurrence for $a< b$
\begin{align}\label{offcenterrecurG2qa}
\M(\overline{G}^{(2)}_{x,y,z}(\textbf{a};\ \textbf{c};\ \textbf{b}))\M(\overline{E}^{(1)}_{x-1,y,z-1}(\textbf{a}^{+1};\ \textbf{c}; \ \textbf{b}^{+1}))&=\M(\overline{K}^{(5)}_{x-1,y-1,z}(\textbf{a};\ \textbf{c}; \ \textbf{b}^{+1})) \M(\overline{F}^{(3)}_{x,y+1,z-1}(\textbf{a}^{+1};\ \textbf{c};\ \textbf{b}))\notag\\
&+
\M(\overline{G}^{(2)}_{x-1,y,z-1}(\textbf{a}^{+1};\ \textbf{c};\ \textbf{b}^{+1})) \M(\overline{E}^{(1)}_{x,y,z}(\textbf{a};\ \textbf{c};\ \textbf{b})).
\end{align}

The $\overline{G}^{(2)}$-recurrence for $a> b$
\begin{align}\label{offcenterrecurG2qb}
\M(\overline{G}^{(2)}_{x,y,z}(\textbf{a};\ \textbf{c};\ \textbf{b}))\M(\overline{E}^{(1)}_{x-1,y,z-1}(\textbf{a}^{+1};\ \textbf{c}; \ \textbf{b}^{+1}))&=\M(\overline{K}^{(5)}_{x-1,y,z}(\textbf{a};\ \textbf{c}; \ \textbf{b}^{+1})) \M(\overline{F}^{(3)}_{x,y,z-1}(\textbf{a}^{+1};\ \textbf{c};\ \textbf{b}))\notag\\
&+
\M(\overline{G}^{(2)}_{x-1,y,z-1}(\textbf{a}^{+1};\ \textbf{c};\ \textbf{b}^{+1})) \M(\overline{E}^{(1)}_{x,y,z}(\textbf{a};\ \textbf{c};\ \textbf{b})).
\end{align}

The $\overline{G}^{(2)}$-recurrence for $a= b$
\begin{align}\label{offcenterrecurG2qc}
\M(\overline{G}^{(2)}_{x,y,z}(\textbf{a};\ \textbf{c};\ \textbf{b}))\M(\overline{E}^{(1)}_{x-1,y,z-1}(\textbf{a}^{+1};\ \textbf{c}; \ \textbf{b}^{+1}))&=\M(\overline{K}^{(5)}_{x-1,y-1,z}(\textbf{a};\ \textbf{c}; \ \textbf{b}^{+1})) \M(\overline{F}^{(3)}_{x,y,z-1}(\textbf{a}^{+1};\ \textbf{c};\ \textbf{b}))\notag\\
&+
\M(\overline{G}^{(2)}_{x-1,y,z-1}(\textbf{a}^{+1};\ \textbf{c};\ \textbf{b}^{+1})) \M(\overline{E}^{(1)}_{x,y,z}(\textbf{a};\ \textbf{c};\ \textbf{b})).
\end{align}

\begin{figure}\centering
\setlength{\unitlength}{3947sp}%
\begingroup\makeatletter\ifx\SetFigFont\undefined%
\gdef\SetFigFont#1#2#3#4#5{%
  \reset@font\fontsize{#1}{#2pt}%
  \fontfamily{#3}\fontseries{#4}\fontshape{#5}%
  \selectfont}%
\fi\endgroup%
\resizebox{15cm}{!}{
\begin{picture}(0,0)%
\includegraphics{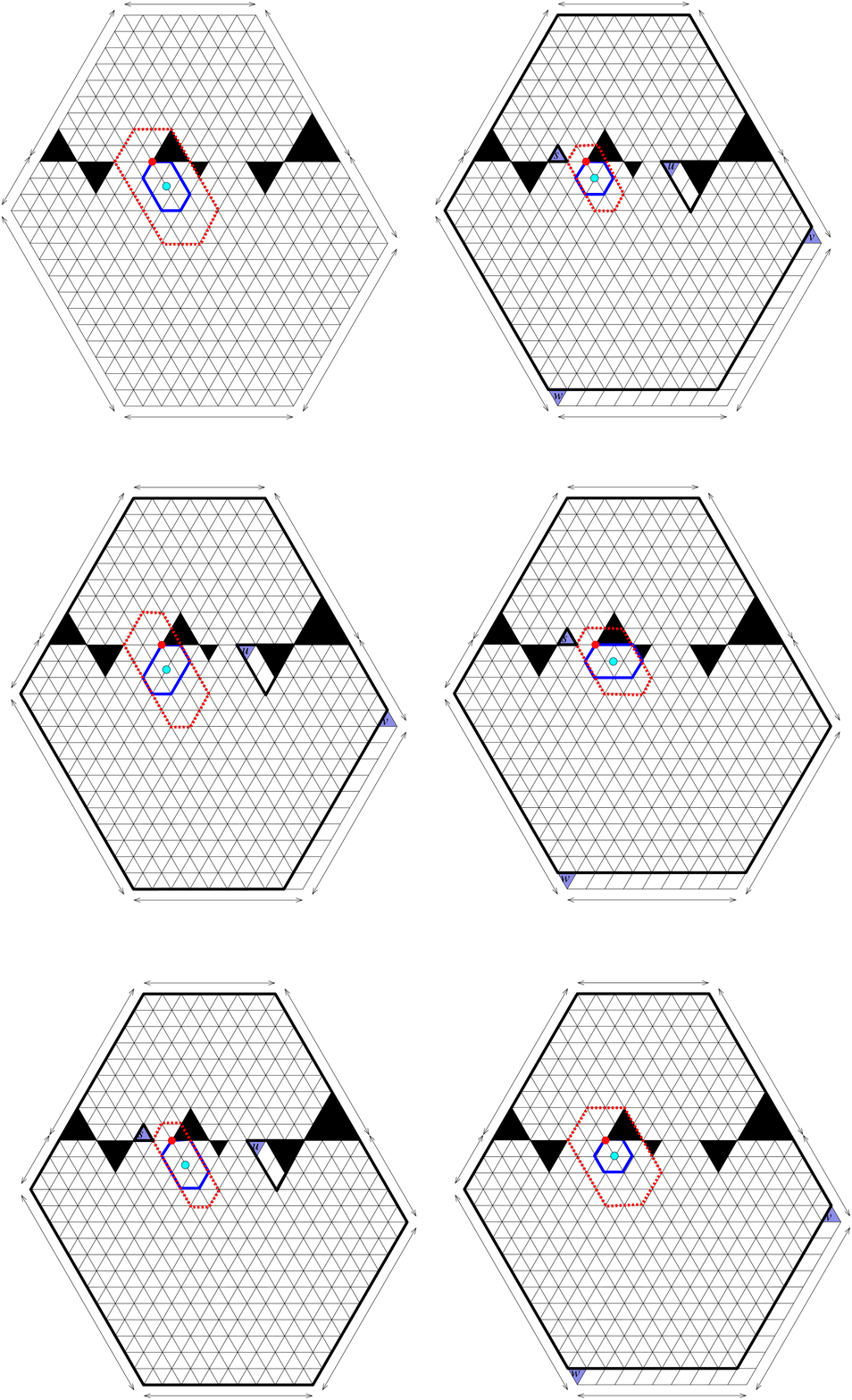}%
\end{picture}%
%
%
\setlength{\unitlength}{3947sp}%
\begingroup\makeatletter\ifx\SetFigFont\undefined%
\gdef\SetFigFont#1#2#3#4#5{%
  \reset@font\fontsize{#1}{#2pt}%
  \fontfamily{#3}\fontseries{#4}\fontshape{#5}%
  \selectfont}%
\fi\endgroup%
\begin{picture}(18976,31596)(537,-32834)
\put(4312,-22809){\makebox(0,0)[lb]{\smash{{\SetFigFont{14}{16.8}{\rmdefault}{\mddefault}{\itdefault}{\color[rgb]{0,0,0}$x+e_a+e_b+e_c$}%
}}}}
\put(7380,-23400){\rotatebox{300.0}{\makebox(0,0)[lb]{\smash{{\SetFigFont{14}{16.8}{\rmdefault}{\mddefault}{\itdefault}{\color[rgb]{0,0,0}$z+o_a+o_b+o_c$}%
}}}}}
\put(9222,-26601){\rotatebox{300.0}{\makebox(0,0)[lb]{\smash{{\SetFigFont{14}{16.8}{\rmdefault}{\mddefault}{\itdefault}{\color[rgb]{0,0,0}$y+3$}%
}}}}}
\put(8321,-31431){\rotatebox{60.0}{\makebox(0,0)[lb]{\smash{{\SetFigFont{14}{16.8}{\rmdefault}{\mddefault}{\itdefault}{\color[rgb]{0,0,0}$y+z+e_a+e_b+e_c+b-a$}%
}}}}}
\put(4721,-32258){\makebox(0,0)[lb]{\smash{{\SetFigFont{14}{16.8}{\rmdefault}{\mddefault}{\itdefault}{\color[rgb]{0,0,0}$x+o_a+o_b+o_c$}%
}}}}
\put(1448,-28608){\rotatebox{300.0}{\makebox(0,0)[lb]{\smash{{\SetFigFont{14}{16.8}{\rmdefault}{\mddefault}{\itdefault}{\color[rgb]{0,0,0}$y+z+e_a+e_b+e_c+3$}%
}}}}}
\put(1161,-27061){\rotatebox{60.0}{\makebox(0,0)[lb]{\smash{{\SetFigFont{14}{16.8}{\rmdefault}{\mddefault}{\itdefault}{\color[rgb]{0,0,0}$y+b-a$}%
}}}}}
\put(2082,-25526){\rotatebox{60.0}{\makebox(0,0)[lb]{\smash{{\SetFigFont{14}{16.8}{\rmdefault}{\mddefault}{\itdefault}{\color[rgb]{0,0,0}$z+o_a+o_b+o_c$}%
}}}}}
\put(13744,-22805){\makebox(0,0)[lb]{\smash{{\SetFigFont{14}{16.8}{\rmdefault}{\mddefault}{\itdefault}{\color[rgb]{0,0,0}$x+e_a+e_b+e_c$}%
}}}}
\put(16812,-23396){\rotatebox{300.0}{\makebox(0,0)[lb]{\smash{{\SetFigFont{14}{16.8}{\rmdefault}{\mddefault}{\itdefault}{\color[rgb]{0,0,0}$z+o_a+o_b+o_c$}%
}}}}}
\put(18654,-26597){\rotatebox{300.0}{\makebox(0,0)[lb]{\smash{{\SetFigFont{14}{16.8}{\rmdefault}{\mddefault}{\itdefault}{\color[rgb]{0,0,0}$y+3$}%
}}}}}
\put(17753,-31427){\rotatebox{60.0}{\makebox(0,0)[lb]{\smash{{\SetFigFont{14}{16.8}{\rmdefault}{\mddefault}{\itdefault}{\color[rgb]{0,0,0}$y+z+e_a+e_b+e_c+b-a$}%
}}}}}
\put(14153,-32254){\makebox(0,0)[lb]{\smash{{\SetFigFont{14}{16.8}{\rmdefault}{\mddefault}{\itdefault}{\color[rgb]{0,0,0}$x+o_a+o_b+o_c$}%
}}}}
\put(10880,-28604){\rotatebox{300.0}{\makebox(0,0)[lb]{\smash{{\SetFigFont{14}{16.8}{\rmdefault}{\mddefault}{\itdefault}{\color[rgb]{0,0,0}$y+z+e_a+e_b+e_c+3$}%
}}}}}
\put(10593,-27057){\rotatebox{60.0}{\makebox(0,0)[lb]{\smash{{\SetFigFont{14}{16.8}{\rmdefault}{\mddefault}{\itdefault}{\color[rgb]{0,0,0}$y+b-a$}%
}}}}}
\put(11514,-25522){\rotatebox{60.0}{\makebox(0,0)[lb]{\smash{{\SetFigFont{14}{16.8}{\rmdefault}{\mddefault}{\itdefault}{\color[rgb]{0,0,0}$z+o_a+o_b+o_c$}%
}}}}}
\put(4881,-11401){\makebox(0,0)[lb]{\smash{{\SetFigFont{20}{24.0}{\familydefault}{\mddefault}{\updefault}{\color[rgb]{0,0,0}(a)}%
}}}}
\put(14321,-11341){\makebox(0,0)[lb]{\smash{{\SetFigFont{20}{24.0}{\familydefault}{\mddefault}{\updefault}{\color[rgb]{0,0,0}(b)}%
}}}}
\put(5241,-22021){\makebox(0,0)[lb]{\smash{{\SetFigFont{20}{24.0}{\familydefault}{\mddefault}{\updefault}{\color[rgb]{0,0,0}(c)}%
}}}}
\put(14641,-21961){\makebox(0,0)[lb]{\smash{{\SetFigFont{20}{24.0}{\familydefault}{\mddefault}{\updefault}{\color[rgb]{0,0,0}(d)}%
}}}}
\put(5401,-32721){\makebox(0,0)[lb]{\smash{{\SetFigFont{20}{24.0}{\familydefault}{\mddefault}{\updefault}{\color[rgb]{0,0,0}(e)}%
}}}}
\put(14801,-32661){\makebox(0,0)[lb]{\smash{{\SetFigFont{20}{24.0}{\familydefault}{\mddefault}{\updefault}{\color[rgb]{0,0,0}(f)}%
}}}}
\put(3888,-1523){\makebox(0,0)[lb]{\smash{{\SetFigFont{14}{16.8}{\rmdefault}{\mddefault}{\itdefault}{\color[rgb]{0,0,0}$x+e_a+e_b+e_c$}%
}}}}
\put(6956,-2114){\rotatebox{300.0}{\makebox(0,0)[lb]{\smash{{\SetFigFont{14}{16.8}{\rmdefault}{\mddefault}{\itdefault}{\color[rgb]{0,0,0}$z+o_a+o_b+o_c$}%
}}}}}
\put(8798,-5315){\rotatebox{300.0}{\makebox(0,0)[lb]{\smash{{\SetFigFont{14}{16.8}{\rmdefault}{\mddefault}{\itdefault}{\color[rgb]{0,0,0}$y+3$}%
}}}}}
\put(7897,-10145){\rotatebox{60.0}{\makebox(0,0)[lb]{\smash{{\SetFigFont{14}{16.8}{\rmdefault}{\mddefault}{\itdefault}{\color[rgb]{0,0,0}$y+z+e_a+e_b+e_c+b-a$}%
}}}}}
\put(4297,-10972){\makebox(0,0)[lb]{\smash{{\SetFigFont{14}{16.8}{\rmdefault}{\mddefault}{\itdefault}{\color[rgb]{0,0,0}$x+o_a+o_b+o_c$}%
}}}}
\put(1024,-7322){\rotatebox{300.0}{\makebox(0,0)[lb]{\smash{{\SetFigFont{14}{16.8}{\rmdefault}{\mddefault}{\itdefault}{\color[rgb]{0,0,0}$y+z+e_a+e_b+e_c+3$}%
}}}}}
\put(737,-5775){\rotatebox{60.0}{\makebox(0,0)[lb]{\smash{{\SetFigFont{14}{16.8}{\rmdefault}{\mddefault}{\itdefault}{\color[rgb]{0,0,0}$y+b-a$}%
}}}}}
\put(1658,-4240){\rotatebox{60.0}{\makebox(0,0)[lb]{\smash{{\SetFigFont{14}{16.8}{\rmdefault}{\mddefault}{\itdefault}{\color[rgb]{0,0,0}$z+o_a+o_b+o_c$}%
}}}}}
\put(13320,-1519){\makebox(0,0)[lb]{\smash{{\SetFigFont{14}{16.8}{\rmdefault}{\mddefault}{\itdefault}{\color[rgb]{0,0,0}$x+e_a+e_b+e_c$}%
}}}}
\put(16388,-2110){\rotatebox{300.0}{\makebox(0,0)[lb]{\smash{{\SetFigFont{14}{16.8}{\rmdefault}{\mddefault}{\itdefault}{\color[rgb]{0,0,0}$z+o_a+o_b+o_c$}%
}}}}}
\put(18230,-5311){\rotatebox{300.0}{\makebox(0,0)[lb]{\smash{{\SetFigFont{14}{16.8}{\rmdefault}{\mddefault}{\itdefault}{\color[rgb]{0,0,0}$y+3$}%
}}}}}
\put(17329,-10141){\rotatebox{60.0}{\makebox(0,0)[lb]{\smash{{\SetFigFont{14}{16.8}{\rmdefault}{\mddefault}{\itdefault}{\color[rgb]{0,0,0}$y+z+e_a+e_b+e_c+b-a$}%
}}}}}
\put(13729,-10968){\makebox(0,0)[lb]{\smash{{\SetFigFont{14}{16.8}{\rmdefault}{\mddefault}{\itdefault}{\color[rgb]{0,0,0}$x+o_a+o_b+o_c$}%
}}}}
\put(10456,-7318){\rotatebox{300.0}{\makebox(0,0)[lb]{\smash{{\SetFigFont{14}{16.8}{\rmdefault}{\mddefault}{\itdefault}{\color[rgb]{0,0,0}$y+z+e_a+e_b+e_c+3$}%
}}}}}
\put(10169,-5771){\rotatebox{60.0}{\makebox(0,0)[lb]{\smash{{\SetFigFont{14}{16.8}{\rmdefault}{\mddefault}{\itdefault}{\color[rgb]{0,0,0}$y+b-a$}%
}}}}}
\put(11090,-4236){\rotatebox{60.0}{\makebox(0,0)[lb]{\smash{{\SetFigFont{14}{16.8}{\rmdefault}{\mddefault}{\itdefault}{\color[rgb]{0,0,0}$z+o_a+o_b+o_c$}%
}}}}}
\put(4092,-12029){\makebox(0,0)[lb]{\smash{{\SetFigFont{14}{16.8}{\rmdefault}{\mddefault}{\itdefault}{\color[rgb]{0,0,0}$x+e_a+e_b+e_c$}%
}}}}
\put(7160,-12620){\rotatebox{300.0}{\makebox(0,0)[lb]{\smash{{\SetFigFont{14}{16.8}{\rmdefault}{\mddefault}{\itdefault}{\color[rgb]{0,0,0}$z+o_a+o_b+o_c$}%
}}}}}
\put(9002,-15821){\rotatebox{300.0}{\makebox(0,0)[lb]{\smash{{\SetFigFont{14}{16.8}{\rmdefault}{\mddefault}{\itdefault}{\color[rgb]{0,0,0}$y+3$}%
}}}}}
\put(8101,-20651){\rotatebox{60.0}{\makebox(0,0)[lb]{\smash{{\SetFigFont{14}{16.8}{\rmdefault}{\mddefault}{\itdefault}{\color[rgb]{0,0,0}$y+z+e_a+e_b+e_c+b-a$}%
}}}}}
\put(4501,-21478){\makebox(0,0)[lb]{\smash{{\SetFigFont{14}{16.8}{\rmdefault}{\mddefault}{\itdefault}{\color[rgb]{0,0,0}$x+o_a+o_b+o_c$}%
}}}}
\put(1228,-17828){\rotatebox{300.0}{\makebox(0,0)[lb]{\smash{{\SetFigFont{14}{16.8}{\rmdefault}{\mddefault}{\itdefault}{\color[rgb]{0,0,0}$y+z+e_a+e_b+e_c+3$}%
}}}}}
\put(941,-16281){\rotatebox{60.0}{\makebox(0,0)[lb]{\smash{{\SetFigFont{14}{16.8}{\rmdefault}{\mddefault}{\itdefault}{\color[rgb]{0,0,0}$y+b-a$}%
}}}}}
\put(1862,-14746){\rotatebox{60.0}{\makebox(0,0)[lb]{\smash{{\SetFigFont{14}{16.8}{\rmdefault}{\mddefault}{\itdefault}{\color[rgb]{0,0,0}$z+o_a+o_b+o_c$}%
}}}}}
\put(13524,-12025){\makebox(0,0)[lb]{\smash{{\SetFigFont{14}{16.8}{\rmdefault}{\mddefault}{\itdefault}{\color[rgb]{0,0,0}$x+e_a+e_b+e_c$}%
}}}}
\put(16592,-12616){\rotatebox{300.0}{\makebox(0,0)[lb]{\smash{{\SetFigFont{14}{16.8}{\rmdefault}{\mddefault}{\itdefault}{\color[rgb]{0,0,0}$z+o_a+o_b+o_c$}%
}}}}}
\put(18434,-15817){\rotatebox{300.0}{\makebox(0,0)[lb]{\smash{{\SetFigFont{14}{16.8}{\rmdefault}{\mddefault}{\itdefault}{\color[rgb]{0,0,0}$y+3$}%
}}}}}
\put(17533,-20647){\rotatebox{60.0}{\makebox(0,0)[lb]{\smash{{\SetFigFont{14}{16.8}{\rmdefault}{\mddefault}{\itdefault}{\color[rgb]{0,0,0}$y+z+e_a+e_b+e_c+b-a$}%
}}}}}
\put(13933,-21474){\makebox(0,0)[lb]{\smash{{\SetFigFont{14}{16.8}{\rmdefault}{\mddefault}{\itdefault}{\color[rgb]{0,0,0}$x+o_a+o_b+o_c$}%
}}}}
\put(10660,-17824){\rotatebox{300.0}{\makebox(0,0)[lb]{\smash{{\SetFigFont{14}{16.8}{\rmdefault}{\mddefault}{\itdefault}{\color[rgb]{0,0,0}$y+z+e_a+e_b+e_c+3$}%
}}}}}
\put(10373,-16277){\rotatebox{60.0}{\makebox(0,0)[lb]{\smash{{\SetFigFont{14}{16.8}{\rmdefault}{\mddefault}{\itdefault}{\color[rgb]{0,0,0}$y+b-a$}%
}}}}}
\put(11294,-14742){\rotatebox{60.0}{\makebox(0,0)[lb]{\smash{{\SetFigFont{14}{16.8}{\rmdefault}{\mddefault}{\itdefault}{\color[rgb]{0,0,0}$z+o_a+o_b+o_c$}%
}}}}}
\end{picture}%
}
\caption{Obtaining a recurrence for $\overline{G}^{(3)}$-type regions with $a\leq b$. Kuo condensation is applied to the region $\overline{G}^{(3)}_{2,2,2}(2,2 ;\ 2,1 ;\ 3,2)$
 (picture (a)) as shown on the picture (b).}\label{fig:kuooff7}
\end{figure}

The $\overline{G}^{(3)}$-recurrence for $a>b$
\begin{align}\label{offcenterrecurG3qb}
\M(\overline{G}^{(3)}_{x,y,z}(\textbf{a};\ \textbf{c};\ \textbf{b}))\M(\overline{E}^{(2)}_{x-1,y,z-1}(\textbf{a}^{+1};\ \textbf{c}; \ \textbf{b}^{+1}))&=\M(\overline{K}^{(6)}_{x-1,y,z}(\textbf{a};\ \textbf{c}; \ \textbf{b}^{+1})) \M(\overline{F}^{(4)}_{x,y,z-1}(\textbf{a}^{+1};\ \textbf{c};\ \textbf{b}))\notag\\
&+
\M(\overline{G}^{(3)}_{x-1,y,z-1}(\textbf{a}^{+1};\ \textbf{c};\ \textbf{b}^{+1})) \M(\overline{E}^{(2)}_{x,y,z}(\textbf{a};\ \textbf{c};\ \textbf{b})).
\end{align}

The $\overline{G}^{(3)}$-recurrence for $a=b$
\begin{align}\label{offcenterrecurG3qc}
\M(\overline{G}^{(3)}_{x,y,z}(\textbf{a};\ \textbf{c};\ \textbf{b}))\M(\overline{E}^{(2)}_{x-1,y,z-1}(\textbf{a}^{+1};\ \textbf{c}; \ \textbf{b}^{+1}))&=\M(\overline{K}^{(6)}_{x-1,y-1,z}(\textbf{a};\ \textbf{c}; \ \textbf{b}^{+1})) \M(\overline{F}^{(4)}_{x,y,z-1}(\textbf{a}^{+1};\ \textbf{c};\ \textbf{b}))\notag\\
&+
\M(\overline{G}^{(3)}_{x-1,y,z-1}(\textbf{a}^{+1};\ \textbf{c};\ \textbf{b}^{+1})) \M(\overline{E}^{(2)}_{x,y,z}(\textbf{a};\ \textbf{c};\ \textbf{b})).
\end{align}

The $\overline{G}^{(4)}$-recurrence for $a< b$
\begin{align}\label{offcenterrecurG4qa}
\M(\overline{F}^{(5)}_{x,y+1,z-1}(\textbf{a}^{+1};\ \textbf{c}; \ \textbf{b})) \M(\overline{E}^{(3)}_{x,y,z}(\textbf{a};\ \textbf{c};\ \textbf{b}))&=\M(\overline{G}^{(4)}_{x,y,z}(\textbf{a};\ \textbf{c};\ \textbf{b}))\M(Q^{\nearrow}_{x,y+1,z-1}(\textbf{a}^{+1};\ \textbf{c}; \ \textbf{b}))\notag\\
&+
\M(\overline{G}^{(4)}_{x+1,y,z-1}(\textbf{a};\ \textbf{c};\ \textbf{b})) \M(Q^{\nearrow}_{x-1,y+1,z}(\textbf{a}^{+1};\ \textbf{c};\ \textbf{b})).
\end{align}

The $\overline{G}^{(4)}$-recurrence for $a\geq b$
\begin{align}\label{offcenterrecurG4qb}
\M(\overline{F}^{(5)}_{x,y,z-1}(\textbf{a}^{+1};\ \textbf{c}; \ \textbf{b})) \M(\overline{E}^{(3)}_{x,y,z}(\textbf{a};\ \textbf{c};\ \textbf{b}))&=\M(\overline{G}^{(4)}_{x,y,z}(\textbf{a};\ \textbf{c};\ \textbf{b}))\M(Q^{\nearrow}_{x,y,z-1}(\textbf{a}^{+1};\ \textbf{c}; \ \textbf{b}))\notag\\
&+
\M(\overline{G}^{(4)}_{x+1,y,z-1}(\textbf{a};\ \textbf{c};\ \textbf{b})) \M(Q^{\nearrow}_{x-1,y,z}(\textbf{a}^{+1};\ \textbf{c};\ \textbf{b})).
\end{align}

The $\overline{G}^{(5)}$-recurrence for $a< b$
\begin{align}\label{offcenterrecurG5qa}
\M(Q^{\leftarrow}_{x,y+1,z-1}(\textbf{b};\ \textbf{c}^{\leftrightarrow}; \ \textbf{a}^{+1})) \M(\overline{E}^{(1)}_{x,y,z}(\textbf{b};\ \textbf{c}^{\leftrightarrow}; \ \textbf{a}))&=\M(\overline{G}^{(5)}_{x,y,z}(\textbf{a};\ \textbf{c};\ \textbf{b}))\M(\overline{K}^{(5)}_{x,y,z-1}(\textbf{b};\ \textbf{c}^{\leftrightarrow}; \ \textbf{a}^{+1}))\notag\\
&+
\M(\overline{G}^{(5)}_{x+1,y,z-1}(\textbf{a};\ \textbf{c};\ \textbf{b})) \M(\overline{K}^{(5)}_{x-1,y,z}(\textbf{b};\ \textbf{c}^{\leftrightarrow}; \ \textbf{a}^{+1})).
\end{align}

The $\overline{G}^{(5)}$-recurrence for $a\geq b$
\begin{align}\label{offcenterrecurG5qb}
\M(Q^{\leftarrow}_{x,y,z-1}(\textbf{b};\ \textbf{c}^{\leftrightarrow}; \ \textbf{a}^{+1})) \M(\overline{E}^{(1)}_{x,y,z}(\textbf{b};\ \textbf{c}^{\leftrightarrow}; \ \textbf{a}))&=\M(\overline{G}^{(5)}_{x,y,z}(\textbf{a};\ \textbf{c};\ \textbf{b}))\M(\overline{K}^{(5)}_{x,y,z-1}(\textbf{b};\ \textbf{c}^{\leftrightarrow}; \ \textbf{a}^{+1}))\notag\\
&+
\M(\overline{G}^{(5)}_{x+1,y,z-1}(\textbf{a};\ \textbf{c};\ \textbf{b})) \M(\overline{K}^{(5)}_{x-1,y-1,z}(\textbf{b};\ \textbf{c}^{\leftrightarrow}; \ \textbf{a}^{+1})).
\end{align}

\subsection{Recurrences for $\overline{K}^{(i)}$-type regions}\label{subsec:offrecurK2}

For the final groups of regions, $\overline{K}^{(5)}$--$\overline{K}^{(8)}$, we apply Kuo condensation in the same way we did for the  $\overline{G}^{(i)}$-type regions above. The placements of the $u$-, $v$-, $w$- and $s$-
triangles are shown particularly in Figure \ref{fig:offK2}.

Let us demonstrate in detail for the case of $\overline{K}^{(8)}$-type region when $a\geq b$. We apply Kuo's Theorem \ref{kuothm2} to the dual graph $G$ of the  region $\overline{K}^{(8)}_{x,y,z}(\textbf{a};\ \textbf{c};\ \textbf{b})$. By considering forced lozenges as in Figures \ref{fig:kuooff8}(a),(b),(d),(e),(f), we get
\begin{equation}\label{kuothm8eq1}
\M(G-\{u,s\})=\M(\overline{E}^{(1)}_{x,y,z-1}(\textbf{b};\ \textbf{c}^{\leftrightarrow}; \ \textbf{a}^{+1}))
\end{equation}
\begin{equation}\label{kuothm8eq2}
\M(G-\{v,w\})=\M(\overline{F}^{(3)}_{x,y,z}(\textbf{b};\ \textbf{c}^{\leftrightarrow}; \ \textbf{a}))
\end{equation}
\begin{equation}\label{kuothm8eq3}
\M(G-\{u,v,w,s\})=\M(\overline{G}^{(2)}_{x,y-1,z-1}(\textbf{b};\ \textbf{c}^{\leftrightarrow}; \ \textbf{a}^{+1}))
\end{equation}
\begin{equation}\label{kuothm8eq4}
\M(G-\{u,w\})=\M(\overline{K}^{(8)}_{x+1,y,z-1}(\textbf{a};\ \textbf{c};\ \textbf{b}))
\end{equation}
\begin{equation}\label{kuothm8eq5}
\M(G-\{v,s\})=\M(\overline{G}^{(2)}_{x-1,y-1,z}(\textbf{b};\ \textbf{c}^{\leftrightarrow}; \ \textbf{a}^{+1}).
\end{equation}
One can realize that we reflected the leftover regions by a vertical line in (\ref{kuothm8eq1}), (\ref{kuothm8eq2}), (\ref{kuothm8eq3}) and (\ref{kuothm8eq5}).

Plugging the above 5 equalities into that of Kuo's Theorem \ref{kuothm2},  we have the $\overline{K}^{(8)}$-recurrence for $a\geq b$ (illustrated in Figure \ref{fig:kuooff8})
\begin{align}\label{offcenterrecurK8qb}
\M(\overline{E}^{(1)}_{x,y,z-1}(\textbf{b};\ \textbf{c}^{\leftrightarrow}; \ \textbf{a}^{+1})) \M(\overline{F}^{(3)}_{x,y,z}(\textbf{b};\ \textbf{c}^{\leftrightarrow}; \ \textbf{a}))&=\M(\overline{K}^{(8)}_{x,y,z}(\textbf{a};\ \textbf{c};\ \textbf{b}))\M(\overline{G}^{(2)}_{x,y-1,z-1}(\textbf{b};\ \textbf{c}^{\leftrightarrow}; \ \textbf{a}^{+1}))\notag\\
&+
\M(\overline{K}^{(8)}_{x+1,y,z-1}(\textbf{a};\ \textbf{c};\ \textbf{b})) \M(\overline{G}^{(2)}_{x-1,y-1,z}(\textbf{b};\ \textbf{c}^{\leftrightarrow}; \ \textbf{a}^{+1})).
\end{align}

We also have 7 more recurrences by applying Kuo condensation:

The $\overline{K}^{(5)}$-recurrence for $a< b$
\begin{align}\label{offcenterrecurK5qa}
\M(\overline{K}^{(5)}_{x,y,z}(\textbf{a};\ \textbf{c};\ \textbf{b}))\M(Q^{\leftarrow}_{x-1,y,z-1}(\textbf{a}^{+1};\ \textbf{c}; \ \textbf{b}^{+1}))&=\M(Q^{\nwarrow}_{x-1,y-1,z}(\textbf{a};\ \textbf{c}; \ \textbf{b}^{+1})) \M(\overline{E}^{(1)}_{x,y+1,z-1}(\textbf{a}^{+1};\ \textbf{c};\ \textbf{b}))\notag\\
&+
\M(\overline{K}^{(5)}_{x-1,y,z-1}(\textbf{a}^{+1};\ \textbf{c};\ \textbf{b}^{+1})) \M(Q^{\leftarrow}_{x,y,z}(\textbf{a};\ \textbf{c};\ \textbf{b})).
\end{align}

The $\overline{K}^{(5)}$-recurrence for $a\geq  b$
\begin{align}\label{offcenterrecurK5qb}
\M(\overline{K}^{(5)}_{x,y,z}(\textbf{a};\ \textbf{c};\ \textbf{b}))\M(Q^{\leftarrow}_{x-1,y,z-1}(\textbf{a}^{+1};\ \textbf{c}; \ \textbf{b}^{+1}))&=\M(Q^{\nwarrow}_{x-1,y,z}(\textbf{a};\ \textbf{c}; \ \textbf{b}^{+1})) \M(\overline{E}^{(1)}_{x,y,z-1}(\textbf{a}^{+1};\ \textbf{c};\ \textbf{b}))\notag\\
&+
\M(\overline{K}^{(5)}_{x-1,y,z-1}(\textbf{a}^{+1};\ \textbf{c};\ \textbf{b}^{+1})) \M(Q^{\leftarrow}_{x,y,z}(\textbf{a};\ \textbf{c};\ \textbf{b})).
\end{align}

The $\overline{K}^{(6)}$-recurrence for $a< b$
\begin{align}\label{offcenterrecurK6qa}
\M(\overline{K}^{(6)}_{x,y,z}(\textbf{a};\ \textbf{c};\ \textbf{b}))\M(\overline{F}^{(5)}_{x-1,y,z-1}(\textbf{a}^{+1};\ \textbf{c}; \ \textbf{b}^{+1}))&=\M(\overline{G}^{(4)}_{x-1,y-1,z}(\textbf{a};\ \textbf{c}; \ \textbf{b}^{+1})) \M(\overline{E}^{(2)}_{x,y+1,z-1}(\textbf{a}^{+1};\ \textbf{c};\ \textbf{b}))\notag\\
&+
\M(\overline{K}^{(6)}_{x-1,y,z-1}(\textbf{a}^{+1};\ \textbf{c};\ \textbf{b}^{+1})) \M(\overline{F}^{(5)}_{x,y,z}(\textbf{a};\ \textbf{c};\ \textbf{b})).
\end{align}

The $\overline{K}^{(6)}$-recurrence for $a\geq  b$
\begin{align}\label{offcenterrecurK6qb}
\M(\overline{K}^{(6)}_{x,y,z}(\textbf{a};\ \textbf{c};\ \textbf{b}))\M(\overline{F}^{(5)}_{x-1,y,z-1}(\textbf{a}^{+1};\ \textbf{c}; \ \textbf{b}^{+1}))&=\M(\overline{G}^{(4)}_{x-1,y,z}(\textbf{a};\ \textbf{c}; \ \textbf{b}^{+1})) \M(\overline{E}^{(2)}_{x,y,z-1}(\textbf{a}^{+1};\ \textbf{c};\ \textbf{b}))\notag\\
&+
\M(\overline{K}^{(6)}_{x-1,y,z-1}(\textbf{a}^{+1};\ \textbf{c};\ \textbf{b}^{+1})) \M(\overline{F}^{(5)}_{x,y,z}(\textbf{a};\ \textbf{c};\ \textbf{b})).
\end{align}

The $\overline{K}^{(7)}$-recurrence for $a< b$
\begin{align}\label{offcenterrecurK7qa}
\M(\overline{E}^{(3)}_{x,y+1,z-1}(\textbf{a}^{+1};\ \textbf{c}; \ \textbf{b})) \M(\overline{F}^{(6)}_{x,y,z}(\textbf{a};\ \textbf{c};\ \textbf{b}))&=\M(\overline{K}^{(7)}_{x,y,z}(\textbf{a};\ \textbf{c};\ \textbf{b}))\M(\overline{G}^{(5)}_{x,y+1,z-1}(\textbf{a}^{+1};\ \textbf{c}; \ \textbf{b}))\notag\\
&+
\M(\overline{K}^{(7)}_{x+1,y,z-1}(\textbf{a};\ \textbf{c};\ \textbf{b})) \M(\overline{G}^{(5)}_{x-1,y+1,z}(\textbf{a}^{+1};\ \textbf{c};\ \textbf{b})).
\end{align}

The $\overline{K}^{(7)}$-recurrence for $a\geq b$
\begin{align}\label{offcenterrecurK7qb}
\M(\overline{E}^{(3)}_{x,y,z-1}(\textbf{a}^{+1};\ \textbf{c}; \ \textbf{b})) \M(\overline{F}^{(6)}_{x,y,z}(\textbf{a};\ \textbf{c};\ \textbf{b}))&=\M(\overline{K}^{(7)}_{x,y,z}(\textbf{a};\ \textbf{c};\ \textbf{b}))\M(\overline{G}^{(5)}_{x,y,z-1}(\textbf{a}^{+1};\ \textbf{c}; \ \textbf{b}))\notag\\
&+
\M(\overline{K}^{(7)}_{x+1,y,z-1}(\textbf{a};\ \textbf{c};\ \textbf{b})) \M(\overline{G}^{(5)}_{x-1,y,z}(\textbf{a}^{+1};\ \textbf{c};\ \textbf{b})).
\end{align}

The $\overline{K}^{(8)}$-recurrence for $a< b$
\begin{align}\label{offcenterrecurK8qa}
\M(\overline{E}^{(1)}_{x,y+1,z-1}(\textbf{b};\ \textbf{c}^{\leftrightarrow}; \ \textbf{a}^{+1})) \M(\overline{F}^{(3)}_{x,y,z}(\textbf{b};\ \textbf{c}^{\leftrightarrow}; \ \textbf{a}))&=\M(\overline{K}^{(8)}_{x,y,z}(\textbf{a};\ \textbf{c};\ \textbf{b}))\M(\overline{G}^{(2)}_{x,y,z-1}(\textbf{b};\ \textbf{c}^{\leftrightarrow}; \ \textbf{a}^{+1}))\notag\\
&+
\M(\overline{K}^{(8)}_{x+1,y,z-1}(\textbf{a};\ \textbf{c};\ \textbf{b})) \M(\overline{G}^{(2)}_{x-1,y,z}(\textbf{b};\ \textbf{c}^{\leftrightarrow}; \ \textbf{a}^{+1})).
\end{align}

\begin{figure}\centering
\setlength{\unitlength}{3947sp}%
\begingroup\makeatletter\ifx\SetFigFont\undefined%
\gdef\SetFigFont#1#2#3#4#5{%
  \reset@font\fontsize{#1}{#2pt}%
  \fontfamily{#3}\fontseries{#4}\fontshape{#5}%
  \selectfont}%
\fi\endgroup
\resizebox{15cm}{!}{
\begin{picture}(0,0)%
\includegraphics{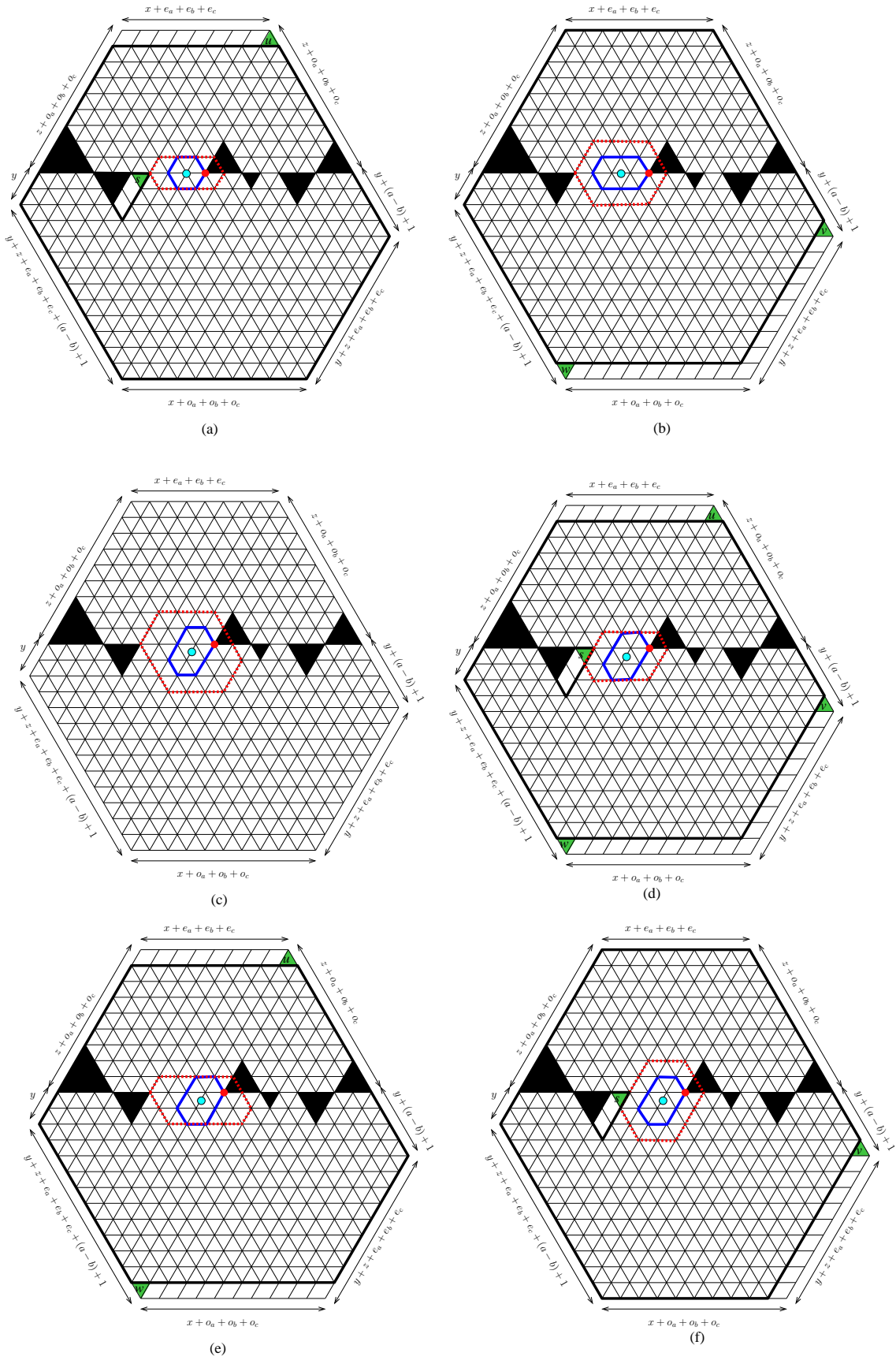}%
\end{picture}%
%
%

\begin{picture}(19871,30376)(1041,-30969)
\put(4398,-11407){\makebox(0,0)[lb]{\smash{{\SetFigFont{14}{16.8}{\rmdefault}{\mddefault}{\itdefault}{\color[rgb]{0,0,0}$x+e_a+e_b+e_c$}%
}}}}
\put(7875,-12115){\rotatebox{300.0}{\makebox(0,0)[lb]{\smash{{\SetFigFont{14}{16.8}{\rmdefault}{\mddefault}{\itdefault}{\color[rgb]{0,0,0}$z+o_a+o_b+o_c$}%
}}}}}
\put(2066,-14182){\rotatebox{60.0}{\makebox(0,0)[lb]{\smash{{\SetFigFont{14}{16.8}{\rmdefault}{\mddefault}{\itdefault}{\color[rgb]{0,0,0}$z+o_a+o_b+o_c$}%
}}}}}
\put(9512,-15021){\rotatebox{300.0}{\makebox(0,0)[lb]{\smash{{\SetFigFont{14}{16.8}{\rmdefault}{\mddefault}{\itdefault}{\color[rgb]{0,0,0}$y+(a-b)+1$}%
}}}}}
\put(1431,-15127){\makebox(0,0)[lb]{\smash{{\SetFigFont{14}{16.8}{\rmdefault}{\mddefault}{\itdefault}{\color[rgb]{0,0,0}$y$}%
}}}}
\put(8694,-19320){\rotatebox{60.0}{\makebox(0,0)[lb]{\smash{{\SetFigFont{14}{16.8}{\rmdefault}{\mddefault}{\itdefault}{\color[rgb]{0,0,0}$y+z+e_a+e_b+e_c$}%
}}}}}
\put(4909,-20206){\makebox(0,0)[lb]{\smash{{\SetFigFont{14}{16.8}{\rmdefault}{\mddefault}{\itdefault}{\color[rgb]{0,0,0}$x+o_a+o_b+o_c$}%
}}}}
\put(1329,-16439){\rotatebox{300.0}{\makebox(0,0)[lb]{\smash{{\SetFigFont{14}{16.8}{\rmdefault}{\mddefault}{\itdefault}{\color[rgb]{0,0,0}$y+z+e_a+e_b+e_c+(a-b)+1$}%
}}}}}
\put(14045,-11494){\makebox(0,0)[lb]{\smash{{\SetFigFont{14}{16.8}{\rmdefault}{\mddefault}{\itdefault}{\color[rgb]{0,0,0}$x+e_a+e_b+e_c$}%
}}}}
\put(17522,-12202){\rotatebox{300.0}{\makebox(0,0)[lb]{\smash{{\SetFigFont{14}{16.8}{\rmdefault}{\mddefault}{\itdefault}{\color[rgb]{0,0,0}$z+o_a+o_b+o_c$}%
}}}}}
\put(11713,-14269){\rotatebox{60.0}{\makebox(0,0)[lb]{\smash{{\SetFigFont{14}{16.8}{\rmdefault}{\mddefault}{\itdefault}{\color[rgb]{0,0,0}$z+o_a+o_b+o_c$}%
}}}}}
\put(19159,-15108){\rotatebox{300.0}{\makebox(0,0)[lb]{\smash{{\SetFigFont{14}{16.8}{\rmdefault}{\mddefault}{\itdefault}{\color[rgb]{0,0,0}$y+(a-b)+1$}%
}}}}}
\put(11078,-15214){\makebox(0,0)[lb]{\smash{{\SetFigFont{14}{16.8}{\rmdefault}{\mddefault}{\itdefault}{\color[rgb]{0,0,0}$y$}%
}}}}
\put(18341,-19407){\rotatebox{60.0}{\makebox(0,0)[lb]{\smash{{\SetFigFont{14}{16.8}{\rmdefault}{\mddefault}{\itdefault}{\color[rgb]{0,0,0}$y+z+e_a+e_b+e_c$}%
}}}}}
\put(14556,-20293){\makebox(0,0)[lb]{\smash{{\SetFigFont{14}{16.8}{\rmdefault}{\mddefault}{\itdefault}{\color[rgb]{0,0,0}$x+o_a+o_b+o_c$}%
}}}}
\put(10976,-16526){\rotatebox{300.0}{\makebox(0,0)[lb]{\smash{{\SetFigFont{14}{16.8}{\rmdefault}{\mddefault}{\itdefault}{\color[rgb]{0,0,0}$y+z+e_a+e_b+e_c+(a-b)+1$}%
}}}}}
\put(4610,-21424){\makebox(0,0)[lb]{\smash{{\SetFigFont{14}{16.8}{\rmdefault}{\mddefault}{\itdefault}{\color[rgb]{0,0,0}$x+e_a+e_b+e_c$}%
}}}}
\put(8087,-22132){\rotatebox{300.0}{\makebox(0,0)[lb]{\smash{{\SetFigFont{14}{16.8}{\rmdefault}{\mddefault}{\itdefault}{\color[rgb]{0,0,0}$z+o_a+o_b+o_c$}%
}}}}}
\put(2278,-24199){\rotatebox{60.0}{\makebox(0,0)[lb]{\smash{{\SetFigFont{14}{16.8}{\rmdefault}{\mddefault}{\itdefault}{\color[rgb]{0,0,0}$z+o_a+o_b+o_c$}%
}}}}}
\put(9724,-25038){\rotatebox{300.0}{\makebox(0,0)[lb]{\smash{{\SetFigFont{14}{16.8}{\rmdefault}{\mddefault}{\itdefault}{\color[rgb]{0,0,0}$y+(a-b)+1$}%
}}}}}
\put(1643,-25144){\makebox(0,0)[lb]{\smash{{\SetFigFont{14}{16.8}{\rmdefault}{\mddefault}{\itdefault}{\color[rgb]{0,0,0}$y$}%
}}}}
\put(8906,-29337){\rotatebox{60.0}{\makebox(0,0)[lb]{\smash{{\SetFigFont{14}{16.8}{\rmdefault}{\mddefault}{\itdefault}{\color[rgb]{0,0,0}$y+z+e_a+e_b+e_c$}%
}}}}}
\put(5121,-30223){\makebox(0,0)[lb]{\smash{{\SetFigFont{14}{16.8}{\rmdefault}{\mddefault}{\itdefault}{\color[rgb]{0,0,0}$x+o_a+o_b+o_c$}%
}}}}
\put(1541,-26456){\rotatebox{300.0}{\makebox(0,0)[lb]{\smash{{\SetFigFont{14}{16.8}{\rmdefault}{\mddefault}{\itdefault}{\color[rgb]{0,0,0}$y+z+e_a+e_b+e_c+(a-b)+1$}%
}}}}}
\put(14840,-21424){\makebox(0,0)[lb]{\smash{{\SetFigFont{14}{16.8}{\rmdefault}{\mddefault}{\itdefault}{\color[rgb]{0,0,0}$x+e_a+e_b+e_c$}%
}}}}
\put(18317,-22132){\rotatebox{300.0}{\makebox(0,0)[lb]{\smash{{\SetFigFont{14}{16.8}{\rmdefault}{\mddefault}{\itdefault}{\color[rgb]{0,0,0}$z+o_a+o_b+o_c$}%
}}}}}
\put(12508,-24199){\rotatebox{60.0}{\makebox(0,0)[lb]{\smash{{\SetFigFont{14}{16.8}{\rmdefault}{\mddefault}{\itdefault}{\color[rgb]{0,0,0}$z+o_a+o_b+o_c$}%
}}}}}
\put(19954,-25038){\rotatebox{300.0}{\makebox(0,0)[lb]{\smash{{\SetFigFont{14}{16.8}{\rmdefault}{\mddefault}{\itdefault}{\color[rgb]{0,0,0}$y+(a-b)+1$}%
}}}}}
\put(11873,-25144){\makebox(0,0)[lb]{\smash{{\SetFigFont{14}{16.8}{\rmdefault}{\mddefault}{\itdefault}{\color[rgb]{0,0,0}$y$}%
}}}}
\put(19136,-29337){\rotatebox{60.0}{\makebox(0,0)[lb]{\smash{{\SetFigFont{14}{16.8}{\rmdefault}{\mddefault}{\itdefault}{\color[rgb]{0,0,0}$y+z+e_a+e_b+e_c$}%
}}}}}
\put(15351,-30223){\makebox(0,0)[lb]{\smash{{\SetFigFont{14}{16.8}{\rmdefault}{\mddefault}{\itdefault}{\color[rgb]{0,0,0}$x+o_a+o_b+o_c$}%
}}}}
\put(11771,-26456){\rotatebox{300.0}{\makebox(0,0)[lb]{\smash{{\SetFigFont{14}{16.8}{\rmdefault}{\mddefault}{\itdefault}{\color[rgb]{0,0,0}$y+z+e_a+e_b+e_c+(a-b)+1$}%
}}}}}
\put(4195,-874){\makebox(0,0)[lb]{\smash{{\SetFigFont{14}{16.8}{\rmdefault}{\mddefault}{\itdefault}{\color[rgb]{0,0,0}$x+e_a+e_b+e_c$}%
}}}}
\put(7672,-1582){\rotatebox{300.0}{\makebox(0,0)[lb]{\smash{{\SetFigFont{14}{16.8}{\rmdefault}{\mddefault}{\itdefault}{\color[rgb]{0,0,0}$z+o_a+o_b+o_c$}%
}}}}}
\put(1863,-3649){\rotatebox{60.0}{\makebox(0,0)[lb]{\smash{{\SetFigFont{14}{16.8}{\rmdefault}{\mddefault}{\itdefault}{\color[rgb]{0,0,0}$z+o_a+o_b+o_c$}%
}}}}}
\put(9309,-4488){\rotatebox{300.0}{\makebox(0,0)[lb]{\smash{{\SetFigFont{14}{16.8}{\rmdefault}{\mddefault}{\itdefault}{\color[rgb]{0,0,0}$y+(a-b)+1$}%
}}}}}
\put(1228,-4594){\makebox(0,0)[lb]{\smash{{\SetFigFont{14}{16.8}{\rmdefault}{\mddefault}{\itdefault}{\color[rgb]{0,0,0}$y$}%
}}}}
\put(8491,-8787){\rotatebox{60.0}{\makebox(0,0)[lb]{\smash{{\SetFigFont{14}{16.8}{\rmdefault}{\mddefault}{\itdefault}{\color[rgb]{0,0,0}$y+z+e_a+e_b+e_c$}%
}}}}}
\put(4706,-9673){\makebox(0,0)[lb]{\smash{{\SetFigFont{14}{16.8}{\rmdefault}{\mddefault}{\itdefault}{\color[rgb]{0,0,0}$x+o_a+o_b+o_c$}%
}}}}
\put(1126,-5906){\rotatebox{300.0}{\makebox(0,0)[lb]{\smash{{\SetFigFont{14}{16.8}{\rmdefault}{\mddefault}{\itdefault}{\color[rgb]{0,0,0}$y+z+e_a+e_b+e_c+(a-b)+1$}%
}}}}}
\put(14033,-874){\makebox(0,0)[lb]{\smash{{\SetFigFont{14}{16.8}{\rmdefault}{\mddefault}{\itdefault}{\color[rgb]{0,0,0}$x+e_a+e_b+e_c$}%
}}}}
\put(17510,-1582){\rotatebox{300.0}{\makebox(0,0)[lb]{\smash{{\SetFigFont{14}{16.8}{\rmdefault}{\mddefault}{\itdefault}{\color[rgb]{0,0,0}$z+o_a+o_b+o_c$}%
}}}}}
\put(11701,-3649){\rotatebox{60.0}{\makebox(0,0)[lb]{\smash{{\SetFigFont{14}{16.8}{\rmdefault}{\mddefault}{\itdefault}{\color[rgb]{0,0,0}$z+o_a+o_b+o_c$}%
}}}}}
\put(19147,-4488){\rotatebox{300.0}{\makebox(0,0)[lb]{\smash{{\SetFigFont{14}{16.8}{\rmdefault}{\mddefault}{\itdefault}{\color[rgb]{0,0,0}$y+(a-b)+1$}%
}}}}}
\put(11066,-4594){\makebox(0,0)[lb]{\smash{{\SetFigFont{14}{16.8}{\rmdefault}{\mddefault}{\itdefault}{\color[rgb]{0,0,0}$y$}%
}}}}
\put(18329,-8787){\rotatebox{60.0}{\makebox(0,0)[lb]{\smash{{\SetFigFont{14}{16.8}{\rmdefault}{\mddefault}{\itdefault}{\color[rgb]{0,0,0}$y+z+e_a+e_b+e_c$}%
}}}}}
\put(14544,-9673){\makebox(0,0)[lb]{\smash{{\SetFigFont{14}{16.8}{\rmdefault}{\mddefault}{\itdefault}{\color[rgb]{0,0,0}$x+o_a+o_b+o_c$}%
}}}}
\put(10964,-5906){\rotatebox{300.0}{\makebox(0,0)[lb]{\smash{{\SetFigFont{14}{16.8}{\rmdefault}{\mddefault}{\itdefault}{\color[rgb]{0,0,0}$y+z+e_a+e_b+e_c+(a-b)+1$}%
}}}}}
\end{picture}%
}
\caption{Obtaining a recurrence for $\overline{K}^{(8)}$-type regions with $a> b$. Kuo condensation is applied to the region $\overline{K}^{(8)}_{3,2,2}(3,2 ;\ 2,1 ;\ 2,2)$
(picture (c)) as shown on the picture (d).}\label{fig:kuooff8}.
\end{figure}

\subsection{Two extremal cases for off-central regions}\label{sec:extremcase}

This subsection is devoted to two special cases when certain parameters in the off-central regions achieve their minimal values. Our first special case is when some of triangles in our ferns have side-length $0$.

\begin{figure}\centering
\setlength{\unitlength}{3947sp}%
\begingroup\makeatletter\ifx\SetFigFont\undefined%
\gdef\SetFigFont#1#2#3#4#5{%
  \reset@font\fontsize{#1}{#2pt}%
  \fontfamily{#3}\fontseries{#4}\fontshape{#5}%
  \selectfont}%
\fi\endgroup%
\resizebox{15cm}{!}{
\begin{picture}(0,0)%
\includegraphics{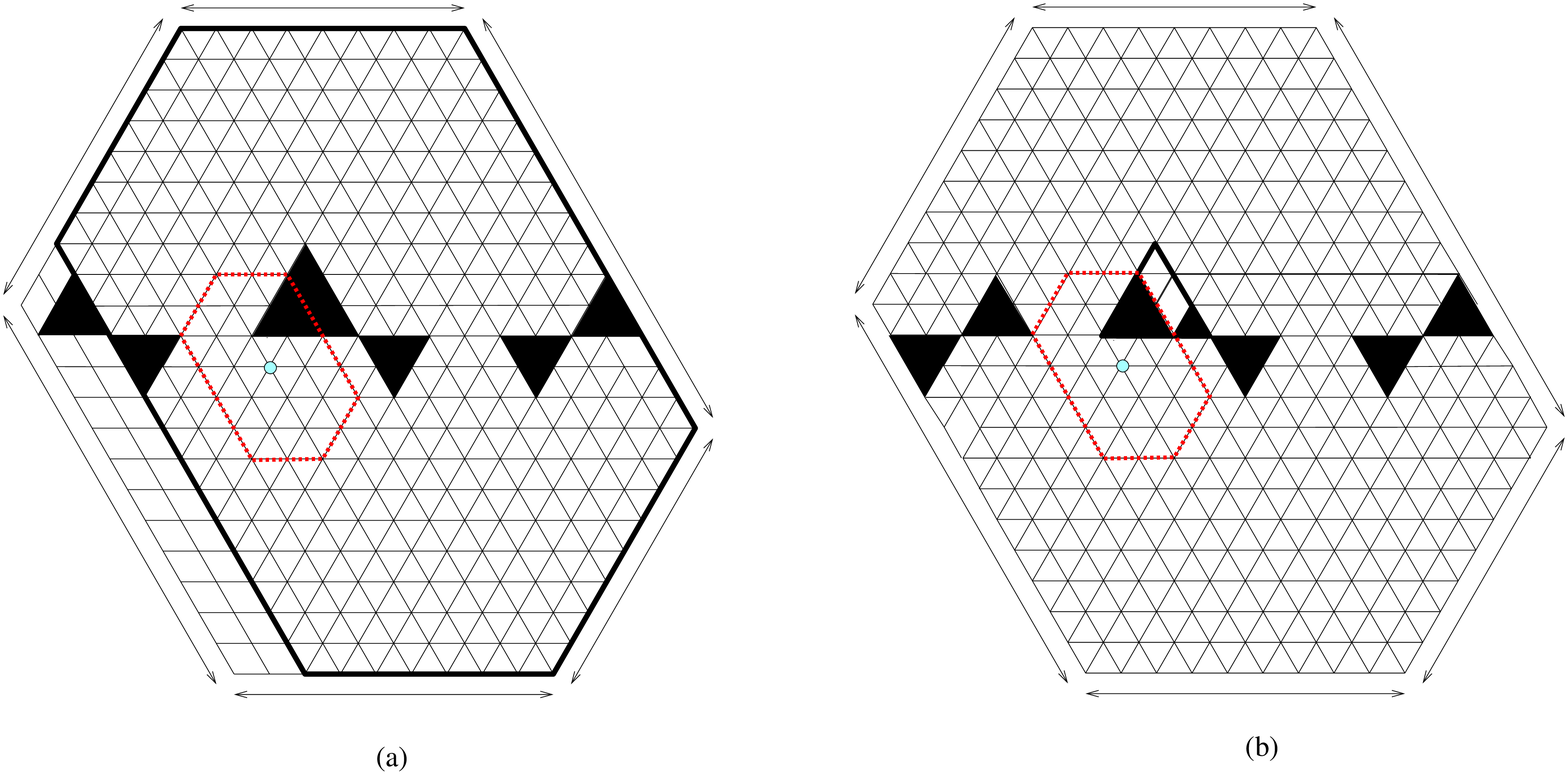}%
\end{picture}%
%
%

\begin{picture}(18044,9252)(1830,-10258)
\put(5320,-9673){\makebox(0,0)[lb]{\smash{{\SetFigFont{14}{16.8}{\familydefault}{\mddefault}{\updefault}{\color[rgb]{0,0,0}$x+o_a+e_b+e_c$}%
}}}}
\put(1944,-5905){\rotatebox{300.0}{\makebox(0,0)[lb]{\smash{{\SetFigFont{14}{16.8}{\familydefault}{\mddefault}{\updefault}{\color[rgb]{0,0,0}$2y+z+e_a+o_b+o_c+|a-b|$}%
}}}}}
\put(2091,-3911){\rotatebox{60.0}{\makebox(0,0)[lb]{\smash{{\SetFigFont{14}{16.8}{\familydefault}{\mddefault}{\updefault}{\color[rgb]{0,0,0}$z+o_a+e_b+e_c$}%
}}}}}
\put(14167,-1287){\makebox(0,0)[lb]{\smash{{\SetFigFont{14}{16.8}{\familydefault}{\mddefault}{\updefault}{\color[rgb]{0,0,0}$x+e_a+o_b+o_c$}%
}}}}
\put(15195,-9673){\makebox(0,0)[lb]{\smash{{\SetFigFont{14}{16.8}{\familydefault}{\mddefault}{\updefault}{\color[rgb]{0,0,0}$x+o_a+e_b+e_c$}%
}}}}
\put(11633,-5735){\rotatebox{300.0}{\makebox(0,0)[lb]{\smash{{\SetFigFont{14}{16.8}{\familydefault}{\mddefault}{\updefault}{\color[rgb]{0,0,0}$2y+z+e_a+o_b+o_c+|a-b|$}%
}}}}}
\put(17775,-2081){\rotatebox{300.0}{\makebox(0,0)[lb]{\smash{{\SetFigFont{14}{16.8}{\familydefault}{\mddefault}{\updefault}{\color[rgb]{0,0,0}$2y+z+o_a+e_b+e_c+|a-b|$}%
}}}}}
\put(11906,-3829){\rotatebox{60.0}{\makebox(0,0)[lb]{\smash{{\SetFigFont{14}{16.8}{\familydefault}{\mddefault}{\updefault}{\color[rgb]{0,0,0}$z+o_a+e_b+e_c$}%
}}}}}
\put(18757,-8852){\rotatebox{60.0}{\makebox(0,0)[lb]{\smash{{\SetFigFont{14}{16.8}{\familydefault}{\mddefault}{\updefault}{\color[rgb]{0,0,0}$z+e_a+o_b+o_c$}%
}}}}}
\put(2491,-5086){\makebox(0,0)[lb]{\smash{{\SetFigFont{14}{16.8}{\rmdefault}{\mddefault}{\updefault}{\color[rgb]{1,1,1}$a_2$}%
}}}}
\put(3281,-5508){\makebox(0,0)[lb]{\smash{{\SetFigFont{14}{16.8}{\rmdefault}{\mddefault}{\updefault}{\color[rgb]{1,1,1}$a_3$}%
}}}}
\put(5057,-5023){\makebox(0,0)[lb]{\smash{{\SetFigFont{14}{16.8}{\rmdefault}{\mddefault}{\updefault}{\color[rgb]{1,1,1}$c_1$}%
}}}}
\put(6167,-5463){\makebox(0,0)[lb]{\smash{{\SetFigFont{14}{16.8}{\rmdefault}{\mddefault}{\updefault}{\color[rgb]{1,1,1}$c_2$}%
}}}}
\put(8657,-5043){\makebox(0,0)[lb]{\smash{{\SetFigFont{14}{16.8}{\rmdefault}{\mddefault}{\updefault}{\color[rgb]{1,1,1}$b_1$}%
}}}}
\put(7787,-5473){\makebox(0,0)[lb]{\smash{{\SetFigFont{14}{16.8}{\rmdefault}{\mddefault}{\updefault}{\color[rgb]{1,1,1}$b_2$}%
}}}}
\put(12422,-5446){\makebox(0,0)[lb]{\smash{{\SetFigFont{14}{16.8}{\rmdefault}{\mddefault}{\updefault}{\color[rgb]{1,1,1}$a_1$}%
}}}}
\put(13154,-4993){\makebox(0,0)[lb]{\smash{{\SetFigFont{14}{16.8}{\rmdefault}{\mddefault}{\updefault}{\color[rgb]{1,1,1}$a_2$}%
}}}}
\put(14658,-5086){\makebox(0,0)[lb]{\smash{{\SetFigFont{14}{16.8}{\rmdefault}{\mddefault}{\updefault}{\color[rgb]{1,1,1}$c_1$}%
}}}}
\put(15378,-5156){\makebox(0,0)[lb]{\smash{{\SetFigFont{14}{16.8}{\rmdefault}{\mddefault}{\updefault}{\color[rgb]{1,1,1}$c_3$}%
}}}}
\put(15918,-5476){\makebox(0,0)[lb]{\smash{{\SetFigFont{14}{16.8}{\rmdefault}{\mddefault}{\updefault}{\color[rgb]{1,1,1}$c_4$}%
}}}}
\put(17673,-5476){\makebox(0,0)[lb]{\smash{{\SetFigFont{14}{16.8}{\rmdefault}{\mddefault}{\updefault}{\color[rgb]{1,1,1}$b_2$}%
}}}}
\put(18395,-5069){\makebox(0,0)[lb]{\smash{{\SetFigFont{14}{16.8}{\rmdefault}{\mddefault}{\updefault}{\color[rgb]{1,1,1}$b_1$}%
}}}}
\put(4604,-1311){\makebox(0,0)[lb]{\smash{{\SetFigFont{14}{16.8}{\familydefault}{\mddefault}{\updefault}{\color[rgb]{0,0,0}$x+e_a+o_b+o_c$}%
}}}}
\put(7861,-2031){\rotatebox{300.0}{\makebox(0,0)[lb]{\smash{{\SetFigFont{14}{16.8}{\familydefault}{\mddefault}{\updefault}{\color[rgb]{0,0,0}$2y+z+o_a+e_b+e_c+|a-b|$}%
}}}}}
\put(8911,-8781){\rotatebox{60.0}{\makebox(0,0)[lb]{\smash{{\SetFigFont{14}{16.8}{\familydefault}{\mddefault}{\updefault}{\color[rgb]{0,0,0}$z+e_a+o_b+o_c$}%
}}}}}
\end{picture}}
\caption{Eliminating triangles of side-length $0$ from the ferns.}\label{Specialoff}
\end{figure}

\begin{lem}\label{lem3}
For any off-central regions, we can find a new region of the same type (1) whose number of tilings is the same,
(2) whose $h$-parameter is at most that of the original region, (3) whose left and right ferns consist of triangles with positive side-lengths, and (4) whose middle fern contain
at most one possible triangle of side length $0$ at the beginning.
\end{lem}
\begin{proof}
The proof is essentially the same as that of Lemma 3.5 in \cite{HoleDent}, however in order to make our paper more self-contained, we still present the full proof here.
We only consider the $R:=E^{(2)}_{x,y,z}(\textbf{a}; \textbf{c}; \textbf{b})$ region, the other regions can be treated similarly.

We will show how to eliminate $0$-triangles from the three ferns without changing the tiling number or increasing the $h$-parameter. We consider the following three $0$-eliminating procedures for the left fern:

(1) If $a_1=a_2=\dotsc=a_{2i}=0$, for some $i\geq 1$, we can simply truncate the first $2i$ zero terms in the sequence $\textbf{a}$. The new
region is `exactly' the old one, however, strictly speaking, it has less $0$-triangles in the left fern.

(2) If $a_1=0$ and $a_2>0$, then we can remove forced lozenges along the southwest side of the region $R$ and obtain the region
 $E^{(2)}_{x,y,z}(a_3,\dots,a_m;\ \textbf{c};\ \textbf{b})$ (see Figure \ref{Specialoff}(a)).  The new region has the same number of tilings as the original one, the
 $h$-parameter $a_1$-unit less than $h$, and less $0$-triangles in the left fern.

(3) If $a_i=0$, for some $i>1$, then we can eliminate this $0$-triangle by combining the $(i-1)$-th and the $(i+1)$-th triangles in the fern (as shown
in Figure \ref{Specialoff}(b)).

Repeating these three procedures if needed, one can eliminate all $0$-triangles from the left fern. Working similarly for the right fern, we obtain a region with no $0$-triangle in the left and right ferns. For the middle fern, we apply the procedure (3) to eliminate all $0$-triangles, except for a possible $0$-triangle at the beginning. This finishes our proof.
\end{proof}

In the next special case, we deal with off-central regions whose $y$-parameter achieves it minimal value.

\begin{lem}\label{lem4}
For any off-central region $R$ with the $y$-parameter minimal,  we can find another off-central region whose number of tilings is the same
 as that of $R$ and whose $h$-parameter is strictly smaller than that of $R$.
 \end{lem}

\begin{proof}
By Lemma \ref{lem3}, we can assume that $a_i$'s, $b_j$'s, $c_t$'s are all positive, for $i\geq 1$, $j\geq 1$, $t\geq 2$.

Recall that if $R$ has the leftmost of its middle fern $d$ unit above (resp., below) the center of the auxiliary hexagon $H_0$, for $d=0,\frac{1}{2}, 1, \frac{3}{2}$, then
$y \geq \max(a-b,-2d)$ (resp., $y \geq \max(b-a,-2d)$). We call these regions \emph{upper regions}  (resp., \emph{lower regions}). We note that
all $\overline{E}^{(i)}$'s, $\overline{F}^{(i)}$'s, $\overline{G}^{(i)}$'s, and $\overline{K}^{(i)}$'s regions are upper ones. There  are
six cases to distinguish:


\begin{figure}\centering
\end{figure}

\medskip

\emph{Case 1a: $R$ is an upper region, and $a\geq b$.}

We consider the case $R$ has the leftmost of the middle fern $d$-unit above the center of the auxiliary hexagon, for  $d=1/2,1,3/2$. If $a\geq b$, then $y\geq 0$, we can find a `bar' or `un-bar' counterpart of the region that has the same number of tilings and has
$h$-parameter  strictly smaller. In particular, by considering forced lozenges, we get
\begin{align}
 \M(E^{(i)}_{x,0,z}(\textbf{a}; \textbf{c}; \textbf{b}))&=\M(\overline{E}^{(i)}_{x,\min(a_1, a-b),z}(a_2,\dotsc,a_m;\textbf{c}; \textbf{b}))\\
 \M(F^{(j)}_{x,0,z}(\textbf{a}; \textbf{c}; \textbf{b}))&=\M(\overline{F}^{(j)}_{x,\min(a_1, a-b),z}(a_2,\dotsc,a_m;\textbf{c}; \textbf{b}))\\
 \M(G^{(k)}_{x,0,z}(\textbf{a}; \textbf{c}; \textbf{b}))&=\M(\overline{G}^{(k)}_{x,\min(a_1, a-b),z}(a_2,\dotsc,a_m;\textbf{c}; \textbf{b}))\\
 \M(K^{(l)}_{x,0,z}(\textbf{a}; \textbf{c}; \textbf{b}))&=\M(\overline{K}^{(l)}_{x,\min(a_1, a-b),z}(a_2,\dotsc,a_m;\textbf{c}; \textbf{b}))
 \end{align}
for $i=1,2,3,4$, $j=3,4,5,6$, $k=2,3,4,5$, and $l=5,6,7,8$ (see Figure \ref{fig:Specialoff1} for examples). Similarly, we get
\begin{align}
 \M(\overline{E}^{(i)}_{x,0,z}(\textbf{a}; \textbf{c}; \textbf{b}))&=\M(E^{(i)}_{x,\min(a_1, a-b),z}(a_2,\dotsc,a_m;\textbf{c}; \textbf{b}))\\
 \M(\overline{F}^{(j)}_{x,0,z}(\textbf{a}; \textbf{c}; \textbf{b}))&=\M(F^{(j)}_{x,\min(a_1, a-b),z}(a_2,\dotsc,a_m;\textbf{c}; \textbf{b}))\\
 \M(\overline{G}^{(k)}_{x,0,z}(\textbf{a}; \textbf{c}; \textbf{b}))&=\M(G^{(k)}_{x,\min(a_1, a-b),z}(a_2,\dotsc,a_m;\textbf{c}; \textbf{b}))\\
 \M(\overline{K}^{(l)}_{x,0,z}(\textbf{a}; \textbf{c}; \textbf{b}))&=\M(K^{(l)}_{x,\min(a_1, a-b),z}(a_2,\dotsc,a_m;\textbf{c}; \textbf{b}))
 \end{align}
for $i=1,2,3,4$, $j=3,4,5,6$, $k=2,3,4,5$, and $l=5,6,7,8$ (see Figure \ref{fig:Specialoff2} for examples).

\begin{figure}\centering
\setlength{\unitlength}{3947sp}%
\begingroup\makeatletter\ifx\SetFigFont\undefined%
\gdef\SetFigFont#1#2#3#4#5{%
  \reset@font\fontsize{#1}{#2pt}%
  \fontfamily{#3}\fontseries{#4}\fontshape{#5}%
  \selectfont}%
\fi\endgroup%
\resizebox{15cm}{!}{
\begin{picture}(0,0)%
\includegraphics{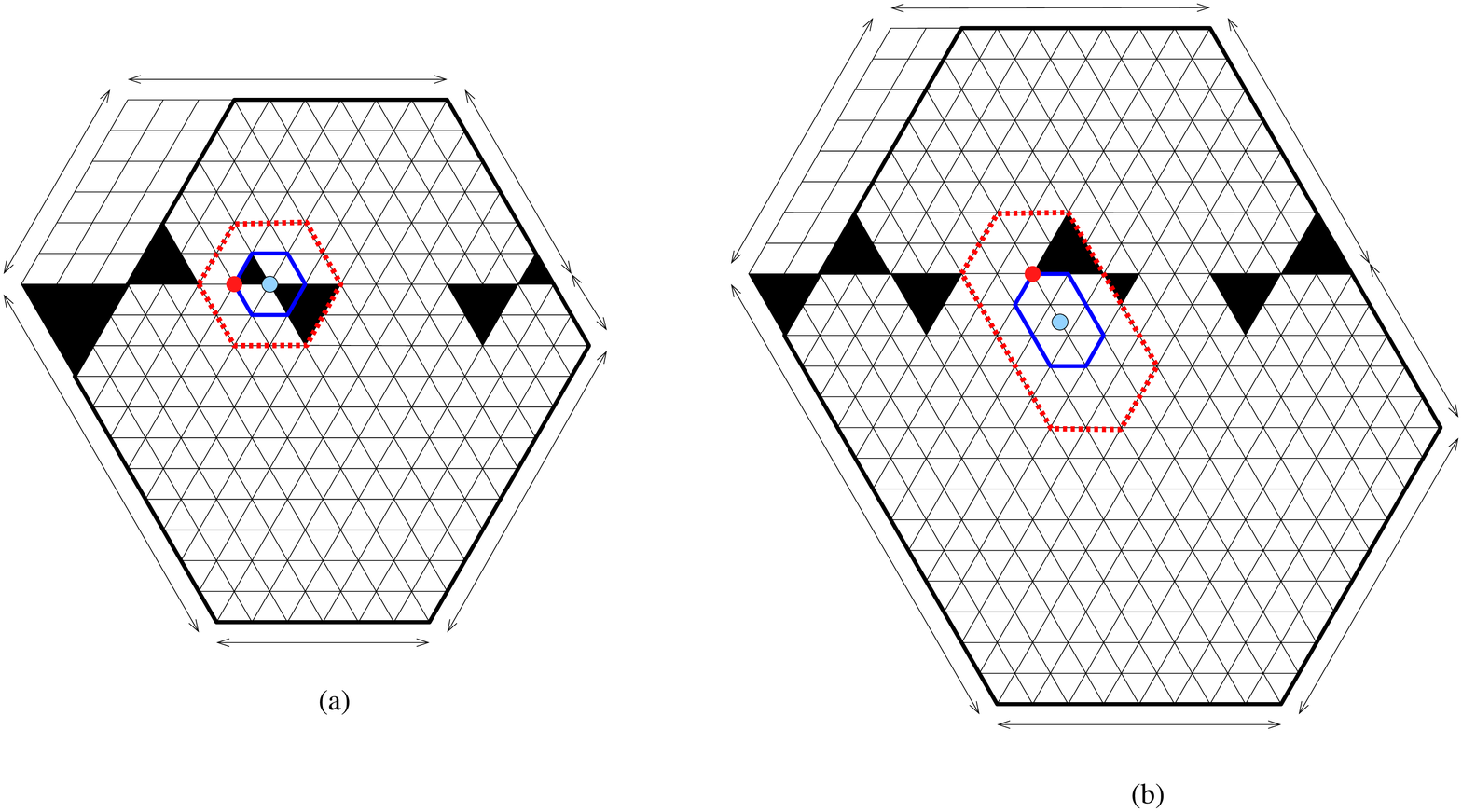}%
\end{picture}%
%
%

\begin{picture}(17137,9724)(2035,-11874)
\put(4641,-3251){\makebox(0,0)[lb]{\smash{{\SetFigFont{14}{16.8}{\rmdefault}{\mddefault}{\updefault}{\color[rgb]{0,0,0}$x+o_a+e_b+e_c$}%
}}}}
\put(7843,-3701){\rotatebox{300.0}{\makebox(0,0)[lb]{\smash{{\SetFigFont{14}{16.8}{\rmdefault}{\mddefault}{\updefault}{\color[rgb]{0,0,0}$z+e_a+o_b+o_c$}%
}}}}}
\put(9061,-5751){\rotatebox{300.0}{\makebox(0,0)[lb]{\smash{{\SetFigFont{14}{16.8}{\rmdefault}{\mddefault}{\updefault}{\color[rgb]{0,0,0}$a-b$}%
}}}}}
\put(7911,-9271){\rotatebox{60.0}{\makebox(0,0)[lb]{\smash{{\SetFigFont{14}{16.8}{\rmdefault}{\mddefault}{\updefault}{\color[rgb]{0,0,0}$z+o_a+e_b+e_c$}%
}}}}}
\put(5001,-10221){\makebox(0,0)[lb]{\smash{{\SetFigFont{14}{16.8}{\rmdefault}{\mddefault}{\updefault}{\color[rgb]{0,0,0}$x+e_a+o_b+o_c$}%
}}}}
\put(2451,-6911){\rotatebox{300.0}{\makebox(0,0)[lb]{\smash{{\SetFigFont{14}{16.8}{\rmdefault}{\mddefault}{\updefault}{\color[rgb]{0,0,0}$z+o_a+e_b+e_c+a-b$}%
}}}}}
\put(2281,-5261){\rotatebox{60.0}{\makebox(0,0)[lb]{\smash{{\SetFigFont{14}{16.8}{\rmdefault}{\mddefault}{\updefault}{\color[rgb]{0,0,0}$z+e_a+o_b+o_c$}%
}}}}}
\put(13161,-2431){\makebox(0,0)[lb]{\smash{{\SetFigFont{14}{16.8}{\rmdefault}{\mddefault}{\updefault}{\color[rgb]{0,0,0}$x+o_a+e_b+e_c$}%
}}}}
\put(14216,-11271){\makebox(0,0)[lb]{\smash{{\SetFigFont{14}{16.8}{\rmdefault}{\mddefault}{\updefault}{\color[rgb]{0,0,0}$x+e_a+o_b+o_c$}%
}}}}
\put(17889,-10229){\rotatebox{60.0}{\makebox(0,0)[lb]{\smash{{\SetFigFont{14}{16.8}{\rmdefault}{\mddefault}{\updefault}{\color[rgb]{0,0,0}$z+o_a+e_b+e_c$}%
}}}}}
\put(16846,-3050){\rotatebox{300.0}{\makebox(0,0)[lb]{\smash{{\SetFigFont{14}{16.8}{\rmdefault}{\mddefault}{\updefault}{\color[rgb]{0,0,0}$z+e_a+o_b+o_c$}%
}}}}}
\put(10627,-5063){\rotatebox{60.0}{\makebox(0,0)[lb]{\smash{{\SetFigFont{14}{16.8}{\rmdefault}{\mddefault}{\updefault}{\color[rgb]{0,0,0}$z+e_a+o_b+o_c$}%
}}}}}
\put(10926,-7035){\rotatebox{300.0}{\makebox(0,0)[lb]{\smash{{\SetFigFont{14}{16.8}{\rmdefault}{\mddefault}{\updefault}{\color[rgb]{0,0,0}$z+o_a+e_b+e_c+a-b+3$}%
}}}}}
\put(18481,-6051){\rotatebox{300.0}{\makebox(0,0)[lb]{\smash{{\SetFigFont{14}{16.8}{\rmdefault}{\mddefault}{\updefault}{\color[rgb]{0,0,0}$a-b+3$}%
}}}}}
\put(17201,-5581){\makebox(0,0)[lb]{\smash{{\SetFigFont{14}{16.8}{\rmdefault}{\mddefault}{\updefault}{\color[rgb]{1,1,1}$b_1$}%
}}}}
\put(12781,-5931){\makebox(0,0)[lb]{\smash{{\SetFigFont{14}{16.8}{\rmdefault}{\mddefault}{\updefault}{\color[rgb]{1,1,1}$a_3$}%
}}}}
\put(16411,-5951){\makebox(0,0)[lb]{\smash{{\SetFigFont{14}{16.8}{\rmdefault}{\mddefault}{\updefault}{\color[rgb]{1,1,1}$b_2$}%
}}}}
\put(2991,-6271){\makebox(0,0)[lb]{\smash{{\SetFigFont{14}{16.8}{\rmdefault}{\mddefault}{\updefault}{\color[rgb]{1,1,1}$a_1$}%
}}}}
\put(3881,-5651){\makebox(0,0)[lb]{\smash{{\SetFigFont{14}{16.8}{\rmdefault}{\mddefault}{\updefault}{\color[rgb]{1,1,1}$a_2$}%
}}}}
\put(8261,-5731){\makebox(0,0)[lb]{\smash{{\SetFigFont{14}{16.8}{\rmdefault}{\mddefault}{\updefault}{\color[rgb]{1,1,1}$b_1$}%
}}}}
\put(7561,-6121){\makebox(0,0)[lb]{\smash{{\SetFigFont{14}{16.8}{\rmdefault}{\mddefault}{\updefault}{\color[rgb]{1,1,1}$b_2$}%
}}}}
\put(11136,-5991){\makebox(0,0)[lb]{\smash{{\SetFigFont{14}{16.8}{\rmdefault}{\mddefault}{\updefault}{\color[rgb]{1,1,1}$a_1$}%
}}}}
\put(11856,-5541){\makebox(0,0)[lb]{\smash{{\SetFigFont{14}{16.8}{\rmdefault}{\mddefault}{\updefault}{\color[rgb]{1,1,1}$a_2$}%
}}}}
\end{picture}%
}
\caption{(a) Obtaining  an $\overline{E}^{(1)}$-type region from the region $E^{(1)}_{x,0,z}(\textbf{a};\textbf{c};\textbf{b})$ with $a\geq b$. (b) Obtaining a $\overline{G}^{(3)}$-type region from the region  $G^{(3)}_{x,0,z}(\textbf{a};\textbf{c};\textbf{b})$ with $a\geq b$.}\label{fig:Specialoff1}
\end{figure}

\begin{figure}\centering
\setlength{\unitlength}{3947sp}%
\begingroup\makeatletter\ifx\SetFigFont\undefined%
\gdef\SetFigFont#1#2#3#4#5{%
  \reset@font\fontsize{#1}{#2pt}%
  \fontfamily{#3}\fontseries{#4}\fontshape{#5}%
  \selectfont}%
\fi\endgroup%
\resizebox{15cm}{!}{
\begin{picture}(0,0)%
\includegraphics{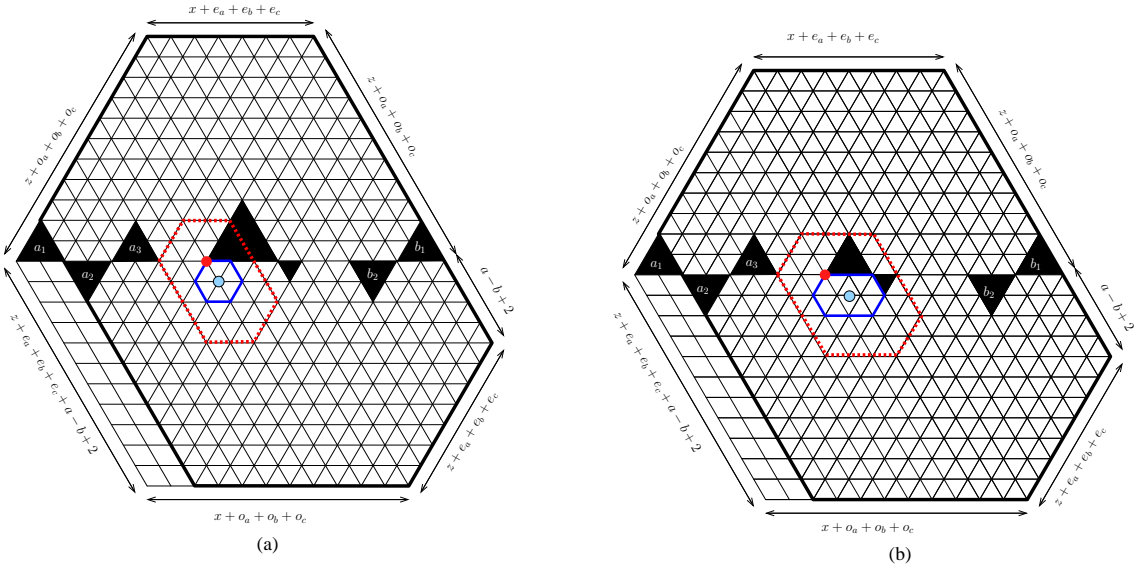}%
\end{picture}%
%
%

\begin{picture}(19555,9890)(2444,-9774)
\put(5627,-165){\makebox(0,0)[lb]{\smash{{\SetFigFont{14}{16.8}{\familydefault}{\mddefault}{\updefault}{\color[rgb]{0,0,0}$x+e_a+e_b+e_c$}%
}}}}
\put(8695,-1299){\rotatebox{300.0}{\makebox(0,0)[lb]{\smash{{\SetFigFont{14}{16.8}{\familydefault}{\mddefault}{\updefault}{\color[rgb]{0,0,0}$z+o_a+o_b+o_c$}%
}}}}}
\put(10639,-4606){\rotatebox{300.0}{\makebox(0,0)[lb]{\smash{{\SetFigFont{14}{16.8}{\familydefault}{\mddefault}{\updefault}{\color[rgb]{0,0,0}$a-b+2$}%
}}}}}
\put(10148,-8138){\rotatebox{60.0}{\makebox(0,0)[lb]{\smash{{\SetFigFont{14}{16.8}{\familydefault}{\mddefault}{\updefault}{\color[rgb]{0,0,0}$z+e_a+e_b+e_c$}%
}}}}}
\put(6061,-8971){\makebox(0,0)[lb]{\smash{{\SetFigFont{14}{16.8}{\familydefault}{\mddefault}{\updefault}{\color[rgb]{0,0,0}$x+o_a+o_b+o_c$}%
}}}}
\put(2558,-5433){\rotatebox{300.0}{\makebox(0,0)[lb]{\smash{{\SetFigFont{14}{16.8}{\familydefault}{\mddefault}{\updefault}{\color[rgb]{0,0,0}$z+e_a+e_b+e_c+a-b+2$}%
}}}}}
\put(2885,-3177){\rotatebox{60.0}{\makebox(0,0)[lb]{\smash{{\SetFigFont{14}{16.8}{\familydefault}{\mddefault}{\updefault}{\color[rgb]{0,0,0}$z+o_a+o_b+o_c$}%
}}}}}
\put(15921,-648){\makebox(0,0)[lb]{\smash{{\SetFigFont{14}{16.8}{\familydefault}{\mddefault}{\updefault}{\color[rgb]{0,0,0}$x+e_a+e_b+e_c$}%
}}}}
\put(19551,-1863){\rotatebox{300.0}{\makebox(0,0)[lb]{\smash{{\SetFigFont{14}{16.8}{\familydefault}{\mddefault}{\updefault}{\color[rgb]{0,0,0}$z+o_a+o_b+o_c$}%
}}}}}
\put(21308,-4907){\rotatebox{300.0}{\makebox(0,0)[lb]{\smash{{\SetFigFont{14}{16.8}{\familydefault}{\mddefault}{\updefault}{\color[rgb]{0,0,0}$a-b+2$}%
}}}}}
\put(20605,-8732){\rotatebox{60.0}{\makebox(0,0)[lb]{\smash{{\SetFigFont{14}{16.8}{\familydefault}{\mddefault}{\updefault}{\color[rgb]{0,0,0}$z+e_a+e_b+e_c$}%
}}}}}
\put(16509,-9152){\makebox(0,0)[lb]{\smash{{\SetFigFont{14}{16.8}{\familydefault}{\mddefault}{\updefault}{\color[rgb]{0,0,0}$x+o_a+o_b+o_c$}%
}}}}
\put(12932,-5398){\rotatebox{300.0}{\makebox(0,0)[lb]{\smash{{\SetFigFont{14}{16.8}{\familydefault}{\mddefault}{\updefault}{\color[rgb]{0,0,0}$z+e_a+e_b+e_c+a-b+2$}%
}}}}}
\put(13331,-3852){\rotatebox{60.0}{\makebox(0,0)[lb]{\smash{{\SetFigFont{14}{16.8}{\familydefault}{\mddefault}{\updefault}{\color[rgb]{0,0,0}$z+o_a+o_b+o_c$}%
}}}}}
\put(15206,-4561){\makebox(0,0)[lb]{\smash{{\SetFigFont{14}{16.8}{\familydefault}{\mddefault}{\updefault}{\color[rgb]{1,1,1}$a_3$}%
}}}}
\put(13566,-4581){\makebox(0,0)[lb]{\smash{{\SetFigFont{14}{16.8}{\familydefault}{\mddefault}{\updefault}{\color[rgb]{1,1,1}$a_1$}%
}}}}
\put(2967,-4299){\makebox(0,0)[lb]{\smash{{\SetFigFont{14}{16.8}{\familydefault}{\mddefault}{\updefault}{\color[rgb]{1,1,1}$a_1$}%
}}}}
\put(3786,-4771){\makebox(0,0)[lb]{\smash{{\SetFigFont{14}{16.8}{\familydefault}{\mddefault}{\updefault}{\color[rgb]{1,1,1}$a_2$}%
}}}}
\put(4604,-4299){\makebox(0,0)[lb]{\smash{{\SetFigFont{14}{16.8}{\familydefault}{\mddefault}{\updefault}{\color[rgb]{1,1,1}$a_3$}%
}}}}
\put(9514,-4299){\makebox(0,0)[lb]{\smash{{\SetFigFont{14}{16.8}{\familydefault}{\mddefault}{\updefault}{\color[rgb]{1,1,1}$b_1$}%
}}}}
\put(8695,-4771){\makebox(0,0)[lb]{\smash{{\SetFigFont{14}{16.8}{\familydefault}{\mddefault}{\updefault}{\color[rgb]{1,1,1}$b_2$}%
}}}}
\put(19298,-5077){\makebox(0,0)[lb]{\smash{{\SetFigFont{14}{16.8}{\familydefault}{\mddefault}{\updefault}{\color[rgb]{1,1,1}$b_2$}%
}}}}
\put(20101,-4591){\makebox(0,0)[lb]{\smash{{\SetFigFont{14}{16.8}{\familydefault}{\mddefault}{\updefault}{\color[rgb]{1,1,1}$b_1$}%
}}}}
\put(14366,-5001){\makebox(0,0)[lb]{\smash{{\SetFigFont{14}{16.8}{\familydefault}{\mddefault}{\updefault}{\color[rgb]{1,1,1}$a_2$}%
}}}}
\end{picture}%
}
\caption{(a) Obtaining an $E^{(2)}$-type region from the region $\overline{E}^{(2)}_{x,0,z}(\textbf{a};\textbf{c};\textbf{b})$ with $a\geq b$.
(b) Obtaining an $E^{(4)}$-type region from the region $\overline{E}^{(4)}_{x,0,z}(\textbf{a};\textbf{c};\textbf{b})$ with $a\geq b$.}\label{fig:Specialoff2}
\end{figure}

\begin{figure}\centering
\setlength{\unitlength}{3947sp}%
\begingroup\makeatletter\ifx\SetFigFont\undefined%
\gdef\SetFigFont#1#2#3#4#5{%
  \reset@font\fontsize{#1}{#2pt}%
  \fontfamily{#3}\fontseries{#4}\fontshape{#5}%
  \selectfont}%
\fi\endgroup%
\resizebox{15cm}{!}{
\begin{picture}(0,0)%
\includegraphics{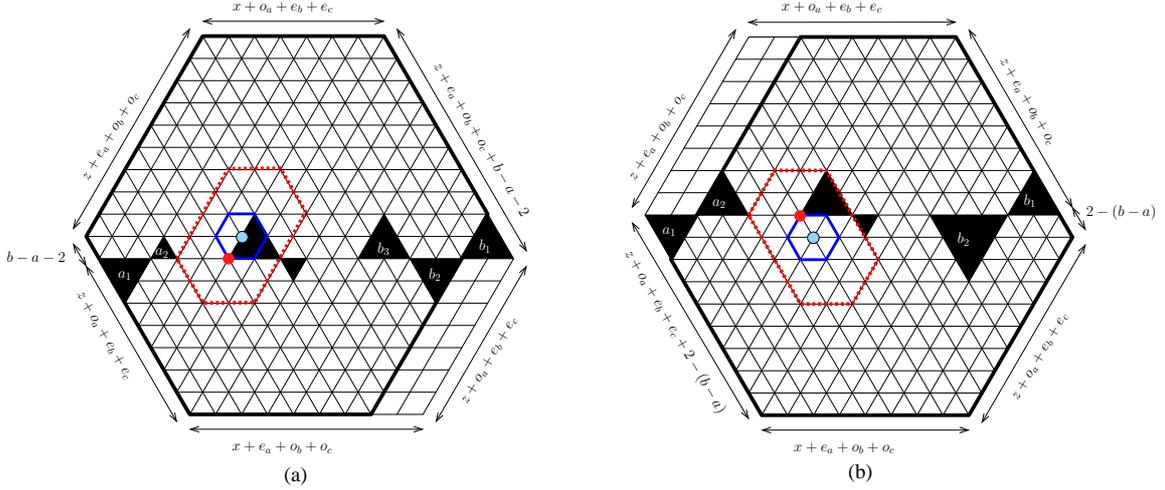}%
\end{picture}%
%
%

\begin{picture}(17951,7909)(751,-10487)
\put(11906,-5981){\makebox(0,0)[lb]{\smash{{\SetFigFont{14}{16.8}{\familydefault}{\mddefault}{\updefault}{\color[rgb]{1,1,1}$a_2$}%
}}}}
\put(11116,-6431){\makebox(0,0)[lb]{\smash{{\SetFigFont{14}{16.8}{\familydefault}{\mddefault}{\updefault}{\color[rgb]{1,1,1}$a_1$}%
}}}}
\put(4331,-2870){\makebox(0,0)[lb]{\smash{{\SetFigFont{14}{16.8}{\familydefault}{\mddefault}{\updefault}{\color[rgb]{0,0,0}$x+o_a+e_b+e_c$}%
}}}}
\put(7503,-3750){\rotatebox{300.0}{\makebox(0,0)[lb]{\smash{{\SetFigFont{14}{16.8}{\familydefault}{\mddefault}{\updefault}{\color[rgb]{0,0,0}$z+e_a+o_b+o_c+b-a-2$}%
}}}}}
\put(8031,-9150){\rotatebox{60.0}{\makebox(0,0)[lb]{\smash{{\SetFigFont{14}{16.8}{\familydefault}{\mddefault}{\updefault}{\color[rgb]{0,0,0}$z+o_a+e_b+e_c$}%
}}}}}
\put(4321,-9880){\makebox(0,0)[lb]{\smash{{\SetFigFont{14}{16.8}{\familydefault}{\mddefault}{\updefault}{\color[rgb]{0,0,0}$x+e_a+o_b+o_c$}%
}}}}
\put(1871,-7390){\rotatebox{300.0}{\makebox(0,0)[lb]{\smash{{\SetFigFont{14}{16.8}{\familydefault}{\mddefault}{\updefault}{\color[rgb]{0,0,0}$z+o_a+e_b+e_c$}%
}}}}}
\put(766,-6879){\makebox(0,0)[lb]{\smash{{\SetFigFont{14}{16.8}{\familydefault}{\mddefault}{\updefault}{\color[rgb]{0,0,0}$b-a-2$}%
}}}}
\put(2051,-5580){\rotatebox{60.0}{\makebox(0,0)[lb]{\smash{{\SetFigFont{14}{16.8}{\familydefault}{\mddefault}{\updefault}{\color[rgb]{0,0,0}$z+e_a+o_b+o_c$}%
}}}}}
\put(13001,-2871){\makebox(0,0)[lb]{\smash{{\SetFigFont{14}{16.8}{\familydefault}{\mddefault}{\updefault}{\color[rgb]{0,0,0}$x+o_a+e_b+e_c$}%
}}}}
\put(16443,-3661){\rotatebox{300.0}{\makebox(0,0)[lb]{\smash{{\SetFigFont{14}{16.8}{\familydefault}{\mddefault}{\updefault}{\color[rgb]{0,0,0}$z+e_a+o_b+o_c$}%
}}}}}
\put(17811,-6151){\makebox(0,0)[lb]{\smash{{\SetFigFont{14}{16.8}{\familydefault}{\mddefault}{\updefault}{\color[rgb]{0,0,0}$2-(b-a)$}%
}}}}
\put(16729,-9089){\rotatebox{60.0}{\makebox(0,0)[lb]{\smash{{\SetFigFont{14}{16.8}{\familydefault}{\mddefault}{\updefault}{\color[rgb]{0,0,0}$z+o_a+e_b+e_c$}%
}}}}}
\put(10523,-6731){\rotatebox{300.0}{\makebox(0,0)[lb]{\smash{{\SetFigFont{14}{16.8}{\familydefault}{\mddefault}{\updefault}{\color[rgb]{0,0,0}$z+o_a+e_b+e_c+2-(b-a)$}%
}}}}}
\put(10687,-5583){\rotatebox{60.0}{\makebox(0,0)[lb]{\smash{{\SetFigFont{14}{16.8}{\familydefault}{\mddefault}{\updefault}{\color[rgb]{0,0,0}$z+e_a+o_b+o_c$}%
}}}}}
\put(13196,-9891){\makebox(0,0)[lb]{\smash{{\SetFigFont{14}{16.8}{\familydefault}{\mddefault}{\updefault}{\color[rgb]{0,0,0}$x+e_a+o_b+o_c$}%
}}}}
\put(2521,-7130){\makebox(0,0)[lb]{\smash{{\SetFigFont{14}{16.8}{\familydefault}{\mddefault}{\updefault}{\color[rgb]{1,1,1}$a_1$}%
}}}}
\put(3111,-6770){\makebox(0,0)[lb]{\smash{{\SetFigFont{14}{16.8}{\familydefault}{\mddefault}{\updefault}{\color[rgb]{1,1,1}$a_2$}%
}}}}
\put(8221,-6690){\makebox(0,0)[lb]{\smash{{\SetFigFont{14}{16.8}{\familydefault}{\mddefault}{\updefault}{\color[rgb]{1,1,1}$b_1$}%
}}}}
\put(7431,-7140){\makebox(0,0)[lb]{\smash{{\SetFigFont{14}{16.8}{\familydefault}{\mddefault}{\updefault}{\color[rgb]{1,1,1}$b_2$}%
}}}}
\put(6621,-6730){\makebox(0,0)[lb]{\smash{{\SetFigFont{14}{16.8}{\familydefault}{\mddefault}{\updefault}{\color[rgb]{1,1,1}$b_3$}%
}}}}
\put(16826,-5971){\makebox(0,0)[lb]{\smash{{\SetFigFont{14}{16.8}{\familydefault}{\mddefault}{\updefault}{\color[rgb]{1,1,1}$b_1$}%
}}}}
\put(15776,-6581){\makebox(0,0)[lb]{\smash{{\SetFigFont{14}{16.8}{\familydefault}{\mddefault}{\updefault}{\color[rgb]{1,1,1}$b_2$}%
}}}}
\end{picture}%
}
\caption{ (a) Obtaining a (horizontally reflected) $\overline{E}^{(2)}$-type region
from the region $E^{(2)}_{x,-2,z}(\textbf{a};\textbf{c};\textbf{b})$ in the case $b-a\geq 2$. (b) Obtaining a $\overline{E}^{(2)}$-type region
from the region $E^{(2)}_{x,-2,z}(\textbf{a};\textbf{c};\textbf{b})$ in the case $0<b-a\leq 2$.}\label{fig:Specialoff4}
\end{figure}

\medskip

\emph{Case 1b: $R$ is an upper region, and $0<2d\leq  b-a$.}

\medskip

We have in this case $y\geq -2d$. If $y=-2d$, then
\begin{align}
  \M(E^{(i)}_{x,-2d,z}(\textbf{a}; \textbf{c}; \textbf{b}))&=\M(\overline{E}^{(i)}_{x,\min(b_1, b-a)-2d,z}(\textbf{a};\ {}^0\textbf{c};\  b_2,\dotsc,b_n))\\
  \M(F^{(j)}_{x,-2d,z}(\textbf{a}; \textbf{c}; \textbf{b}))&=\M(\overline{F}^{(j)}_{x,\min(b_1, b-a)-2d,z}(\textbf{a};\ {}^0\textbf{c};\  b_2,\dotsc,b_n))\\
  \M(G^{(k)}_{x,-2d,z}(\textbf{a}; \textbf{c}; \textbf{b}))&=\M(\overline{G}^{(k)}_{x,\min(b_1, b-a)-2d,z}(\textbf{a};\ {}^0\textbf{c};\ b_2,\dotsc,b_n))\\
  \M(K^{(l)}_{x,-2d,z}(\textbf{a}; \textbf{c}; \textbf{b}))&=\M(\overline{K}^{(l)}_{x,\min(b_1, b-a)-2d,z}(\textbf{a};\ {}^0\textbf{c};\ b_2,\dotsc,b_n)),
 \end{align}
for $i=2,3$, $j=4,5,6$, $k=2,3,4,5$, and $l=5,6,7,8$ (illustrated in Figure \ref{fig:Specialoff4}(a)). Recall that we use the notation ${}^0\textbf{c}$ for the sequence obtained from $\textbf{c}$ by including a new $0$ term in the front.
It is easy to see that the regions on the right-hand sides of the above equalities satisfy the conditions of the lemma.

\medskip

\emph{Case 1c: $R$ is an upper region, and $0< b-a\leq 2d$.}

\medskip

We have here $y\geq a-b$. If $y=a-b$, then
\begin{align}
  \M(E^{(i)}_{x,a-b,z}(\textbf{a}; \textbf{c}; \textbf{b}))&=\M(\overline{E}^{(i)}_{x,a-b,z}(a_2,\dotsc,a_m;\textbf{c};\textbf{b}))\\
  \M(F^{(j)}_{x,-2d,z}(\textbf{a}; \textbf{c}; \textbf{b}))&=\M(\overline{F}^{(j)}_{x,a-b,z}(a_2,\dotsc,a_m;\textbf{c};\textbf{b}))\\
  \M(G^{(k)}_{x,-2d,z}(\textbf{a}; \textbf{c}; \textbf{b}))&=\M(\overline{G}^{(k)}_{x,a-b,z}(a_2,\dotsc,a_m;\textbf{c};\textbf{b}))\\
  \M(K^{(l)}_{x,-2d,z}(\textbf{a}; \textbf{c}; \textbf{b}))&=\M(\overline{K}^{(l)}_{x,a-b,z}(a_2,\dotsc,a_m;\textbf{c};\textbf{b})),
 \end{align}
for $i=2,3$, $j=4,5,6$, $k=2,3,4,5$, and $l=5,6,7,8$  (see Figure \ref{fig:Specialoff4}(b) for an example). It is easy to see that the regions on the right-hand sides of the above equalities satisfy the conditions of the lemma.

\medskip

\emph{Case 2a: $R$ is a lower region and $a\leq b$.}

Then $y\geq 0$, and we get by removing forced lozenges:
\begin{align}
  \M(E^{(i)}_{x,0,z}(\textbf{a}; \textbf{c}; \textbf{b}))&=\M(\overline{E}^{(3+i \mod 6)}_{x,\min(b_1, b-a),z}(b_2,\dotsc,b_n;\overline{\textbf{c}}; \textbf{a}))\\
 \M(F^{(j)}_{x,0,z}(\textbf{a}; \textbf{c}; \textbf{b}))&=\M(\overline{F}^{(4+j \mod 8)}_{x,\min(b_1, b-a),z}(b_2,\dotsc,b_n;\overline{\textbf{c}}; \textbf{a}))\\
 \M(G^{(k)}_{x,0,z}(\textbf{a}; \textbf{c}; \textbf{b}))&=\M(\overline{G}^{(4+k \mod 8)}_{x,\min(b_1, b-a),z}(b_2,\dotsc,b_n;\overline{\textbf{c}}; \textbf{a}))\\
 \M(K^{(l)}_{x,0,z}(\textbf{a}; \textbf{c}; \textbf{b}))&=\M(\overline{K}^{(4+l \mod 8)}_{x,\min(b_1, b-a),z}(b_2,\dotsc,b_n;\overline{\textbf{c}}; \textbf{a})),
 \end{align}
for $i=1,4,5,6$, $j=1,2,3,7,8$, $k=1,6,7,8$, and $l=1,2,3,4$. Recall that $\overline{\textbf{c}}$ is the sequence obtained from $\textbf{c}$ by reverting the order of the terms if we have an even number of terms, otherwise, we revert the sequence and include a new $0$ in front of the resulting one. One readily sees that the regions on the right-hand sides of the above equalities satisfy the conditions of the lemma.

\begin{figure}\centering\setlength{\unitlength}{3947sp}%
\begingroup\makeatletter\ifx\SetFigFont\undefined%
\gdef\SetFigFont#1#2#3#4#5{%
  \reset@font\fontsize{#1}{#2pt}%
  \fontfamily{#3}\fontseries{#4}\fontshape{#5}%
  \selectfont}%
\fi\endgroup%
\resizebox{15cm}{!}{
\begin{picture}(0,0)%
\includegraphics{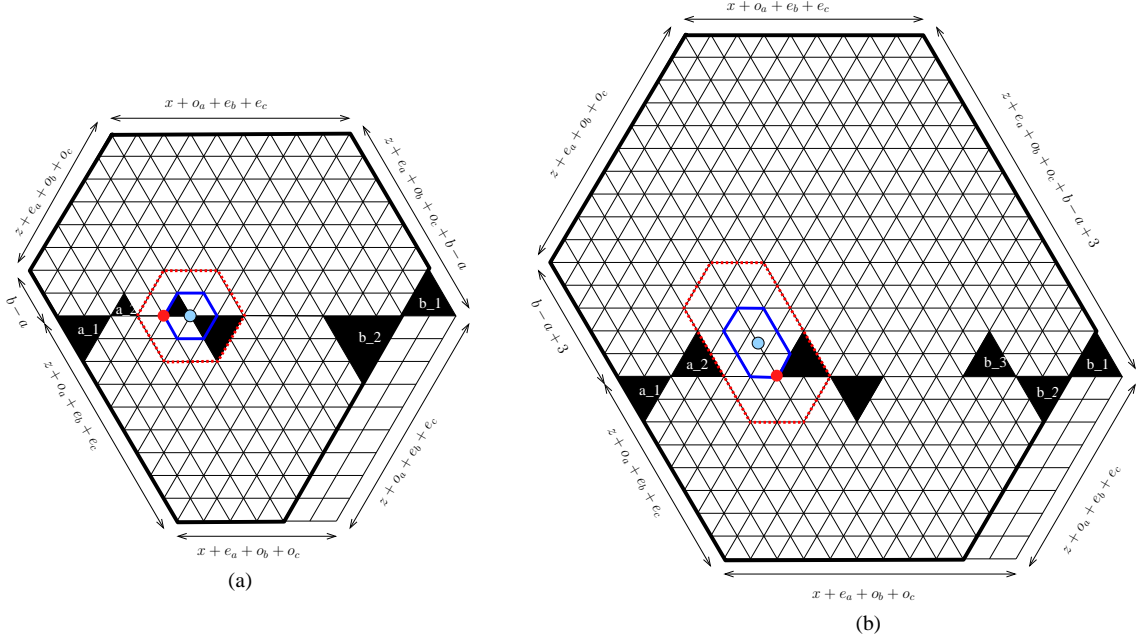}%
\end{picture}%
%
%

\begin{picture}(17517,10024)(1184,-10294)
\put(3660,-2050){\makebox(0,0)[lb]{\smash{{\SetFigFont{14}{16.8}{\familydefault}{\mddefault}{\updefault}{\color[rgb]{0,0,0}$x+o_a+e_b+e_c$}%
}}}}
\put(7042,-2550){\rotatebox{300.0}{\makebox(0,0)[lb]{\smash{{\SetFigFont{14}{16.8}{\familydefault}{\mddefault}{\updefault}{\color[rgb]{0,0,0}$z+e_a+o_b+o_c+b-a$}%
}}}}}
\put(7050,-8270){\rotatebox{60.0}{\makebox(0,0)[lb]{\smash{{\SetFigFont{14}{16.8}{\familydefault}{\mddefault}{\updefault}{\color[rgb]{0,0,0}$z+o_a+e_b+e_c$}%
}}}}}
\put(4158,-9023){\makebox(0,0)[lb]{\smash{{\SetFigFont{14}{16.8}{\familydefault}{\mddefault}{\updefault}{\color[rgb]{0,0,0}$x+e_a+o_b+o_c$}%
}}}}
\put(1792,-6030){\rotatebox{300.0}{\makebox(0,0)[lb]{\smash{{\SetFigFont{14}{16.8}{\familydefault}{\mddefault}{\updefault}{\color[rgb]{0,0,0}$z+o_a+e_b+e_c$}%
}}}}}
\put(1212,-5050){\rotatebox{300.0}{\makebox(0,0)[lb]{\smash{{\SetFigFont{14}{16.8}{\familydefault}{\mddefault}{\updefault}{\color[rgb]{0,0,0}$b-a$}%
}}}}}
\put(1430,-4080){\rotatebox{60.0}{\makebox(0,0)[lb]{\smash{{\SetFigFont{14}{16.8}{\familydefault}{\mddefault}{\updefault}{\color[rgb]{0,0,0}$z+e_a+o_b+o_c$}%
}}}}}
\put(9301,-5108){\rotatebox{300.0}{\makebox(0,0)[lb]{\smash{{\SetFigFont{14}{16.8}{\familydefault}{\mddefault}{\updefault}{\color[rgb]{0,0,0}$b-a+3$}%
}}}}}
\put(12321,-551){\makebox(0,0)[lb]{\smash{{\SetFigFont{14}{16.8}{\familydefault}{\mddefault}{\updefault}{\color[rgb]{0,0,0}$x+o_a+e_b+e_c$}%
}}}}
\put(9677,-3173){\rotatebox{60.0}{\makebox(0,0)[lb]{\smash{{\SetFigFont{14}{16.8}{\familydefault}{\mddefault}{\updefault}{\color[rgb]{0,0,0}$z+e_a+o_b+o_c$}%
}}}}}
\put(16553,-1851){\rotatebox{300.0}{\makebox(0,0)[lb]{\smash{{\SetFigFont{14}{16.8}{\familydefault}{\mddefault}{\updefault}{\color[rgb]{0,0,0}$z+e_a+o_b+o_c+b-a+3$}%
}}}}}
\put(17589,-9009){\rotatebox{60.0}{\makebox(0,0)[lb]{\smash{{\SetFigFont{14}{16.8}{\familydefault}{\mddefault}{\updefault}{\color[rgb]{0,0,0}$z+o_a+e_b+e_c$}%
}}}}}
\put(10478,-7069){\rotatebox{300.0}{\makebox(0,0)[lb]{\smash{{\SetFigFont{14}{16.8}{\familydefault}{\mddefault}{\updefault}{\color[rgb]{0,0,0}$z+o_a+e_b+e_c$}%
}}}}}
\put(13636,-9681){\makebox(0,0)[lb]{\smash{{\SetFigFont{14}{16.8}{\familydefault}{\mddefault}{\updefault}{\color[rgb]{0,0,0}$x+e_a+o_b+o_c$}%
}}}}
\end{picture}%
}
\caption{(a) Obtaining  an $\overline{E}^{(1)}$-type region from the region $E^{(1)}_{x,0,z}(\textbf{a};\textbf{c};\textbf{b})$ with $a\leq b$. (b) Obtaining a $\overline{K}^{(6)}$-type region from the region  $K^{(2)}_{x,0,z}(\textbf{a};\textbf{c};\textbf{b})$ with $a\leq b$.}\label{fig:Specialoff3}
\end{figure}

\begin{figure}\centering
\setlength{\unitlength}{3947sp}%
\begingroup\makeatletter\ifx\SetFigFont\undefined%
\gdef\SetFigFont#1#2#3#4#5{%
  \reset@font\fontsize{#1}{#2pt}%
  \fontfamily{#3}\fontseries{#4}\fontshape{#5}%
  \selectfont}%
\fi\endgroup%
\resizebox{15cm}{!}{
\begin{picture}(0,0)%
\includegraphics{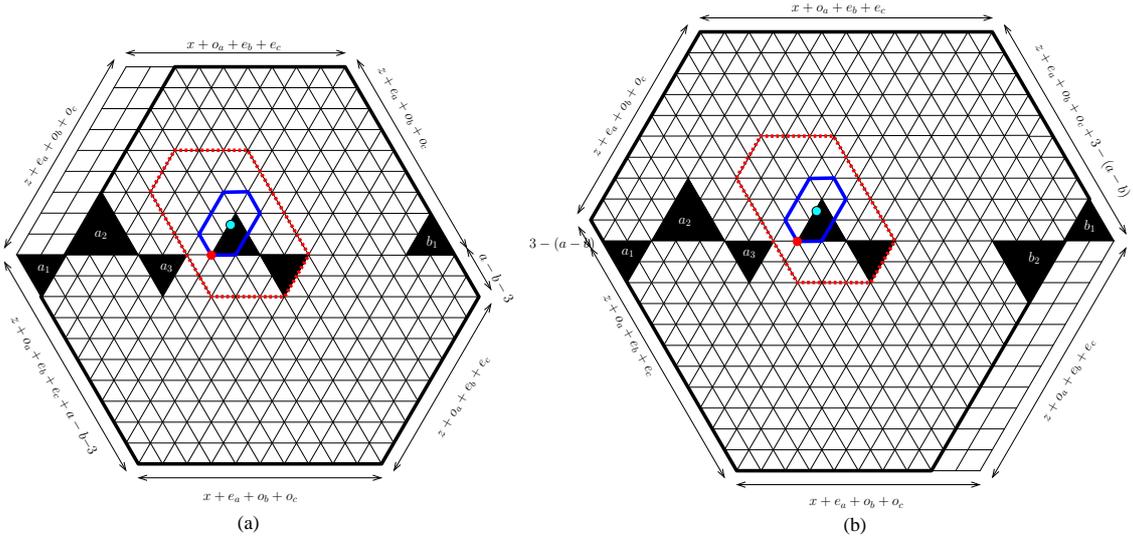}%
\end{picture}%
%
%

\begin{picture}(19117,9233)(3641,-11177)
\put(11018,-9474){\rotatebox{60.0}{\makebox(0,0)[lb]{\smash{{\SetFigFont{14}{16.8}{\familydefault}{\mddefault}{\updefault}{\color[rgb]{0,0,0}$z+o_a+e_b+e_c$}%
}}}}}
\put(6994,-10522){\makebox(0,0)[lb]{\smash{{\SetFigFont{14}{16.8}{\familydefault}{\mddefault}{\updefault}{\color[rgb]{0,0,0}$x+e_a+o_b+o_c$}%
}}}}
\put(11631,-6389){\rotatebox{300.0}{\makebox(0,0)[lb]{\smash{{\SetFigFont{14}{16.8}{\familydefault}{\mddefault}{\updefault}{\color[rgb]{0,0,0}$a-b-3$}%
}}}}}
\put(9929,-3220){\rotatebox{300.0}{\makebox(0,0)[lb]{\smash{{\SetFigFont{14}{16.8}{\familydefault}{\mddefault}{\updefault}{\color[rgb]{0,0,0}$z+e_a+o_b+o_c$}%
}}}}}
\put(16886,-2235){\makebox(0,0)[lb]{\smash{{\SetFigFont{14}{16.8}{\familydefault}{\mddefault}{\updefault}{\color[rgb]{0,0,0}$x+o_a+e_b+e_c$}%
}}}}
\put(20978,-2838){\rotatebox{300.0}{\makebox(0,0)[lb]{\smash{{\SetFigFont{14}{16.8}{\familydefault}{\mddefault}{\updefault}{\color[rgb]{0,0,0}$z+e_a+o_b+o_c+3-(a-b)$}%
}}}}}
\put(21203,-9203){\rotatebox{60.0}{\makebox(0,0)[lb]{\smash{{\SetFigFont{14}{16.8}{\familydefault}{\mddefault}{\updefault}{\color[rgb]{0,0,0}$z+o_a+e_b+e_c$}%
}}}}}
\put(12480,-6248){\makebox(0,0)[lb]{\smash{{\SetFigFont{14}{16.8}{\familydefault}{\mddefault}{\updefault}{\color[rgb]{0,0,0}$3-(a-b)$}%
}}}}
\put(17185,-10618){\makebox(0,0)[lb]{\smash{{\SetFigFont{14}{16.8}{\familydefault}{\mddefault}{\updefault}{\color[rgb]{0,0,0}$x+e_a+o_b+o_c$}%
}}}}
\put(13715,-7078){\rotatebox{300.0}{\makebox(0,0)[lb]{\smash{{\SetFigFont{14}{16.8}{\familydefault}{\mddefault}{\updefault}{\color[rgb]{0,0,0}$z+o_a+e_b+e_c$}%
}}}}}
\put(13633,-4715){\rotatebox{60.0}{\makebox(0,0)[lb]{\smash{{\SetFigFont{14}{16.8}{\familydefault}{\mddefault}{\updefault}{\color[rgb]{0,0,0}$z+e_a+o_b+o_c$}%
}}}}}
\put(6721,-2862){\makebox(0,0)[lb]{\smash{{\SetFigFont{14}{16.8}{\familydefault}{\mddefault}{\updefault}{\color[rgb]{0,0,0}$x+o_a+e_b+e_c$}%
}}}}
\put(4063,-5222){\rotatebox{60.0}{\makebox(0,0)[lb]{\smash{{\SetFigFont{14}{16.8}{\familydefault}{\mddefault}{\updefault}{\color[rgb]{0,0,0}$z+e_a+o_b+o_c$}%
}}}}}
\put(3755,-7452){\rotatebox{300.0}{\makebox(0,0)[lb]{\smash{{\SetFigFont{14}{16.8}{\familydefault}{\mddefault}{\updefault}{\color[rgb]{0,0,0}$z+o_a+e_b+e_c+a-b-$3}%
}}}}}
\put(4211,-6631){\makebox(0,0)[lb]{\smash{{\SetFigFont{14}{16.8}{\familydefault}{\mddefault}{\updefault}{\color[rgb]{1,1,1}$a_1$}%
}}}}
\put(5171,-6061){\makebox(0,0)[lb]{\smash{{\SetFigFont{14}{16.8}{\familydefault}{\mddefault}{\updefault}{\color[rgb]{1,1,1}$a_2$}%
}}}}
\put(6271,-6621){\makebox(0,0)[lb]{\smash{{\SetFigFont{14}{16.8}{\familydefault}{\mddefault}{\updefault}{\color[rgb]{1,1,1}$a_3$}%
}}}}
\put(10751,-6211){\makebox(0,0)[lb]{\smash{{\SetFigFont{14}{16.8}{\familydefault}{\mddefault}{\updefault}{\color[rgb]{1,1,1}$b_1$}%
}}}}
\put(14031,-6371){\makebox(0,0)[lb]{\smash{{\SetFigFont{14}{16.8}{\familydefault}{\mddefault}{\updefault}{\color[rgb]{1,1,1}$a_1$}%
}}}}
\put(14991,-5841){\makebox(0,0)[lb]{\smash{{\SetFigFont{14}{16.8}{\familydefault}{\mddefault}{\updefault}{\color[rgb]{1,1,1}$a_2$}%
}}}}
\put(16071,-6381){\makebox(0,0)[lb]{\smash{{\SetFigFont{14}{16.8}{\familydefault}{\mddefault}{\updefault}{\color[rgb]{1,1,1}$a_3$}%
}}}}
\put(21811,-5941){\makebox(0,0)[lb]{\smash{{\SetFigFont{14}{16.8}{\familydefault}{\mddefault}{\updefault}{\color[rgb]{1,1,1}$b_1$}%
}}}}
\put(20875,-6487){\makebox(0,0)[lb]{\smash{{\SetFigFont{14}{16.8}{\familydefault}{\mddefault}{\updefault}{\color[rgb]{1,1,1}$b_2$}%
}}}}
\end{picture}%
}
\caption{(a) Obtaining an $\overline{K}^{(3)}$-type region from the region $K^{(3)}_{x,-3,z}(\textbf{a};\textbf{c};\textbf{b})$ with $a-b \geq 3$.
(b) Obtaining an (horizontally reflected) $\overline{K}^{(7)}$-type region from the region $K^{(3)}_{x,a-b,z}(\textbf{a};\textbf{c};\textbf{b})$ with $0<a-b\leq 3$. }\label{fig:Specialoff5}
\end{figure}

\medskip

\emph{Case 2b: $R$ is a lower region and $ 0<2d\leq a-b$.}

\medskip

We have $y\geq -2d$. If $y=-2d$, then by investigating forced lozenges
\begin{align}
  \M(E^{(i)}_{x,a-b,z}(\textbf{a}; \textbf{c}; \textbf{b}))&=\M(\overline{E}^{(3+i \mod 6)}_{x,\min(a_1,a-b)-2d,z}(a_2,\dotsc,a_m;\textbf{c};\textbf{b}))\\
  \M(F^{(j)}_{x,-2d,z}(\textbf{a}; \textbf{c}; \textbf{b}))&=\M(\overline{F}^{(4+j \mod 8)}_{x,\min(a_1,a-b)-2d,z}(a_2,\dotsc,a_m;\textbf{c};\textbf{b}))\\
  \M(G^{(k)}_{x,-2d,z}(\textbf{a}; \textbf{c}; \textbf{b}))&=\M(\overline{G}^{(4+k \mod 8)}_{x,\min(a_1,a-b)-2d,z}(a_2,\dotsc,a_m;\textbf{c};\textbf{b}))\\
  \M(K^{(l)}_{x,-2d,z}(\textbf{a}; \textbf{c}; \textbf{b}))&=\M(\overline{K}^{(4+l \mod 8)}_{x,\min(a_1,a-b)-2d,z}(a_2,\dotsc,a_m;\textbf{c};\textbf{b})),
 \end{align}
for $i=5,6$, $j=1,2,8$, $k=1,6,7,8$, and $l=1,2,3,4$ (see Figure \ref{fig:Specialoff5}(a) for an example). Again, we can check that the new regions (the one on the right-hand sides of the above identities) have $h$-parameter strictly less than that of the corresponding region on the left-hand side.

\medskip

\emph{Case 2c: $R$ is a lower region and  $0<a-b\leq 2d$.}

\medskip

We have $y\geq b-a$ (in this case $b-a$ is negative). If $y=b-a$, then
\begin{align}
  \M(E^{(i)}_{x,a-b,z}(\textbf{a}; \textbf{c}; \textbf{b}))&=\M(\overline{E}^{(3+i \mod 6)}_{x,b-a,z}(\textbf{a};\ {}^0\textbf{c};\ b_2,\dotsc,b_n))\\
  \M(F^{(j)}_{x,-2d,z}(\textbf{a}; \textbf{c}; \textbf{b}))&=\M(\overline{F}^{(4+j \mod 8)}_{x,b-a,z}(\textbf{a};\ {}^0\textbf{c};\ b_2,\dotsc,b_n))\\
  \M(G^{(k)}_{x,-2d,z}(\textbf{a}; \textbf{c}; \textbf{b}))&=\M(\overline{G}^{(4+k \mod 8)}_{x,b-a,z}(\textbf{a};\ {}^0\textbf{c};\ b_2,\dotsc,b_n))\\
  \M(K^{(l)}_{x,-2d,z}(\textbf{a}; \textbf{c}; \textbf{b}))&=\M(\overline{K}^{(4+l \mod 8)}_{x,b-a,z}(\textbf{a};\ {}^0\textbf{c};\ b_2,\dotsc,b_n)),
 \end{align}
for $i=5,6$, $j=1,2,8$, $k=1,6,7,8$, and $l=1,2,3,4$ (illustrated in Figure \ref{fig:Specialoff5}(b)).  One readily sees that the regions on the right-hand sides of the above equalities satisfy the conditions of the lemma.
 This finishes our proof.
\end{proof}

\subsection{Combined proof for  all off-central regions}

We prove all $30$ tiling formulas in Section 2 at the same time by induction on $h=p+x+z$. Recall that $p$ denotes the quasi-perimeter of the region, i.e. the perimeter of the base hexagon that our region is obtained by removing three ferns from.

First by Lemma \ref{lem3}, we can assume that our three ferns consists of all triangles of positive side lengths, except for a possible $0$-triangle in front of the middle fern.

The base cases are the cases $x=0$, $z=0$, and $p<6$.

\begin{figure}\centering
\setlength{\unitlength}{3947sp}%
\begingroup\makeatletter\ifx\SetFigFont\undefined%
\gdef\SetFigFont#1#2#3#4#5{%
  \reset@font\fontsize{#1}{#2pt}%
  \fontfamily{#3}\fontseries{#4}\fontshape{#5}%
  \selectfont}%
\fi\endgroup%
\resizebox{8cm}{!}{
\begin{picture}(0,0)%
\includegraphics{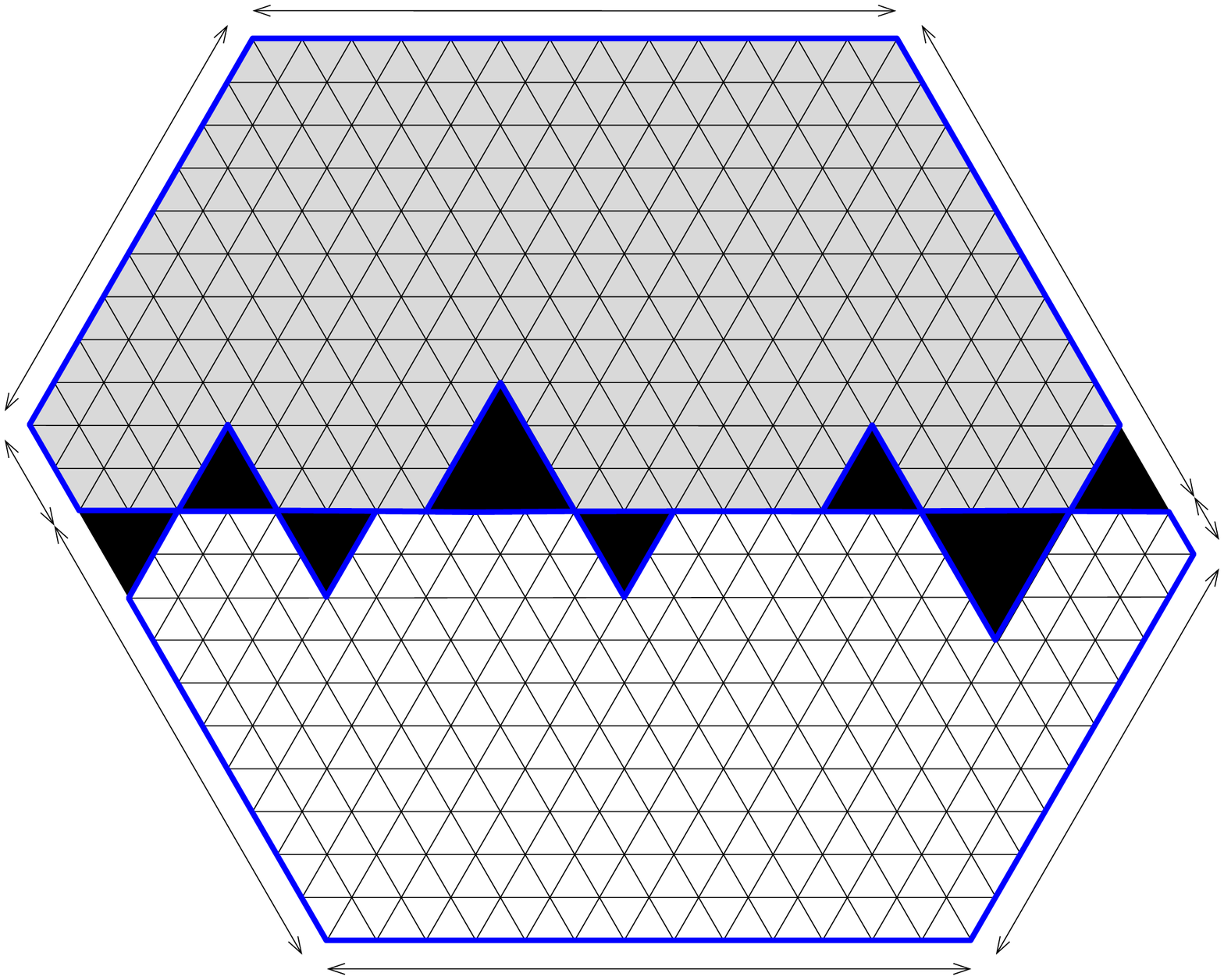}%
\end{picture}%
%
%

\begin{picture}(10594,8721)(1465,-10616)
\put(2151,-4851){\rotatebox{60.0}{\makebox(0,0)[lb]{\smash{{\SetFigFont{14}{16.8}{\rmdefault}{\mddefault}{\updefault}{\color[rgb]{0,0,0}$z+e_a+o_b+o_c$}%
}}}}}
\put(2661,-6731){\makebox(0,0)[lb]{\smash{{\SetFigFont{14}{16.8}{\rmdefault}{\mddefault}{\updefault}{\color[rgb]{1,1,1}$a_1$}%
}}}}
\put(3561,-6331){\makebox(0,0)[lb]{\smash{{\SetFigFont{14}{16.8}{\rmdefault}{\mddefault}{\updefault}{\color[rgb]{1,1,1}$a_2$}%
}}}}
\put(4361,-6731){\makebox(0,0)[lb]{\smash{{\SetFigFont{14}{16.8}{\rmdefault}{\mddefault}{\updefault}{\color[rgb]{1,1,1}$a_3$}%
}}}}
\put(5729,-6248){\makebox(0,0)[lb]{\smash{{\SetFigFont{14}{16.8}{\rmdefault}{\mddefault}{\updefault}{\color[rgb]{1,1,1}$c_1$}%
}}}}
\put(10961,-6321){\makebox(0,0)[lb]{\smash{{\SetFigFont{14}{16.8}{\rmdefault}{\mddefault}{\updefault}{\color[rgb]{1,1,1}$b_1$}%
}}}}
\put(9921,-6891){\makebox(0,0)[lb]{\smash{{\SetFigFont{14}{16.8}{\rmdefault}{\mddefault}{\updefault}{\color[rgb]{1,1,1}$b_2$}%
}}}}
\put(8931,-6331){\makebox(0,0)[lb]{\smash{{\SetFigFont{14}{16.8}{\rmdefault}{\mddefault}{\updefault}{\color[rgb]{1,1,1}$b_3$}%
}}}}
\put(5013,-6720){\makebox(0,0)[lb]{\smash{{\SetFigFont{14}{16.8}{\rmdefault}{\mddefault}{\updefault}{\color[rgb]{0,0,0}$\frac{x}{2}-1$}%
}}}}
\put(7468,-6838){\makebox(0,0)[lb]{\smash{{\SetFigFont{14}{16.8}{\rmdefault}{\mddefault}{\updefault}{\color[rgb]{0,0,0}$\frac{x}{2}+1$}%
}}}}
\put(5601,-2181){\makebox(0,0)[lb]{\smash{{\SetFigFont{14}{16.8}{\rmdefault}{\mddefault}{\updefault}{\color[rgb]{0,0,0}$x+o_a+e_b+e_c$}%
}}}}
\put(10071,-3331){\rotatebox{300.0}{\makebox(0,0)[lb]{\smash{{\SetFigFont{14}{16.8}{\rmdefault}{\mddefault}{\updefault}{\color[rgb]{0,0,0}$y+z+e_a+o_b+o_c$}%
}}}}}
\put(11891,-6471){\makebox(0,0)[lb]{\smash{{\SetFigFont{14}{16.8}{\rmdefault}{\mddefault}{\updefault}{\color[rgb]{0,0,0}$y$}%
}}}}
\put(10741,-9531){\rotatebox{60.0}{\makebox(0,0)[lb]{\smash{{\SetFigFont{14}{16.8}{\rmdefault}{\mddefault}{\updefault}{\color[rgb]{0,0,0}$z+o_a+e_b+e_c$}%
}}}}}
\put(5901,-10601){\makebox(0,0)[lb]{\smash{{\SetFigFont{14}{16.8}{\rmdefault}{\mddefault}{\updefault}{\color[rgb]{0,0,0}$x+e_a+o_b+o_c$}%
}}}}
\put(2371,-7411){\rotatebox{300.0}{\makebox(0,0)[lb]{\smash{{\SetFigFont{14}{16.8}{\rmdefault}{\mddefault}{\updefault}{\color[rgb]{0,0,0}$y+z+o_a+e_b+e_c$}%
}}}}}
\put(1551,-6041){\rotatebox{300.0}{\makebox(0,0)[lb]{\smash{{\SetFigFont{14}{16.8}{\rmdefault}{\mddefault}{\updefault}{\color[rgb]{0,0,0}$y+b-a$}%
}}}}}
\put(6854,-6720){\makebox(0,0)[lb]{\smash{{\SetFigFont{14}{16.8}{\rmdefault}{\mddefault}{\updefault}{\color[rgb]{1,1,1}$c_2$}%
}}}}
\end{picture}}
\caption{Dividing the region $E^{(1)}_{x,y,0}(\textbf{a};\textbf{c};\textbf{b})$ in to two dented semihexagons.}\label{Basecaseoff1}
\end{figure}

\begin{figure}\centering
\setlength{\unitlength}{3947sp}%
\begingroup\makeatletter\ifx\SetFigFont\undefined%
\gdef\SetFigFont#1#2#3#4#5{%
  \reset@font\fontsize{#1}{#2pt}%
  \fontfamily{#3}\fontseries{#4}\fontshape{#5}%
  \selectfont}%
\fi\endgroup%
\resizebox{8cm}{!}{
\begin{picture}(0,0)%
\includegraphics{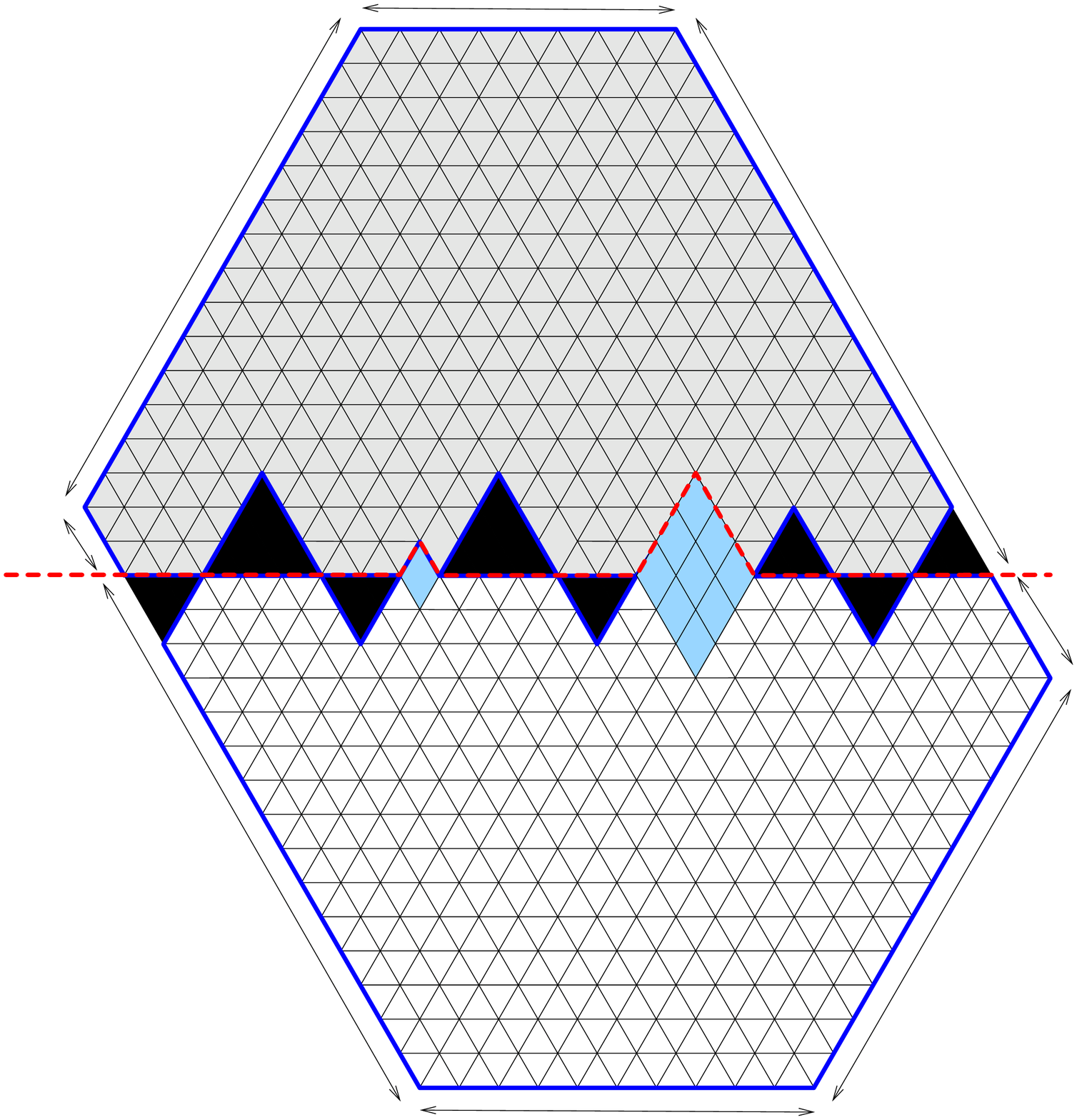}%
\end{picture}%
%
%

\begin{picture}(11521,12166)(980,-15096)
\put(5291,-3211){\makebox(0,0)[lb]{\smash{{\SetFigFont{14}{16.8}{\rmdefault}{\mddefault}{\updefault}{\color[rgb]{0,0,0}$x+o_a+e_b+e_c$}%
}}}}
\put(9371,-4801){\rotatebox{300.0}{\makebox(0,0)[lb]{\smash{{\SetFigFont{14}{16.8}{\rmdefault}{\mddefault}{\updefault}{\color[rgb]{0,0,0}$y+z+e_a+o_b+o_c$}%
}}}}}
\put(11871,-9441){\rotatebox{300.0}{\makebox(0,0)[lb]{\smash{{\SetFigFont{14}{16.8}{\rmdefault}{\mddefault}{\updefault}{\color[rgb]{0,0,0}$y+a-b$}%
}}}}}
\put(10421,-13781){\rotatebox{60.0}{\makebox(0,0)[lb]{\smash{{\SetFigFont{14}{16.8}{\rmdefault}{\mddefault}{\updefault}{\color[rgb]{0,0,0}$z+o_a+e_b+e_c$}%
}}}}}
\put(6281,-15081){\makebox(0,0)[lb]{\smash{{\SetFigFont{14}{16.8}{\rmdefault}{\mddefault}{\updefault}{\color[rgb]{0,0,0}$x+e_a+o_b+o_c$}%
}}}}
\put(2803,-11211){\rotatebox{300.0}{\makebox(0,0)[lb]{\smash{{\SetFigFont{14}{16.8}{\rmdefault}{\mddefault}{\updefault}{\color[rgb]{0,0,0}$y+z+o_a+e_b+e_c$}%
}}}}}
\put(1509,-8971){\makebox(0,0)[lb]{\smash{{\SetFigFont{14}{16.8}{\rmdefault}{\mddefault}{\updefault}{\color[rgb]{0,0,0}$y$}%
}}}}
\put(2501,-6611){\rotatebox{60.0}{\makebox(0,0)[lb]{\smash{{\SetFigFont{14}{16.8}{\rmdefault}{\mddefault}{\updefault}{\color[rgb]{0,0,0}$z+e_a+o_b+o_c$}%
}}}}}
\put(5115,-9933){\makebox(0,0)[lb]{\smash{{\SetFigFont{14}{16.8}{\rmdefault}{\mddefault}{\updefault}{\color[rgb]{0,0,0}$\frac{z}{2}-1$}%
}}}}
\put(2599,-9437){\makebox(0,0)[lb]{\smash{{\SetFigFont{14}{16.8}{\rmdefault}{\mddefault}{\updefault}{\color[rgb]{1,1,1}$a_1$}%
}}}}
\put(3544,-8846){\makebox(0,0)[lb]{\smash{{\SetFigFont{14}{16.8}{\rmdefault}{\mddefault}{\updefault}{\color[rgb]{1,1,1}$a_2$}%
}}}}
\put(4607,-9437){\makebox(0,0)[lb]{\smash{{\SetFigFont{14}{16.8}{\rmdefault}{\mddefault}{\updefault}{\color[rgb]{1,1,1}$a_3$}%
}}}}
\put(6036,-8964){\makebox(0,0)[lb]{\smash{{\SetFigFont{14}{16.8}{\rmdefault}{\mddefault}{\updefault}{\color[rgb]{1,1,1}$c_1$}%
}}}}
\put(7059,-9555){\makebox(0,0)[lb]{\smash{{\SetFigFont{14}{16.8}{\rmdefault}{\mddefault}{\updefault}{\color[rgb]{1,1,1}$c_2$}%
}}}}
\put(9095,-9082){\makebox(0,0)[lb]{\smash{{\SetFigFont{14}{16.8}{\rmdefault}{\mddefault}{\updefault}{\color[rgb]{1,1,1}$b_3$}%
}}}}
\put(9922,-9555){\makebox(0,0)[lb]{\smash{{\SetFigFont{14}{16.8}{\rmdefault}{\mddefault}{\updefault}{\color[rgb]{1,1,1}$b_2$}%
}}}}
\put(10639,-9086){\makebox(0,0)[lb]{\smash{{\SetFigFont{14}{16.8}{\rmdefault}{\mddefault}{\updefault}{\color[rgb]{1,1,1}$b_1$}%
}}}}
\put(7775,-9318){\makebox(0,0)[lb]{\smash{{\SetFigFont{14}{16.8}{\rmdefault}{\mddefault}{\updefault}{\color[rgb]{0,0,0}$\frac{z}{2}+1$}%
}}}}
\end{picture}}
\caption{Dividing the region $E^{(1)}_{0,y,z}(\textbf{a};\textbf{c};\textbf{b})$ in to two dented semihexagons.}\label{Basecaseoff2}
\end{figure}

If $z=0$,  we divide our region into two dented semihexagons with dents by cutting along the horizontal lattice line containing our three ferns. This semihexagons corresponds to the two
$s$-terms in our tiling formulas (shown in Figure \ref{Basecaseoff1}). By the Region-splitting Lemma \ref{RS}, the tiling number of our region is equal to the product of tiling numbers of the two dented semihexagons. Then our theorems follow from Cohn--Larsen--Propp's formula (\ref{semieq}).

If $x=0$, we also partition the region into two parts along the same lattice line above, the only difference is that we now add two `bump' of the same size as the gaps between the three ferns. The upper part is a dented semihexagon, and the lower part, after removing several vertical forced lozenges at the two bumps, is congruent with another dented semihexagon (illustrated in Figure \ref{Basecaseoff2}). Again, the theorems follows from Cohn--Larsen--Propp formula (\ref{semieq}) and the Region-splitting Lemma \ref{RS}.

If $p< 6$, by Lemma \ref{claimp}, we always have $p\geq 2x+4z$. This implies that $2x+4z<6$, and at least one of $x$ and $z$ is $0$. This base case is then reduced to the above base cases.

For the induction step, we assume that $x$ and $z$ are positive, that $p\geq 6$, and that all our 30  tiling formulas hold for off-central regions whose $h$-parameter strictly less than $h=p+x+z$.

If $y$ does not achieve its minimal value, by the $66$ recurrences in Sections 3.3--3.10, the tiling number of each of our off-central regions can be written in terms of tiling numbers of other regions, that are either the regions in the central case treated in \cite[Theorems 2.2--2.9]{HoleDent} or other off-central regions with $h$-parameter strictly less than $h$. This means that by Theorem 2.2--2.9 in \cite{HoleDent} and by the induction hypothesis, we get an explicit formula for the tiling number of each of our regions. Then, by performing a straightforward simplification, one can verify that the latter  explicit tiling formula agrees with that in our theorem.

If $y$ achieves its minimal value, then our recurrences dot not work well anymore. However, by Lemma \ref{lem4}, each of our regions has the same tiling number as that of a off-central region whose $h$-parameter strictly less than. Thus, our theorems follow from the induction hypothesis. This finishes our proof.



\section*{Acknowledgement}

The author would like to thank Dennis Stanton and Hjalmar Rosengren for pointing out to him paper \cite{Rosen}.

\end{document}